\documentclass[reqno,english]{amsart}%
\usepackage{amsfonts,amsmath,latexsym,verbatim,amscd,mathrsfs,color,array}
\usepackage[colorlinks=true]{hyperref}
\usepackage{amsmath,amssymb,amsthm,amsfonts,graphicx,color}
\usepackage{amssymb}
\usepackage{pdfsync}
\usepackage{epstopdf}
\usepackage{cite}
\usepackage{graphicx}
\usepackage{amsmath}
\usepackage{amsfonts}%
\providecommand{\U}[1]{\protect\rule{.1in}{.1in}}
\numberwithin{equation}{section}

\newtheorem{theorem}{Theorem}[section]
\newtheorem{corollary}{Corollary}[section]
\newtheorem{lemma}{Lemma}[section]
\newtheorem{proposition}{Proposition}[section]
\newtheorem{remark}{Remark}[section]
\newtheorem{definition}{Definition}[section]

\numberwithin{equation}{section}

\newcommand{\bbr}{\mathbb{R}}

\newcommand{\bbn}{\mathbb{N}}
\newcommand{\ve}{\varepsilon}
\newcommand{\bd}{\begin{definition}}
\newcommand{\ed}{\end{definition}}

\newcommand{\br}{\begin{remark}}
\newcommand{\er}{\end{remark}}

\newcommand{\be}{\begin{equation}}
\newcommand{\ee}{\end{equation}}

\newcommand{\bc}{\begin{corollary}}
\newcommand{\ec}{\end{corollary}}
\begin{document}

\title[Caffarelli-Kohn-Nirenberg inequality]{Stability of  the Caffarelli-Kohn-Nirenberg inequality along Felli-Schneider curve: critical points at infinity}

\author[J. Wei]{Juncheng Wei}
\address{\noindent Department of Mathematics, Chinese University of Hong Kong,
Shatin, NT, Hong Kong}
\email{wei@math.cuhk.edu.hk}

\author[Y.Wu]{Yuanze Wu}
\address{\noindent  School of Mathematics, China
University of Mining and Technology, Xuzhou, 221116, P.R. China }
\email{wuyz850306@cumt.edu.cn}

\begin{abstract}
In this paper, we consider the following Caffarelli-Kohn-Nirenberg (CKN for short) inequality
\begin{eqnarray*}
\bigg(\int_{\bbr^d}|x|^{-b(p+1)}|u|^{p+1}dx\bigg)^{\frac{2}{p+1}}\leq S_{a,b}\int_{\bbr^d}|x|^{-2a}|\nabla u|^2dx,
\end{eqnarray*}
where $u\in D^{1,2}_{a}(\bbr^d)$, $d\geq2$, $p=\frac{d+2(1+a-b)}{d-2(1+a-b)}$ and
\begin{eqnarray}\label{eq0003}
\left\{\aligned
&a<b<a+1,\quad d=2,\\
&a\leq b<a+1,\quad d\geq3.
\endaligned
\right.
\end{eqnarray}
Based on the ideas of \cite{DSW2024,FP2024}, we develop a suitable strategy to derive the following sharp stability of the critical points at infinity of the above CKN inequality in the degenerate case $d\geq2$, $b=b_{FS}(a)$ (Felli-Schneider curve)  and $a<0$:  let $\nu \in {\mathbb N}$ and $u\in D^{1,2}_{a}(\bbr^d)$ be an nonnegative function such that
\begin{eqnarray}\label{eqqqnew0001}
\left(\nu-\frac12\right)\left(S_{a,b}^{-1}\right)^{\frac{p+1}{p-1}}<\|u\|^2_{D^{1,2}_a(\bbr^d)}<\left(\nu+\frac12\right)\left(S_{a,b}^{-1}\right)^{\frac{p+1}{p-1}}
\end{eqnarray}
Then
\begin{eqnarray*}
\inf_{\overrightarrow{\alpha}_{\nu}\in\left(\bbr_+\right)^{\nu}, \overrightarrow{\lambda}_{\nu}\in\bbr^\nu}\left\|u-\sum_{j=1}^{\nu}\alpha_jW_{\lambda_j}\right\|\lesssim\left\|-div(|x|^{-a}\nabla u)-|x|^{-b(p+1)}|u|^{p-1}u\right\|_{W^{-1,2}_a(\bbr^d)}^{\frac{1}{3}}
\end{eqnarray*}
as $\left\|-div(|x|^{-a}\nabla u)-|x|^{-b(p+1)}|u|^{p-1}u\right\|_{W^{-1,2}_a(\bbr^d)}\to0$,
where $b_{FS}(a)$ is the well known Felli-Schneider curve.  The above stability is sharp in the sense that the order of the right hand side can not be improved.  The significant finding in our result is that in the degenerate case, {\it the power of the optimal stability is an absolute constant $1/3$} (independent of $p$ and $\nu$) which is quite different from the non-degenerate case \cite{DSW2024,WW2022}.  We also believe that our strategy of proofs will  be very useful in studying many other degenerate problems.

\vspace{3mm} \noindent{\bf Keywords:} Caffarelli-Kohn-Nirenberg inequality; Sharp stability; the Felli-Schneider curve; Critical points at infinity.

\vspace{3mm}\noindent {\bf AMS} Subject Classification 2010: 35B09; 35B33; 35B40; 35J20.%

\end{abstract}

\date{}
\maketitle

\section{Introduction}
\subsection{Background and Previous Results.}
Let $d\geq2$ be a positive integer and $D^{1,2}_{a}(\bbr^d)$ be the Hilbert space given by
\begin{eqnarray}\label{eqn886}
D^{1,2}_{a}(\bbr^d)=\left\{u\in D^{1,2}(\bbr^d)\mid \int_{\bbr^d}|x|^{-2a}|\nabla u|^2dx<+\infty\right\}
\end{eqnarray}
with the inner product
\begin{eqnarray*}
\langle u,v \rangle_{D^{1,2}_{a}(\bbr^d)}=\int_{\bbr^d}|x|^{-2a}\nabla u\nabla vdx
\end{eqnarray*}
and the induced norm $\|\cdot\|_{D^{1,2}_a(\bbr^d)}=\langle \cdot,\cdot \rangle_{D^{1,2}_{a}(\bbr^d)}^{\frac12}$, where $D^{1,2}(\bbr^d)=\dot{W}^{1,2}(\bbr^d)$ is the usual homogeneous Sobolev space (cf. \cite[Definition~2.1]{FG2021}) with $D^{-1,2}(\bbr^d)$ being the dual space.  Then the following Caffarelli-Kohn-Nirenberg (CKN for short in what follows) inequality
\begin{eqnarray}\label{eq0001}
\bigg(\int_{\bbr^d}|x|^{-b(p+1)}|u|^{p+1}dx\bigg)^{\frac{2}{p+1}}\leq S_{a,b}\int_{\bbr^d}|x|^{-2a}|\nabla u|^2dx,
\end{eqnarray}
which is established by Caffarelli, Kohn and Nirenberg in the celebrated paper \cite{CKN1984} in a more general version,
holds for all $u\in D^{1,2}_{a}(\bbr^d)$,
where $d\geq2$, $p=\frac{d+2(1+a-b)}{d-2(1+a-b)}$ and
\begin{eqnarray}\label{eq0003}
\left\{\aligned
&a<b<a+1,\quad d=2,\\
&a\leq b<a+1,\quad d\geq3.
\endaligned
\right.
\end{eqnarray}
Here, for the sake of simplicity, we denote $a_c=\frac{d-2}{2}$, as in \cite{DELT2009,DEL2012,DEL2016}.

\vskip0.12in

As pointed out by Catrina and Wang in \cite{CW2001}, a fundamental task in understanding  a functional inequality is to study the best constants, existence (and nonexistence) of extremal functions, as well as their qualitative properties and classifications, which have played important roles in many applications by virtue of the complete knowledge on the minimizers.  For the CKN inequality~\eqref{eq0001}, it is known that up to dilations $u_{\tau}(x)=\tau^{a_c-a}u(\tau x)$ and scalar multiplications $Cu(x)$ (also up to translations $u(x+y)$ for the spacial case $a=b=0$), the radial function $W(x)$ given by
\begin{eqnarray}\label{eq0004}
W(x)=\left(2(p+1)(a_c-a)^2\right)^{\frac{1}{(p-1)}}\bigg(1+|x|^{(a_c-a)(p-1)}\bigg)^{-\frac{2}{p-1}}
\end{eqnarray}
is the unique nonnegative solution of \eqref{eq0018} in $D^{1,2}_a(\bbr^d)$ for $d\geq2$ under the conditions
\begin{eqnarray}\label{eq0003}
\left\{\aligned
&b_{FS}(a)\leq b<a+1,\quad a<0,\\
&a\leq b<a+1,\quad a\geq0,
\endaligned
\right.
\end{eqnarray}
where
\begin{eqnarray*}
b_{FS}(a)=\frac{d(a_c-a)}{2\sqrt{(a_c-a)^2+(d-1)}}+a-a_c>a
\end{eqnarray*}
is the well known Felli-Schneider curve found in \cite{FS2003}.  Precisely, Aubin and Talanti classified the extremal functions of the CKN inequality~\eqref{eq0001} for $a=b=0$ in \cite{A1976,T1976}, respectively.  As a special case, Lieb classified the extremal functions of the CKN inequality~\eqref{eq0001} for $a=0$ and $0<b<1$ in \cite{L1983}.  Chou and Chu classified the extremal functions of the CKN inequality~\eqref{eq0001} for $a\geq0$ in \cite{CC1993}.  Felli and Schneider proved in \cite{FS2003} that extremal functions of the CKN inequality~\eqref{eq0001} must be nonradial if $a<0$ and $a<b<b_{FS}(a)$.  Dolbeault, Esteban, Loss and Tarantello finally classified the extremal functions of the CKN inequality~\eqref{eq0001} in \cite{DEL2016,DET2008} under the conditions~\eqref{eq0003}.  Catrina and Wang proved in \cite{LW2004} that extremal functions of \eqref{eq0001} must have $\mathcal{O}(N-1)$ symmetry for $a<b<b_{FS}(a)$ with $a<0$.
Moreover, it is also well known that $W(x)$ is nondegenerate in $D^{1,2}_a(\bbr^d)$ under the condition~\eqref{eq0003} except $b=b_{FS}(a)$ (cf. \cite{FS2003}).  That is, up to scalar multiplications $CV(x)$,
\begin{eqnarray}\label{eq0010}
V(x):=\nabla W(x)\cdot x-(a_c-a)W(x)=\frac{\partial}{\partial\lambda}\left(\lambda^{-(a_c-a)}W(\lambda x)\right)|_{\lambda=1}
\end{eqnarray}
is the only nonzero solution in $D^{1,2}_a(\bbr^d)$ to the linearization of \eqref{eq0018} around $W$ which is given by
\begin{eqnarray}\label{eq0017}
-div(|x|^{-a}\nabla u)=p|x|^{-b(p+1)}W^{p-1}u, \quad u\in D^{1,2}_{a}(\bbr^d).
\end{eqnarray}
However, if the parameters $a$ and $b$ lie on the Felli-Schneider curve, that is, $b=b_{FS}(a)$ with $a<0$, then the bubble $W(x)$ is degenerate in $D^{1,2}_a(\bbr^d)$ (cf. \cite{FP2024}).  For the sake of simplicity, we introduce the set
\begin{eqnarray*}
\mathcal{Z}=\left\{cW_\tau(x)\mid c\in\bbr\backslash\{0\}\text{ and }\tau>0\right\}
\end{eqnarray*}
and the usual weighted Lebesgue space $L^{p+1}(|x|^{-b(p+1)},\bbr^d)$ with the norm
\begin{eqnarray*}
\|u\|_{L^{p+1}(|x|^{-b(p+1)},\bbr^d)}=\bigg(\int_{\bbr^d}|x|^{-b(p+1)}|u|^{p+1}dx\bigg)^{\frac{1}{p+1}}.
\end{eqnarray*}
in what follows.

\vskip0.12in

As pointed out by Dolbeault and Esteban in \cite{DE2022} (see also Figalli in \cite{F2013}), once optimal constants are known and the set of extremal functions has been characterised, the next question is to understand stability: which kind of distance is measured by the deficit, that is, the difference of the two terms in the functional inequality, written with the optimal constant.  These studies were initialed by Brezis and Lieb in \cite{BL1985} by raising an open question for the classical Sobolev inequality,
\begin{eqnarray}\label{eqin0001}
S\bigg(\int_{\bbr^d}|u|^{\frac{2d}{d-2}}dx\bigg)^{\frac{d-2}{d}}\leq\int_{\bbr^d}|\nabla u|^2dx,\quad u\in D^{1,2}(\bbr^d),
\end{eqnarray}
which was settled partially by Egnell-Pacella-Tricarico in \cite{EPT1989} and completely by Bianchi-Egnell in \cite{BE1991} by proving that
\begin{eqnarray}\label{eqin0002}
0<s_{BE}=\inf_{u\in D^{1,2}(\bbr^d)\backslash\{0\}}\frac{\|\nabla u\|^2_{L^2(\bbr^d)}-S\|u\|^2_{L^{\frac{2d}{d-2}}(\bbr^d)}}{dist_{D^{1,2}}^2(u, \mathcal{M})},
\end{eqnarray}
where $\|\cdot\|_{L^p(\bbr^d)}$ is the usual norm in the Lebesgue space $L^p(\bbr^d)$ and
\begin{eqnarray*}
\mathcal{M}=\{U_{\lambda,z,c}\mid (\lambda,z,c)\in\bbr_+\times\bbr^{d+1}\}.
\end{eqnarray*}
Due to the non-Hilbert property of $W^{1,p}(\bbr^d)$ for $p\not=2$, the generalization of the Bianchi-Egnell stability~\eqref{eqin0002} to the general $L^p$-Sobolev inequality takes a long time to introduce new ideas and develop new techniques by Cianchi in \cite{C2006}, Cianchi-Fusco-Maggi-Pratelli in \cite{CFMP2009}, Figalli-Magggi-Pratelli in \cite{FMP2013}, Figalli-Neumayer in \cite{FN2019}, Fusco in \cite{F2015}, Fusco-Maggi-Pratelli in \cite{FMP2007}, Neumayer in \cite{N2020} and finally, Figalli and Zhang proved the optimal Bianchi-Egnell stability of the general $L^p$-Sobolev inequality in \cite{FZ2020}.  The Bianchi-Egnell type stability like \eqref{eqin0002} was also generalized to many other famous inequalities.  Since the literature on this topic is so vast and this direction is not the main topic in our paper, we only refer the readers to \cite{C2017,CFW2013,D2011,DJ2014} for the Hardy-Littlewood-Sobolev inequality, \cite{DT2016,LW2000,N2019,R2014,S2016} for the Gagliardo-Nirenberg-Sobolev inequality,  \cite{BGRS2014,CLT2023,E2015,FI2016,FIPR2017,IM2014,IK2021,W78} for the logarithmic Sobolev inequality, \cite{ADM2022,CFLL2024,DGK2024,FP2024,WW2003,WW2022} for the Caffarelli-Kohn-Nirenberg inequality, \cite{BWW2003,BGKM2022,CK2024,F2022,L2005} for other kinds of Sobolev inequalities and \cite{CFL2014,DSW2023,ENS2022,FI2013,FJ2015,FJ2017,FMP2010,FMP2009,F1989,J1992,M2008,NV2024} for many kinds of geometric inequalities.  We would like to highlight the survey \cite{DE2022} and the Lecture notes \cite{F2013,F2023} to the readers for their detailed introductions and references about the studies on the stability of inequalities.
In particular, the Bianchi-Egnell type stability of the CKN inequality~\eqref{eq0001} reads as follows:
\begin{enumerate}
\item[$(1)$] The nondegenerate case (\cite{WW2022,WW2024}).\quad  Let $d\geq2$ and either
\begin{enumerate}
\item[$(i)$]\quad $b_{FS}(a)< b<a+1$ with $a<0$ or
\item[$(ii)$]\quad $a\leq b<a+1$ with $a\geq0$ and $a+b>0$.
\end{enumerate}
Then
\begin{eqnarray*}
0<c_{BE}=\inf_{u\in D^{1,2}_a(\bbr^d)\backslash\{0\}}\frac{\|u\|^2_{D^{1,2}_a(\bbr^d)}-S_{a,b}^{-1}\|u\|^2_{L^{p+1}(|x|^{-b(p+1)},\bbr^d)}}{dist_{D^{1,2}_{a}}^2(u, \mathcal{Z})}
\end{eqnarray*}
for all $u\in D^{1,2}_a(\bbr^d)$.
\item[$(2)$]The degenerate case (\cite{FP2024}). \quad Let $d\geq2$ and $b=b_{FS}(a)$ with $a<0$.  Then
\begin{eqnarray*}
0<c_{BE}=\inf_{u\in D^{1,2}_a(\bbr^N)\backslash\mathcal{Z}}\frac{\|u\|^2_{D^{1,2}_a(\bbr^d)}\left(\|u\|^2_{D^{1,2}_a(\bbr^d)}-S_{a,b}^{-1}\|u\|^2_{L^{p+1}(|x|^{-b(p+1)},\bbr^d)}\right)}{dist^4(u, \mathcal{Z})}
\end{eqnarray*}
for all $u\in D^{1,2}_a(\bbr^d)$.
\end{enumerate}
We remark that Bianchi and Egnell's arguments for \eqref{eqin0002} depends on the nondegeneracy of the Talanti bubble $U$ in $D^{1,2}(\bbr^d)$.  Thus, to establish the Bianchi-Egnell type stability of the CKN inequality~\eqref{eq0001} in the degenerate case, Frank and Peteranderl \cite{FP2024} introduced new idea and developed new techniques to expand the deficit of the CKN inequality~\eqref{eq0001} up to the fourth order terms in \cite{FP2024}, as that in \cite{F2022}.  We would like to mention the paper \cite{CF2013} where Carlen and Figalli proved a quantitative convergence result for the critical mass Keller-Segel system by the Bianchi-Egnell type stability of Gagliardo-Nirenberg-Sobolev inequality  and the logarithmic Hardy-Littlewood-Sobolev inequality, which provides the potential applications of the studies on the stability of many other inequalities.  We also want to mention the paper \cite{C2017}, where Carlen develop a dual method to establish stability of functional inequalities.  Finally  we want to mention is that in the very recent papers \cite{K2022,K2023,K2023-1}, Konig proved that $s_{BE}$ is attainable which gives a positive answer to the open question proposed by Dolbeault, Esteban, Figalli, Frank and Loss in \cite{DEFFL2022} and makes the key step in answering the long-standing open question of determining the best constant $s_{BE}$.  Konig's result on $s_{BE}$ has been generalized to $c_{BE}$ in the nondegenerate case in our very recent paper \cite{WW2024}.

\vskip0.12in

On the other hand, it is well known that all critical points at infinity
of the corresponding functional of the Sobolev inequality~\eqref{eqin0001} are induced by limits of sums of Talenti bubbles (at least if we
consider only nonnegative functions) which can be precisely stated as follows.
\begin{theorem}\label{thm0003}
(Struwe \cite{S1984})\quad Let $d\geq3$ and $\nu\geq1$ be positive integers.  Let $\{u_n\}\subset D^{1,2}(\bbr^d)$ be a nonnegative sequence with
\begin{eqnarray*}
(\nu-\frac12)S^\frac{d}{2}<\|u_n\|_{D^{1,2}(\bbr^d)}^2<(\nu+\frac12)S^\frac{d}{2},
\end{eqnarray*}
where $S$ is the best Sobolev constant.
Assume that $\|\Delta u_n+|u_n|^{\frac{4}{d-2}}u_n\|_{D^{-1,2}}\to0$ as $n\to\infty$, then there exist a sequence $(z_1^{(n)}, z_2^{(n)}, \cdots, z_\nu^{(n)})$ of $\nu$-tuples of points in $\bbr^d$ and a sequence of $(\lambda_1^{(n)}, \lambda_2^{(n)}, \cdots, \lambda_\nu^{(n)})$ of $\nu$-tuples of positive real numbers such that
\begin{eqnarray*}
\|\nabla u_n-\sum_{i=1}^{\nu}\nabla U[z_i^{(n)}, \lambda_i^{(n)}]\|_{L^2(\bbr^d)}\to0\quad\text{as }n\to\infty.
\end{eqnarray*}
\end{theorem}
Based on the above well-known Struwe decomposition, Ciraolo, Figalli, Glaudo and Maggi proposed the following question on the stability of critical points of the corresponding functional of the Sobolev inequality~\eqref{eqin0001} at infinity:
\begin{enumerate}
\item[$(Q)$]\quad Let $d\geq3$ and $\nu\geq1$ be positive integers.  If $\{u\}\subset D^{1,2}(\bbr^d)$ be nonnegative,
\begin{eqnarray*}
(\nu-\frac12)S^\frac{d}{2}<\|u\|_{D^{1,2}(\bbr^d)}^2<(\nu+\frac12)S^\frac{d}{2}
\end{eqnarray*}
and $\|\Delta u+|u|^{\frac{4}{d-2}}u\|_{D^{-1,2}}<<1$, does there exist a constant $C(d,\nu)$ such that
    \begin{eqnarray*}
    dist_{D^{1,2}}(u, \mathcal{M})\leq C(d,\nu)\|\Delta u+|u|^{\frac{4}{N-2}}u\|_{D^{-1,2}}?
    \end{eqnarray*}
\end{enumerate}
\begin{remark}
The original question (\cite[Problem~1.2]{FG2021}) is more general than $(Q)$ stated here in the sense that, $u$ could be sign-changing if $u$ is close to the sum of $U[z_i, \lambda_i]$ in $D^{1,2}(\bbr^d)$ where $U[z_i, \lambda_i]$ are weakly interacting (the definition of weakly interaction can be found in \cite[Definition~3.1]{FG2021}).  We choose to state the question~$(Q)$ since it is more close to Theorem~\ref{thm0003} (Struwe \cite{S1984}).
\end{remark}
In the recent papers \cite{CFM2017,FG2021}, Ciraolo, Figalli, Glaudo and Maggi proved the following results by the energy method:
\begin{enumerate}
\item[$(1)$] (Ciraolo-Figalli-Maggi \cite{CFM2017})\quad Let $d\geq3$ and $u\in D^{1,2}(\bbr^d)$ be positive such that $\|\nabla u\|_{L^2(\bbr^d)}^2\leq\frac{3}{2}S^{\frac{d}{2}}$ and $\|\Delta u+|u|^{\frac{4}{d-2}}u\|_{D^{-1,2}}\leq\delta$ for some $\delta>0$ sufficiently small, then $dist_{D^{1,2}}(u, \mathcal{M})\lesssim\|\Delta u+|u|^{\frac{4}{d-2}}u\|_{d^{-1,2}}$.
\item[$(2)$] (Figalli-Glaudo \cite{FG2021})\quad Let $u\in D^{1,2}(\bbr^d)$ be nonnegative such that
\begin{eqnarray*}
(\nu-\frac12)S^\frac{d}{2}<\|u\|_{D^{1,2}(\bbr^d)}^2<(\nu+\frac12)S^\frac{d}{2}
\end{eqnarray*}
and $\|\Delta u+|u|^{\frac{4}{d-2}}u\|_{D^{-1,2}}\leq\delta$ for some $\delta>0$ sufficiently small, then $dist_{D^{1,2}}(u, \mathcal{M}^\nu)\lesssim\|\Delta u+|u|^{\frac{4}{d-2}}u\|_{D^{-1,2}}$ for $3\leq N\leq5$, where
$\mathcal{M}^\nu=\{(U[z_1, \lambda_1],U[z_2, \lambda_2],\cdots,U[z_\nu, \lambda_\nu])\mid z_i\in\bbr^N, \lambda_i>0\}$.
\end{enumerate}
It is worth pointing out that a significant finding in \cite{FG2021} is that Figalli and Glaudo construct a counterexample for $\nu=2$ and $d\geq6$ to show that the answer of the question~$(Q)$ for $\nu\geq2$ and $d\geq6$ is negative!   Based on their counterexample for $d\geq6$, Figalli and Glaudo conjectured in \cite{FG2021} that the stability of critical points of the corresponding functional of the Sobolev inequality~\eqref{eqin0001} at infinity should be of the following nonlinear form:
\begin{eqnarray*}
dist_{D^{1,2}}(u, \mathcal{M}^\nu)\lesssim\left\{\aligned&\|\Delta u+|u|u\|_{H^{-1}}|\ln(\|\Delta u+|u|u\|_{D^{-1,2}})|,\quad \nu\geq2\text{ and }d=6;\\
&\|\Delta u+|u|^{\frac{4}{N-2}}u\|_{D^{-1,2}}^{\gamma(d)},\quad \nu\geq2\text{ and }d\geq7
\endaligned\right.
\end{eqnarray*}
with $0<\gamma(d)<1$.  We would like to point out that, besides its own mathematical interests, the stability of critical points of the corresponding functional of the Sobolev inequality~\eqref{eqin0001} at infinity can be used to prove quantitative convergence results for the fast diffusion equation, see, for example \cite{CFM2017,FG2021}.  In the very recent work \cite{DSW2024}, the first author, together with Deng and Sun, proved that tthe stability of critical points of the corresponding functional of the Sobolev inequality~\eqref{eqin0001} at infinity is actually of the following nonlinear form by combining the energy method, the reduction argument and the blow-up analysis:
\begin{eqnarray*}
dist_{D^{1,2}}(u, \mathcal{M}^\nu)\lesssim\left\{\aligned&\|\Delta u+|u|u\|_{D^{-1,2}}|\ln(\|\Delta u+|u|u\|_{D^{-1,2}})|^{\frac12},\ \ \nu\geq2\text{ and }d=6;\\
&\|\Delta u+|u|^{\frac{4}{d-2}}u\|_{H^{-1}}^{\frac{d+2}{2(d-2)}},\ \ \nu\geq2\text{ and }d\geq7.
\endaligned\right.
\end{eqnarray*}
Moreover, the orders of the right hand sides in above estimates are shown to be optimal by constructing related examples.  We remark that due to the mathematical interests and potential applications, the stability of critical points of the corresponding functional of other famous functional inequalities at infinity have already been carried out, see, for example by Aryan in \cite{A2023} and De Nitti and Konig in \cite{NK2023} for the fractional Sobolev inequality, and by us in \cite{WW2022} for the CKN inequality~\eqref{eq0001} in the nondegenerate case.  In particular, the stability of critical points of the corresponding functional of the CKN inequality~\eqref{eq0001} at infinity is stated as follows.
\begin{theorem}\label{thm0002}
Let $d\geq2$ and $\nu\geq1$ be positive integers and either
\begin{enumerate}
\item[$(i)$]\quad $b_{FS}(a)< b<a+1$ with $a<0$ or
\item[$(ii)$]\quad $a\leq b<a+1$ with $a\geq0$ and $a+b>0$.
\end{enumerate}
Then for any nonnegative $u\in D^{1,2}_a(\bbr^d)$ such that
\begin{eqnarray*}
(\nu-\frac12)(S_{a,b}^{-1})^{\frac{p+1}{p-1}}<\|u\|_{D^{1,2}_a(\bbr^d)}^2<(\nu+\frac12)(S_{a,b}^{-1})^{\frac{p+1}{p-1}}
\end{eqnarray*}
and $\Gamma(u)\leq\delta$ with some $\delta>0$ sufficiently small,
we have
\begin{eqnarray*}
dist_{D_a^{1,2}}(u, \mathcal{Z}^\nu)\lesssim\left\{\aligned &\Gamma(u),\quad p>2\text{ or }\nu=1,\\
&\Gamma(u)|\log\Gamma(u)|^{\frac12},\quad p=2\text{ and }\nu\geq2,\\
&\Gamma(u)^{\frac{p}{2}},\quad 1<p<2\text{ and }\nu\geq2,
\endaligned\right.
\end{eqnarray*}
where $\Gamma (u)=\|div(|x|^{-a}\nabla u)+|x|^{-b(p+1)}|u|^{p-1}u\|_{D_a^{-1,2}}$
and
$$
\mathcal{Z}^\nu=\{(W_{\tau_1},W_{\tau_2},\cdots,W_{\tau_\nu})\mid \tau_i>0\}.
$$
Moreover, the orders of the right hand sides in above estimates are sharp.
\end{theorem}
We remark that Theorem~\ref{thm0002} is a direct generalization of the Ciraolo-Figalli-Maggi,
Figalli-Glaudo and Deng-Sun-Wei results in \cite{CFM2017,FG2021,DSW2024} for the Sobolev inequality~\ref{eqin0001} to the CKN inequality~\eqref{eq0001} in the nondegenerate case, which was mainly based on the following Struwe decomposition of critical points of the corresponding functional of the CKN inequality~\eqref{eq0001} at infinity.
\begin{proposition}\label{prop0002}
(\cite[Proposition~3.2]{WW2022} or \cite[Lemma~4.2]{CW2001})\quad
Let $d\geq2$ and $\nu\geq1$ be positive integers and either
\begin{enumerate}
\item[$(i)$]\quad $b_{FS}(a)< b<a+1$ with $a<0$ or
\item[$(ii)$]\quad $a\leq b<a+1$ with $a\geq0$ and $a+b>0$.
\end{enumerate}
If $\{w_n\}$ be a nonnegative sequence with
\begin{eqnarray*}
(\nu-\frac12)(S_{a,b}^{-1})^{\frac{p+1}{p-1}}<\|w_n\|_{D^{1,2}_a(\bbr^d)}^2<(\nu+\frac12)(S_{a,b}^{-1})^{\frac{p+1}{p-1}}
\end{eqnarray*}
Then there exists $\{\tau_{i,n}\}\subset\bbr_+$, satisfying
\begin{eqnarray*}
\min_{i\not=j}\bigg\{\max\bigg\{\frac{\tau_{i,n}}{\tau_{j,n}},\frac{\tau_{j,n}}{\tau_{i,n}}\bigg\}\bigg\}\to+\infty
\end{eqnarray*}
as $n\to\infty$ for $\nu\geq2$, such that
\begin{enumerate}
\item[$(1)$]\quad $w_n=\sum_{i=1}^{\nu}(W)_{\tau_{i,n}}+o_n(1)$ in $D^{1,2}_a(\bbr^d)$.
\item[$(2)$]\quad $\|w_n\|_{D^{1,2}_a(\bbr^d)}^2=\nu\|W\|_{D^{1,2}_a(\bbr^d)}^2+o_n(1)$.
\end{enumerate}
\end{proposition}

\subsection{Main result}
Since the Struwe decomposition (Proposition~\ref{prop0002}) of critical points of the corresponding functional of the CKN inequality~\eqref{eq0001} at infinity also holds in the degenerate case.  It is natural to ask the following question:
\begin{enumerate}
\item[$(\mathcal{Q})$]\quad  Does the stability of critical points of the corresponding functional of the CKN inequality~\eqref{eq0001} at infinity hold true in the degenerate case?
\end{enumerate}

We shall answer the natural question~$(\mathcal{Q})$ by proving the following sharp result.
\begin{theorem}\label{thmn0001}
Let $d\geq2$, $a<0$ and $b=b_{FS}(a)$.  Suppose $u\in D^{1,2}_{a}(\bbr^d)$ be a nonnegative function such that
\begin{eqnarray}\label{eqqqnew0001}
\left(\nu-\frac12\right)\left(S_{a,b}^{-1}\right)^{\frac{p+1}{p-1}}<\|u\|^2_{D^{1,2}_a(\bbr^d)}<\left(\nu+\frac12\right)\left(S_{a,b}^{-1}\right)^{\frac{p+1}{p-1}}.
\end{eqnarray}
\begin{enumerate}
\item[$(a)$]\quad Then as $\left\|-div(|x|^{-a}\nabla u)-|x|^{-b(p+1)}|u|^{p-1}u\right\|_{W^{-1,2}_a(\bbr^d)}\to0$, we have
\begin{eqnarray*}
\inf_{\overrightarrow{\alpha}_{\nu}\in\left(\bbr_+\right)^{\nu}, \overrightarrow{\lambda}_{\nu}\in\bbr^\nu}\left\|u-\sum_{j=1}^{\nu}\alpha_jW_{\lambda_j}\right\|\lesssim\left\|-div(|x|^{-a}\nabla u)-|x|^{-b(p+1)}|u|^{p-1}u\right\|_{W^{-1,2}_a(\bbr^d)}^{\frac{1}{3}}.
\end{eqnarray*}
\item[$(b)$]\quad   There exists $\{u_n\}\subset D^{1,2}_{a}(\bbr^d)$, nonnegative and satisfies \eqref{eqqqnew0001}, such that
\begin{eqnarray*}
\inf_{\overrightarrow{\alpha}_{\nu}\in\left(\bbr_+\right)^{\nu}, \overrightarrow{\lambda}_{\nu}\in\bbr^\nu}\left\|u-\sum_{j=1}^{\nu}\alpha_jW_{\lambda_j}\right\|\sim\left\|-div(|x|^{-a}\nabla u)-|x|^{-b(p+1)}|u|^{p-1}u\right\|_{W^{-1,2}_a(\bbr^d)}^{\frac{1}{3}}.
\end{eqnarray*}
\end{enumerate}
\end{theorem}

\begin{remark}
In preparing this paper, we knew from personal communications with Professor W. Zou that their group was also working on the question~$(\mathcal{Q})$ for the one-bubble case.  Moreover, we notice that in a very recent paper \cite{ZZ2024}, the optimal stability for the one-bubble case has been established by Zhou and Zou \cite{ZZ2024}, while for the multi-bubble case only a partial result is obtained. By assuming that the projection to the nontrivial kernel is much smaller than the interaction, they obtained the exponent $p/2$, which is the same as in the non-degenerate case Theorem \ref{thm0002}. Note that from Theorem \ref{thmn0001} the most important contribution comes exactly from the projection to the nontrivial kernels.

\end{remark}

\begin{remark}
Theorem~\ref{thmn0001} is rather surprising since the optimal power of the stability is an absolute constant $\frac{1}{3}$ which is independent of $p$ and $\nu$.  This is a completely new finding in the studies on the stability of critical points of the corresponding functional of functional inequalities at infinity.  This new finding can be explained by the optimal example of the stability stated in Theorem~\ref{thmn0001} which is given in section~9.  Roughly speaking, for the two-bubble case, the optimal power of the stability of critical points of the corresponding functional of the CKN inequality~\eqref{eq0001} at infinity depends on two values, the interaction between bubbles which is measured by the distance of these bubbles and the projections of these bubbles on their nontrivial kernels.  If the interaction wins the projections then the optimal power of the stability of critical points of the corresponding functional at infinity will be values in Theorem~\ref{thm0002} which depends on $p$.  If the projections wins the interaction then the optimal power of the stability of critical points of the corresponding functional at infinity will be the absolute constant $\frac{1}{3}$.  If the projections and the interaction are comparable then the optimal power of the stability of critical points of the corresponding functional at infinity will lie between the values of Theorem~\ref{thm0002} and the absolute constant $\frac{1}{3}$.  We refer the readers to Remark~\ref{rmkn0001} for more details.  Since the function $u\in D^{1,2}_a(\bbr^d)$ discussed in Theorem~\ref{thmn0001} is arbitrary, the optimal power of the stability must be the absolute constant $\frac{1}{3}$.
\end{remark}

\subsection{Sketch of the proof}
The basic idea in proving Theorem~\ref{thmn0001} is still to apply the Deng-Sun-Wei arguments in \cite{DSW2024}, as in \cite{WW2022}.  Since the bubble $W$ is degenerate now, we need also employ the Frank-Peteranderl strategy in \cite{FP2024}.  However, since our problem is in the critical point setting, new ideas and new techniques are also needed to develop.  Let us now explain our strategy in proving Theorem~\ref{thmn0001} in what follows.

In the first step, we need to set a good problem.  Suppose that  $u\in D^{1,2}_{a}(\bbr^d)$ be a nonnegative function.  We first transform the problem onto the cylinder $\mathcal{C}=\bbr\times \mathbb{S}^{d-1}$, as usual.  Then, based on the Struwe decomposition (Proposition~\ref{prop0002}), the basic idea is to decompose $v$, the the image of the bubble $u$ on the cylinder $\mathcal{C}$, into two parts, as in \cite{CFM2017,FG2021,DSW2024}, by considering the following minimizing problem
\begin{eqnarray*}
\inf_{\overrightarrow{\alpha}_{\nu}\in\left(\bbr_+\right)^{\nu}, \overrightarrow{s}_{\nu}\in\bbr^\nu}\left\|v-\sum_{j=1}^{\nu}\alpha_j\Psi_{s_j}\right\|^2,
\end{eqnarray*}
so that we can write $v=\sum_{j=1}^{\nu}\alpha_j^*\Psi_{s_j^*}+\rho$ where the remaining term $\rho$ is orthogonal to $\{\Psi_{s_j^*}\}$ in $H^1(\mathcal{C})$.  Since the bubble $\Psi$ which is the image of the bubble $W$ on the cylinder $\mathcal{C}$ is degenerate now, we need further decompose the remaining term $\rho$ and further write
\begin{eqnarray*}
v=\sum_{j=1}^{\nu}\alpha_j^*\Psi_{s_j^*}+\left(\sum_{j=1}^{\nu}\sum_{l=1}^{d}\beta_{j,l}^*w_{j,l}\right)+\rho_*,
\end{eqnarray*}
where $\{w_{j,l}\}$ are the nontrivial kernels of $\Psi_{s_j^*}$ and the remaining term $\rho_*$ is orthogonal to $\{\Psi_{s_j^*}\}$ and $\{w_{j,l}\}$ in $H^1(\mathcal{C})$, as in \cite{FP2024}.  Since we are in the critical point setting, the remaining term $\rho_*$ will also satisfy an elliptic equation:
\begin{eqnarray*}
\left\{\aligned&\mathcal{L}(\rho_*)=f+\mathcal{R}_{int}+\mathcal{N},\quad \text{in }\mathcal{C},\\
&\langle \Psi_{j}, \rho_*\rangle=\langle \partial_t\Psi_{j}, \rho_*\rangle=\langle w_{j,l}, \rho_*\rangle=0\quad\text{for all }1\leq j\leq\nu\text{ and }1\leq l\leq d.\endaligned\right.
\end{eqnarray*}
Now, our aim is to control $\{\beta_{k,l}^*\}$ and $\|\rho_*\|$ by $\|f\|_{H^{-1}}$, which also needs us to control $\sum_{j=1}^{\nu}\left|\alpha_j^*-1\right|$ and the interaction between bubbles by $\|f\|_{H^{-1}}$.

In the second step, we need to expand the nonlinear part $\mathcal{N}$ and the regular data $\mathcal{R}_{int}$ of the equation of $\rho_*$ to control $\sum_{j=1}^{\nu}\left|\alpha_j^*-1\right|$, $\{\beta_{k,l}^*\}$, the interaction between bubbles and $\|\rho_*\|$ by $\|f\|_{H^{-1}}$, as in \cite{DSW2024}.  Roughly speaking, we shall decompose $\rho_*$ into two parts, the first part is regular enough so that we can control it very well in any reasonable sense and the second part is (possible) singular due to the (possible) singularity of the data $f\in H^{-1}$ but it can lie in the positive define part of the linear operator $\mathcal{L}$ and is small enough.  We notice that in the functional inequality setting, Frank and Peteranderl have proved in \cite{FP2024} that the optimal Bianchi-Egnell stability of the CKN inequality~\eqref{eq0001} in the degenerate case is quartic and the projection onto nontrivial kernel dominates the remaining term, thus, it is reasonable to expand the nonlinear part $\mathcal{N}$ at least up to the fourth order terms and to ensure that the (possible) singular part of the remaining term $\rho_*$ is smaller than $\max_{k,l}\left|\beta_{k,l}^*\right|^4$ by decomposing the remaining term $\rho_*$ in a suitable way.  Keeping this in minds, we pick up all regular part of $\mathcal{R}_{int}+\mathcal{N}$ which are potentially larger than or equal to $\max_{k,l}\left|\beta_{k,l}^*\right|^4$
and solve several linear equations to decompose $\rho_*$ into $\rho_*=\rho_0+\rho_{**}^{\perp}$, where $\rho_0$ is the regular part and $\rho_{**}^{\perp}$ is the (possible) singular part.  We remark that to pick up all regular part of $\mathcal{R}_{int}+\mathcal{N}$ which are potentially larger than or equal to $\max_{k,l}\left|\beta_{k,l}^*\right|^4$, we need to expand the nonlinear part $\mathcal{N}$ up to the sixth order terms.

In the third step, we need to multiply the equation of $\rho_*$ by $\{\Psi_j\}$, $\{\partial_t\Psi_j\}$ and $\{w_{j,l}\}$, and multiply the equation of $\rho_{**}^{\perp}$ by $\rho_{**}^{\perp}$ to establish the relations of $\sum_{j=1}^{\nu}\left|\alpha_j^*-1\right|$, $\{\beta_{k,l}^*\}$, the interaction between bubbles, $\|\rho_*\|$ and $\|f\|_{H^{-1}}$, as in \cite{DSW2024}.  However, these estimates are not good enough to finally control $\sum_{j=1}^{\nu}\left|\alpha_j^*-1\right|$, $\{\beta_{k,l}^*\}$, the interaction between bubbles and $\|\rho_*\|$ only by $\|f\|_{H^{-1}}$.  This is mainly because $\max_{k,l}\left|\beta_{k,l}^*\right|^4$ can only be bounded from above by a very special form, as observed by Frank and Peteranderl in \cite{FP2024}.  Thus, we need to find out the right third equation to march this special form and ensure that we will not enlarge the upper bounds in the original estimates of $\sum_{j=1}^{\nu}\left|\alpha_j^*-1\right|$, $\{\beta_{k,l}^*\}$, the interaction between bubbles, $\|\rho_*\|$ and $\|f\|_{H^{-1}}$ in this progress.  We remark that in order to ensure that the (possible) singular part of the remaining term $\rho_*$ is smaller than $\max_{k,l}\left|\beta_{k,l}^*\right|^4$, we also need to full use the symmetry and the orthogonality.

In the final step, we use all above estimates of $\sum_{j=1}^{\nu}\left|\alpha_j^*-1\right|$, $\{\beta_{k,l}^*\}$, the interaction between bubbles, $\|\rho_*\|$ and $\|f\|_{H^{-1}}$ and the estimates of $\max_{k,l}\left|\beta_{k,l}^*\right|^4$ established by Frank and Peteranderl in \cite{FP2024} to derive the desired estimate in $(a)$ of Theorem~\ref{thmn0001}.  The proof of $(b)$ of Theorem~\ref{thmn0001} is to construct an example by considering the case $\nu=2$ and using the good ansatz $\sum_{j=1}^{2}\alpha_j^*\Psi_{s_j^*}+\left(\sum_{j=1}^{2}\sum_{l=1}^{d}\beta_{j,l}^*w_{j,l}\right)+\rho_0$ in the proof of $(a)$ of Theorem~\ref{thmn0001}.

\vskip0.12in

We believe that our strategy of proofs may be useful  to study many other problems in which degeneracy appears.

\subsection{Structure of this paper}. In section~2, we give some preliminaries.  In section~3, we introduce the setting of the problem as stated above by decomposing a given function into three parts, the projection on bubbles, the projection on nontrivial kernels and the remaining term.  In section~4, we expand the nonlinear part of the remaining term up to the fourth order term in the first time to pick up all possible leading order terms in it and use these possible leading order terms to decompose the remaining term into the regular part and the (possible) singular part.  In section~5, we refine the expansion of the nonlinear part by the decomposition of the remaining term by adding the regular part of the remaining term into the ansatz and estimates the differences of the projection on bubbles.  In section~6, we expand the nonlinear part of the remaining term up to the sixth order term to estimate the interaction between bubbles.  In section~7, we estimate the (possible) singular part.  In section~8, we finally estimate the projection on nontrivial kernels and prove $(a)$ of Theorem~\ref{thmn0001}.  In section~9, we construct an optimal example and prove $(b)$ of Theorem~\ref{thmn0001}.

\section{Preliminaries}
The CKN inequality~\eqref{eq0001} can be rewritten as the following minimizing problem:
\begin{eqnarray}\label{eq0002}
S_{a,b}^{-1}=\inf_{u\in D^{1,2}_{a}(\bbr^d)\backslash\{0\}}\frac{\|u\|^2_{D^{1,2}_a(\bbr^d)}}{\|u\|^2_{L^{p+1}(|x|^{-b(p+1)},\bbr^d)}},
\end{eqnarray}
where $L^{p+1}(|x|^{-b(p+1)},\bbr^d)$ is the usual weighted Lebesgue space and its usual norm is given by
\begin{eqnarray*}
\|u\|_{L^{p+1}(|x|^{-b(p+1)},\bbr^d)}=\bigg(\int_{\bbr^d}|x|^{-b(p+1)}|u|^{p+1}dx\bigg)^{\frac{1}{p+1}}.
\end{eqnarray*}
The Euler-Lagrange equation of the minimizing problem~\eqref{eq0002} is given by
\begin{eqnarray}\label{eq0018}
-div(|x|^{-a}\nabla u)=|x|^{-b(p+1)}|u|^{p-1}u, \quad u\in D^{1,2}_{a}(\bbr^d).
\end{eqnarray}
It is well known (cf. \cite[Proposition~2.2]{CW2001}) that $D^{1,2}_{a}(\bbr^d)$, the Hilbert space given by \eqref{eqn886}, is isomorphic to the Hilbert space $H^1(\mathcal{C})$ by the transformation
\begin{eqnarray}\label{eq0007}
u(x)=|x|^{-(a_c-a)}v\left(-\ln|x|,\frac{x}{|x|}\right),
\end{eqnarray}
where we denote $a_c=\frac{d-2}{2}$ as in \cite{DEL2016, DET2008}, $\mathcal{C}=\bbr\times \mathbb{S}^{d-1}$ is the standard cylinder,  $H^1(\mathcal{C})$ is the Hilbert space with the inner product given by
\begin{eqnarray*}
\langle w,v \rangle_{H^1(\mathcal{C})}=\int_{\mathcal{C}}\left(\nabla w\nabla v+(a_c-a)^2 uv\right) d\mu
\end{eqnarray*}
and $d\mu$ the volume element on $\mathcal{C}$.
By \eqref{eq0007}, the minimizing problem~\eqref{eq0002} is equivalent to the following minimizing problem:
\begin{eqnarray}\label{eq0009}
S_{a,b}^{-1}=\inf_{v\in H^1(\mathcal{C})\backslash\{0\}}\frac{\|v\|^2_{H^{1}(\mathcal{C})}}{\|v\|^2_{L^{p+1}(\mathcal{C})}},
\end{eqnarray}
where $\|\cdot\|_{L^{p+1}(\mathcal{C})}$ is the usual norm in the Lebesgue space $L^{p+1}(\mathcal{C})$.  For the sake of simplicity, we denote
\begin{eqnarray*}
L^{p+1}:=L^{p+1}(\mathcal{C})\quad\text{and}\quad H^1:=H^1(\mathcal{C})
\end{eqnarray*}
in what follows.  Let $t=-\ln|x|$ and $\theta=\frac{x}{|x|}$ for $x\in\bbr^N\backslash\{0\}$, then the Euler-Lagrange equation of \eqref{eq0002} in terms of $u$ given by \eqref{eq0018} is equivalent to the following Euler-Lagrange equation of \eqref{eq0009} in terms of $v$:
\begin{eqnarray}\label{eq0006}
-\Delta_{\theta}v-\partial_t^2v+(a_c-a)^2v=|v|^{p-1}v, \quad v\in H^1(\mathcal{C}),
\end{eqnarray}
where $\Delta_{\theta}$ is the Laplace-Beltrami operator on $\mathbb{S}^{d-1}$.

\vskip0.12in

Clearly, minimizers of \eqref{eq0002} are ground states of \eqref{eq0018}.  Moreover,
by the transformation~\eqref{eq0007}, the linear equation~\eqref{eq0017} can be rewritten as follows:
\begin{eqnarray}\label{eq0016}
-\Delta_{\theta}v-\partial_t^2v+(a_c-a)^2v=p\Psi^{p-1}v, \quad v\in H^1(\mathcal{C}),
\end{eqnarray}
and $W(x)$ given by \eqref{eq0004} can be rewritten as
\begin{eqnarray}\label{eq0026}
\Psi(t)=\bigg(\frac{(p+1)(a_c-a)^2}{2}\bigg)^{\frac{1}{p-1}}\bigg(cosh\left(\frac{(a_c-a)(p-1)}{2}t\right)\bigg)^{-\frac{2}{p-1}}.
\end{eqnarray}
Since \eqref{eq0016} is translational invariance, it follows from \eqref{eq0010} and the transformation~\eqref{eq0007} that
\begin{eqnarray*}
\Psi_s'(t)=\Psi'(t-s)=\partial_t\Psi(t-s)=-\partial_s\Psi(t-s)
\end{eqnarray*}
is the only nonzero solution of \eqref{eq0016} in $H^1(\mathcal{C})$ under the condition~\eqref{eq0003} except $b=b_{FS}(a)$.

\vskip0.12in

For the special case $b=b_{FS}(a)$, the bubble $\Psi(t)$ is degenerate in $H^1(\mathcal{C})$ in the sense that the nonzero solution of \eqref{eq0016} in $H^1(\mathcal{C})$ is not only generated by the translational invariance of \eqref{eq0006}.  Fortunately, we have the following lemma which provides a completely understanding of the solutions of the linear equation~\eqref{eq0016} in $H^1(\mathcal{C})$.
\begin{lemma}\label{lem0001}
(\cite[Lemma~7]{FP2024})\quad Let $d\geq2$, $a<0$ and $b=b_{FS}(a)$.  Then any solution of the linear equation~\eqref{eq0016} in $H^1(\mathcal{C})$ is the linear combination of $\partial_t\Psi$ and $\Psi^{\frac{p+1}{2}}\theta_{1}, \Psi^{\frac{p+1}{2}}\theta_{2}, \cdots, \Psi^{\frac{p+1}{2}}\theta_{d}$, where $\theta_l$ are the standard spherical harmonics of degree $1$ on $\mathbb{S}^{d-1}$.
\end{lemma}

\begin{remark}\label{rmk0001}
As in \cite{FP2024}, we call $\partial_t\Psi$ the trivial kernel of the the linear equation~\eqref{eq0016} in $H^1(\mathcal{C})$ and call $\Psi^{\frac{p+1}{2}}\theta_{1}, \Psi^{\frac{p+1}{2}}\theta_{2}, \cdots, \Psi^{\frac{p+1}{2}}\theta_{d}$ the nontrivial kernel of the linear equation~\eqref{eq0016} in $H^1(\mathcal{C})$.  Moreover, since $\theta_{l}$ are odd on $\mathbb{S}^{d-1}$, $\partial_t\Psi$ is odd in $\bbr$ and $\Psi$ is even in $\bbr$, by \eqref{eq0006} and \eqref{eq0016}, we have the following orthogonal conditions:
\begin{eqnarray*}
\left\langle \Psi, \partial_t\Psi\right\rangle=0,\quad \left\langle \Psi, w_{l}\right\rangle=0,\quad\left\langle \partial_t\Psi, w_{l}\right\rangle=0\quad \text{and}\quad \left\langle w_{j}, w_{l}\right\rangle=0
\end{eqnarray*}
for all $1\leq j\not=l\leq d$, where $w_l=\Psi^{\frac{p+1}{2}}\theta_{l}$.
\end{remark}

\section{Setting of the problem}
Let $d\geq2$, $a<0$ and $b=b_{FS}(a)$.  Then direct calculations show that
\begin{eqnarray*}
(a_c-a)^2=\frac{4(d-1)}{(p+1)^2-4}.
\end{eqnarray*}
For the sake of simplicity, we denote that
\begin{eqnarray*}
S_{FS}:=S_{a,b}\quad\text{and}\quad\Lambda_{FS}:=\frac{4(d-1)}{(p+1)^2-4}
\end{eqnarray*}
for $d\geq2$, $a<0$ and $b=b_{FS}(a)$.  Let $v\in H^1(\mathcal{C})$ be a nonnegative function such that
\begin{eqnarray*}
\left(\nu-\frac12\right)\left(S_{FS}^{-1}\right)^{\frac{p+1}{p-1}}<\|v\|_{H^{1}}^2<\left(\nu+\frac12\right)\left(S_{FS}^{-1}\right)^{\frac{p+1}{p-1}}
\end{eqnarray*}
for some positive integer $\nu\geq2$ and denote
\begin{eqnarray}\label{eq0060}
f:=-\Delta_{\theta}v-\partial_t^2v+\Lambda_{FS}v-v^{p}.
\end{eqnarray}
Then it is easy to see that $f\in H^{-1}(\mathcal{C})$.

\vskip0.12in

Since $a<b_{FS}(a)<a+1$ for $a<0$ and $d\geq2$, by \cite[Lemma~4.2]{CW2001} (see also \cite[Proposition~3.2]{WW2022}), there exists $(\alpha_{1,\natural}, \alpha_{2,\natural}, \cdots, \alpha_{\nu,\natural})$ and $(s_{1,\natural}, s_{2,\natural}, \cdots, s_{\nu,\natural})$ satisfying
\begin{eqnarray*}
\max_{1\leq j\leq\nu}|\alpha_{j, \natural}-1|\to0\text{ and }\min_{i\not=j}\left|s_{i, \natural}-s_{j, \natural}\right|\to+\infty\quad\text{as}\quad \|f\|_{H^{-1}(\mathcal{C})}\to0
\end{eqnarray*}
such that
\begin{eqnarray}\label{eqn0005}
\left\|v-\sum_{j=1}^{\nu}\alpha_{j,\natural}\Psi_{s_{j,\natural}}\right\|^2\to0\quad\text{as}\quad \|f\|_{H^{-1}}\to0
\end{eqnarray}
where
\begin{eqnarray*}
\|\cdot\|^2=\|\cdot\|^2_{H^1}+\Lambda_{FS}\|\cdot\|^2_{L^2}
\end{eqnarray*}
is the equivalent norm in $H^1$.  We also denote the related inner product in $H^1$, which is induced by the norm $\|\cdot\|$, by $\langle \cdot,\cdot \rangle$.
Thus, we can rewrite
\begin{eqnarray*}
v=\sum_{j=1}^{\nu}\alpha_{j,\natural}\Psi_{s_{j,\natural}}+\text{a remaining term}
\end{eqnarray*}
in $H^1$ as $\|f\|_{H^{-1}}\to0$.  To obtain an optimal decomposition as above, let us consider the following minimizing problem:
\begin{eqnarray}\label{eqn0001}
\inf_{\overrightarrow{\alpha}_{\nu}\in\left(\bbr_+\right)^{\nu}, \overrightarrow{s}_{\nu}\in\bbr^\nu}\left\|v-\sum_{j=1}^{\nu}\alpha_j\Psi_{s_j}\right\|^2,
\end{eqnarray}
where $\overrightarrow{\alpha}_{\nu}=(\alpha_1, \alpha_2, \cdots, \alpha_{\nu})$ and $\overrightarrow{s}_{\nu}=(s_1, s_2, \cdots, s_{\nu})$.  By \eqref{eqn0005} and similar arguments used for \cite[Proposition~4.1]{WW2022}, we know that the variational problem~\eqref{eqn0001} has minimizers, say $(\overrightarrow{\alpha}_{\nu}^*, \overrightarrow{s}_{\nu}^*)$, such that
\begin{eqnarray}\label{eqn1005}
\max_{1\leq j\leq\nu}|\alpha_j^*-1|\to0\text{ and }\min_{i\not=j}\left|s_i^*-s_j^*\right|\to+\infty\quad\text{as}\quad \|f\|_{H^{-1}(\mathcal{C})}\to0.
\end{eqnarray}
Thus, we can decompose
\begin{eqnarray}\label{eqnewnew0002}
v=\sum_{j=1}^{\nu}\alpha_j^*\Psi_{s_j^*}+\rho
\end{eqnarray}
where
\begin{eqnarray}\label{eqn0008}
\|\rho\|^2=\inf_{\overrightarrow{\alpha}_{\nu}\in\left(\bbr_+\right)^{\nu}, \overrightarrow{s}_{\nu}\in\bbr^\nu}\left\|v-\sum_{j=1}^{\nu}\alpha_j\Psi_{s_j}\right\|^2\to0\quad\text{as }\|f\|_{H^{-1}}\to0
\end{eqnarray}
and by the minimality of $(\overrightarrow{\alpha}_{\nu}^*, \overrightarrow{s}_{\nu}^*)$,
\begin{eqnarray}\label{eqn0004}
\left\langle \rho, \Psi_{s_j^*}\right\rangle=0 \quad\text{and}\quad \left\langle \rho, \partial_t\Psi_{s_j^*}\right\rangle=0\quad\text{for all }j=1,2,\cdots,\nu.
\end{eqnarray}
Since by Lemma~\ref{lem0001}, the linear equation~\eqref{eq0016} has nontrivial kernels in $H^1(\mathcal{C})$ for $d\geq2$, $a<0$ and and $b=b_{FS}(a)$, we need further decompose the remaining term as follows:
\begin{eqnarray}\label{eqn0002}
\rho=\left(\sum_{j=1}^{\nu}\sum_{l=1}^{d}\beta_{j,l}^*w_{j,l}\right)+\rho_*,
\end{eqnarray}
where for the sake of simplicity, we denote $w_{j,l}=\Psi_{j}^{\frac{p+1}{2}}\theta_{l}=w_l(t-s_j^*)$ and $\{\beta_{j,l}^*\}$ is chosen such that $\left\langle w_{j,l}, \rho_*\right\rangle=0$ for all $1\leq j\leq \nu$ and $1\leq l\leq d$.  The above facts can be summarized into the following lemma.
\begin{lemma}\label{lem0002}
Let $d\geq2$, $a<0$ and $b=b_{FS}(a)$.  Then we have the following decomposition of $v$:
\begin{eqnarray}\label{eqnewnew0003}
v=\sum_{j=1}^{\nu}\alpha_j^*\Psi_{s_j^*}+\left(\sum_{j=1}^{\nu}\sum_{l=1}^{d}\beta_{j,l}^*w_{j,l}\right)+\rho_*,
\end{eqnarray}
where the remaining term $\rho_*$ satisfies the following orthogonal conditions:
\begin{eqnarray}\label{eqn0003}
\left\langle \rho_*, \Psi_{s_j^*}\right\rangle=0, \quad \left\langle \rho_*, \partial_t\Psi_{s_j^*}\right\rangle=0\quad\text{and}\quad\left\langle \rho_*, w_{j,l}\right\rangle=0
\end{eqnarray}
for all $1\leq j\leq \nu$ and $1\leq l\leq d$ with
\begin{eqnarray}\label{eqnewnew1005}
\max_{1\leq j\leq\nu}|\alpha_j^*-1|\to0,\quad\min_{i\not=j}\left|s_i^*-s_j^*\right|\to+\infty\quad\text{and}\quad\|\rho_*\|\to0
\end{eqnarray}
as $\|f\|_{H^{-1}(\mathcal{C})}\to0$.  Moreover, we also have
\begin{eqnarray}\label{eqnewnew0001}
\inf_{\overrightarrow{\alpha}_{\nu}\in\left(\bbr_+\right)^{\nu}, \overrightarrow{s}_{\nu}\in\bbr^\nu}\left\|v-\sum_{j=1}^{\nu}\alpha_j\Psi_{s_j}\right\|^2\sim\sum_{k=1}^{\nu}\sum_{l=1}^{d}\left(\beta_{k,l}^*\right)^2+\|\rho_*\|^2.
\end{eqnarray}
\end{lemma}
\begin{proof}
\eqref{eqnewnew0003} can be obtained by \eqref{eqnewnew0002} and \eqref{eqn0002}, directly, while the orthogonal conditions of $\rho_*$ is obtained by the choice of $\{\beta_{j,l}^*\}$, the orthogonal condition of $w_{j,l}$ given in Remark~\ref{rmk0001} and the orthogonal condition of $\rho$ given by \eqref{eqn0004}.  By \eqref{eqn1005}, it remains to show that $\|\rho_*\|\to0$ as $\|f\|_{H^{-1}(\mathcal{C})}\to0$.  Indeed, by the orthogonal condition of $w_{j,l}$ given in Remark~\ref{rmk0001} and the orthogonal condition of $\rho$ given by \eqref{eqn0003},
\begin{eqnarray}\label{eqn0006}
\|\rho\|^2
&=&\sum_{k=1}^{\nu}\sum_{l=1}^{d}\left(\beta_{k,l}^*\right)^2\left\|w_{d}\right\|^2+\|\rho_*\|^2\notag\\
&&+2\sum_{m,n=1;m<n}^{\nu}\sum_{l=1}^{d}\beta_{n,l}^*\beta_{m,l}^*\left\langle w_{n,l}, w_{m,l}\right\rangle,
\end{eqnarray}
where we have used the invariance of $\mathbb{S}^{d-1}$ and the norm $\|\cdot\|$ under the action of orthogonal matrix $O(d)$.  Clearly, by \eqref{eqn1005} and \eqref{eqn0006}, it is easy to see that
\begin{eqnarray*}
\|\rho\|^2\sim\sum_{k=1}^{\nu}\sum_{l=1}^{d}\left(\beta_{k,l}^*\right)^2+\|\rho_*\|^2,
\end{eqnarray*}
which, together with \eqref{eqn0008}, implies that $\|\rho_*\|\to0$ and \eqref{eqnewnew0001} as $\|f\|_{H^{-1}(\mathcal{C})}\to0$.
\end{proof}

\vskip0.12in

For the sake of simplicity, we use the notations $\Psi_{j}:=\Psi_{s_{j}^*}$, $\Psi_{j}^*=\alpha_j^*\Psi_j$,
\begin{eqnarray}\label{eqn0140}
\mathcal{U}:=\sum_{j=1}^{\nu}\Psi_{j}^*,\quad \mathcal{U}_j=\mathcal{U}-\Psi_{j}^*=\sum_{i=1;i\not=j}^{\nu}\Psi_{i}^*
\end{eqnarray}
and
\begin{eqnarray}\label{eqn0040}
\mathcal{V}_j=\sum_{l=1}^{d}\beta_{j,l}^*w_{j,l}=\Psi_{j}^{\frac{p+1}{2}}\sum_{l=1}^{d}\beta_{j,l}^*\theta_l,\quad \mathcal{V}:=\sum_{j=1}^{\nu}\mathcal{V}_j.
\end{eqnarray}
Since $\Psi_{j}$ are solutions of \eqref{eq0006} and $w_{j,l}$ are solutions of \eqref{eq0016}, by \eqref{eq0060}, \eqref{eqn0002} and \eqref{eqn0003}, it is easy to see that the remaining term $\rho_*$ satisfies:
\begin{eqnarray}\label{eq0014}
\left\{\aligned&\mathcal{L}(\rho_*)=f+\mathcal{R}_1+\mathcal{R}_2+\mathcal{N},\quad \text{in }\mathcal{C},\\
&\langle \Psi_{j}, \rho_*\rangle=\langle \partial_t\Psi_{j}, \rho_*\rangle=\langle w_{j,l}, \rho_*\rangle=0\quad\text{for all }1\leq j\leq\nu\text{ and }1\leq l\leq d,\endaligned\right.
\end{eqnarray}
where $\mathcal{L}(\rho_*)$ is the linear operator given by
\begin{eqnarray}
\mathcal{L}(\rho_*)&=&-\partial_t^2\rho_*-\Delta_\theta\rho_*+\Lambda_{FS}\rho_*-p\mathcal{U}^{p-1}\rho_*\notag\\
&=&\left(-\partial_t^2\rho_*-\Delta_\theta\rho_*+\Lambda_{FS}\rho_*-p\left(\Psi_{j}^*\right)^{p-1}\rho_*\right)-p\left(\mathcal{U}^{p-1}-\left(\Psi_{j}^*\right)^{p-1}\right)\rho_*\notag\\
&=&\mathcal{L}_j(\rho_*)-\mathcal{L}_{j,ex}(\rho_*)\label{eqn0044}
\end{eqnarray}
for all $j=1,2,\cdots,\nu$, $\mathcal{R}_1$ and $\mathcal{R}_2$ are the errors given by
\begin{eqnarray}\label{eqn0020}
\mathcal{R}_1&=&\mathcal{U}^{p}-\sum_{j=1}^{\nu}\left(\Psi_{j}^*\right)^p+\sum_{j=1}^{\nu}\left(\left(\alpha_j^*\right)^p-\alpha_j^*\right)\Psi_{j}^p\notag\\
&:=&\mathcal{R}_{1,ex}+\sum_{j=1}^{\nu}\mathcal{R}_{1,j}
\end{eqnarray}
and
\begin{eqnarray}\label{eqn0021}
\mathcal{R}_2&=&p\sum_{j=1}^{\nu}\left(\mathcal{U}^{p-1}-\left(\Psi_j^*\right)^{p-1}+\left((\alpha_j^*)^{p-1}-1\right)\Psi_{j}^{p-1}\right)\mathcal{V}_j\notag\\
&=&p\sum_{j=1}^{\nu}\left(\mathcal{U}^{p-1}-\left(\Psi_j^*\right)^{p-1}\right)\mathcal{V}_j+p\sum_{j=1}^{\nu}\left(\left((\alpha_j^*)^{p-1}-1\right)\Psi_{j}^{p-1}\right)\mathcal{V}_j\notag\\
&:=&\mathcal{R}_{2,ex}+\sum_{j=1}^{\nu}\mathcal{R}_{2,j},
\end{eqnarray}
and $\mathcal{N}$ is the only nonlinear part of $\rho_*$ given by
\begin{eqnarray}\label{eqnewnew8856}
\mathcal{N}=\left(\mathcal{U}+\mathcal{V}+\rho_*\right)^p-\mathcal{U}^p-p\mathcal{U}^{p-1}\left(\mathcal{V}+\rho_*\right).
\end{eqnarray}
By \eqref{eqnewnew0001}, to establish stability inequality of the CKN inequality in the critical point setting for $d\geq2$, $a<0$ and $b=b_{FS}(a)$ as in \cite{DSW2024,FG2021,WW2022}, we shall control $\{\beta_{k,l}^*\}$ and $\|\rho_*\|$ by $\|f\|_{H^{-1}}$.

\vskip0.12in

\section{first expansion of $\mathcal{N}$ and further decomposition of $\rho_*$}
To get optimal control of $\{\beta_{k,l}^*\}$ and $\|\rho_*\|$ by $\|f\|_{H^{-1}}$, we shall
apply the ideas in \cite{DSW2024} (see also \cite{WW2022}).  Roughly speaking, we need to further decompose the remaining part $\rho_*$ into two parts.  The first part, say $\rho_0$, is regular enough in the sense that $\rho_0$ can be controlled by a good weighted $L^\infty$ norm.  The second part, say $\rho_{**}^{\perp}$, is (possible) singular according to the (possible) lack of regularity of $f$ which is much smaller than $\rho_0$ in $H^1(\mathcal{C})$.  For this purpose, we need to expand the nonlinear part $\mathcal{N}$ to pick up all possible leading order terms of the remaining part $\rho_*$.

\subsection{First expansion of $\mathcal{N}$}
Since by \cite[Theorem~1]{FP2024}, the optimal Bianchi-Egnell stability of the CKN inequality for $d\geq2$, $a<0$ and $b=b_{FS}(a)$ is quartic, it is reasonable to first expand the nonlinear part $\mathcal{N}$ up to the fourth order term.
\begin{lemma}\label{lem0003}
Let $d\geq2$, $a<0$ and $b=b_{FS}(a)$.  Then we have the following expansion of the nonlinear part $\mathcal{N}$:
\begin{eqnarray}\label{eqn0018}
\mathcal{N}&=&A_p\mathcal{U}^{p-2}\left(\mathcal{V}^2+2\mathcal{V}\rho_*\right)+B_p\mathcal{U}^{p-3}\left(\mathcal{V}^3+3\mathcal{V}^2\rho_*\right)\notag\\
&&+\mathcal{O}\left(\mathcal{U}^{p}\beta_*^4+\chi_{p\geq2}|\rho_*|^{2}+|\rho_*|^{p}+\sum_{l=2}^{3}\beta_*^{\frac{2(l-p)_+}{p+1}}|\rho_*|^{l-\frac{2(l-p)_+}{p+1}}\right)\notag\\
&:=&\mathcal{N}_{*}+\mathcal{N}_{rem}
\end{eqnarray}
in $\mathcal{C}$, where $\rho_*$, $\mathcal{U}$ and $\mathcal{V}$ are given by \eqref{eqnewnew0003}, \eqref{eqn0140} and \eqref{eqn0040}, respectively, $A_p=\frac{p(p-1)}{2}$, $B_p=\frac{p(p-1)(p-2)}{6}$ and
\begin{eqnarray*}
\chi_{p\geq2}=\left\{\aligned
&1,\quad p\geq2,\\
&0,\quad 1<p<2.
\endaligned\right.
\end{eqnarray*}
\end{lemma}
\begin{proof}
As in the proof of \cite[Lemma~8]{FP2024}, we introduce the set
\begin{eqnarray*}
\mathcal{A}=\left\{(\theta,t)\in\mathcal{C}\mid |\rho_*|\leq\left|\mathcal{V}\right|\right\}.
\end{eqnarray*}
Note that by \eqref{eqnewnew1005}, \eqref{eqn0040} and $p>1$, we have
\begin{eqnarray}\label{eqn0191}
\left|\mathcal{V}\right|\lesssim\beta_*\mathcal{U}^{\frac{p+1}{2}}\lesssim\beta_*\mathcal{U},
\end{eqnarray}
where $\beta_*=\max_{j,l}\left|\beta_{j,l}^*\right|$.  Thus, we can apply the ideas in the proof of \cite[Lemma~8]{FP2024} to expand the nonlinear part $\mathcal{N}$ in $\mathcal{A}$ and $\mathcal{A}^c$, respectively, as follows:
\begin{eqnarray}\label{eqn0009}
\mathcal{N}&=&A_p\mathcal{U}^{p-2}\left(\mathcal{V}^2+2\mathcal{V}\rho_*\right)+B_p\mathcal{U}^{p-3}\left(\mathcal{V}^3+3\mathcal{V}^2\rho_*\right)\notag\\
&&+\mathcal{O}\left(\mathcal{U}^{p-4}\left(\mathcal{V}+\rho_*\right)^4+\mathcal{U}^{p-2}\left|\rho_*\right|^2\right)\notag\\
&=&A_p\mathcal{U}^{p-2}\left(\mathcal{V}^2+2\mathcal{V}\rho_*\right)+B_p\mathcal{U}^{p-3}\left(\mathcal{V}^3+3\mathcal{V}^2\rho_*\right)\notag\\
&&+\mathcal{O}\left(\mathcal{U}^{3p-2}\beta_*^4+\chi_{p\geq2}|\rho_*|^{2}+|\rho_*|^{p}\right)
\end{eqnarray}
in $\mathcal{A}$ and
\begin{eqnarray}\label{eqn0010}
\mathcal{N}=\mathcal{O}\left(\chi_{p\geq2}|\rho_*|^{2}+|\rho_*|^{p}\right)
\end{eqnarray}
in $\mathcal{A}^c$.
Since $\frac{2(2-p)}{p+1}\in(0, 1)$ for $1<p<2$ and $\frac{2(3-p)}{p+1}\in(0, 2)$ for $1<p<3$, by \eqref{eqn0140} and \eqref{eqn0040}, we have
\begin{eqnarray}\label{eqn0017}
\sum_{l=2}^{3}\beta_*^{\frac{2(l-p)_+}{p+1}}|\rho_*|^{l-\frac{2(l-p)_+}{p+1}}\gtrsim A_p\mathcal{U}^{p-2}\left(\mathcal{V}^2+2\left|\mathcal{V}\rho_*\right|\right)+B_p\mathcal{U}^{p-3}\left(\left|\mathcal{V}\right|^3+3\mathcal{V}^2|\rho_*|\right)
\end{eqnarray}
in $\mathcal{A}^c$, where $a_{\pm}=\max\{\pm a,0\}$.  Thus, \eqref{eqn0018} is obtained by combining \eqref{eqn0009}, \eqref{eqn0010} and \eqref{eqn0017}.
\end{proof}

\vskip0.12in

We need to further expand the nonlinear part $\mathcal{N}_*$ to separate the bubbles, for this purpose, we introduce some necessary notations.  For the sake of simplicity and without loss of generality, we assume that
\begin{eqnarray*}
-\infty:=s_0<s_1<s_2<\cdots<s_\nu<s_{\nu+1}:=+\infty.
\end{eqnarray*}
We also denote
\begin{eqnarray}\label{eqn0240}
\tau_{j}=s_{j+1}-s_j,\quad\tau=\min_{j=1,2,\cdots,\nu-1}\tau_j
\end{eqnarray}
and
\begin{eqnarray*}
\mathcal{B}_1&=&\left[s_1^*-\frac{\tau_{1}}{2}, s_1^*+\frac{\tau_1}{2}\right]\times\mathbb{S}^{d-1},\\
\mathcal{B}_j&=&\left[s_j^*-\frac{\tau_{j-1}}{2}, s_j^*+\frac{\tau_j}{2}\right]\times\mathbb{S}^{d-1},\quad 2\leq j\leq\nu-1,\\
\mathcal{B}_\nu&=&\left[s_{\nu}^*-\frac{\tau_{\nu-1}}{2}, s_\nu^*+\frac{\tau_{\nu-1}}{2}\right]\times\mathbb{S}^{d-1}.
\end{eqnarray*}
\begin{lemma}\label{lem0004}
Let $d\geq2$, $a<0$ and $b=b_{FS}(a)$.  Then the nonlinear part $\mathcal{N}$ which is given by \eqref{eqnewnew8856} can be further expanded as follows:
\begin{eqnarray}\label{eqn0045}
\mathcal{N}&=&\sum_{j=1}^{\nu}\left(A_p\left(\Psi_j^*\right)^{p-2}\left(\mathcal{V}_j^2+2\mathcal{V}_j\rho_*\right)+B_p\left(\Psi_j^*\right)^{p-3}\left(\mathcal{V}_j^3+3\mathcal{V}_j^2\rho_*\right)\right)\chi_{\mathcal{B}_j}\notag\\
&&+\sum_{j=1}^{\nu}2A_p\left(\mathcal{U}^{p-2}\mathcal{V}-\left(\Psi_j^*\right)^{p-2}\mathcal{V}_j\right)\rho_*\chi_{\mathcal{B}_j}+\sum_{j=1}^{\nu}\mathcal{O}\left(\beta_*^2\mathcal{U}_j\Psi_{j}^{2p-3}\left(\Psi_j+\rho_*\right)\right)\chi_{\mathcal{B}_j}\notag\\
&&+\left(2A_p\mathcal{U}^{p-2}\mathcal{V}\rho_*+\mathcal{U}^{p-3}\mathcal{V}^{2}\left(A_p\mathcal{U}+B_p\mathcal{V}\right)\right)\chi_{\mathcal{C}\backslash\left(\cup_{j=1}^{\nu}\mathcal{B}_j\right)}+\mathcal{N}_{rem}
\end{eqnarray}
in $\mathcal{C}$, where $\mathcal{N}_{rem}$ is given in \eqref{eqn0018} and $A_p$, $B_p$ are given in Lemma~\ref{lem0003}.
\end{lemma}
\begin{proof}
Since $\mathcal{U}_j>0$ in $\mathcal{C}$ for all $1\leq j\leq\nu$ by \eqref{eqn0140}, by \eqref{eq0026}, \eqref{eqn0040} and the Taylor expansion, we have
\begin{eqnarray*}
\mathcal{U}^{p-\alpha}\mathcal{V}^{\alpha-1}&=&\left(\Psi_{j}^*\right)^{p-\alpha}\mathcal{V}_j^{\alpha-1}+\mathcal{O}\left(\Psi_{j}^{p-\alpha-1}\mathcal{V}_j^{\alpha-1}\mathcal{U}_j+\Psi_{j}^{p-\alpha}\left|\sum_{i=1;i\not=j}^{\nu}\mathcal{V}_j^{\alpha-2}\mathcal{V}_i\right|\right)\notag\\
&=&\left(\Psi_{j}^*\right)^{p-\alpha}\mathcal{V}_j^{\alpha-1}+\mathcal{O}\left(\beta_*^{\alpha-1}\left(\Psi_j^{\frac{\alpha(p-1)+p-3}{2}}\mathcal{U}_j+\Psi_{j}^{\frac{\alpha(p-1)-2}{2}}\mathcal{U}_j^{\frac{p+1}{2}}\right)\right)
\end{eqnarray*}
in $\mathcal{B}_j$ for all $j=1,2,\cdots,\nu$ and
\begin{eqnarray*}
\mathcal{U}^{p-\alpha}\mathcal{V}^{\alpha-1}=\mathcal{O}\left(\beta_*^{\alpha-1}\mathcal{U}^{\frac{(\alpha+1)(p-1)}{2}}\right)
\end{eqnarray*}
in $\mathcal{C}\backslash\left(\cup_{j=1}^{\nu}\mathcal{B}_j\right)$, where $\alpha=2$ or $\alpha=3$.  Similarly,
\begin{eqnarray*}
\mathcal{U}^{p-\alpha}\mathcal{V}^{\alpha}&=&\left(\Psi_{j}^*\right)^{p-\alpha}\mathcal{V}_j^{\alpha}+\mathcal{O}\left(\Psi_{j}^{p-\alpha-1}\mathcal{V}_j^{\alpha}\mathcal{U}_j+\Psi_{j}^{p-\alpha}\left|\sum_{i=1;i\not=j}^{\nu}\mathcal{V}_j^{\alpha-1}\mathcal{V}_i\right|\right)\notag\\
&=&\left(\Psi_{j}^*\right)^{p-\alpha}\mathcal{V}_j^{\alpha}+\mathcal{O}\left(\beta_*^{\alpha}\left(\Psi_j^{\frac{\alpha(p-1)}{2}+p-1}\mathcal{U}_j+\Psi_{j}^{\frac{(\alpha+1)(p-1)}{2}}\mathcal{U}_j^{\frac{p+1}{2}}\right)\right)
\end{eqnarray*}
in $\mathcal{B}_j$ for all $j=1,2,\cdots,\nu$ and
\begin{eqnarray*}
\mathcal{U}^{p-\alpha}\mathcal{V}^{\alpha}=\mathcal{O}\left(\beta_*^{\alpha}\mathcal{U}^{\frac{\alpha(p-1)}{2}+p}\right)
\end{eqnarray*}
in $\mathcal{C}\backslash\left(\cup_{j=1}^{\nu}\mathcal{B}_j\right)$.  Thus, summarizing the above estimates of $\mathcal{U}^{p-\alpha}\mathcal{V}^{\alpha-1}$ and $\mathcal{U}^{p-\alpha}\mathcal{V}^{\alpha}$ in $\mathcal{N}_{*}$ and by $p>1$ and Lemma~\ref{lem0003}, we have the desired expansion of $\mathcal{N}$ given by \eqref{eqn0045}, where $\mathcal{N}_{rem}$ is given in \eqref{eqn0018}.
\end{proof}

\vskip0.12in

\subsection{Further decomposition of $\rho_{*}$}
Recall that we shall control $\{\beta_{k,l}^*\}$ and $\|\rho_*\|$ by $\|f\|_{H^{-1}}$.  However, due to the errors $\mathcal{R}_1$ and $\mathcal{R}_2$, we have two additional terms which need to control.  The first one is $\sum_{j=1}^{\nu}\left(\left(\alpha_j^*\right)^p-\alpha_j^*\right)$, which the (possible) difference of the decomposition of $v$ given in Lemma~\ref{lem0002} on the bubbles.  The second one is the interaction between bubbles.  To measure the interaction between bubbles, we denote
\begin{eqnarray}\label{eqn19993}
Q_{j}=e^{-\sqrt{\Lambda_{FS}}\tau_{j}},\quad\varphi_{s^*_i}(t)=e^{-\sqrt{\Lambda_{FS}}|t-s^*_i|}\quad\text{and}\quad Q=e^{-\sqrt{\Lambda_{FS}}\tau}
\end{eqnarray}
where $i,j=1,2,\cdots,\nu$.
\begin{lemma}\label{lem0005}
Let $d\geq2$, $a<0$ and $b=b_{FS}(a)$.  Then for every $\alpha,\beta\in\bbr$ such that $\alpha+\beta\geq0$, we have
\begin{eqnarray}\label{eqnew0001}
\int_{\mathcal{B}_i}\Psi_{i}^{\alpha}\mathcal{U}_i^\beta d\mu\lesssim\left\{\aligned
&Q^{\beta},\quad \alpha>\beta,\\
&Q^{\beta}\left|\log Q\right|,\quad \alpha=\beta,\\
&Q^{\frac{\alpha+\beta}{2}},\quad \alpha<\beta
\endaligned\right.
\end{eqnarray}
and
\begin{eqnarray}\label{eqnew0002}
\int_{\mathcal{C}\backslash\left(\cup_{l=1}^{\nu}\mathcal{B}_l\right)}\Psi_{i}^{\alpha}\Psi_j^\beta d\mu\lesssim Q^{\frac{\alpha+\beta}{2}}.
\end{eqnarray}
\end{lemma}
\begin{proof}
Recall that
\begin{eqnarray*}
s_1<s_2<\cdots<s_{\nu-1}<s_\nu,
\end{eqnarray*}
thus, by \eqref{eq0026}, \eqref{eqn0240}, \eqref{eqn19993} and similar estimates for \eqref{eqn0045}, we have
\begin{eqnarray}\label{eqnewnew0006}
\Psi_{i}^{\alpha}\mathcal{U}_i^\beta\sim e^{-\alpha\sqrt{\Lambda_{FS}}(t-s^*_i)}e^{-\beta\sqrt{\Lambda_{FS}}(s^*_{i+1}-t)}\sim Q_i^{\beta}e^{-(\alpha-\beta)\sqrt{\Lambda_{FS}}(t-s^*_i)}
\end{eqnarray}
in the region
\begin{eqnarray}\label{eqnewnew0010}
\mathcal{B}_{i,+}:=\left[s^*_i, s^*_i+\frac{\tau_i}{2}\right]\times\mathbb{S}^{d-1}
\end{eqnarray}
for all $i=1,2,\cdots,\nu-1$,
while in the region
\begin{eqnarray}\label{eqnewnew0011}
\mathcal{B}_{i,-}:=\left[s_i^*-\frac{\tau_{i-1}}{2}, s_i^*\right]\times\mathbb{S}^{d-1}
\end{eqnarray}
for all $i=2,3,\cdots,\nu$, we have
\begin{eqnarray}\label{eqnewnew0007}
\Psi_{i}^{\alpha}\mathcal{U}_i^\beta\sim e^{-\alpha\sqrt{\Lambda_{FS}}(s^*_i-t)}e^{-\beta\sqrt{\Lambda_{FS}}(t-s^*_{i-1})}\sim Q_{i-1}^{\beta}e^{-(\alpha-\beta)\sqrt{\Lambda_{FS}}(s^*_i-t)}.
\end{eqnarray}
Thus, by direct calculations, we have
\begin{eqnarray*}
\int_{\mathcal{B}_i}\Psi_{i}^{\alpha}\mathcal{U}_i^\beta d\mu&\sim&\int_{\mathcal{B}_{i,+}}\Psi_{i}^{\alpha}\Psi_{i+1}^\beta d\mu+\int_{\mathcal{B}_{i,-}}\Psi_{i}^{\alpha}\Psi_{i-1}^\beta d\mu\\
&\lesssim&\left\{\aligned
&Q^{\beta},\quad \alpha>\beta,\\
&Q^{\beta}\left|\log Q\right|,\quad \alpha=\beta,\\
&Q^{\frac{\alpha+\beta}{2}}\quad \alpha<\beta,
\endaligned\right.
\end{eqnarray*}
which implies that \eqref{eqnew0001} holds true.  To prove \eqref{eqnew0002}, we denote
\begin{eqnarray}\label{eqnewnew0012}
\mathcal{B}_{1,-,*}&=&\left(-\infty, s_1^*\right]\times\mathbb{S}^{d-1}\notag\\
&=&\left(-\infty, s_1^*-\frac{\tau_1}{2}\right)\times\mathbb{S}^{d-1}\cup\left[s_1^*-\frac{\tau_1}{2}, s_1^*\right]\times\mathbb{S}^{d-1}\notag\\
&:=&\left(\mathcal{C}\backslash\cup_{j=1}^{\nu}\mathcal{B}_{j}\right)^-\cup\mathcal{B}_{1,-},
\end{eqnarray}
and
\begin{eqnarray}\label{eqnewnew0013}
\mathcal{B}_{\nu,+,*}&=&\left[s_\nu^*, +\infty\right)\times\mathbb{S}^{d-1}\notag\\
&=&\left[s_\nu^*, s_\nu^*+\frac{\tau_{\nu-1}}{2}\right]\times\mathbb{S}^{d-1}\cup\left(s_\nu^*+\frac{\tau_{\nu-1}}{2}, +\infty\right)\times\mathbb{S}^{d-1}\notag\\
&:=&\mathcal{B}_{\nu,+}\cup\left(\mathcal{C}\backslash\cup_{j=1}^{\nu}\mathcal{B}_{j}\right)^+.
\end{eqnarray}
Then by \eqref{eq0026}, \eqref{eqn0240} and similar estimates for \eqref{eqn0045},
we have
\begin{eqnarray}\label{eqnewnew0008}
\Psi_{i}^{\alpha}\Psi_j^\beta\lesssim \left\{\aligned
&e^{-(\alpha+\beta)\sqrt{\Lambda_{FS}}(s^*_1-t)},\quad \text{in }\mathcal{B}_{1,-,*},\\
&e^{-(\alpha+\beta)\sqrt{\Lambda_{FS}}(t-s^*_\nu)},\quad \text{in }\mathcal{B}_{\nu,+,*}.
\endaligned\right.
\end{eqnarray}
Thus, \eqref{eqnew0002} is also obtained by direct calculations.
\end{proof}

\vskip0.12in

To further decompose the remaining part $\rho_*$ and pick up the leading order term,
let us first consider the following equation:
\begin{eqnarray}\label{eqn0011}
\left\{\aligned&\mathcal{L}(\gamma_{1,ex})=\mathcal{R}_{1,ex}-\sum_{j=1}^{\nu}\Psi_{j}^{p-1}\left(c_{1,ex,j}\partial_t\Psi_{j}+\sum_{l=1}^{d}\varsigma_{1,ex,j,l}w_{i,l}\right),\quad \text{in }\mathcal{C},\\
&\langle \partial_t\Psi_{j}, \gamma_{1,ex}\rangle=\langle w_{j,l}, \gamma_{1,ex}\rangle=0\quad\text{for all }1\leq j\leq\nu\text{ and all }1\leq l\leq d,\endaligned\right.
\end{eqnarray}
where $\mathcal{R}_{1,ex}$ is given by \eqref{eqn0020}.
\begin{lemma}\label{lem0006}
Let $d\geq2$, $a<0$ and $b=b_{FS}(a)$.  Then \eqref{eqn0011} is uniquely solvable.  Moreover, the solution $\gamma_{1,ex}$ in even on $\mathbb{S}^{d-1}$ and satisfies
\begin{eqnarray}\label{eqn0047}
1\gtrsim\left\{\aligned
&\|\gamma_{1,ex}\|_{\sharp},\quad p\geq3,\\
&\|\gamma_{1,ex}\|_{\natural,1,*},\quad 1<p<3,
\endaligned\right.
\end{eqnarray}
where the Lagrange multipliers $\{c_{1,ex,j}\}$ and $\{\varsigma_{1,ex,j,l}\}$ are chosen such that the right hand side of the equation~\eqref{eqn0011} is orthogonal to $\{\Psi_j^{p-1}\partial_t\Psi_{j}\}$ and $\{\Psi_j^{p-1}w_{j,l}\}$ in $L^2(\mathcal{C})$,
\begin{eqnarray*}
\|\gamma_{1,ex}\|_{\sharp}&:=&\sum_{i=2}^{\nu-1}\sup_{\mathcal{B}_{i}}\frac{|\gamma_{1,ex}|}{Q\varphi_{s^*_i}^{1-\sigma}(t)}+\sup_{\mathcal{B}_{1}\cup\left(\mathcal{C}\backslash\cup_{j=1}^{\nu}\mathcal{B}_{j}\right)^-}\frac{|\gamma_{1,ex}|}{Q\varphi_{s^*_1}^{1-\sigma}(t)}\notag\\
&&+\sup_{\mathcal{B}_{\nu}\cup\left(\mathcal{C}\backslash\cup_{j=1}^{\nu}\mathcal{B}_{j}\right)^+}\frac{|\gamma_{1,ex}|}{Q\varphi_{s^*_\nu}^{1-\sigma}(t)}
\end{eqnarray*}
and
\begin{eqnarray*}
\|\gamma_{1,ex}\|_{\natural,1,*}&:=&\sum_{i=1}^{\nu-1}\sup_{\mathcal{B}_{i,+}\backslash\mathcal{B}_{i,0}}\frac{|\gamma_{1,ex}|}{Q_{i}\varphi_{s^*_i}^{p-2}(t)}+\sup_{\left(\mathcal{B}_{\nu,+}\backslash\mathcal{B}_{\nu,0}\right)\cup\left(\mathcal{C}\backslash\cup_{j=1}^{\nu}\mathcal{B}_{j}\right)^+}\frac{|\gamma_{1,ex}|}{Q_{\nu}\varphi_{s^*_\nu}^{1-\sigma}(t)}\notag\\
&&+\sum_{i=2}^{\nu}\sup_{\mathcal{B}_{i,-}\backslash\mathcal{B}_{i,0}}\frac{|\gamma_{1,ex}|}{Q_{i-1}\varphi_{s^*_i}^{p-2}(t)}+\sup_{\left(\mathcal{B}_{1,-}\backslash\mathcal{B}_{1,0}\right)\cup\left(\mathcal{C}\backslash\cup_{j=1}^{\nu}\mathcal{B}_{j}\right)^-}\frac{|\gamma_{1,ex}|}{Q_{1}\varphi_{s^*_1}^{1-\sigma}(t)}\notag\\
&&+\sum_{i=1}^{\nu}\sup_{\mathcal{B}_{i,0}}\frac{|\gamma_{1,ex}|}{Q}
\end{eqnarray*}
with $\sigma>0$ being a small constant which can be taken arbitrary small if necessary and
\begin{eqnarray*}
\mathcal{B}_{i,0}=[s_i^*-1,s_i^*+1]\times\mathbb{S}^{d-1}\cup\left[\frac{s_i^*+s_{i+1}^*}{2}-1,\frac{s_i^*+s_{i+1}^*}{2}+1\right]\times\mathbb{S}^{d-1}.
\end{eqnarray*}
The Lagrange multipliers also satisfy $|c_{1,ex,j}|\lesssim Q$ for all $1\leq j\leq \nu$ and $\varsigma_{1,ex,j,l}=0$ for all $1\leq j\leq \nu$ and all $1\leq l\leq d$.
\end{lemma}
\begin{proof}
By \eqref{eqn0020} and similar estimates for \eqref{eqnewnew0006}, \eqref{eqnewnew0007} and \eqref{eqnewnew0008},
\begin{eqnarray}\label{eqnewnew0009}
\mathcal{R}_{1,ex}\sim \left\{\aligned
&Q_i\varphi_{s^*_i}^{p-2}(t),\quad \mathcal{B}_{i,+}\text{ for }1\leq i\leq\nu-1,\\
&Q_{i-1}\varphi_{s^*_{i-1}}^{p-2}(t),\quad \mathcal{B}_{i,-}\text{ for }2\leq i\leq \nu,\\
&Q_1\varphi_{s^*_{1}}^{p}(t),\quad \mathcal{B}_{1,-,*},\\
&Q_\nu \varphi_{s^*_{\nu}}^{p}(t),\quad \mathcal{B}_{\nu,+,*},
\endaligned\right.
\end{eqnarray}
where $\mathcal{B}_{i,+}$, $\mathcal{B}_{i,-}$, $\mathcal{B}_{1,-,*}$ and $\mathcal{B}_{\nu,+,*}$ are given by \eqref{eqnewnew0010}, \eqref{eqnewnew0011}, \eqref{eqnewnew0012} and \eqref{eqnewnew0013}.
Thus, we have
\begin{eqnarray}\label{eqn19997}
\|\mathcal{R}_{1,ex}\|_{\natural,1}&:=&\sum_{i=1}^{\nu-1}\sup_{\mathcal{B}_{i,+}\backslash\mathcal{B}_{i,0,*}}\frac{|\mathcal{R}_{1,ex}|}{Q_{i}\varphi_{s^*_i}^{p-2}(t)}+\sup_{\left(\mathcal{B}_{\nu,+}\backslash\mathcal{B}_{\nu,0,*}\right)\cup\left(\mathcal{C}\backslash\cup_{j=1}^{\nu}\mathcal{B}_{j}\right)^+}\frac{|\mathcal{R}_{1,ex}|}{Q_{\nu}\varphi_{s^*_\nu}^{1-\sigma}(t)}\notag\\
&&+\sum_{i=2}^{\nu}\sup_{\mathcal{B}_{i,-}\backslash\mathcal{B}_{i,0,*}}\frac{|\mathcal{R}_{1,ex}|}{Q_{i-1}\varphi_{s^*_i}^{p-2}(t)}+\sup_{\left(\mathcal{B}_{1,-}\backslash\mathcal{B}_{1,0,*}\right)\cup\left(\mathcal{C}\backslash\cup_{j=1}^{\nu}\mathcal{B}_{j}\right)^-}\frac{|\mathcal{R}_{1,ex}|}{Q_{1}\varphi_{s^*_1}^{1-\sigma}(t)}\notag\\
&&+\sum_{i=1}^{\nu}\sup_{\mathcal{B}_{i,0,*}}\frac{|\mathcal{R}_{1,ex}|}{Q}\notag\\
&\lesssim&1
\end{eqnarray}
for $1<p<3$ with
\begin{eqnarray}\label{eqnewnew8876}
\mathcal{B}_{i,0,*}=[s_i^*-2,s_i^*+2]\times\mathbb{S}^{d-1}\left[\frac{s_i^*+s_{i+1}^*}{2}-2,\frac{s_i^*+s_{i+1}^*}{2}+2\right]\times\mathbb{S}^{d-1}
\end{eqnarray}
and
\begin{eqnarray}\label{eqn19996}
\|\mathcal{R}_{1,ex}\|_{\sharp}&:=&\sum_{i=2}^{\nu-1}\sup_{\mathcal{B}_{i}}\frac{|\mathcal{R}_{1,ex}|}{Q\varphi_{s^*_i}^{1-\sigma}(t)}+\sup_{\mathcal{B}_{1}\cup\left(\mathcal{C}\backslash\cup_{j=1}^{\nu}\mathcal{B}_{j}\right)^-}\frac{|\mathcal{R}_{1,ex}|}{Q\varphi_{s^*_i}^{1-\sigma}(t)}\notag\\
&&+\sup_{\mathcal{B}_{\nu}\cup\left(\mathcal{C}\backslash\cup_{j=1}^{\nu}\mathcal{B}_{j}\right)^+}\frac{|\mathcal{R}_{1,ex}|}{Q\varphi_{s^*_i}^{1-\sigma}(t)}\notag\\
&\lesssim&1
\end{eqnarray}
for $p\geq3$, where $\sigma>0$ is a small constant which can be taken arbitrary small.
Since
it is easy to check that $\varphi_{s^*_i}^{1-\sigma}(t)$ for all $\sigma\in(0, 1)$ and $\varphi_{s^*_i}^{p-2}(t)$ for all $p\in(1, 3)$ are supersolutions of the equation $\mathcal{L}(\rho)=0$ in $\mathcal{B}_i$ for all $1\leq i\leq \nu$, by Lemma~\ref{lem0001} and applying similar cut-off functions and blow-up arguments for \cite[Proposition~3.1]{LWW2024} to \eqref{eqn0011} and by \eqref{eqn19997} and \eqref{eqn19996}, we can show the existence and uniqueness of $\gamma_{1,ex}$ with the desired estimates~\eqref{eqn0047}.  Moreover, since $\mathcal{R}_{1,ex}$ is even on $\mathbb{S}^{d-1}$, by uniqueness, we also have that $\gamma_{1,ex}$ is even on $\mathbb{S}^{d-1}$.  It remains to estimate the Lagrange multipliers.  By the orthogonal conditions and the oddness of $w_{j,l}$ on $\mathbb{S}^{d-1}$, we have
\begin{eqnarray*}
\sum_{i=1}^{\nu}\left\langle \Psi_i^{p-1}\partial_t\Psi_i,\Psi_j^{p-1}\partial_t\Psi_j\right\rangle_{L^2}c_{1,ex,i}=-\left\langle \mathcal{R}_{1,ex},\Psi_j^{p-1}\partial_t\Psi_j\right\rangle_{L^2}
\end{eqnarray*}
and
\begin{eqnarray*}
\sum_{m=1}^{\nu}\sum_{n=1}^{d}\left\langle \Psi_m^{p-1}w_{m,n},\Psi_j^{p-1}w_{j,l}\right\rangle_{L^2}\varsigma_{1,ex,m,n}=0
\end{eqnarray*}
for all $1\leq j\leq\nu$ and all $1\leq l\leq d$.  Since the matrix $\left[\left\langle \Psi_i^{p-1}\partial_t\Psi_i,\Psi_j^{p-1}\partial_t\Psi_j\right\rangle_{L^2}\right]$ is diagonally dominant by \eqref{eqnewnew1005}, by $p>1$, Lemma~\ref{lem0005} and \eqref{eqnewnew0009},
\begin{eqnarray}\label{eqn1047}
|c_{1,ex,j}|&\sim&\sum_{l=1}^{\nu}\left|\left\langle \mathcal{R}_{1,ex},\Psi_l^{p-1}\partial_t\Psi_l\right\rangle_{L^2}\right|\notag\\
&\lesssim&\sum_{i=1}^{l-1}\int_{\mathcal{B}_{i,+}}Q_{i}\Psi_i^{p-2}\Psi_l^{p}d\mu+\sum_{i=l+1}^{\nu}\int_{\mathcal{B}_{i,-}}Q_{i-1}\Psi_{i}^{p-2}\Psi_l^{p}d\mu\notag\\
&&+p\int_{\mathcal{B}_l}\Psi_l^{2p-1}l\mathcal{U}_ld\mu\notag\\
&\lesssim&\int_{\mathcal{B}_{j,+}}\Psi_j^{2p-1}\Psi_{j+1}d\mu
+\int_{\mathcal{B}_{j,-}}\Psi_j^{2p-1}\Psi_{j-1}d\mu\notag\\
&\lesssim&Q
\end{eqnarray}
for all $1\leq j\leq \nu$.
Moreover, by \eqref{eqnewnew1005} and the orthogonal conditions of $\{w_l\}$ on $\mathbb{S}^{d-1}$, the matrix $\left[\left\langle \Psi_m^{p-1}w_{m,n},\Psi_j^{p-1}w_{j,l}\right\rangle_{L^2}\right]$ is also diagonally dominant.  Thus, it is also easy to see that $\varsigma_{1,ex,j,l}=0$ for all $1\leq j\leq \nu$ and all $1\leq l\leq d$.
\end{proof}

\vskip0.12in

We next consider the following equation:
\begin{eqnarray}\label{eqn0012}
\left\{\aligned&\mathcal{L}(\gamma_{1,j})=\mathcal{R}_{1,j}-\sum_{i=1}^{\nu}\Psi_{i}^{p-1}\left(c_{1,j,i}\partial_t\Psi_{i}+\sum_{l=1}^{d}\varsigma_{1,j,i,l}w_{i,l}\right),\quad \text{in }\mathcal{C},\\
&\langle \partial_t\Psi_{j}, \gamma_{1,j}\rangle=\langle w_{j,l}, \gamma_{1,j}\rangle=0\quad\text{for all }1\leq j\leq\nu\text{ and all }1\leq l\leq d,\endaligned\right.
\end{eqnarray}
where $\mathcal{R}_{1,j}$ is given by \eqref{eqn0020}.
\begin{lemma}\label{lem0007}
Let $d\geq2$, $a<0$ and $b=b_{FS}(a)$.  Then \eqref{eqn0012} is uniquely solvable.  Moreover, the solution $\gamma_{1,j}$ is even in terms of $t-s_j^*$ and satisfies
\begin{eqnarray}\label{eqn0048}
\left|(\alpha_j^*)^{p-1}-1\right|\gtrsim\sup_{(t,\theta)\in\mathcal{C}}\frac{|\gamma_{1,j}|}{\varphi_{s_j^*}^{1-\sigma}(t)},
\end{eqnarray}
where the Lagrange multipliers $\{c_{1,j,i}\}$ and $\{\varsigma_{1,j,i,l}\}$ are chosen such that the right hand side of the equation~\eqref{eqn0012} is orthogonal to $\{\Psi_j^{p-1}\partial_t\Psi_{j}\}$ and $\{\Psi_j^{p-1}w_{j,l}\}$ in $L^2(\mathcal{C})$ and $\sigma>0$ is chosen such that
\begin{eqnarray}\label{eqnewnew3345}
\left(Q+\beta_*^2+\sum_{j=1}^{\nu}\left|(\alpha_j^*)^{p-1}-1\right|\right)\leq\frac18\mathcal{U}^{\sigma}\text{ in }\overline{\mathcal{B}}_*=[s_1^*-\tau_1,s_\nu^*+\tau_{\nu-1}]\times\mathbb{S}^{d-1}.
\end{eqnarray}
The Lagrange multipliers also satisfy $\varsigma_{1,j,i,l}=0$ and
\begin{eqnarray*}
\left|c_{1,j,i}\right|\lesssim \left\{\aligned
&\left|(\alpha_j^*)^{p-1}-1\right|Q^{p}\left|\log Q\right|,\quad i\not=j,\\
& \left|(\alpha_j^*)^{p-1}-1\right|Q^{2p}\left|\log Q\right|^2,\quad i=j,\\
\endaligned\right.
\end{eqnarray*}
for all $1\leq j\leq \nu$, $1\leq l\leq d$ and all $1\leq i\leq\nu$.
\end{lemma}
\begin{proof}
Again, since $p>1$, by Lemma~\ref{lem0001} and applying similar blow-up arguments for \cite[Lemma~5.4]{WW2022} to \eqref{eqn0012} and by \eqref{eq0026} and \eqref{eqn0020}, we can show the existence and uniqueness of $\gamma_{1,j}$ with the desired estimates~\eqref{eqn0048}.  Moreover, since $\mathcal{R}_{1,j}$ is even in terms of $t-s_j^*$, by uniqueness, $\gamma_{1,j}$ is also even in terms of $t-s_j^*$.  For the estimates of the Lagrange multipliers, by the orthogonal conditions and the oddness of $w_{j,l}$ on $\mathbb{S}^{d-1}$, we have
\begin{eqnarray*}
\sum_{i=1}^{\nu}\left\langle \Psi_i^{p-1}\partial_t\Psi_i,\Psi_j^{p-1}\partial_t\Psi_j\right\rangle_{L^2}c_{1,j,i}=-\left\langle \mathcal{R}_{1,j},\Psi_l^{p-1}\partial_t\Psi_l\right\rangle_{L^2}
\end{eqnarray*}
and
\begin{eqnarray*}
\sum_{m=1}^{\nu}\sum_{n=1}^{d}\left\langle \Psi_m^{p-1}w_{m,n},\Psi_j^{p-1}w_{j,l}\right\rangle_{L^2}\varsigma_{1,j,m,n}=0
\end{eqnarray*}
for all $1\leq j\leq\nu$ and all $1\leq l\leq d$.  Again, the matrix $\left[\left\langle \Psi_i^{p-1}\partial_t\Psi_i,\Psi_j^{p-1}\partial_t\Psi_j\right\rangle_{L^2}\right]$ is diagonally dominant by \eqref{eqnewnew1005}.  Note that by the oddness of $\partial_t\Psi$ in $\bbr$, we have
\begin{eqnarray*}
\left\langle \mathcal{R}_{1,j},\Psi_j^{p-1}\partial_t\Psi_j\right\rangle_{L^2}=0.
\end{eqnarray*}
Thus,
\begin{eqnarray*}
\left|c_{1,j,i}\right|\lesssim\left\{\aligned
&\left|(\alpha_j^*)^{p-1}-1\right|\int_{\mathcal{C}}\Psi_i^{p}\Psi_j^{p}d\mu
\lesssim \left|(\alpha_j^*)^{p-1}-1\right|Q^{p}\left|\log Q\right|,\quad i\not=j,\\
&\left|(\alpha_j^*)^{p-1}-1\right|\left(\int_{\mathcal{C}}\Psi_i^{p}\Psi_j^{p}d\mu\right)^2
\lesssim \left|(\alpha_j^*)^{p-1}-1\right|Q^{2p}\left|\log Q\right|^2,\quad i=j,
\endaligned\right.
\end{eqnarray*}
for all $1\leq i\leq \nu$.
Moreover, by \eqref{eqnewnew1005} and the orthogonal conditions of $\{w_l\}$ on $\mathbb{S}^{d-1}$, the matrix $\left[\left\langle \Psi_m^{p-1}w_{m,n},\Psi_j^{p-1}w_{j,l}\right\rangle_{L^2}\right]$ is also diagonally dominant.  Thus,
it is also easy to see that $\varsigma_{1,j,i,l}=0$
for all $1\leq i,j\leq \nu$ and all $1\leq l\leq d$.
\end{proof}

\vskip0.12in

We also need to consider the following equation:
\begin{eqnarray}\label{eqn0013}
\left\{\aligned&\mathcal{L}(\overline{\gamma}_{2,ex})=\mathcal{R}_{2,ex}-\sum_{j=1}^{\nu}\Psi_{j}^{p-1}\left(c_{2,ex,j}\partial_t\Psi_{j}+\sum_{l=1}^{d}\varsigma_{2,ex,j,l}w_{j,l}\right),\quad \text{in }\mathcal{C},\\
&\langle \partial_t\Psi_{j}, \overline{\gamma}_{2,ex}\rangle=\langle w_{j,l}, \overline{\gamma}\rangle=0\quad\text{for all }1\leq j\leq\nu\text{ and all }1\leq l\leq d,\endaligned\right.
\end{eqnarray}
where $\mathcal{R}_{2,ex}$ is given by \eqref{eqn0021}.
\begin{lemma}\label{lem0008}
Let $d\geq2$, $a<0$ and $b=b_{FS}(a)$.  Then \eqref{eqn0013} is uniquely solvable.  Moreover, the solution $\overline{\gamma}_{2,ex}$ is odd on $\mathbb{S}^{d-1}$ and satisfies
\begin{eqnarray}\label{eqn2047}
\beta_*\gtrsim\left\{\aligned
&\|\overline{\gamma}_{2,ex}\|_{\sharp},\quad p\geq\frac{7}{3},\\
&\|\overline{\gamma}_{2,ex}\|_{\natural,2,*},\quad 1<p<\frac{7}{3},
\endaligned\right.
\end{eqnarray}
where the Lagrange multipliers $\{c_{2,ex,j}\}$ and $\{\varsigma_{2,ex,j,l}\}$ are chosen such that the right hand side of the equation~\eqref{eqn0013} is orthogonal to $\{\Psi_j^{p-1}\partial_t\Psi_{j}\}$ and $\{\Psi_j^{p-1}w_{j,l}\}$ in $L^2(\mathcal{C})$,  $\|\cdot\|_{\sharp}$ and $\mathcal{B}_{i,0}$ are given in Lemma~\ref{lem0006}, and
\begin{eqnarray*}
\|\overline{\gamma}_{2,ex}\|_{\natural,2,*}&:=&\sum_{i=1}^{\nu-1}\sup_{\mathcal{B}_{i,+}\backslash\mathcal{B}_{i,0}}\frac{|\overline{\gamma}_{2,ex}|}{Q_{i}\varphi_{s^*_i}^{\frac{3p-5}{2}}(t)}+\sup_{\left(\mathcal{B}_{\nu,+}\backslash\mathcal{B}_{\nu,0}\right)\cup\left(\mathcal{C}\backslash\cup_{j=1}^{\nu}\mathcal{B}_{j}\right)^+}\frac{|\overline{\gamma}_{2,ex}|}{Q_{\nu}\varphi_{s^*_\nu}^{1-\sigma}(t)}\notag\\
&&+\sum_{i=2}^{\nu}\sup_{\mathcal{B}_{i,-}\backslash\mathcal{B}_{i,0}}\frac{|\overline{\gamma}_{2,ex}|}{Q_{i-1}\varphi_{s^*_i}^{\frac{3p-5}{2}}(t)}+\sup_{\left(\mathcal{B}_{1,-}\backslash\mathcal{B}_{1,0,*}\right)\cup\left(\mathcal{C}\backslash\cup_{j=1}^{\nu}\mathcal{B}_{j}\right)^-}\frac{|\overline{\gamma}_{2,ex}|}{Q_{1}\varphi_{s^*_1}^{1-\sigma}(t)}\notag\\
&&+\sum_{i=1}^{\nu}\sup_{\mathcal{B}_{i,0}}\frac{|\overline{\gamma}_{2,ex}|}{Q}
\end{eqnarray*}
with $\sigma>0$ being a small constant which can be taken arbitrary small if necessary.
The Lagrange multipliers also satisfy $c_{2,ex,j}=0$ for all $1\leq j\leq \nu$ and $|\varsigma_{2,ex,j,l}|\lesssim\beta_*Q$ for all $1\leq j\leq \nu$ and $1\leq l\leq d$.
\end{lemma}
\begin{proof}
Similar to \eqref{eqn0045}, by \eqref{eq0026}, \eqref{eqn0040} and \eqref{eqn0021}, we have
\begin{eqnarray}\label{eqn3147}
\left|\mathcal{R}_{2,ex}\right|\lesssim\beta_*\left(\sum_{i=1}^{\nu}\Psi_i^{\frac{3(p-1)}{2}}\mathcal{U}_i\chi_{\mathcal{B}_i}+\mathcal{U}^{\frac{3p-1}{2}}\chi_{\mathcal{C}\backslash\left(\cup_{i=1}^{\nu}\mathcal{B}_i\right)}\right).
\end{eqnarray}
Thus, similar to \eqref{eqn19997} and \eqref{eqn19996},
\begin{eqnarray}\label{eqn29997}
\|\mathcal{R}_{2,ex}\|_{\natural,2}&:=&\sum_{i=1}^{\nu-1}\sup_{\mathcal{B}_{i,+}\backslash\mathcal{B}_{i,0,*}}\frac{|\mathcal{R}_{2,ex}|}{Q_{i}\varphi_{s^*_i}^{\frac{3p-5}{2}}(t)}+\sup_{\left(\mathcal{B}_{\nu,+}\backslash\mathcal{B}_{\nu,0,*}\right)\cup\left(\mathcal{C}\backslash\cup_{j=1}^{\nu}\mathcal{B}_{j}\right)^+}\frac{|\mathcal{R}_{2,ex}|}{Q_{\nu}\varphi_{s^*_\nu}^{1-\sigma}(t)}\notag\\
&&+\sum_{i=2}^{\nu}\sup_{\mathcal{B}_{i,-}\backslash\mathcal{B}_{i,0,*}}\frac{|\mathcal{R}_{2,ex}|}{Q_{i-1}\varphi_{s^*_i}^{\frac{3p-5}{2}}(t)}+\sup_{\left(\mathcal{B}_{1,-}\backslash\mathcal{B}_{1,0,*}\right)\cup\left(\mathcal{C}\backslash\cup_{j=1}^{\nu}\mathcal{B}_{j}\right)^-}\frac{|\mathcal{R}_{2,ex}|}{Q_{1}\varphi_{s^*_1}^{1-\sigma}(t)}\notag\\
&&+\sum_{i=1}^{\nu}\sup_{\mathcal{B}_{i,0,*}}\frac{|\mathcal{R}_{2,ex}|}{Q}\notag\\
&\lesssim&\beta_*
\end{eqnarray}
for $1<p<\frac{7}{3}$ and $\|\mathcal{R}_{2,ex}\|_{\sharp}\lesssim\beta_*$
for $p\geq\frac{7}{3}$, where $\sigma>0$ is a small constant which can be taken arbitrary small if necessary and $\mathcal{B}_{i,0,*}$ is given by \eqref{eqnewnew8876}.  Then by Lemma~\ref{lem0001} and applying similar cut-off functions and blow-up arguments for \cite[Proposition~3.1]{LWW2024} to \eqref{eqn0013}, we can show the existence and uniqueness of $\overline{\gamma}_{2,ex}$ with the desired estimates~\eqref{eqn2047}.  On the other hand, by \eqref{eqn0040} and \eqref{eqn0021}, we have
\begin{eqnarray*}
\mathcal{R}_{2,ex}&=&p\sum_{j=1}^{\nu}\left(\mathcal{U}^{p-1}-\left(\Psi_j^*\right)^{p-1}\right)\mathcal{V}_j\\
&=&\sum_{l=1}^{d}\left(p\sum_{j=1}^{\nu}\left(\mathcal{U}^{p-1}-\left(\Psi_j^*\right)^{p-1}\right)\Psi_{j}^{\frac{p+1}{2}}\beta_{j,l}^*\right)\theta_l,
\end{eqnarray*}
which is odd on $\mathbb{S}^{d-1}$.  Thus, by uniqueness of $\overline{\gamma}_{2,ex}$, we know that $\overline{\gamma}_{2,ex}$ is also odd on $\mathbb{S}^{d-1}$.
It remains to estimate the Lagrange multipliers.  By the orthogonal conditions and the oddness of $w_{j,l}$ on $\mathbb{S}^{d-1}$, we have
\begin{eqnarray*}
\sum_{i=1}^{\nu}\left\langle \Psi_i^{p-1}\partial_t\Psi_i,\Psi_j^{p-1}\partial_t\Psi_j\right\rangle_{L^2}c_{2,ex,i}=0
\end{eqnarray*}
and
\begin{eqnarray*}
\sum_{m=1}^{\nu}\sum_{n=1}^{d}\left\langle \Psi_m^{p-1}w_{m,n},\Psi_j^{p-1}w_{j,l}\right\rangle_{L^2}\varsigma_{2,ex,m,n}=-\left\langle \mathcal{R}_{2,ex},\Psi_j^{p-1}w_{j,l}\right\rangle_{L^2}
\end{eqnarray*}
for all $1\leq j\leq\nu$ and $1\leq l\leq d$.  Again, the matrix $\left[\left\langle \Psi_i^{p-1}\partial_t\Psi_i,\Psi_j^{p-1}\partial_t\Psi_j\right\rangle_{L^2}\right]$ is diagonally dominant by \eqref{eqnewnew1005}.  Thus, $c_{2,ex,j}=0$ for all $1\leq j\leq\nu$.  Moreover, the matrix $\left[\left\langle \Psi_m^{p-1}w_{m,n},\Psi_j^{p-1}w_{j,l}\right\rangle_{L^2}\right]$ is also diagonally dominant by \eqref{eqnewnew1005} and the orthogonal conditions of $\{w_l\}$ on $\mathbb{S}^{d-1}$.  Thus, by Lemma~\ref{lem0005} and \eqref{eqn3147}, we also have
\begin{eqnarray*}
\left|\varsigma_{2,ex,j,l}\right|&=&\left|\left\langle \mathcal{R}_{2,ex}, \Psi_j^{p-1}w_{j,l}\right\rangle_{L^2}\right|\notag\\
&\lesssim&\beta_*\left\langle\Psi_{j-1}^{\frac{3(p-1)}{2}}\mathcal{U}_{j-1}\chi_{\mathcal{B}_{j-1,+}}+\Psi_{j+1}^{\frac{3(p-1)}{2}}\mathcal{U}_{j+1}\chi_{\mathcal{B}_{j+1,-}},\Psi_{j}^{\frac{3p-1}{2}}\right\rangle_{L^2}\notag\\
&&+\beta_*\left\langle\Psi_{j}^{3p-2},\mathcal{U}_{j}\right\rangle_{L^2(\mathcal{B}_{j})}\notag\\
&\sim&\beta_*Q.
\end{eqnarray*}
for all $1\leq j\leq \nu$ and all $1\leq l\leq d$.
\end{proof}

\vskip0.12in

We finally consider the following equation:
\begin{eqnarray}\label{eqn0015}
\left\{\aligned&\mathcal{L}(\gamma_{\mathcal{N},led,*})=\mathcal{N}_{led}-\sum_{j=1}^{\nu}\Psi_{j}^{p-1}\left(c_{\mathcal{N},led,j}\partial_t\Psi_{j}+\sum_{l=1}^{d}\varsigma_{\mathcal{N},led,j,l}w_{j,l}\right),\quad \text{in }\mathcal{C},\\
&\langle \partial_t\Psi_{j}, \gamma_{\mathcal{N},led,*}\rangle=\langle w_{j,l}, \gamma_{\mathcal{N},led,*}\rangle=0\quad\text{for all }1\leq j\leq\nu\text{ and all }1\leq l\leq d,\endaligned\right.
\end{eqnarray}
where
\begin{eqnarray}\label{eqn3045}
\mathcal{N}_{led}=\sum_{j=1}^{\nu}\left(\Psi_j^*\right)^{p-3}\mathcal{V}_j^2\left(A_p\Psi_j^*+B_p\mathcal{V}_j\right)\chi_{\mathcal{B}_j}+\mathcal{U}^{p-3}\mathcal{V}^{2}\left(A_p\mathcal{U}+B_p\mathcal{V}\right)\chi_{\mathcal{C}\backslash\left(\cup_{j=1}^{\nu}\mathcal{B}_j\right)}.
\end{eqnarray}
\begin{lemma}\label{lem0009}
Let $d\geq2$, $a<0$ and $b=b_{FS}(a)$.  Then \eqref{eqn0015} is uniquely solvable.  Moreover, the solution $\gamma_{\mathcal{N},led,*}$ satisfies
\begin{eqnarray}\label{eqn0049}
\beta_*^2\gtrsim\sup_{(t,\theta)\in\mathcal{C}}\frac{|\gamma_{\mathcal{N},led,*}|}{\sum_{j=1}^{\nu}\Psi_{j}^{1-\sigma}(t)}
\end{eqnarray}
where the Lagrange multipliers $\{c_{\mathcal{N},led,j}\}$ and $\{\varsigma_{\mathcal{N},led,j,l}\}$ are chosen such that the right hand side of the equation~\eqref{eqn0015} is orthogonal to $\{\Psi_j^{p-1}\partial_t\Psi_{j}\}$ and $\{\Psi_j^{p-1}w_{j,l}\}$ in $L^2(\mathcal{C})$ and $\sigma>0$ is chosen to satisfy \eqref{eqnewnew3345}.
The Lagrange multipliers also satisfy $|c_{\mathcal{N},led,j}|\lesssim\beta_*^2 Q^p$ and $\left|\varsigma_{\mathcal{N},led,j,l}\right|\lesssim\beta_*^3$
for all $1\leq j\leq \nu$, $1\leq l\leq d$ and $1\leq i\leq\nu$.
\end{lemma}
\begin{proof}
Similar to \eqref{eqn3147}, by \eqref{eqn3045}, we have
\begin{eqnarray}\label{eqn3045}
\left|\mathcal{N}_{led}\right|\lesssim\sum_{j=1}^{\nu}\beta_*^2\Psi_j^{2p-1}\chi_{\mathcal{B}_j}+\beta_*^2\mathcal{U}^{2p-1}\chi_{\mathcal{C}\backslash\left(\cup_{j=1}^{\nu}\mathcal{B}_j\right)}.
\end{eqnarray}
Since $p>1$, by Lemma~\ref{lem0001} and applying similar blow-up arguments for \cite[Lemma~5.4]{WW2022} to \eqref{eqn0015} and by \eqref{eq0026}, we can show the existence of $\gamma_{\mathcal{N},led,*}$ with the desired estimates~\eqref{eqn0049}.
It remains to estimate the Lagrange multipliers.  By the oddness of $w_{l}$ on $\mathbb{S}^{d-1}$ and $\partial_{t}\Psi$ in $\bbr$, we have
\begin{eqnarray*}
&&\sum_{i=1}^{\nu}\left\langle \Psi_i^{p-1}\partial_t\Psi_i,\Psi_j^{p-1}\partial_t\Psi_j\right\rangle_{L^2}c_{\mathcal{N},led,i}\\
&=&-\left\langle \sum_{i=1}^{\nu}\left(\Psi_i^*\right)^{p-2}\mathcal{V}_i^2\chi_{\mathcal{B}_i}+\mathcal{U}^{p-2}\mathcal{V}^{2}\chi_{\mathcal{C}\backslash\left(\cup_{i=1}^{\nu}\mathcal{B}_i\right)}, \Psi_j^{p-1}\partial_{t}\Psi_j\right\rangle_{L^2}
\end{eqnarray*}
and
\begin{eqnarray*}
&&\sum_{m=1}^{\nu}\sum_{n=1}^{d}\left\langle \Psi_m^{p-1}w_{m,n},\Psi_j^{p-1}w_{j,l}\right\rangle_{L^2}\varsigma_{\mathcal{N},led,m,n}\\
&=&-\left\langle \sum_{i=1}^{\nu}\left(\Psi_i^*\right)^{p-3}\mathcal{V}_i^3\chi_{\mathcal{B}_i}+\mathcal{U}^{p-3}\mathcal{V}^{3}\chi_{\mathcal{C}\backslash\left(\cup_{i=1}^{\nu}\mathcal{B}_i\right)}, \Psi_j^{p-1}w_{j,l}\right\rangle_{L^2}
\end{eqnarray*}
for all $1\leq j\leq\nu$ and all $1\leq l\leq d$.
Thus, similar to \eqref{eqn1047}, by Lemma~\ref{lem0005} and the oddness of $\partial_{t}\Psi$ in $\bbr$, we have
\begin{eqnarray}\label{eqn1049}
|c_{\mathcal{N},led,j}|&\lesssim&\int_{\mathcal{B}_{j-1}}\Psi_{j-1}^{p-2}\mathcal{V}_{j-1}^2\Psi_j^pd\mu+\int_{\mathcal{B}_{j}}\Psi_{j}^{p-2}\mathcal{V}_{j}^2\Psi_j^{p-1}\partial_t\Psi_jd\mu\notag\\
&\lesssim&\beta_*^2Q^{p}+\beta_*^2\int_{s_j^*+\frac{\tau}{2}}^{s_j^*+\frac{\tau_j}{2}}\Psi_j^{3p-1}dt\notag\\
&\lesssim&\beta_*^2Q^{p}
\end{eqnarray}
and by the oddness of $w_{l}$ on $\mathbb{S}^{d-1}$, we have
\begin{eqnarray}\label{eqn3049}
\left|\varsigma_{\mathcal{N},led,j,l}\right|\lesssim\left|\left\langle\Psi_j^{p-3}\mathcal{V}_j^3\chi_{\mathcal{B}_j}+\mathcal{U}^{p-3}\mathcal{V}^{3}\chi_{\mathcal{C}\backslash\left(\cup_{i=1}^{\nu}\mathcal{B}_i\right)}, w_{j,l}\right\rangle_{L^2}\right|\lesssim\beta_*^3
\end{eqnarray}
for all $1\leq j\leq \nu$ and all $1\leq l\leq d$.
\end{proof}

\vskip0.12in

By Lemmas~\ref{lem0006}, \ref{lem0007}, \ref{lem0008} and \ref{lem0009}, we have picked up all possible leading order terms of $\rho_*$ in terms of $Q$, $\beta_*$ and $\sum_{j=1}^{\nu}\left|(\alpha_j^*)^{p-1}-1\right|$.   Now, let
\begin{eqnarray*}
\rho_{**,0}=\rho_*-\gamma_{1,ex}-\overline{\gamma}_{2,ex}-\sum_{j=1}^{\nu}\gamma_{1,j}-\gamma_{\mathcal{N},led,*}.
\end{eqnarray*}
Since $\gamma_{1,ex}$, $\gamma_{1,j}$ and $\gamma_{\mathcal{N},led,*}$ may have projections on span$\{\Psi_{j}\}$, we further decompose $\rho_{**,0}=\sum_{j=1}^{\nu}\alpha_{j,0}^{**}\Psi_{j}+\rho_{**,0}^{\perp}$,
where $\{\alpha_{j,0}^{**}\}$ is chosen such that
\begin{eqnarray*}
\left\langle \rho_{**,0}^{\perp}, \Psi_{j}\right\rangle=\left\langle \rho_{**,0}^{\perp}, \partial_t\Psi_{j}\right\rangle=\left\langle \rho_{**,0}^{\perp}, w_{j,l}\right\rangle=0
\end{eqnarray*}
for all $1\leq j\leq \nu$ and all $1\leq l\leq d$.  On the other hand, by the orthogonal conditions of $\rho_*$ given in \eqref{eq0014} and Lemmas~\ref{lem0006}, \ref{lem0007} and \ref{lem0008}, we have
\begin{eqnarray}\label{eqn0052}
\sum_{l=1}^{\nu}\left\langle \Psi_l, \Psi_{j}\right\rangle\alpha_{l,0}^{**}=-\left\langle \gamma_{1,ex}, \Psi_{j}\right\rangle-\sum_{i=1}^{\nu}\left\langle\gamma_{1,i}, \Psi_{j}\right\rangle-\left\langle \gamma_{\mathcal{N},led,*}, \Psi_{j}\right\rangle
\end{eqnarray}
for all $1\leq j\leq\nu$.  Since $\Psi$ is a solution of \eqref{eq0006}, by Lemmas~\ref{lem0005}, \ref{lem0006}, \ref{lem0007} and \ref{lem0009},
\begin{eqnarray*}
\left\{\aligned
&\left|\left\langle \gamma_{1,ex}, \Psi_{j}\right\rangle\right|=\left|\left\langle \gamma_{1,ex}, \Psi_{j}^{p}\right\rangle_{L^2}\right|\lesssim Q,\\
&\left|\left\langle \sum_{i=1}^{\nu}\gamma_{1,i}, \Psi_{j}\right\rangle\right|=\left|\left\langle \sum_{i=1}^{\nu}\gamma_{1,i}, \Psi_{j}^{p}\right\rangle_{L^2}\right|\lesssim\sum_{i=1}^{\nu}\left|\left(\alpha_i^{*}\right)^{p-1}-1\right|,\\
&\left|\left\langle \gamma_{\mathcal{N},led,*}, \Psi_{j}\right\rangle\right|=\left|\left\langle \gamma_{\mathcal{N},led,*}, \Psi_{j}^{p}\right\rangle_{L^2}\right|\lesssim\beta_*^2.
\endaligned\right.
\end{eqnarray*}
Intersecting these estimates into \eqref{eqn0052}, we have
\begin{eqnarray}\label{eqn0068}
\sum_{j=1}^{\nu}\left|\alpha_{j,0}^{**}\right|\lesssim Q+\sum_{j=1}^{\nu}\left|\left(\alpha_j^{*}\right)^{p-1}-1\right|+\beta_*^2.
\end{eqnarray}
Moreover, by \eqref{eq0014}, \eqref{eqn0011}, \eqref{eqn0012} and \eqref{eqn0015}, $\rho_{**,0}^{\perp}$ satisfies
\begin{eqnarray}\label{eqn5114}
\left\{\aligned&\mathcal{L}(\rho_{**,0}^{\perp})=f+\mathcal{R}_{new},\quad \text{in }\mathcal{C},\\
&\langle \Psi_{j}, \rho_{**}^{\perp}\rangle=\langle \partial_t\Psi_{j}, \rho_{**,0}^{\perp}\rangle=\langle w_{j,l}, \rho_{**,0}^{\perp}\rangle=0\quad\text{for }1\leq j\leq\nu\text{ and }1\leq l\leq d,\endaligned\right.
\end{eqnarray}
where by \eqref{eq0014} and Lemmas~\ref{lem0004}, \ref{lem0006} \ref{lem0007}, \ref{lem0008} and \ref{lem0009},
\begin{eqnarray}
\mathcal{R}_{new}&=&\sum_{i=1}^{\nu}\Psi_{i}^{p-1}\left((c_{1,ex,i}+c_{1,j,i}+c_{3,led,i})\partial_t\Psi_{i}+\sum_{l=1}^{d}(\varsigma_{2,ex,i,l}+\varsigma_{3,led,i,l})w_{i,l}\right)\notag\\
&&+\sum_{j=1}^{\nu}\left(2A_p\left(\Psi_j^*\right)^{p-2}\mathcal{V}_j+3B_p\left(\Psi_j^*\right)^{p-3}\mathcal{V}_j^2\right)\rho_*\chi_{\mathcal{B}_j}+\sum_{j=1}^{\nu}\mathcal{R}_{2,j}+\mathcal{N}_{rem}\notag\\
&&+\sum_{j=1}^{\nu}2A_p\left(\mathcal{U}^{p-2}\mathcal{V}-\sum_{j=1}^{\nu}\left(\Psi_j^*\right)^{p-2}\mathcal{V}_j\right)\rho_*\chi_{\mathcal{B}_j}\notag\\
&&+\sum_{j=1}^{\nu}\mathcal{O}\left(\beta_*^2\mathcal{U}_j\Psi_{j}^{2p-3}\left(\Psi_j+\rho_*\right)\right)\chi_{\mathcal{B}_j}+2A_p\mathcal{U}^{p-2}\mathcal{V}\rho_*\chi_{\mathcal{C}\backslash\cup_{j=1}^{\nu}\mathcal{B}_j}\notag\\
&&+\sum_{j=1}^{\nu}\alpha_{j,0}^{**}\left(\mathcal{U}^{p-1}-\Psi_j^{p-1}\right)\Psi_{j},\label{eqn5061}
\end{eqnarray}
where $A_p$ and $B_p$ are given in Lemma~\ref{lem0003}.
Even though we have picked up all possible leading order terms of $\rho_*$ in terms of $Q$, $\beta_*$ and $\sum_{j=1}^{\nu}\left|(\alpha_j^*)^{p-1}-1\right|$, the data $\mathcal{R}_{new}$, given by \eqref{eqn5061}, is not good enough to control $\rho_{**,0}^{\perp}$ in a desired size.  This is mainly because the optimal Bianchi-Egnell stability of the CKN inequality for $d\geq2$, $a<0$ and $b=b_{FS}(a)$ is quartic, as shown in \cite[Theorem~1]{FP2024}, which implies that we only have the opportunity to control the terms of order $\beta_*^4$ from above.  Thus, we need to ensure that the (possible) singular part should be smaller than $\beta_*^4$.  Keep this in mind, we need to eliminate the lower order terms (compared to the $\beta_*^4$ terms) in the data $\mathcal{R}_{new}$.  For this purpose, we need the following.
\begin{lemma}\label{lem0010}
Let $d\geq2$, $a<0$ and $b=b_{FS}(a)$.  Then we can decompose
\begin{eqnarray*}
\gamma_{\mathcal{N},led,*}=\gamma_{\mathcal{N},led,j}+\gamma_{\mathcal{N},led,rem,j,*},
\end{eqnarray*}
where $\gamma_{\mathcal{N},led,j}$ is even in terms of $t-s_j^*$ and satisfies the equation
\begin{eqnarray*}
\left\{\aligned&\mathcal{L}(\gamma_{\mathcal{N},led,j})=\mathcal{N}_{led,j}-\sum_{i=1}^{\nu}\Psi_{i}^{p-1}\left(c_{\mathcal{N},led,j,i}\partial_t\Psi_{i}+\sum_{l=1}^{d}\varsigma_{\mathcal{N},led,j,i,l}w_{i,l}\right),\quad \text{in }\mathcal{C},\\
&\langle \partial_t\Psi_{i}, \gamma_{\mathcal{N},led,j}\rangle=\langle w_{i,l}, \gamma_{\mathcal{N},led,j}\rangle=0\quad\text{for all }1\leq i\leq\nu\text{ and all }1\leq l\leq d,\endaligned\right.
\end{eqnarray*}
with $\mathcal{N}_{led,j}=\left(\Psi_j^*\right)^{p-3}\mathcal{V}_j^2\left(\Psi_j^*+\mathcal{V}_j\right)$ and
\begin{eqnarray*}
\beta_*^2\gtrsim\sup_{(t,\theta)\in\mathcal{C}}\frac{|\gamma_{\mathcal{N},led,j}|}{\Psi_{j}^{1-\sigma}(t)}+\sup_{(t,\theta)\in\mathcal{C}}\frac{|\gamma_{\mathcal{N},led,rem,j,*}|}{\sum_{i=1;i\not=j}^{\nu}\Psi_{i}^{1-\sigma}(t)}.
\end{eqnarray*}
Moreover, we can decompose $\gamma_{\mathcal{N},led,j}=\gamma_{\mathcal{N},led,j,*}+\gamma_{\mathcal{N},led,j,**}$ with $\gamma_{\mathcal{N},led,j,*}$ being even on $\mathbb{S}^{d-1}$, $\gamma_{\mathcal{N},led,j,**}$ being odd on $\mathbb{S}^{d-1}$ and
\begin{eqnarray*}
1\gtrsim\sup_{(t,\theta)\in\mathcal{C}}\frac{|\gamma_{\mathcal{N},led,j,*}|}{\beta_*^2\Psi_{j}^{1-\sigma}(t)}+\sup_{(t,\theta)\in\mathcal{C}}\frac{|\gamma_{\mathcal{N},led,j,**}|}{\beta_*^3\Psi_{i}^{1-\sigma}(t)},
\end{eqnarray*}
where $\sigma>0$ is chosen to satisfy \eqref{eqnewnew3345}.
\end{lemma}
\begin{proof}
The proof is similar to that of Lemma~\ref{lem0009} so we omit it.  Moreover, since $\mathcal{N}_{led,j}$ is even in terms of $t-s_j^*$, by uniqueness, we also have that $\gamma_{\mathcal{N},led,j}$ is even in terms of $t-s_j^*$.  On the other hand, the decomposition of $\gamma_{\mathcal{N},led,j}$ is generated by the data $\mathcal{N}_{led,j}=\left(\Psi_j^*\right)^{p-2}\mathcal{V}_j^2+\left(\Psi_j^*\right)^{p-3}\mathcal{V}_j^3$.  The first part $\gamma_{\mathcal{N},led,j,*}$ is obtained by the data $\left(\Psi_j^*\right)^{p-2}\mathcal{V}_j^2$ which is even on $\mathbb{S}^{d-1}$ while, the second part $\gamma_{\mathcal{N},led,j,**}$ is obtained by the data $\left(\Psi_j^*\right)^{p-3}\mathcal{V}_j^3$ which is odd on $\mathbb{S}^{d-1}$.
\end{proof}

\vskip0.12in

By Lemma~\ref{lem0010}, we can consider the following equation:
\begin{eqnarray}\label{eqn4015}
\left\{\aligned&\mathcal{L}(\rho_{**,1,j}^{\perp})=\mathcal{R}_{new,*,j}+\sum_{i=1}^{\nu}\Psi_{i}^{p-1}\left(c_{new,*,j,i}\partial_t\Psi_{i}+\sum_{l=1}^{d}\varsigma_{new,*,j,i,l}w_{i,l}\right),\quad \text{in }\mathcal{C},\\
&\langle \partial_t\Psi_{j}, \rho_{**,1}^{\perp}\rangle=\langle w_{j,l}, \rho_{**,1}^{\perp}\rangle=0\quad\text{for }1\leq j\leq\nu\text{ and }1\leq l\leq d,\endaligned\right.
\end{eqnarray}
where $\mathcal{R}_{new,*,j}=2A_p\left(\Psi_j^*\right)^{p-2}\mathcal{V}_j\gamma_{\mathcal{N},led,j}$ with $A_p$ given in Lemma~\ref{lem0003}.
\begin{lemma}\label{lem0011}
Let $d\geq2$, $a<0$ and $b=b_{FS}(a)$.  Then \eqref{eqn4015} is uniquely solvable.  Moreover, the solution $\rho_{**,1,j}^{\perp}$ is even in terms of $t-s_j^*$ and satisfies
\begin{eqnarray*}
\beta_*^3\gtrsim\sup_{(t,\theta)\in\mathcal{C}}\frac{|\rho_{**,1,j}^{\perp}|}{\Psi_{j}^{1-\sigma}(t)},
\end{eqnarray*}
where the Lagrange multipliers $\{c_{new,*,j,i}\}$ and $\{\varsigma_{new,*,j,i,l}\}$ are chosen such that the right hand side of the equation~\eqref{eqn4015} is orthogonal to $\{\Psi_j^{p-1}\partial_t\Psi_{j}\}$ and $\{\Psi_j^{p-1}w_{j,l}\}$ in $L^2(\mathcal{C})$ and $\sigma>0$ is chosen to satisfy \eqref{eqnewnew3345}.
The Lagrange multipliers also satisfy $|c_{new,*,j,i}|\lesssim\beta_*^3 Q^p$ and $\left|\varsigma_{new,*,j,i,l}\right|\lesssim\beta_*^3$
for all $1\leq i,j\leq \nu$ and all $1\leq l\leq d$.
\end{lemma}
\begin{proof}
The proof is similar to that of Lemma~\ref{lem0009} so we omit it.  Moreover, since $\mathcal{R}_{new,*,j}$ is even in terms of $t-s_j^*$ by Lemma~\ref{lem0010}, by uniqueness, we also have that $\rho_{**,1,j}^{\perp}$ is even in terms of $t-s_j^*$.  Moreover, similar to that of \eqref{eqn1049} and \eqref{eqn3049}, by the oddness of $\partial_t\Psi$ in $\bbr$, we also have
\begin{eqnarray*}
\sum_{i=1}^{\nu}|c_{new,*,j,i}|\lesssim\beta_*^3Q^p\quad\text{and}\quad\sum_{i=1}^{\nu}|\varsigma_{new,*,j,i,l}|\lesssim\beta_*^3
\end{eqnarray*}
for all $1\leq j\leq\nu$.
\end{proof}

\vskip0.12in

Let
\begin{eqnarray*}
\rho_{**,1}^{\perp}=\sum_{j=1}^{\nu}\rho_{**,1,j}^{\perp}\quad\text{and}\quad \mathcal{R}_{new,*}=\sum_{j=1}^{\nu}\mathcal{R}_{new,*,j}.
\end{eqnarray*}
Clearly, $\rho_{**,1}^{\perp}$ may also have projections on $span\{\Psi_l\}$.  Thus, as above, we decompose
\begin{eqnarray*}
\rho_{**,1}^{\perp}=\sum_{l=1}^{\nu}\alpha_{l,1}^{**}\Psi_l+\rho_{**,2}^{\perp},
\end{eqnarray*}
where $\{\alpha_{l,1}^{**}\}$ is chosen such that
\begin{eqnarray}\label{eqn2004}
\left\langle \rho_{**,2}^{\perp}, \Psi_{j}\right\rangle=\left\langle \rho_{**,2}^{\perp}, \partial_t\Psi_{j}\right\rangle=\left\langle \rho_{**,2}^{\perp}, w_{j,l}\right\rangle=0.
\end{eqnarray}
Moreover, by \eqref{eqn4015}, we know that $\rho_{**,2}^{\perp}$ satisfies the following equation:
\begin{eqnarray}\label{eqn6114}
\left\{\aligned&\mathcal{L}(\rho_{**,2}^{\perp})=\mathcal{R}_{new,**}+\sum_{i=1}^{\nu}\Psi_{i}^{p-1}\left(c_{new,*,i}\partial_t\Psi_{i}+\sum_{l=1}^{d}\varsigma_{new,*,i,l}w_{i,l}\right),\quad \text{in }\mathcal{C},\\
&\left\langle \rho_{**,2}^{\perp}, \Psi_{j}\right\rangle=\langle \partial_t\Psi_{j}, \rho_{**,2}^{\perp}\rangle=\langle w_{j,l}, \rho_{**,2}^{\perp}\rangle=0\quad\text{for }1\leq j\leq\nu\text{ and }1\leq l\leq d,\endaligned\right.
\end{eqnarray}
where
\begin{eqnarray}\label{eqn6061}
\mathcal{R}_{new,**}=\sum_{j=1}^{\nu}2A_p\left(\Psi_j^*\right)^{p-2}\mathcal{V}_j\gamma_{\mathcal{N},led,j}+\sum_{l=1}^{\nu}p\alpha_{l,1}^{**}\left(\mathcal{U}^{p-1}-\Psi_l^{p-1}\right)\Psi_l,
\end{eqnarray}
with $A_p$ given by Lemma~\ref{lem0003}.
\begin{lemma}\label{lem0012}
Let $d\geq2$, $a<0$ and $b=b_{FS}(a)$.  Then we have
\begin{eqnarray}\label{eqn5068}
\sum_{j=1}^{\nu}\left|\alpha_{j,1}^{**}\right|\lesssim \beta_*^4+\beta_*^3Q
\end{eqnarray}
and
\begin{eqnarray}\label{eqn9079}
\beta_*^4+\beta_*^3Q\gtrsim\left\{\aligned
&\|\rho_{**,2}^{\perp}\|_{\sharp}+\sup_{(t,\theta)\in\mathcal{C}}\frac{|\rho_{**,2}^{\perp}|}{\sum_{j=1}^{\nu}\Psi_{j}^{1-\sigma}(t)},\quad p\geq3,\\
&\|\rho_{**,2}^{\perp}\|_{\natural,1,*}+\sup_{(t,\theta)\in\mathcal{C}}\frac{|\rho_{**,2}^{\perp}|}{\sum_{j=1}^{\nu}\Psi_{j}^{1-\sigma}(t)},\quad 1<p<3,
\endaligned\right.
\end{eqnarray}
where $\|\cdot\|_{\sharp}$ and $\|\cdot\|_{\natural,1,*}$ are given in Lemma~\ref{lem0006} and $\sigma>0$ is chosen to satisfy \eqref{eqnewnew3345}.
\end{lemma}
\begin{proof}
By \eqref{eqn4015} and \eqref{eqn2004},
\begin{eqnarray}\label{eqn5052}
\sum_{l=1}^{\nu}\left\langle \Psi_l, \Psi_{j}\right\rangle\alpha_{l,1}^{**}&=&\left\langle \mathcal{R}_{new,*}, \Psi_{j}\right\rangle_{L^2}+p\left\langle \mathcal{U}^{p-1}\rho_{**,1}^{\perp},\Psi_j\right\rangle_{L^2}\notag\\
&=&\sum_{l=1}^{\nu}p\alpha_{l,1}^{**}\left\langle \mathcal{U}^{p-1}\Psi_l,\Psi_j\right\rangle_{L^2}+p\left\langle \mathcal{U}^{p-1}\rho_{**,2}^{\perp},\Psi_j\right\rangle_{L^2}\notag\\
&&+\left\langle \mathcal{R}_{new,*}, \Psi_{j}\right\rangle_{L^2}
\end{eqnarray}
for all $1\leq j\leq\nu$.  By the oddness of $\{\mathcal{V}_{j}\}$ on $\mathbb{S}^{d-1}$ and Lemmas~\ref{lem0005} and \ref{lem0010}, we have
\begin{eqnarray*}
\left|\left\langle \mathcal{R}_{new,*}, \Psi_{j}\right\rangle_{L^2}\right|&\lesssim&\beta_*^4+\sum_{i=1;i\not=j}^{\nu}\beta_*^3\left\langle \Psi_{i}^{\frac{3p-1-2\sigma}{2}}, \Psi_{j}\right\rangle_{L^2}+\sum_{l=1}^{\nu}\alpha_{l,1}^{**}\left\langle \Psi_{l}^{p-1}, \mathcal{U}_l\Psi_{j}\right\rangle_{L^2}\\
&\lesssim&\beta_*^4+\beta_*^3Q+\sum_{l=1}^{\nu}Q\alpha_{l,1}^{**}.
\end{eqnarray*}
On the other hand, by Lemma~\ref{lem0005} and similar estimates of \eqref{eqn3147},
\begin{eqnarray*}
\sum_{l=1}^{\nu}p\alpha_{l,1}^{**}\left\langle \mathcal{U}^{p-1}\Psi_l,\Psi_j\right\rangle_{L^2}&=&\sum_{l=1}^{\nu}p\alpha_{l,1}^{**}\left\langle \sum_{i=1}^{\nu}\left(\Psi_i^*\right)^{p-1}\chi_{\mathcal{B}_i}\Psi_l,\Psi_j\right\rangle_{L^2}\\
&&+\sum_{l=1}^{\nu}p\alpha_{l,1}^{**}\left\langle \left(\mathcal{U}^{p-1}-\sum_{i=1}^{\nu}\left(\Psi_i^*\right)^{p-1}\chi_{\mathcal{B}_i}\right)\Psi_l,\Psi_j\right\rangle_{L^2}\\
&&+\sum_{l=1}^{\nu}p\alpha_{l,1}^{**}\left\langle \mathcal{U}^{p-1}\chi_{\mathcal{C}\backslash\cup_{i=1}^{\nu}\mathcal{B}_i}\Psi_l,\Psi_j\right\rangle_{L^2}\\
&=&p\|\Psi\|^2\alpha_{j,1}^{**}+\sum_{l=1;l\not=j}^{\nu}\mathcal{O}(Q)\alpha_{l,1}^{**}
\end{eqnarray*}
and further by \eqref{eqn2004}, we have
\begin{eqnarray*}
\left\langle \mathcal{U}^{p-1}\rho_{**,2}^{\perp},\Psi_j\right\rangle_{L^2}&=&\left\langle \left(\mathcal{U}^{p-1}-\left(\Psi_j^*\right)^{p-1}\right)\Psi_j,\rho_{**,2}^{\perp}\right\rangle_{L^2}\\
&=&\|\rho_{**,2}^{\perp}\|\times\left\{\aligned
&\mathcal{O}\left(Q\right),\quad p>2,\\
&\mathcal{O}\left(Q\left|\log Q\right|^{\frac12}\right),\quad p=2,\\
&\mathcal{O}\left(Q^{\frac{p}{2}}\right),\quad 1<p<2.
\endaligned\right.
\end{eqnarray*}
It follows from \eqref{eqn5052} that
\begin{eqnarray*}
\sum_{j=1}^{\nu}\left|\alpha_{j,1}^{**}\right|\lesssim \beta_*^4+\beta_*^3Q+Q^{\frac{p}{2}\wedge1}\left|\log Q\right|\|\rho_{**,2}^{\perp}\|.
\end{eqnarray*}
Now, by multiplying \eqref{eqn6114} with $\rho_{**,2}^{\perp}$ on both sides and integrating by parts, we have $\|\rho_{**,2}^{\perp}\|\lesssim\beta_*^4+\beta_*^3Q$, which implies that \eqref{eqn5068} holds true.  To obtain the estimate~\eqref{eqn9079}, we shall decompose $\mathcal{R}_{new,**}$ into two parts, where $\mathcal{R}_{new,**}$ is given by \eqref{eqn6061}.  The first part is given by $\sum_{j=1}^{\nu}p(p-1)\left(\Psi_j^*\right)^{p-2}\mathcal{V}_j\gamma_{\mathcal{N},led,j}$, which generates the bound
\begin{eqnarray*}
\sup_{(t,\theta)\in\mathcal{C}}\frac{|\rho_{**,2}^{\perp}|}{\sum_{j=1}^{\nu}\Psi_{j}^{1-\sigma}(t)}\lesssim\beta_*^4+\beta_*^3Q
\end{eqnarray*}
as that of $\gamma_{\mathcal{N},led}$.  The second part is given by $\sum_{l=1}^{\nu}p\alpha_{l,1}^{**}\left(\mathcal{U}^{p-1}-\Psi_l^{p-1}\right)\Psi_l$, which generates the bound
\begin{eqnarray*}
\beta_*^4+\beta_*^3Q\gtrsim\left\{\aligned
&\|\rho_{**,2}^{\perp}\|_{\sharp},\quad p\geq3,\\
&\|\rho_{**,2}^{\perp}\|_{\natural,1,*},\quad 1<p<3,
\endaligned\right.
\end{eqnarray*}
as that of $\gamma_{1,ex}$.
\end{proof}

\vskip0.12in

We also need to consider the following equation:
\begin{eqnarray}\label{eqn5013}
\left\{\aligned&\mathcal{L}(\rho_{**,3}^{\perp})=\mathcal{R}_{3,ex}-\sum_{j=1}^{\nu}\Psi_{j}^{p-1}\left(c_{3,ex,j}\partial_t\Psi_{j}+\sum_{l=1}^{d}\varsigma_{3,ex,j,l}w_{j,l}\right),\quad \text{in }\mathcal{C},\\
&\langle \partial_t\Psi_{j}, \rho_{**,3}^{\perp}\rangle=\langle w_{j,l}, \rho_{**,3}^{\perp}\rangle=0\quad\text{for all }1\leq j\leq\nu\text{ and all }1\leq l\leq d,\endaligned\right.
\end{eqnarray}
where $\mathcal{R}_{3,ex}=2A_p\mathcal{U}^{p-2}\mathcal{V}\gamma_{1,ex}$.
\begin{lemma}\label{lem0013}
Let $d\geq2$, $a<0$ and $b=b_{FS}(a)$.  Then \eqref{eqn5013} is uniquely solvable.  Moreover, the solution $\rho_{**,3}^{\perp}$ is odd on $\mathbb{S}^{d-1}$ and satisfies
\begin{eqnarray}\label{eqn4447}
\beta_*\gtrsim\left\{\aligned
&\|\rho_{**,3}^{\perp}\|_{\sharp},\quad p\geq\frac{7}{3},\\
&\|\rho_{**,3}^{\perp}\|_{\natural,2,*},\quad 1<p<\frac{7}{3},
\endaligned\right.
\end{eqnarray}
where the Lagrange multipliers $\{c_{3,ex,j}\}$ and $\{\varsigma_{3,ex,j,l}\}$ are chosen such that the right hand side of the equation~\eqref{eqn5013} is orthogonal to $\{\Psi_j^{p-1}\partial_t\Psi_{j}\}$ and $\{\Psi_j^{p-1}w_{j,l}\}$ in $L^2(\mathcal{C})$ and
the norms $\|\cdot\|_{\sharp}$ and $\|\cdot\|_{\natural,2,*}$ are given in Lemma~\ref{lem0008}.
The Lagrange multipliers also satisfy $c_{3,ex,j}=0$ for all $1\leq j\leq \nu$ and $|\varsigma_{3,ex,j,l}|\lesssim\beta_*Q$ for all $1\leq j\leq \nu$ and $1\leq l\leq d$.
\end{lemma}
\begin{proof}
By Lemma~\ref{lem0006}, $\gamma_{1,ex}$ is even on $\mathbb{S}^{d-1}$.  Moreover, similar to \eqref{eqn3147}, direct calculations show that
\begin{eqnarray}\label{eqn4548}
\left|\mathcal{R}_{3,ex}\right|\lesssim\left\{\aligned
&\beta_*\left(\sum_{i=1}^{\nu}\Psi_i^{\frac{3p-1-2\sigma}{2}}\left(Q_i\chi_{\mathcal{B}_{i,+}}+Q_{i-1}\chi_{\mathcal{B}_{i,-}}\right)+Q\mathcal{U}^{\frac{3p-1-2\sigma}{2}}\chi_{\mathcal{C}\backslash\left(\cup_{i=1}^{\nu}\mathcal{B}_i\right)}\right),\quad p\geq3,\\
&\beta_*\left(\sum_{i=1}^{\nu}\Psi_i^{\frac{5p-7}{2}}\left(Q_i\chi_{\mathcal{B}_{i,+}}+Q_{i-1}\chi_{\mathcal{B}_{i,-}}\right)+Q\mathcal{U}^{\frac{3p-1-2\sigma}{2}}\chi_{\mathcal{C}\backslash\left(\cup_{i=1}^{\nu}\mathcal{B}_i\right)}\right),\quad 1<p<3.
\endaligned\right.
\end{eqnarray}
Since $\Psi_i^{\frac{5p-7}{2}}\lesssim\Psi_i^{\frac{3p-5}{2}}$ in $\mathcal{B}_i$ for all $1\leq i\leq \nu$, the rest of the proof is the same of that of Lemma~\ref{lem0008}, so we omit it here.
\end{proof}

Let $\rho_{**}^{\perp}=\rho_{**,0}^{\perp}-\rho_{**,2}^{\perp}-\rho_{**,3}^{\perp}$, then we have the following decomposition of $\rho_*$.
\begin{proposition}\label{prop0001}
Let $d\geq2$, $a<0$ and $b=b_{FS}(a)$.  Then we have $\rho_*=\rho_0+\rho_{**}^{\perp}$,
where
\begin{enumerate}
\item[$(1)$] \quad the regular part $\rho_0=\gamma_{ex}+\gamma_{*}+\gamma_{\mathcal{N},led}$
and
\begin{enumerate}
\item[$(i)$]\quad$\gamma_{ex}=\sum_{l=1}^{2}\gamma_{l,ex}$ with $\gamma_{1,ex}$ even on $\mathbb{S}^{d-1}$ and $\gamma_{2,ex}=\overline{\gamma}_{2,ex}+\rho_{**,3}^{\perp}$ odd on $\mathbb{S}^{d-1}$ satisfying
\begin{eqnarray*}
1\gtrsim\left\{\aligned
&\left\|\gamma_{ex}\right\|_{\sharp}+\left\|\gamma_{1,ex}\right\|_{\sharp},\quad p\geq3,\\
&\left\|\gamma_{ex}\right\|_{\natural,1,*}+\left\|\gamma_{1,ex}\right\|_{\natural,1,*},\quad 1<p<3
\endaligned\right.
\end{eqnarray*}
and
\begin{eqnarray*}
\beta_*\gtrsim\left\{\aligned
&\|\gamma_{2,ex}\|_{\sharp},\quad p\geq\frac{7}{3},\\
&\|\gamma_{2,ex}\|_{\natural,2,*},\quad 1<p<\frac{7}{3},
\endaligned\right.
\end{eqnarray*}

\item[$(ii)$]\quad $\gamma_{*}=\sum_{j=1}^{\nu}\gamma_{j,*}$ is even on $\mathbb{S}^{d-1}$ with $\gamma_{j,*}$ even in terms of $t-s_j^*$ in $\mathbb{R}$ and satisfying
\begin{eqnarray*}
Q+\sum_{j=1}^{\nu}\left|\left(\alpha_j^{*}\right)^{p-1}-1\right|+\beta_*^2\gtrsim\sup_{(t,\theta)\in\mathcal{C}}\frac{|\gamma_{*}|}{\sum_{i=1}^{\nu}\Psi_{j}^{1-\sigma}(t)},
\end{eqnarray*}
where $\gamma_{j,*}=\gamma_{1,j}+\alpha_{j,0}^{**}\Psi_{j}$ and $\sigma>0$ is chosen to satisfy \eqref{eqnewnew3345}.
\item[$(iii)$]\quad $\gamma_{\mathcal{N},led}=\gamma_{\mathcal{N},led,*}+\rho_{**,2}^{\perp}$ with the symmetrical part of $\gamma_{\mathcal{N},led}$, in terms of $t-s_j^*$,  given by $\gamma_{\mathcal{N},led,j}+\rho_{**,1,j}^{\perp}-\alpha_{j,1}^{**}\Psi_j$ and the remaining parts, denoted by
$\gamma_{\mathcal{N},led,rem,j}$, satisfies the following estimates
\begin{eqnarray*}
\beta_*^2\gtrsim\sup_{(t,\theta)\in\mathcal{C}}\frac{|\gamma_{\mathcal{N},led,rem,j}|}{\sum_{i=1;i\not=j}^{\nu}\Psi_{j}^{1-\sigma}(t)}.
\end{eqnarray*}
Moreover, $\gamma_{\mathcal{N},led}$ satisfies the following estimates
\begin{eqnarray*}
1\gtrsim\left\{\aligned
&\frac{\|\rho_{**,2}^{\perp}\|_{\sharp}}{\beta_*^4+\beta_*^3Q}+\sup_{(t,\theta)\in\mathcal{C}}\frac{|\gamma_{\mathcal{N},led,*}|}{\sum_{j=1}^{\nu}\beta_*^2\Psi_{j}^{1-\sigma}(t)},\quad p\geq3,\\
&\frac{\|\rho_{**,2}^{\perp}\|_{\natural,1,*}}{\beta_*^4+\beta_*^3Q}+\sup_{(t,\theta)\in\mathcal{C}}\frac{|\gamma_{\mathcal{N},led,*}|}{\sum_{j=1}^{\nu}\beta_*^2\Psi_{j}^{1-\sigma}(t)},\quad 1<p<3,
\endaligned\right.
\end{eqnarray*}
where $\sigma>0$ is chosen to satisfy \eqref{eqnewnew3345}.
\end{enumerate}
\item[$(2)$]\quad The singular part $\rho_{**}^{\perp}$ satisfies the following equation:
\begin{eqnarray}\label{eqn1114}
\left\{\aligned&\mathcal{L}(\rho_{**}^{\perp})=f+\mathcal{R}_{new,0},\quad \text{in }\mathcal{C},\\
&\langle \Psi_{j}, \rho_{**}^{\perp}\rangle=\langle \partial_t\Psi_{j}, \rho_{**}^{\perp}\rangle=\langle w_{j,l}, \rho_{**}^{\perp}\rangle=0\quad\text{for }1\leq j\leq\nu\text{ and }1\leq l\leq d,\endaligned\right.
\end{eqnarray}
where
\begin{eqnarray}
\mathcal{R}_{new,0}&=&\sum_{i=1}^{\nu}(c_{1,ex,i}+c_{1,j,i}+c_{3,led,i}-c_{new,*,i})\Psi_{i}^{p-1}\partial_t\Psi_{i}+\sum_{j=1}^{\nu}\mathcal{R}_{2,j}\notag\\
&&+\sum_{i=1}^{\nu}\sum_{l=1}^{d}(\varsigma_{2,ex,i,l}+\varsigma_{3,led,i,l}-\varsigma_{new,*,i,l})\Psi_{i}^{p-1}w_{i,l}+\mathcal{N}_{rem}\notag\\
&&+\sum_{j=1}^{\nu}\left(2A_p\left(\Psi_j^*\right)^{p-2}\mathcal{V}_j(\rho_*-\gamma_{1,ex}-\gamma_{\mathcal{N},led,j})+3B_p\left(\Psi_j^*\right)^{p-3}\mathcal{V}_j^2\rho_*\right)\chi_{\mathcal{B}_j}\notag\\
&&+\sum_{j=1}^{\nu}\mathcal{O}\left(\beta_*^2\mathcal{U}_j\Psi_{j}^{2p-3}\left(\Psi_j+\rho_*\right)+\beta_*|\rho_*-\gamma_{1,ex}|\Psi_j^{\frac{3p-5}{2}}\mathcal{U}_j\right)\chi_{\mathcal{B}_j}\notag\\
&&+\mathcal{O}\left(\beta_*|\rho_*-\gamma_{1,ex}|\mathcal{U}^{\frac{3(p-1)}{2}}\right)\chi_{\mathcal{C}\backslash\left(\cup_{j=1}^{\nu}\mathcal{B}_j\right)}\notag\\
&&+\sum_{j=1}^{\nu}\alpha_{j}^{**}\left(\mathcal{U}^{p-1}-\Psi_j^{p-1}\right)\Psi_{j}\label{eqn0061}
\end{eqnarray}
with $A_p$ and $B_p$ given in Lemma~\ref{lem0003}, $\alpha_j^{**}=\alpha_{j,0}^{**}-\alpha_{j,1}^{**}$ and
\begin{eqnarray*}
\sum_{j=1}^{\nu}\left|\alpha_j^{**}\right|\lesssim Q+\sum_{j=1}^{\nu}\left|\left(\alpha_j^{*}\right)^{p-1}-1\right|+\beta_*^2.
\end{eqnarray*}
\end{enumerate}
\end{proposition}
\begin{proof}
Since $\frac{3p-5}{2}>p-2$ for $p>1$, by \eqref{eqn2047}, \eqref{eqn29997}, \eqref{eqn4447} and \eqref{eqn4548}, we also have
\begin{eqnarray}\label{eqqnew0009}
\|\mathcal{R}_{2,ex}+\mathcal{R}_{3,ex}\|_{\natural,1,*}\lesssim\beta_*\quad\text{and}\quad\|\gamma_{2,ex}\|_{\natural,1,*}\lesssim\beta_*.
\end{eqnarray}
Thus, the rest proof of $(i)$ of $(1)$ follows from Lemmas~\ref{lem0006}, \ref{lem0008} and \ref{lem0013}.  The conclusion of $(ii)$ of $(1)$ follows from Lemma~\ref{lem0007} and \eqref{eqn0068}.  The conclusion of $(iii)$ of $(1)$ follows from Lemmas~\ref{lem0010}, \ref{lem0011} and \ref{lem0012}.  The conclusion of $(2)$ follows from \eqref{eqn0068}, \eqref{eqn5114}, \eqref{eqn4015} and Lemma~\ref{lem0012}.
\end{proof}

\section{Refined expansion of $\mathcal{N}$ and estimate of $\left\{\alpha_j^{*}\right\}$}
As we stated before, inspired by the optimal Bianchi-Egnell stability of the CKN inequality for $d\geq2$, $a<0$ and $b=b_{FS}(a)$ proved in \cite[Theorem~1]{FP2024}, we need to eliminate the lower order terms (compared to the $\beta_*^4$ terms) in the data $\mathcal{R}_{new,0}$ to get the desired stability inequality.  Thus, we need to refine the expansion of $\mathcal{N}$ since we have picked up a regular part $\rho_0$ in the remaining term $\rho_*$.
\begin{lemma}\label{lemn0001}
Let $d\geq2$, $a<0$ and $b=b_{FS}(a)$.  Then the nonlinear part $\mathcal{N}$, which is given by \eqref{eqnewnew8856}, can be refinedly expanded as follows:
\begin{eqnarray*}
\mathcal{N}&=&A_p\mathcal{U}^{p-2}\left(\overline{\mathcal{V}}^2+2\overline{\mathcal{V}}\rho_{**}^{\perp}\right)+B_p\mathcal{U}^{p-3}\left(\overline{\mathcal{V}}^3+3\overline{\mathcal{V}}^2\rho_{**}^{\perp}\right)+\overline{\mathcal{N}}_{rem}\notag\\
&=&A_p\mathcal{U}^{p-2}\left(\mathcal{V}^2+2\mathcal{V}\rho_{*}+\rho_0^2+2\rho_0\rho_{**}^{\perp}\right)+\overline{\mathcal{N}}_{rem}\notag\\
&&+B_p\mathcal{U}^{p-3}\left(\mathcal{V}^3+3\mathcal{V}^2\rho_{*}+3\mathcal{V}\rho_0^2+\rho_0^3+6\mathcal{V}\rho_0\rho_{**}^{\perp}+3\rho_0^2\rho_{**}^{\perp}\right)\notag\\
&=&\mathcal{N}_{*}+\overline{\mathcal{N}}_{rem}+A_p\mathcal{U}^{p-2}\left(\rho_0^2+2\rho_0\rho_{**}^{\perp}\right)\notag\\
&&+B_p\mathcal{U}^{p-3}\left(3\mathcal{V}\rho_0^2+\rho_0^3+6\mathcal{V}\rho_0\rho_{**}^{\perp}+3\rho_0^2\rho_{**}^{\perp}\right)\notag\\
&:=&\mathcal{N}_{*}+\overline{\mathcal{N}}_{rem}+\mathcal{N}_{0}
\end{eqnarray*}
in $\mathcal{C}$, where $\mathcal{N}_{*}$ is given in \eqref{eqn0018}, $A_p$ and $B_p$ are given in Lemma~\ref{lem0003}, $\overline{\mathcal{V}}=\mathcal{V}+\rho_{0}$ and
\begin{eqnarray*}
\overline{\mathcal{N}}_{rem}&=&\mathcal{O}\left(\mathcal{U}^{p-4\sigma}\left(\beta_*+Q^{\frac{2\wedge p}{2}}\left|\log Q\right|+\sum_{j=1}^{\nu}\left|\left(\alpha_{j}^{*}\right)^{p-1}-1\right|\right)^4\right)\notag\\
&&+\mathcal{O}\left(\chi_{p\geq2}\left|\rho_{**}^{\perp}\right|^{2}+\left|\rho_{**}^{\perp}\right|^{p}+\left|\rho_{**}^{\perp}\right|^{1+\ve}+|\gamma_*+\gamma_{\mathcal{N},led}|^{1+\ve}\chi_{\mathcal{C}\backslash\overline{\mathcal{B}}_*}\right)
\end{eqnarray*}
where $\ve>0$ is a constant which depends on $p>1$ and $\sigma$ and $\overline{\mathcal{B}}_*$ are given by \eqref{eqnewnew3345}.
\end{lemma}
\begin{proof}
We improve the set $\mathcal{A}$, used in the proof of Lemma~\ref{lem0003}, by introducing the set
\begin{eqnarray*}
\mathcal{A}_*=\left\{(\theta,t)\in\overline{\mathcal{B}}_*\mid |\rho_{**}^{\perp}|\leq\left|\mathcal{V}+\rho_{0}\right|\right\},
\end{eqnarray*}
where $\rho_0$ is the regular part of $\rho_*$ given in $(1)$ of Proposition~\ref{prop0001}.  By $(1)$ of Proposition~\ref{prop0001}, we have
\begin{eqnarray}\label{eqn0190}
\|\gamma_{ex}\|_{L^{\infty}}+\|\gamma_{*}\|_{L^{\infty}}+\|\gamma_{\mathcal{N},led}\|_{L^{\infty}}=o(1),
\end{eqnarray}
thus, $\left\|\overline{\mathcal{V}}\right\|_{L^{\infty}(\mathcal{C})}\to0$ as $\|f\|_{H^{-1}}\to0$.
Now, as that of \eqref{eqn0009}, by the choice of $\sigma$ given by \eqref{eqnewnew3345} and $(1)$ of Proposition~\ref{prop0001}, we can expand $\mathcal{N}$ as follows:
\begin{eqnarray}\label{eqn9009}
\mathcal{N}&=&A_p\mathcal{U}^{p-2}\left(\overline{\mathcal{V}}^2+2\overline{\mathcal{V}}\rho_{**}^{\perp}\right)+B_p\mathcal{U}^{p-3}\left(\overline{\mathcal{V}}^3+3\overline{\mathcal{V}}^2\rho_{**}^{\perp}\right)\notag\\
&&+\mathcal{O}\left(\mathcal{U}^{p-4}\left(\overline{\mathcal{V}}+\rho_{**}^{\perp}\right)^4+\mathcal{U}^{p-2}\left|\rho_{**}^{\perp}\right|^2\right)\notag\\
&=&A_p\mathcal{U}^{p-2}\left(\overline{\mathcal{V}}^2+2\overline{\mathcal{V}}\rho_{**}^{\perp}\right)+B_p\mathcal{U}^{p-3}\left(\overline{\mathcal{V}}^3+3\overline{\mathcal{V}}^2\rho_{**}^{\perp}\right)\notag\\
&&+\mathcal{O}\left(\mathcal{U}^{p-4\sigma}\left(\beta_*+Q^{\frac{2\wedge p}{2}}\left|\log Q\right|+\sum_{j=1}^{\nu}\left|\left(\alpha_{j}^{*}\right)^{p-1}-1\right|\right)^4\right)\notag\\
&&+\mathcal{O}\left(\chi_{p\geq2}|\rho_{**}^{\perp}|^{2}+|\rho_{**}^{\perp}|^{p}\right)
\end{eqnarray}
in $\mathcal{A}_*$.  In $\mathcal{A}_*^c$, if $\left|\mathcal{V}+\rho_{0}\right|\leq |\rho_{**}^{\perp}|$, then as that of \eqref{eqn0010}, we have
\begin{eqnarray}\label{eqn9010}
\mathcal{N}=\mathcal{O}\left(\chi_{p\geq2}|\rho_{**}^{\perp}|^{2}+|\rho_{**}^{\perp}|^{p}\right).
\end{eqnarray}
Otherwise, we have $(\theta,t)\in\mathcal{C}\backslash\overline{\mathcal{B}}_*$ and $|\rho_{**}^{\perp}|\leq\left|\mathcal{V}+\rho_{0}\right|$ in $\mathcal{A}_*^c$.  Since we always have $|\gamma_{ex}|=o(\mathcal{U})$ by $(i)$ of $(1)$ of Proposition~\ref{prop0001}, as that of \eqref{eqn0009} and \eqref{eqn0010}, either we have the expansion~\eqref{eqn9009} if we further have $|\gamma_*+\gamma_{\mathcal{N},led}|\leq\frac12\mathcal{U}$ or we have
\begin{eqnarray}\label{eqn9110}
\mathcal{N}=\mathcal{O}\left(|\gamma_*+\gamma_{\mathcal{N},led}|^{2\wedge p}\right)
\end{eqnarray}
if we further have $\frac12\mathcal{U}\leq|\gamma_*+\gamma_{\mathcal{N},led}|$.  Since by $(1)$ of Proposition~\ref{prop0001}, we have $\left|\overline{\mathcal{V}}\right|\lesssim \mathcal{U}^{1-\sigma}$ and by \eqref{eqnewnew3345}, we can take $\sigma>0$ arbitrary small if necessary.  Thus, similar to \eqref{eqn0017},
we have
\begin{eqnarray}\label{eqn5017}
|\rho_{**}^{\perp}|^{1+\ve}\gtrsim A_p\mathcal{U}^{p-2}\left(\overline{\mathcal{V}}^2+2\left|\overline{\mathcal{V}}\rho_{**}^{\perp}\right|\right)+B_p\mathcal{U}^{p-3}\left(\left|\overline{\mathcal{V}}\right|^3+3\overline{\mathcal{V}}^2|\rho_{**}^{\perp}|\right)
\end{eqnarray}
if $\left|\mathcal{V}+\rho_{0}\right|\leq |\rho_{**}^{\perp}|$ in $\mathcal{A}^c$ and
\begin{eqnarray}\label{eqn5217}
|\gamma_*+\gamma_{\mathcal{N},led}|^{1+\ve}\gtrsim A_p\mathcal{U}^{p-2}\left(\overline{\mathcal{V}}^2+2\left|\overline{\mathcal{V}}\rho_{**}^{\perp}\right|\right)+B_p\mathcal{U}^{p-3}\left(\left|\overline{\mathcal{V}}\right|^3+3\overline{\mathcal{V}}^2|\rho_{**}^{\perp}|\right)
\end{eqnarray}
if $(\theta,t)\in\mathcal{C}\backslash\overline{\mathcal{B}}_*$, $|\rho_{**}^{\perp}|\leq\left|\mathcal{V}+\rho_{0}\right|$ and $\frac12\mathcal{U}\leq|\gamma_*+\gamma_{\mathcal{N},led}|$ in $\mathcal{A}_*^c$, where $\ve>0$ is a constant which depends on $p>1$.  The conclusion then follows from \eqref{eqn9009}, \eqref{eqn9010}, \eqref{eqn9110}, \eqref{eqn5017} and \eqref{eqn5217}.
\end{proof}

\vskip0.12in

By multiplying \eqref{eq0014} with $\Psi_j$ on both sides and integrating by parts and by the orthogonal conditions of $\rho_*$ given in \eqref{eq0014} and the oddness of $\{\mathcal{V}_i\}$ on $\mathbb{S}^{d-1}$, we have
\begin{eqnarray}\label{eqn0121}
-\left\langle f, \Psi_j\right\rangle_{H^1}&=&\left\langle \mathcal{R}_{1,j}, \Psi_j\right\rangle_{L^2}+\left\langle \mathcal{N}, \Psi_j\right\rangle_{L^2}+\sum_{i=1;i\not=j}^{\nu}\left\langle \mathcal{R}_{1,i}, \Psi_j\right\rangle_{L^2}\notag\\
&&+\left\langle \mathcal{R}_{1,ex}, \Psi_j\right\rangle_{L^2}+\left\langle \mathcal{L}_{j,ex}(\rho_*), \Psi_j\right\rangle_{L^2}
\end{eqnarray}
for all $j=1,2,\cdots,\nu$.  In what follows, by using the equality~\eqref{eqn0121}, we shall derive the estimate of $\sum_{j=1}^{\nu}\left|\left(\alpha_j^*\right)^p-\alpha_j^*\right|$.
\begin{proposition}\label{propn0001}
Let $d\geq2$, $a<0$ and $b=b_{FS}(a)$.  Then we have
\begin{eqnarray*}
\left(\sum_{j=1}^{\nu}\left(\left(\alpha_{j}^*\right)^{p}-\alpha_{j}^*\right)\right)
&=&-\left(\overline{B}_1+o(1)\right)Q-\left\langle f, \sum_{j=1}^{\nu}\frac{\Psi_j}{\|\Psi\|^2+o(1)}\right\rangle\notag\\
&&-\left(\overline{A}_1+o(1)\right)\beta_*^2+\mathcal{O}\left(\|\rho_{**}^{\perp}\|^{1+\ve}\right),
\end{eqnarray*}
where
\begin{eqnarray*}
\overline{B}_1=\frac{p\int_{\mathcal{C}}\Psi^{p-1}d\mu}{\|\Psi\|^2}\quad\text{and}\quad\overline{A}_1=\lim_{\|f\|_{H^{-1}}\to0}\frac{\sum_{j=1}^{\nu}A_p\left\langle\left(\Psi_j^*\right)^{p-2}\mathcal{V}_j^2, \Psi_j\right\rangle_{L^2}}{\beta_*^2\|\Psi\|^2}.
\end{eqnarray*}
\end{proposition}
\begin{proof}
By the the orthogonal conditions of $\Psi_j$, $\partial_t\Psi_j$, $w_{j,l}$ and $\rho_*$, and the oddness of $w_{j,l}$ on $\mathbb{S}^{d-1}$ and $\partial_t\Psi_j$ in $\bbr$, we also have
\begin{eqnarray}\label{eqn0025}
\left\langle \mathcal{N}_{j}, \Psi_j\right\rangle_{L^2}=A_p\left\langle\left(\Psi_j^*\right)^{p-2}\mathcal{V}_j^2, \Psi_j\right\rangle_{L^2}+3B_p\left\langle\left(\Psi_j^*\right)^{p-3}\mathcal{V}_j^2\rho_*, \Psi_j\right\rangle_{L^2}
\end{eqnarray}
for all $j=1,2,\cdots,\nu$, where
\begin{eqnarray}\label{eqnew9999}
\mathcal{N}_{j}=A_p\left(\Psi_j^*\right)^{p-2}\left(\mathcal{V}_j^2+2\mathcal{V}_j\rho_*\right)+B_p\left(\Psi_j^*\right)^{p-3}\left(\mathcal{V}_j^3+3\mathcal{V}_j^2\rho_*\right).
\end{eqnarray}
Intersecting \eqref{eqn0025} into \eqref{eqn0121}, we have
\begin{eqnarray}\label{eqn0028}
-\sum_{j=1}^{\nu}\left\langle f, \Psi_j\right\rangle_{H^1}&=&\sum_{j=1}^{\nu}\left(\left(\alpha_j^*\right)^p-\alpha_j^*\right)\|\Psi\|^2+\sum_{j=1}^{\nu}\sum_{i=1;i\not=j}^{\nu}\left\langle \mathcal{R}_{1,i}, \Psi_j\right\rangle_{L^2}\notag\\
&&+\sum_{j=1}^{\nu}\left\langle \mathcal{N}-\mathcal{N}_{j}, \Psi_j\right\rangle_{L^2}+\sum_{j=1}^{\nu}\left\langle \mathcal{N}_{j}, \Psi_j\right\rangle_{L^2}\notag\\
&&+\sum_{j=1}^{\nu}\left\langle \mathcal{R}_{1,ex}, \Psi_j\right\rangle_{L^2}+\sum_{j=1}^{\nu}\left\langle \mathcal{L}_{j,ex}(\rho_*), \Psi_j\right\rangle_{L^2}.
\end{eqnarray}
The rest of the proof is to estimate every terms in \eqref{eqn0028}.

{\bf Step.~1}\quad The estimate of $\sum_{j=1}^{\nu}\left\langle \mathcal{N}_{j}, \Psi_j\right\rangle_{L^2}$.

By Lemma~\ref{lem0002},
\begin{eqnarray*}
\sum_{j=1}^{\nu}\left\langle \mathcal{N}_{j}, \Psi_j\right\rangle_{L^2}&=&\sum_{j=1}^{\nu}\left(A_p\left\langle\left(\Psi_j^*\right)^{p-2}\mathcal{V}_j^2, \Psi_j\right\rangle_{L^2}+3B_p\left\langle\left(\Psi_j^*\right)^{p-3}\mathcal{V}_j^2\rho_*, \Psi_j\right\rangle_{L^2}\right)\notag\\
&=&(\overline{A}_{1,*}+o(1))\beta_*^2,
\end{eqnarray*}
where
\begin{eqnarray*}
\overline{A}_{1,*}=\lim_{\|f\|_{H^{-1}}\to0}\frac{\sum_{j=1}^{\nu}A_p\left\langle\left(\Psi_j^*\right)^{p-2}\mathcal{V}_j^2, \Psi_j\right\rangle_{L^2}}{\beta_*^2}>0.
\end{eqnarray*}

{\bf Step.~2}\quad The estimate of $\sum_{j=1}^{\nu}\sum_{i=1;i\not=j}^{\nu}\left\langle \mathcal{R}_{1,i}, \Psi_j\right\rangle_{L^2}$.

By \eqref{eqn0020} and Lemma~\ref{lem0005},
\begin{eqnarray*}
\sum_{j=1}^{\nu}\sum_{i=1;i\not=j}^{\nu}\left\langle \mathcal{R}_{1,i}, \Psi_j\right\rangle_{L^2}=\sum_{j=1}^{\nu}\sum_{i=1;i\not=j}^{\nu}\left(\left(\alpha_{l}^*\right)^{p-1}-1\right)\left\langle \Psi_i^p, \Psi_j\right\rangle_{L^2}=o(Q).
\end{eqnarray*}

{\bf Step.~3}\quad The estimate of $\sum_{j=1}^{\nu}\left\langle \mathcal{R}_{1,ex}, \Psi_j\right\rangle_{L^2}$.

By the Taylor expansion, \eqref{eqn0020} and Lemma~\ref{lem0005},
\begin{eqnarray*}
\sum_{j=1}^{\nu}\left\langle \mathcal{R}_{1,ex}, \Psi_j\right\rangle_{L^2}&=&\sum_{j=1}^{\nu}\sum_{i=1}^{\nu}\left\langle p\Psi_i^{p-1}\mathcal{U}_i\chi_{\mathcal{B}_{i}}, \Psi_j\right\rangle_{L^2}\notag\\
&&+\mathcal{O}\left(\sum_{j=1}^{\nu}\left\langle\Psi_j^{p-1}\chi_{\mathcal{B}_{j}},\mathcal{U}_j^2\right\rangle_{L^2}+\|\mathcal{U}\|_{L^{p+1}(\mathcal{C}\backslash\cup_{j=1}^{\nu}\mathcal{B}_{j})}\right)\notag\\
&=&(\overline{B}_{1,*}+o(1))Q,
\end{eqnarray*}
where $\overline{B}_{1,*}=p\int_{\mathcal{C}}\Psi^{p-1}d\mu$ is a positive constant.

{\bf Step.~4}\quad The estimate of $\sum_{j=1}^{\nu}\left\langle \mathcal{L}_{j,ex}(\rho_*), \Psi_j\right\rangle_{L^2}$.

By \eqref{eqn0044} and $(1)$ of Proposition~\ref{prop0001},
\begin{eqnarray*}
\left\langle \mathcal{L}_{j,ex}(\rho_*), \Psi_j\right\rangle_{L^2}&=&\left\langle p\left(\mathcal{U}^{p-1}-\left(\Psi_{j}^*\right)^{p-1}\right)\rho_*, \Psi_j\right\rangle_{L^2}\notag\\
&=&\left\langle p\left(\mathcal{U}^{p-1}-\left(\Psi_{j}^*\right)^{p-1}\right)\Psi_j, \rho_0+\rho_{**}^{\perp}\right\rangle_{L^2}.
\end{eqnarray*}
Similar to \eqref{eqn3147},
\begin{eqnarray*}
\left|\left(\mathcal{U}^{p-1}-\left(\Psi_{j}^*\right)^{p-1}\right)\Psi_j\right|&\lesssim&\left(\sum_{i=1}^{\nu}\Psi_i^{p-1}\mathcal{U}_i\chi_{\mathcal{B}_{i}}\right)+\mathcal{U}^{p}\chi_{\mathcal{C}\backslash\cup_{j=1}^{\nu}\mathcal{B}_j}.
\end{eqnarray*}
By Lemma~\ref{lem0005} and $(i)$ of $(1)$ of Proposition~\ref{prop0001},
\begin{eqnarray*}
\left|\left\langle \left(\mathcal{U}^{p-1}-\left(\Psi_{j}^*\right)^{p-1}\right)\Psi_j, \gamma_{ex}\right\rangle_{L^2}\right|&\lesssim&\left\{\aligned
&\sum_{i=1}^{\nu}Q\left\langle\Psi_i^{p-\sigma}\chi_{\mathcal{B}_{i}}, \mathcal{U}_{i}\right\rangle_{L^2},\quad p\geq3,\\
&\sum_{i=1}^{\nu}Q\left\langle\Psi_i^{2p-3}\chi_{\mathcal{B}_{i}}, \mathcal{U}_{i}\right\rangle_{L^2},\quad 1<p<3,\\
\endaligned\right.\\
&=&o(Q).
\end{eqnarray*}
By Lemma~\ref{lem0005} and $(ii)$ of $(1)$ of Proposition~\ref{prop0001},
\begin{eqnarray*}
&&\left|\left\langle \left(\mathcal{U}^{p-1}-\left(\Psi_{j}^*\right)^{p-1}\right)\Psi_j, \gamma_{*}\right\rangle_{L^2}\right|\\
&\lesssim&\sum_{i=1}^{\nu}\left(Q+\sum_{l=1}^{\nu}\left|\left(\alpha_{l}^*\right)^{p-1}-1\right|+\beta_*^2\right)\left\langle\Psi_i^{p-\sigma}\chi_{\mathcal{B}_{i}}, \mathcal{U}_{i}\right\rangle_{L^2}\\
&=&o(Q).
\end{eqnarray*}
By Lemma~\ref{lem0005} and $(iii)$ of $(1)$ of Proposition~\ref{prop0001},
\begin{eqnarray*}
\left|\left\langle \left(\mathcal{U}^{p-1}-\left(\Psi_{j}^*\right)^{p-1}\right)\Psi_j, \gamma_{\mathcal{N},led}\right\rangle_{L^2}\right|&\lesssim&\sum_{i=1}^{\nu}\beta_*^2\left\langle\Psi_i^{p-\sigma}\chi_{\mathcal{B}_{i}}, \mathcal{U}_{i}\right\rangle_{L^2}\\
&=&o(Q).
\end{eqnarray*}
By Lemma~\ref{lem0005},
\begin{eqnarray*}
\left|\left\langle\left(\mathcal{U}^{p-1}-\Psi_j^{p-1}\right)\Psi_{j},\rho_{**}^{\perp}\right\rangle_{L^2}\right|&\lesssim&\left\langle\sum_{i=1}^{\nu}\Psi_i^{p-1}\mathcal{U}_{i}\chi_{\mathcal{B}_{i}},\left|\rho_{**}^{\perp}\right|\right\rangle_{L^2}\notag\\
&\lesssim&\left\{\aligned
&Q\|\rho_{**}^{\perp}\|,\quad p>2,\\
&Q\left|\log Q\right|^{\frac{1}{2}}\|\rho_{**}^{\perp}\|,\quad p=2,\\
&Q^{\frac{p}{2}}\|\rho_{**}^{\perp}\|,\quad 1<p<2.
\endaligned\right.\\
&=&o(Q)+\|\rho_{**}^{\perp}\|^{2+\ve}.
\end{eqnarray*}
Summarizing the above estimates, we have
\begin{eqnarray}\label{eqn0074}
\sum_{j=1}^{\nu}\left\langle \mathcal{L}_{j,ex}(\rho_*), \Psi_j\right\rangle_{L^2}= o(Q)+\mathcal{O}\left(\|\rho_{**}^{\perp}\|^{2+\ve}\right).
\end{eqnarray}

{\bf Step.~5}\quad The estimate of $\sum_{j=1}^{\nu}\left\langle \mathcal{N}-\mathcal{N}_{j}, \Psi_j\right\rangle_{L^2}$.

By Lemmas~\ref{lem0004} and \ref{lemn0001}, \eqref{eqnew9999} and the oddness of $\{\mathcal{V}_i\}$ on $\mathbb{S}^{d-1}$, we have
\begin{eqnarray}
\left\langle \mathcal{N}-\mathcal{N}_{j}, \Psi_j\right\rangle_{L^2}&=&
A_p\sum_{i=1;i\not=j}^{\nu}\left\langle\left(\Psi_i^*\right)^{p-2}\left(\mathcal{V}_i^2+2\mathcal{V}_i\rho_*\right)\chi_{\mathcal{B}_i}, \Psi_j\right\rangle_{L^2}\notag\\
&&+3B_p\sum_{i=1;i\not=j}^{\nu}\left\langle\left(\Psi_i^*\right)^{p-3}\mathcal{V}_i^2\rho_*\chi_{\mathcal{B}_i}, \Psi_j\right\rangle_{L^2}\notag\\
&&-A_p\left\langle\left(\Psi_j^*\right)^{p-2}\left(\mathcal{V}_j^2+2\mathcal{V}_j\rho_*\right)\chi_{\mathcal{C}\backslash\mathcal{B}_j}, \Psi_j\right\rangle_{L^2}\notag\\
&&-3B_p\left\langle\left(\Psi_j^*\right)^{p-3}\mathcal{V}_j^2\rho_*\chi_{\mathcal{C}\backslash\mathcal{B}_j}, \Psi_j\right\rangle_{L^2}\notag\\
&&+\sum_{i=1}^{\nu}\left\langle\mathcal{O}\left(\beta_*|\rho_*|\Psi_{i}^{\frac{3p-5}{2}}\mathcal{U}_i+\beta_*^2\Psi_{i}^{2p-2}\mathcal{U}_i\right)\chi_{\mathcal{B}_i}, \Psi_j\right\rangle_{L^2}\notag\\
&&+\left\langle\mathcal{O}\left(\mathcal{U}^{p-2}\mathcal{V}^{2}+\beta_*|\rho_*|\mathcal{U}^{\frac{3(p-1)}{2}}\right)\chi_{\mathcal{C}\backslash\left(\cup_{j=1}^{\nu}\mathcal{B}_j\right)}, \Psi_j\right\rangle_{L^2}\notag\\
&&+\left\langle\overline{\mathcal{N}}_{rem}+\mathcal{N}_{0}, \Psi_j\right\rangle_{L^2}.\label{eqnew0020}
\end{eqnarray}

{\bf Step.~5.1}\quad The estimate of $\sum_{i=1;i\not=j}^{\nu}\left\langle\left(\Psi_i^*\right)^{p-2}\left(\mathcal{V}_i^2+2\mathcal{V}_i\rho_*\right)\chi_{\mathcal{B}_i}, \Psi_j\right\rangle_{L^2}$.

By $(1)$ of Proposition~\ref{prop0001},
\begin{eqnarray*}
\sum_{i=1;i\not=j}^{\nu}\left\langle\left(\Psi_i^*\right)^{p-2}\left(\mathcal{V}_i^2+2\mathcal{V}_i\rho_*\right)\chi_{\mathcal{B}_i}, \Psi_j\right\rangle_{L^2}&=&\sum_{i=1;i\not=j}^{\nu}\left\langle\left(\Psi_i^*\right)^{p-2}\left(\mathcal{V}_i^2+2\mathcal{V}_i\gamma_{2,ex}\right)\chi_{\mathcal{B}_i}, \Psi_j\right\rangle_{L^2}\\
&&\sum_{i=1;i\not=j}^{\nu}\left\langle\left(\Psi_i^*\right)^{p-2}\mathcal{V}_i\left(\gamma_{\mathcal{N},led}+\rho_{**}^{\perp}\right)\chi_{\mathcal{B}_i}, \Psi_j\right\rangle_{L^2}
\end{eqnarray*}
By Lemma~\ref{lem0005} and $(i)$ of $(1)$ of Proposition~\ref{prop0001},
\begin{eqnarray*}
&&\sum_{i=1;i\not=j}^{\nu}\left|\left\langle\left(\Psi_i^*\right)^{p-2}\left(\mathcal{V}_i^2+2\mathcal{V}_i\gamma_{2,ex}\right)\chi_{\mathcal{B}_i}, \Psi_j\right\rangle_{L^2}\right|\\
&\lesssim&\left\{\aligned
&\sum_{i=1;i\not=j}^{\nu}\left\langle\left(\beta_*^2\Psi_i^{2p-1}+\beta_*^2Q\Psi_i^{\frac{3p-1-2\sigma}{2}}\right)\chi_{\mathcal{B}_{i}}, \Psi_j\right\rangle_{L^2},\quad p\geq\frac{7}{3},\\
&\sum_{i=1;i\not=j}^{\nu}\left\langle\left(\beta_*^2\Psi_i^{2p-1}+\beta_*^2Q\Psi_i^{3p-4}\right)\chi_{\mathcal{B}_{i}}, \Psi_j\right\rangle_{L^2},\quad 1<p<\frac{7}{3}
\endaligned\right.\\
&=&o(Q).
\end{eqnarray*}
By Lemma~\ref{lem0005} and $(iii)$ of $(1)$ of Proposition~\ref{prop0001},
\begin{eqnarray*}
&&\sum_{i=1;i\not=j}^{\nu}\left|\left\langle\left(\Psi_i^*\right)^{p-2}\mathcal{V}_i\gamma_{\mathcal{N},led}\chi_{\mathcal{B}_i}, \Psi_j\right\rangle_{L^2}\right|\\
&\lesssim&\left\{\aligned
&\sum_{i=1;i\not=j}^{\nu}\left\langle\beta_*^3\Psi_i^{\frac{3p-1-2\sigma}{2}}\chi_{\mathcal{B}_{i}}, \Psi_j\right\rangle_{L^2},\quad p\geq3,\\
&\sum_{i=1;i\not=j}^{\nu}\left\langle\left(\beta_*^3\Psi_i^{\frac{3p-1-2\sigma}{2}}+\beta_*^4Q\Psi_i^{\frac{5p-7}{2}}\right)\chi_{\mathcal{B}_{i}}, \Psi_j\right\rangle_{L^2},\quad 1<p<3
\endaligned\right.\\
&=&o(Q).
\end{eqnarray*}
By Lemma~\ref{lem0005},
\begin{eqnarray*}
\sum_{i=1;i\not=j}^{\nu}\left|\left\langle\left(\Psi_i^*\right)^{p-2}\mathcal{V}_i\rho_{**}^{\perp}\chi_{\mathcal{B}_i}, \Psi_j\right\rangle_{L^2}\right|
&\lesssim&\sum_{i=1;i\not=j}^{\nu}\beta_*\|\rho_{**}^{\perp}\|\left\|\Psi_{i}^{\frac{3p-3}{2}}\Psi_j\right\|_{L^2(\mathcal{B}_i)}\\
&=&o(Q)+\mathcal{O}\left(\|\rho_{**}^{\perp}\|^{2+\ve}\right).
\end{eqnarray*}
Summarizing the above estimates, we have
\begin{eqnarray*}
\sum_{j=1}^{\nu}\sum_{i=1;i\not=j}^{\nu}\left|\left\langle\left(\Psi_i^*\right)^{p-2}\left(\mathcal{V}_i^2+2\mathcal{V}_i\rho_*\right)\chi_{\mathcal{B}_i}, \Psi_j\right\rangle_{L^2}\right|
=o(Q)+\mathcal{O}\left(\|\rho_{**}^{\perp}\|^{2+\ve}\right).
\end{eqnarray*}

{\bf Step.~5.2}\quad The estimate of $\left\langle\left(\Psi_j^*\right)^{p-2}\left(\mathcal{V}_j^2+2\mathcal{V}_j\rho_*\right)\chi_{\mathcal{C}\backslash\mathcal{B}_j}, \Psi_j\right\rangle_{L^2}$.

By Lemma~\ref{lem0005} and $(1)$ of Proposition~\ref{prop0001},
\begin{eqnarray*}
\left|\left\langle\left(\Psi_j^*\right)^{p-2}\left(\mathcal{V}_j^2+2\mathcal{V}_j\rho_*\right)\chi_{\mathcal{C}\backslash\mathcal{B}_j}, \Psi_j\right\rangle_{L^2}\right|
&\lesssim&\beta_*^2Q^{p}+\beta_*Q^{\frac{3p-1}{4}}\|\rho_{**}^{\perp}\|\\
&&+\beta_*\left\langle\Psi_j^{\frac{3p-1}{2}}\chi_{\mathcal{C}\backslash\mathcal{B}_j}, \left|\gamma_{2,ex}+\gamma_{\mathcal{N},led}\right|\right\rangle_{L^2}.
\end{eqnarray*}
By Lemma~\ref{lem0005} and $(i)$ of $(1)$ of Proposition~\ref{prop0001},
\begin{eqnarray*}
\left\langle\Psi_j^{\frac{3p-1}{2}}\chi_{\mathcal{C}\backslash\mathcal{B}_j}, \left|\gamma_{2,ex}\right|\right\rangle_{L^2}&\lesssim&\left\{\aligned
&\sum_{i=1;i\not=j}^{\nu}\beta_*Q\left\langle\Psi_i^{1-\sigma}\chi_{\mathcal{B}_i}, \Psi_j^{\frac{3p-1}{2}}\right\rangle_{L^2},\quad p\geq\frac{7}{3},\\
&\sum_{i=1;i\not=j}^{\nu}\beta_*Q\left\langle\Psi_i^{\frac{3p-5}{2}}, \Psi_j^{\frac{3p-1}{2}}\right\rangle_{L^2},\quad 1<p<\frac{7}{3}
\endaligned\right.\\
&=&o(Q).
\end{eqnarray*}
By Lemma~\ref{lem0005} and $(iii)$ of $(1)$ of Proposition~\ref{prop0001},
\begin{eqnarray*}
&&\left\langle\Psi_j^{\frac{3p-1}{2}}\chi_{\mathcal{C}\backslash\mathcal{B}_j}, \left|\gamma_{\mathcal{N},led}\right|\right\rangle_{L^2}\\
&\lesssim&\left\{\aligned
&\sum_{i=1;i\not=j}^{\nu}\beta_*^2\left\langle\Psi_i^{1-\sigma}\chi_{\mathcal{B}_i}, \Psi_j^{\frac{3p-1}{2}}\right\rangle_{L^2},\quad p\geq3,\\
&\sum_{i=1;i\not=j}^{\nu}\left(\beta_*^2\left\langle\Psi_i^{1-\sigma}\chi_{\mathcal{B}_i}, \Psi_j^{\frac{3p-1}{2}}\right\rangle_{L^2}+\beta_*^3Q\left\langle\Psi_i^{p-2}\chi_{\mathcal{B}_i}, \Psi_j^{\frac{3p-1}{2}}\right\rangle_{L^2}\right),\quad 1<p<3
\endaligned\right.\\
&=&o(Q).
\end{eqnarray*}
Summarizing the above estimates, we have
\begin{eqnarray*}
\left\langle\left(\Psi_j^*\right)^{p-2}\left(\mathcal{V}_j^2+2\mathcal{V}_j\rho_*\right)\chi_{\mathcal{C}\backslash\mathcal{B}_j}, \Psi_j\right\rangle_{L^2}=o(Q)+\mathcal{O}\left(\|\rho_{**}^{\perp}\|^{2+\ve}\right).
\end{eqnarray*}

{\bf Step.~5.3}\quad The estimate of $\sum_{i=1;i\not=j}^{\nu}\left\langle\left(\Psi_i^*\right)^{p-3}\mathcal{V}_i^2\rho_*\chi_{\mathcal{B}_i}, \Psi_j\right\rangle_{L^2}$.

By $(i)$ of $(1)$ of Proposition~\ref{prop0001},
\begin{eqnarray*}
\sum_{i=1;i\not=j}^{\nu}\left\langle\left(\Psi_i^*\right)^{p-3}\mathcal{V}_i^2\rho_*\chi_{\mathcal{B}_i}, \Psi_j\right\rangle_{L^2}
&=&\sum_{i=1;i\not=j}^{\nu}\left\langle\left(\Psi_i^*\right)^{p-3}\mathcal{V}_i^2\left(\gamma_{\mathcal{N},led}+\gamma_{*}\right)\chi_{\mathcal{B}_i}, \Psi_j\right\rangle_{L^2}\\
&&+\sum_{i=1;i\not=j}^{\nu}\left\langle\left(\Psi_i^*\right)^{p-3}\mathcal{V}_i^2\left(\gamma_{1,ex}+\rho_{**}^{\perp}\right)\chi_{\mathcal{B}_i}, \Psi_j\right\rangle_{L^2}.
\end{eqnarray*}
By Lemma~\ref{lem0005} and $(i)$ of $(1)$ of Proposition~\ref{prop0001},
\begin{eqnarray*}
\sum_{i=1;i\not=j}^{\nu}\left|\left\langle\left(\Psi_i^*\right)^{p-3}\mathcal{V}_i^2\gamma_{1,ex}\chi_{\mathcal{B}_i}, \Psi_j\right\rangle_{L^2}\right|
&\lesssim&\left\{\aligned
&\beta_*^2Q\left\langle\Psi_i^{2p-1-\sigma}\chi_{\mathcal{B}_{i}}, \Psi_j\right\rangle_{L^2},\quad p\geq3,\\
&\beta_*^2Q\left\langle\Psi_i^{3p-4}\chi_{\mathcal{B}_{i}}, \Psi_j\right\rangle_{L^2},\quad 1<p<3
\endaligned\right.\notag\\
&=&o(Q).
\end{eqnarray*}
By Lemma~\ref{lem0005} and $(ii)$ of $(1)$ of Proposition~\ref{prop0001},
\begin{eqnarray*}
&&\sum_{i=1;i\not=j}^{\nu}\left|\left\langle\left(\Psi_i^*\right)^{p-3}\mathcal{V}_i^2\gamma_{*}\chi_{\mathcal{B}_i}, \Psi_j\right\rangle_{L^2}\right|\notag\\
&\lesssim&\sum_{i=1;i\not=j}^{\nu}\beta_*^2\left(\beta_*^2+\sum_{l=1}^{\nu}\left(\left|\left(\alpha_l^*\right)^{p-1}-1\right|+Q\right)\right)\left\langle\Psi_i^{2p-1-\sigma}\chi_{\mathcal{B}_i}, \Psi_j\right\rangle_{L^2}\notag\\
&=&o(Q).
\end{eqnarray*}
By Lemma~\ref{lem0005} and $(iii)$ of $(1)$ of Proposition~\ref{prop0001},
\begin{eqnarray*}
\sum_{i=1;i\not=j}^{\nu}\left|\left\langle\left(\Psi_i^*\right)^{p-3}\mathcal{V}_i^2\gamma_{\mathcal{N},led}\chi_{\mathcal{B}_i}, \Psi_j\right\rangle_{L^2}\right|&\lesssim&\sum_{i=1;i\not=j}^{\nu}\beta_*^4\left\langle\Psi_i^{2p-1-\sigma}\chi_{\mathcal{B}_i}, \Psi_j\right\rangle_{L^2}\notag\\
&=&o(Q).
\end{eqnarray*}
By Lemma~\ref{lem0005},
\begin{eqnarray*}
\sum_{i=1;i\not=j}^{\nu}\left|\left\langle\left(\Psi_i^*\right)^{p-3}\mathcal{V}_i^2\rho_{**}^{\perp}\chi_{\mathcal{B}_i}, \Psi_j\right\rangle_{L^2}\right|&\lesssim&\beta_*^2\|\rho_{**}^{\perp}\|\left\|\Psi_i^{2p-2}\Psi_j\right\|_{L^2(\mathcal{B}_i)}\\
&=&o(Q)+\mathcal{O}\left(\|\rho_{**}^{\perp}\|^{2+\ve}\right).
\end{eqnarray*}
Summarizing the above estimates, we have
\begin{eqnarray*}
\sum_{i=1;i\not=j}^{\nu}\left\langle\left(\Psi_i^*\right)^{p-3}\mathcal{V}_i^2\rho_*\chi_{\mathcal{B}_i}, \Psi_j\right\rangle_{L^2}=o(Q)+\mathcal{O}\left(\|\rho_{**}^{\perp}\|^{2+\ve}\right).
\end{eqnarray*}

{\bf Step.~5.4}\quad The estimate of $\left\langle\left(\Psi_j^*\right)^{p-3}\mathcal{V}_j^2\rho_*\chi_{\mathcal{C}\backslash\mathcal{B}_j}, \Psi_j\right\rangle_{L^2}$.

By $(i)$ of $(1)$ of Proposition~\ref{prop0001},
\begin{eqnarray*}
\left\langle\left(\Psi_j^*\right)^{p-3}\mathcal{V}_j^2\rho_*\chi_{\mathcal{C}\backslash\mathcal{B}_j}, \Psi_j\right\rangle_{L^2}
&=&\left\langle\left(\Psi_i^*\right)^{p-3}\mathcal{V}_i^2\left(\gamma_{\mathcal{N},led}+\gamma_{*}\right)\chi_{\mathcal{C}\backslash\mathcal{B}_j}, \Psi_j\right\rangle_{L^2}\\
&&+\left\langle\left(\Psi_i^*\right)^{p-3}\mathcal{V}_i^2\left(\gamma_{1,ex}+\rho_{**}^{\perp}\right)\chi_{\mathcal{C}\backslash\mathcal{B}_j}, \Psi_j\right\rangle_{L^2}.
\end{eqnarray*}
By Lemma~\ref{lem0005} and $(i)$ of $(1)$ of Proposition~\ref{prop0001},
\begin{eqnarray*}
\left|\left\langle\left(\Psi_j^*\right)^{p-3}\mathcal{V}_j^2\gamma_{1,ex}\chi_{\mathcal{C}\backslash\mathcal{B}_j}, \Psi_j\right\rangle_{L^2}\right|&\lesssim&\left\{\aligned
&\sum_{i=1;i\not=j}^{\nu}\beta_*^2Q\left\langle\Psi_i^{1-\sigma}\chi_{\mathcal{B}_{i}}, \Psi_j^{2p-1}\right\rangle_{L^2},\quad p\geq3,\\
&\sum_{i=1;i\not=j}^{\nu}\beta_*^2Q\left\langle\Psi_i^{p-2}\chi_{\mathcal{B}_{i}}, \Psi_j^{2p-1}\right\rangle_{L^2},\quad 1<p<3,\\
\endaligned\right.\\
&=&o(Q).
\end{eqnarray*}
By Lemma~\ref{lem0005} and $(ii)$ of $(1)$ of Proposition~\ref{prop0001},
\begin{eqnarray*}
&&\left|\left\langle\left(\Psi_j^*\right)^{p-3}\mathcal{V}_j^2\gamma_{*}\chi_{\mathcal{C}\backslash\mathcal{B}_j}, \Psi_j\right\rangle_{L^2}\right|\\
&\lesssim&
\sum_{i=1;i\not=j}^{\nu}\beta_*^2\left(\beta_*^2+\sum_{l=1}^{\nu}\left|\left(\alpha_l^*\right)^{p-1}-1\right|+Q\right)\left\langle\Psi_i^{1-\sigma}\chi_{\mathcal{B}_{i}}, \Psi_j^{2p-1}\right\rangle_{L^2}\\
&=&o(Q).
\end{eqnarray*}
By Lemma~\ref{lem0005} and $(iii)$ of $(1)$ of Proposition~\ref{prop0001},
\begin{eqnarray*}
\left|\left\langle\left(\Psi_j^*\right)^{p-3}\mathcal{V}_j^2\gamma_{\mathcal{N},led}\chi_{\mathcal{C}\backslash\mathcal{B}_j}, \Psi_j\right\rangle_{L^2}\right|=o(Q).
\end{eqnarray*}
By Lemma~\ref{lem0005},
\begin{eqnarray*}
\left|\left\langle\left(\Psi_j^*\right)^{p-3}\mathcal{V}_j^2\rho_{**}^{\perp}\chi_{\mathcal{C}\backslash\mathcal{B}_j}, \Psi_j\right\rangle_{L^2}\right|
=o(Q)+\mathcal{O}\left(\|\rho_{**}^{\perp}\|^{2+\ve}\right).
\end{eqnarray*}
Summarizing the above estimates, we have
\begin{eqnarray*}
\left|\left\langle\left(\Psi_j^*\right)^{p-3}\mathcal{V}_j^2\rho_*\chi_{\mathcal{C}\backslash\mathcal{B}_j}, \Psi_j\right\rangle_{L^2}\right|=o(Q)+\mathcal{O}\left(\|\rho_{**}^{\perp}\|^{2+\ve}\right).
\end{eqnarray*}

{\bf Step.~5.5}\quad The estimate of $\sum_{i=1}^{\nu}\left\langle\beta_*\rho_*\Psi_{i}^{\frac{3p-5}{2}}\mathcal{U}_i\chi_{\mathcal{B}_i}, \Psi_j\right\rangle_{L^2}$.

By $(1)$ of Proposition~\ref{prop0001},
\begin{eqnarray*}
\sum_{i=1}^{\nu}\left\langle\beta_*\rho_*\Psi_{i}^{\frac{3p-5}{2}}\mathcal{U}_i\chi_{\mathcal{B}_i}, \Psi_j\right\rangle_{L^2}
&=&\sum_{i=1}^{\nu}\beta_*\left\langle\Psi_{i}^{\frac{3p-3}{2}}\mathcal{U}_{i}\chi_{\mathcal{B}_{i}}, \gamma_{\mathcal{N},led}+\gamma_{*}\right\rangle_{L^2}\\
&&+\sum_{i=1}^{\nu}\beta_*\left\langle\Psi_{i}^{\frac{3p-3}{2}}\mathcal{U}_{i}\chi_{\mathcal{B}_{i}}, \gamma_{ex}+\rho_{**}^{\perp}\right\rangle_{L^2}.
\end{eqnarray*}
By Lemma~\ref{lem0005} and $(i)$ of $(1)$ of Proposition~\ref{prop0001},
\begin{eqnarray*}
\sum_{i=1}^{\nu}\beta_*\left|\left\langle\Psi_{i}^{\frac{3p-3}{2}}\mathcal{U}_{i}\chi_{\mathcal{B}_{i}}, \gamma_{ex}\right\rangle_{L^2}\right|
&\lesssim&\left\{\aligned
&\sum_{i=1}^{\nu}\beta_*Q\left\langle\Psi_{i}^{\frac{3p-1-2\sigma}{2}}\chi_{\mathcal{B}_{i}},\mathcal{U}_{i}\right\rangle_{L^2},\quad p\geq3,\\
&\sum_{i=1}^{\nu}\beta_*Q\left\langle\Psi_{i}^{\frac{5p-7}{2}}\chi_{\mathcal{B}_{i}},\mathcal{U}_{i}\right\rangle_{L^2},\quad 1<p<3
\endaligned\right.\\
&=&o(Q).
\end{eqnarray*}
By Lemma~\ref{lem0005} and $(ii)$ of $(1)$ of Proposition~\ref{prop0001},
\begin{eqnarray*}
&&\sum_{i=1}^{\nu}\beta_*\left|\left\langle\Psi_{i}^{\frac{3p-3}{2}}\mathcal{U}_{i}\chi_{\mathcal{B}_{i}}, \gamma_{*}\right\rangle_{L^2}\right|\\
&\lesssim&\sum_{i=1}^{\nu}\beta_*\left(\beta_*^2+\sum_{l=1}^{\nu}\left|\left(\alpha_l^*\right)^{p-1}-1\right|+Q\right)\left\langle\Psi_{i}^{\frac{3p-1-2\sigma}{2}},\mathcal{U}_i\chi_{\mathcal{B}_i}\right\rangle_{L^2}\\
&=&o(Q).
\end{eqnarray*}
By Lemma~\ref{lem0005} and $(iii)$ of $(1)$ of Proposition~\ref{prop0001},
\begin{eqnarray*}
\sum_{i=1}^{\nu}\beta_*\left|\left\langle\Psi_{i}^{\frac{3p-3}{2}}\mathcal{U}_{i}\chi_{\mathcal{B}_{i}}, \gamma_{\mathcal{N},led}\right\rangle_{L^2}\right|
=o(Q).
\end{eqnarray*}
By Lemma~\ref{lem0005},
\begin{eqnarray*}
\sum_{i=1}^{\nu}\beta_*\left|\left\langle\Psi_{i}^{\frac{3p-3}{2}}\mathcal{U}_{i}\chi_{\mathcal{B}_{i}}, \rho_{**}^{\perp}\right\rangle_{L^2}\right|&\lesssim&\sum_{i=1}^{\nu}\|\rho_{**}^{\perp}\|\left\|\Psi_{i}^{\frac{3p-3}{2}}\mathcal{U}_{i}\right\|_{L^2(\mathcal{B}_i)}\\
&=&o(Q)+\mathcal{O}\left(\|\rho_{**}^{\perp}\|^{2+\ve}\right).
\end{eqnarray*}
Summarizing the above estimates, we have
\begin{eqnarray*}
\sum_{i=1}^{\nu}\left|\left\langle\beta_*\rho_*\Psi_{i}^{\frac{3p-5}{2}}\mathcal{U}_i\chi_{\mathcal{B}_i}, \Psi_j\right\rangle_{L^2}\right|
=o(Q)+\mathcal{O}\left(\|\rho_{**}^{\perp}\|^{2+\ve}\right).
\end{eqnarray*}

{\bf Step.~5.6}\quad The estimate of $\sum_{i=1}^{\nu}\beta_*^2\left\langle\Psi_{i}^{2p-2}\mathcal{U}_i\chi_{\mathcal{B}_i}, \Psi_j\right\rangle_{L^2}$.

By Lemma~\ref{lem0005},
\begin{eqnarray*}
\sum_{i=1}^{\nu}\beta_*^2\left|\left\langle\Psi_{i}^{2p-2}\mathcal{U}_i\chi_{\mathcal{B}_i}, \Psi_j\right\rangle_{L^2}\right|\lesssim\sum_{i=1}^{\nu}\beta_*^2\left\langle\Psi_{i}^{2p-1}\chi_{\mathcal{B}_{i}}, \mathcal{U}_{i}\right\rangle_{L^2}=o(Q).
\end{eqnarray*}

{\bf Step.~5.7}\quad The estimate of $\left\langle\left(\mathcal{U}^{p-2}\mathcal{V}^{2}+\beta_*\rho_*\mathcal{U}^{\frac{3(p-1)}{2}}\right)\chi_{\mathcal{C}\backslash\left(\cup_{j=1}^{\nu}\mathcal{B}_j\right)}, \Psi_j\right\rangle_{L^2}$.

By $(1)$ of Proposition~\ref{prop0001},
\begin{eqnarray*}
&&\left|\left\langle\left(\mathcal{U}^{p-2}\mathcal{V}^{2}+\beta_*\rho_*\mathcal{U}^{\frac{3(p-1)}{2}}\right)\chi_{\mathcal{C}\backslash\left(\cup_{j=1}^{\nu}\mathcal{B}_j\right)}, \Psi_j\right\rangle_{L^2}\right|\\
&\lesssim&\left|\left\langle\left(\mathcal{U}^{p-2}\mathcal{V}^{2}+\beta_*|\gamma_{ex}+\gamma_{*}+\gamma_{\mathcal{N},led}|\mathcal{U}^{\frac{3(p-1)}{2}}\right)\chi_{\mathcal{C}\backslash\left(\cup_{j=1}^{\nu}\mathcal{B}_j\right)}, \Psi_j\right\rangle_{L^2}\right|\\
&&+\left|\left\langle\beta_*|\rho_{**}^{\perp}|\mathcal{U}^{\frac{3(p-1)}{2}}\chi_{\mathcal{C}\backslash\left(\cup_{j=1}^{\nu}\mathcal{B}_j\right)}, \Psi_j\right\rangle_{L^2}\right|.
\end{eqnarray*}
By Lemma~\ref{lem0005} and $(1)$ of Proposition~\ref{prop0001},
\begin{eqnarray*}
&&\left|\left\langle\left(\mathcal{U}^{p-2}\mathcal{V}^{2}+\beta_*|\gamma_{ex}+\gamma_{*}+\gamma_{\mathcal{N},led}|\mathcal{U}^{\frac{3(p-1)}{2}}\right)\chi_{\mathcal{C}\backslash\left(\cup_{j=1}^{\nu}\mathcal{B}_j\right)}, \Psi_j\right\rangle_{L^2}\right|\\
&\lesssim&\beta_*\left(Q+\beta_*^2+\sum_{l=1}^{\nu}\left(\left|\left(\alpha_l^*\right)^{p-1}-1\right|\right)\right)\left\|\mathcal{U}^{\frac{3p+1-2\sigma}{2}}\right\|_{L^1\left(\mathcal{C}\backslash\left(\cup_{j=1}^{\nu}\mathcal{B}_j\right)\right)}\\
&&+\beta_*^2\left\|\mathcal{U}^{2p}\right\|_{L^1\left(\mathcal{C}\backslash\left(\cup_{j=1}^{\nu}\mathcal{B}_j\right)\right)}\\
&=&o(Q).
\end{eqnarray*}
By Lemma~\ref{lem0005},
\begin{eqnarray*}
\left|\left\langle\beta_*|\rho_{**}^{\perp}|\mathcal{U}^{\frac{3(p-1)}{2}}\chi_{\mathcal{C}\backslash\left(\cup_{j=1}^{\nu}\mathcal{B}_j\right)}, \Psi_j\right\rangle_{L^2}\right|&\lesssim&\beta_*\|\rho_{**}^{\perp}\|\left\|\mathcal{U}^{\frac{3p-1}{2}}\right\|_{L^2\left(\mathcal{C}\backslash\left(\cup_{j=1}^{\nu}\mathcal{B}_j\right)\right)}\\
&=&o(Q)+\mathcal{O}\left(\|\rho_{**}^{\perp}\|^{2+\ve}\right).
\end{eqnarray*}
Summarizing the above estimates, we have
\begin{eqnarray*}
\left\langle\left(\mathcal{U}^{p-2}\mathcal{V}^{2}+\beta_*|\rho_*|\mathcal{U}^{\frac{3(p-1)}{2}}\right)\chi_{\mathcal{C}\backslash\left(\cup_{j=1}^{\nu}\mathcal{B}_j\right)}, \Psi_j\right\rangle_{L^2}=o(Q)+\mathcal{O}\left(\|\rho_{**}^{\perp}\|^{2+\ve}\right).
\end{eqnarray*}

{\bf Step.~5.8}\quad The estimate of $\left\langle\mathcal{N}_{0}, \Psi_j\right\rangle_{L^2}$.

By \eqref{eqn0191}, \eqref{eqn0190} and Lemma~\ref{lemn0001},
\begin{eqnarray*}
\left|\left\langle\mathcal{N}_{0}, \Psi_j\right\rangle_{L^2}\right|\lesssim\left\langle\mathcal{U}^{p-2-\sigma}\Psi_j, \gamma_{ex}^2+\left|\gamma_{\mathcal{N},led}+\gamma_{*}\right|^2\right\rangle_{L^2}+\|\rho_{**}^{\perp}\|^2.
\end{eqnarray*}
By Lemma~\ref{lem0005} and $(i)$ of $(1)$ of Proposition~\ref{prop0001},
\begin{eqnarray}\label{eqnew9998}
\left\langle\mathcal{U}^{p-2-\sigma}\Psi_j, \gamma_{ex}^2\right\rangle_{L^2}
&\lesssim&\left\{\aligned
&\sum_{i=1}^{\nu}Q^2\left\langle \Psi_i^{p-3\sigma}, \Psi_j\right\rangle_{L^2(\mathcal{B}_i)},\quad p\geq3,\\
&\sum_{i=1}^{\nu}Q^2\left\langle \Psi_i^{3(p-2)-\sigma}, \Psi_j\right\rangle_{L^2(\mathcal{B}_i)},\quad 1<p<3
\endaligned\right.\notag\\
&=&o(Q).
\end{eqnarray}
By Lemma~\ref{lem0005} and $(ii)$ of $(1)$ of Proposition~\ref{prop0001},
\begin{eqnarray*}
\left\langle\mathcal{U}^{p-2}\Psi_j, \left|\gamma_{*}\right|^2\right\rangle_{L^2}
\lesssim\left(Q+\beta_*^2+\sum_{l=1}^{\nu}\left|\left(\alpha_l^*\right)^{p-1}-1\right|\right)^{2}.
\end{eqnarray*}
By Lemma~\ref{lem0005} and $(iii)$ of $(1)$ of Proposition~\ref{prop0001},
\begin{eqnarray*}
\left\langle\mathcal{U}^{p-2}\Psi_j, \gamma_{\mathcal{N},led}^2\right\rangle_{L^2}=o(Q)+\mathcal{O}\left(\beta_*^2+\sum_{l=1}^{\nu}\left|\left(\alpha_l^*\right)^{p-1}-1\right|\right)^{2}.
\end{eqnarray*}
Thus, summarizing the above estimates, we have
\begin{eqnarray*}
\left|\left\langle\mathcal{N}_{0}, \Psi_j\right\rangle_{L^2}\right|\lesssim\|\rho_{**}^{\perp}\|^{2}+o(Q)+\left(\beta_*^2+\sum_{l=1}^{\nu}\left|\left(\alpha_l^*\right)^{p-1}-1\right|\right)^{2}.
\end{eqnarray*}

{\bf Step.~5.9}\quad The estimate of $\left\langle\overline{\mathcal{N}}_{rem}, \Psi_j\right\rangle_{L^2}$.

By \eqref{eqnewnew3345}, Lemma~\ref{lemn0001} and $(ii)$ and $(iii)$ of Proposition~\ref{prop0001},
\begin{eqnarray*}
\left|\left\langle\overline{\mathcal{N}}_{rem}, \Psi_j\right\rangle_{L^2}\right|\lesssim\|\rho_{**}^{\perp}\|^{1+\ve}+\left(\beta_*+\sum_{l=1}^{\nu}\left|\left(\alpha_l^*\right)^{p-1}-1\right|\right)^{4}+o(Q).
\end{eqnarray*}

By summarizing the estimates from Step.~5.1 to Step.~5.9, we have
\begin{eqnarray*}
\left|\left\langle \mathcal{N}-\mathcal{N}_{j}, \Psi_j\right\rangle_{L^2}\right|
\lesssim\|\rho_{**}^{\perp}\|^{1+\ve}+o\left(Q\right)+\left(\beta_*^2+\sum_{l=1}^{\nu}\left|\left(\alpha_l^*\right)^{p-1}-1\right|\right)^2.
\end{eqnarray*}

The conclusion follows from the estimates from Step.~1 to Step.5.
\end{proof}

\section{Final expansion of $\mathcal{N}$ and estimates of $Q$}
Again, we emphasize that we need to eliminate the lower order terms (compared to the $\beta_*^4$ terms) in the data $\mathcal{R}_{new,0}$ which is given in $(2)$ of Proposition~\ref{prop0001} to get the desired stability inequality.
Thus, we need to finally further expand $\overline{\mathcal{N}}_{rem}$, the remaining term in the expansion of $\mathcal{N}$ given by Lemma~\ref{lemn0001}, into higher order terms.
\begin{lemma}\label{lemn0002}
Let $d\geq2$, $a<0$ and $b=b_{FS}(a)$.  Then $\overline{\mathcal{N}}_{rem}$, the remaining term in the expansion of $\mathcal{N}$ given by Lemma~\ref{lemn0001}, can be further expanded as follows:
\begin{eqnarray*}
\overline{\mathcal{N}}_{rem}&=&C_p\mathcal{U}^{p-4}\left(\overline{\mathcal{V}}+\rho_{**}^{\perp}\right)^4+D_p\mathcal{U}^{p-5}\left(\overline{\mathcal{V}}+\rho_{**}^{\perp}\right)^5\\
&&+\mathcal{O}\left(\mathcal{U}^{p-6\sigma}\left(\beta_*+Q^{\frac{2\wedge p}{2}}+\sum_{j=1}^{\nu}\left|\left(\alpha_{j}^{*}\right)^{p-1}-1\right|\right)^6\right)\notag\\
&&+\mathcal{O}\left(\chi_{p\geq2}|\rho_{**}^{\perp}|^{2}+|\rho_{**}^{\perp}|^{p}\right)\\
&=&C_p\mathcal{U}^{p-4}\left(\overline{\mathcal{V}}^4+4\overline{\mathcal{V}}^3\rho_{**}^{\perp}\right)+D_p\mathcal{U}^{p-5}\left(\overline{\mathcal{V}}^5+5\overline{\mathcal{V}}^4\rho_{**}^{\perp}\right)\\
&&+\mathcal{O}\left(\mathcal{U}^{p-6\sigma}\left(\beta_*+Q^{\frac{2\wedge p}{2}}\left|\log Q\right|+\sum_{j=1}^{\nu}\left|\left(\alpha_{j}^{*}\right)^{p-1}-1\right|\right)^6\right)\notag\\
&&+\mathcal{O}\left(\chi_{p\geq2}|\rho_{**}^{\perp}|^{2}+|\rho_{**}^{\perp}|^{p}+|\gamma_*+\gamma_{\mathcal{N},led}|^{1+\ve}\chi_{\mathcal{C}\backslash\overline{\mathcal{B}}_*}\right)
\end{eqnarray*}
in $\mathcal{C}$ where $\overline{\mathcal{V}}=\mathcal{V}+\rho_{0}$ with $\mathcal{V}$ given by \eqref{eqn0040}, $\rho_0$ given in Proposition~\ref{prop0001}, $C_p=\frac{p(p-1)(p-2)(p-3)}{24}$, $D_p=\frac{p(p-1)(p-2)(p-3)(p-4)}{120}$ and and $\sigma$ and $\overline{\mathcal{B}}_*$ are given by \eqref{eqnewnew3345}.
\end{lemma}
\begin{proof}
The proof is a direct application of the Taylor expansion to $\mathcal{N}$ in the set $\mathcal{A}_*$, which is introduced in the proof of Lemma~\ref{lemn0001}, up to the sixth order term.
\end{proof}

\vskip0.12in

By multiplying \eqref{eq0014} with $\partial_t\Psi_j$ on both sides and integrating by parts, the orthogonal conditions of $\rho_*$ given in \eqref{eq0014} and the oddness of $\{\mathcal{V}_i\}$ on $\mathbb{S}^{d-1}$ and $\partial_t\Psi_j$ in $\bbr$, we have
\begin{eqnarray}\label{eqn0022}
-\left\langle f, \sum_{j=1}^{\nu}\partial_t\Psi_j\right\rangle_{H^1}&=&\sum_{j=1}^{\nu}\sum_{i=1;i\not=j}^{\nu}\left\langle \mathcal{R}_{1,i}, \partial_t\Psi_j\right\rangle_{L^2}+\sum_{j=1}^{\nu}\left\langle \mathcal{L}_{j,ex}(\rho_*), \partial_t\Psi_j\right\rangle_{L^2}\notag\\
&&+\sum_{j=1}^{\nu}\left\langle \mathcal{N}, \partial_t\Psi_j\right\rangle_{L^2}+\sum_{j=1}^{\nu}\left\langle \mathcal{R}_{1,ex}, \partial_t\Psi_j\right\rangle_{L^2}.
\end{eqnarray}
In what follows, we shall derive the estimate of $Q$ from \eqref{eqn0022}.
\begin{proposition}\label{propn0002}
Let $d\geq2$, $a<0$ and $b=b_{FS}(a)$.  Then we have
\begin{eqnarray*}
Q=\mathcal{O}\left(\beta_*^{6}+\|\rho_{**}^{\perp}\|^{1+\ve}+\|f\|_{H^{-1}}\right).
\end{eqnarray*}
\end{proposition}
\begin{proof}
By the the orthogonal conditions of $\Psi_j$, $\partial_t\Psi_j$, $w_{j,l}$ and $\rho_*$, and the oddness of $w_{j,l}$ on $\mathbb{S}^{d-1}$ and $\partial_t\Psi_j$ in $\bbr$, we also have
\begin{eqnarray}\label{eqn0026}
\left\langle \mathcal{N}_{j}, \partial_t\Psi_j\right\rangle_{L^2}=2A_p\left\langle\left(\Psi_j^*\right)^{p-2}\mathcal{V}_j\rho_*, \partial_t\Psi_j\right\rangle_{L^2}+3B_p\left\langle\left(\Psi_j^*\right)^{p-3}\mathcal{V}_j^2\rho_*, \partial_t\Psi_j\right\rangle_{L^2}
\end{eqnarray}
for all $j=1,2,\cdots,\nu$, where $\mathcal{N}_{j}$ is given by \eqref{eqnew9999}.
Intersecting \eqref{eqn0026} into \eqref{eqn0022}, we have
\begin{eqnarray}\label{eqn2022}
-\sum_{j=1}^{\nu}\left\langle f, \partial_t\Psi_j\right\rangle_{H^1}&=&\sum_{j=1}^{\nu}\left\langle \mathcal{R}_{1,ex}, \partial_t\Psi_j\right\rangle_{L^2}+\sum_{j=1}^{\nu}\left\langle \mathcal{N}_{j}, \partial_t\Psi_j\right\rangle_{L^2}\notag\\
&&+\sum_{j=1}^{\nu}\left\langle \mathcal{N}-\mathcal{N}_{j}, \partial_t\Psi_j\right\rangle_{L^2}+\sum_{j=1}^{\nu}\left\langle \mathcal{L}_{j,ex}(\rho_*), \partial_t\Psi_j\right\rangle_{L^2}\notag\\
&&+\sum_{j=1}^{\nu}\sum_{i=1;i\not=j}^{\nu}\left\langle \mathcal{R}_{1,i}, \partial_t\Psi_j\right\rangle_{L^2}.
\end{eqnarray}
As in the proof of Proposition~\ref{prop0001}, the rest of the proof is to estimate every terms in \eqref{eqn2022}.

{\bf Step.~1}\quad The estimate of $\sum_{j=1}^{\nu}\left\langle \mathcal{R}_{1,ex}, \partial_t\Psi_j\right\rangle_{L^2}$.

By \eqref{eq0026}, \eqref{eqn0020}, Lemma~\ref{lem0005} and the Taylor expansion,
\begin{eqnarray*}
\sum_{j=1}^{\nu}\left\langle \mathcal{R}_{1,ex}, \partial_t\Psi_j\right\rangle_{L^2}&=&\sum_{j=1}^{\nu}\sum_{i=1}^{\nu}\int_{\mathcal{B}_i}\left(\mathcal{U}^p-\sum_{l=1}^{\nu}\left(\Psi_l^*\right)^p\right)\partial_t\Psi_jd\mu+\mathcal{O}\left(Q^{\frac{p+1}{2}}\right)\\
&=&\sum_{j=1}^{\nu}p\int_{\mathcal{B}_j}\left(\Psi_j^*\right)^{p-1}\left(\Psi_{j+1}^*+\Psi_{j-1}^*\right)\partial_t\Psi_jd\mu\\
&&+\mathcal{O}\left(\sum_{i=1}^{\nu}\int_{\mathcal{B}_{i}}\Psi_i^{p-1}\mathcal{U}_{i}^2d\mu\right)+\mathcal{O}\left(Q^{\frac{p+1}{2}}\right)\\
&=&p\int_{\mathcal{B}_j}\left(\Psi_j^*\right)^{p-1}\left(\Psi_{j+1}^*+\Psi_{j-1}^*\right)\partial_t\Psi_jd\mu+o(Q),
\end{eqnarray*}
where by \eqref{eq0026} and Lemma~\ref{lem0005} again,
\begin{eqnarray*}
\sum_{j=1}^{\nu}p\int_{\mathcal{B}_j}\left(\Psi_j^*\right)^{p-1}\left(\Psi_{j+1}^*+\Psi_{j-1}^*\right)\partial_t\Psi_jd\mu&=&\sum_{j=1}^{\nu}\frac{\int_{\mathcal{C}}\partial_t\left(\Psi_j^*\right)^p\left(\Psi_{j+1}^*+\Psi_{j-1}^*\right)d\mu}{\alpha_j^*}\\
&&+\mathcal{O}\left(Q^{\frac{p+1}{2}}\right)\\
&=&-\sum_{j=1}^{\nu}\frac{\int_{\mathcal{C}}\left(\Psi_j^*\right)^p\partial_t\left(\Psi_{j+1}^*+\Psi_{j-1}^*\right)d\mu}{\alpha_j^*}\\
&&+\mathcal{O}\left(Q^{\frac{p+1}{2}}\right)\\
&=&(B_2+o(1))Q
\end{eqnarray*}
with $B_2>0$ being a constant.
Thus, summarizing the above estimates, we have
\begin{eqnarray*}
\left\langle \mathcal{R}_{1,ex}, \partial_t\Psi_j\right\rangle_{L^2}=(B_2+o(1))Q.
\end{eqnarray*}

{\bf Step.~2}\quad The estimate of $\sum_{j=1}^{\nu}\sum_{i=1;i\not=j}^{\nu}\left\langle \mathcal{R}_{1,i}, \partial_t\Psi_j\right\rangle_{L^2}$.

By \eqref{eqn0020} and Lemma~\ref{lem0005},
\begin{eqnarray*}
\left|\sum_{i=1;i\not=j}^{\nu}\left\langle \mathcal{R}_{1,i}, \partial_t\Psi_j\right\rangle_{L^2}\right|\lesssim\sum_{i=1;i\not=j}^{\nu}\left|\left(\alpha_{i}^*\right)^{p-1}-1\right|\left\langle \Psi_{i}^{p}, \Psi_j\right\rangle_{L^2}=o(Q).
\end{eqnarray*}

{\bf Step.~3}\quad The estimate of $\left\langle \mathcal{L}_{j,ex}(\rho_*), \partial_t\Psi_j\right\rangle_{L^2}$.

By \eqref{eq0026} and \eqref{eqn0074},
\begin{eqnarray*}
\left|\left\langle \mathcal{L}_{j,ex}(\rho_*), \partial_t\Psi_j\right\rangle_{L^2}\right|
\lesssim\left\langle \left|\mathcal{L}_{j,ex}(\rho_*)\right|, \Psi_j\right\rangle_{L^2}
=o\left(Q\right)+\mathcal{O}\left(\|\rho_{**}^{\perp}\|^{2+\ve}\right).
\end{eqnarray*}

{\bf Step.~4}\quad The estimate of $\left\langle \mathcal{N}_{j}, \partial_t\Psi_j\right\rangle_{L^2}$.

By \eqref{eqn0026} and $(1)$ of Proposition~\ref{prop0001},
\begin{eqnarray*}
\left\langle \mathcal{N}_{j}, \partial_t\Psi_j\right\rangle_{L^2}&=&2A_p\left\langle \left(\Psi_j^*\right)^{p-2}\mathcal{V}_j\partial_t\Psi_j,\gamma_{2,ex}+\gamma_{\mathcal{N},led,rem,j}+\rho_{**}^{\perp}\right\rangle_{L^2}\\
&&+3B_p\left\langle \left(\Psi_j^*\right)^{p-3}\mathcal{V}_j^2\partial_t\Psi_j,\gamma_{1,ex}+\rho_{**}^{\perp}\right\rangle_{L^2}\\
&&+3B_p\left\langle \left(\Psi_j^*\right)^{p-3}\mathcal{V}_j^2\partial_t\Psi_j,\sum_{l=1;l\not=j}^{\nu}\gamma_{1,l}+\gamma_{\mathcal{N},led,rem,j}\right\rangle_{L^2}.
\end{eqnarray*}
By \eqref{eq0026} and Lemma~\ref{lem0005},
\begin{eqnarray*}
\left|\left\langle \left(\Psi_j^*\right)^{p-3}\mathcal{V}_j^2\partial_t\Psi_j,\gamma_{1,ex}\right\rangle_{L^2}\right|
&\lesssim&\left\{\aligned
&\sum_{i=1}^{\nu}\beta_*^2Q\left\langle \Psi_i^{1-\sigma}\chi_{\mathcal{B}_{i}},\Psi_j^{2p-1}\right\rangle_{L^2},\quad p\geq3,\\
&\sum_{i=1}^{\nu}\beta_*^2Q\left\langle \Psi_i^{p-2}\chi_{\mathcal{B}_{i}},\Psi_j^{2p-1}\right\rangle_{L^2},\quad 1<p<3
\endaligned\right.\\
&=&o(Q).
\end{eqnarray*}
By \eqref{eq0026}, Lemma~\ref{lem0007} and $(iii)$ of $(1)$ of Proposition~\ref{prop0001},
\begin{eqnarray*}
&&\left|\left\langle \left(\Psi_j^*\right)^{p-3}\mathcal{V}_j^2\partial_t\Psi_j,\sum_{l=1;l\not=j}^{\nu}\gamma_{1,l}+\gamma_{\mathcal{N},led,rem,j}\right\rangle_{L^2}\right|\\
&\lesssim&\beta_*^2\left(\beta_*^2+\sum_{i=1;i\not=j}^{\nu}\left|\left(\alpha_{i}^*\right)^{p-1}-1\right|\right)\sum_{l=1;l\not=j}^{\nu}\left\langle \Psi_i^{1-\sigma}\chi_{\mathcal{B}_{l}},\Psi_j^{2p-1}\right\rangle_{L^2}\\
&=&o(Q).
\end{eqnarray*}
By \eqref{eq0026}, Lemma~\ref{lem0005} and $(i)$ of $(1)$ of Proposition~\ref{prop0001},
\begin{eqnarray*}
\left|\left\langle \left(\Psi_j^*\right)^{p-2}\mathcal{V}_j\partial_t\Psi_j,\gamma_{2,ex}\right\rangle_{L^2}\right|
&\lesssim&\left\{\aligned
&\sum_{i=1}^{\nu}\beta_*^2Q\left\langle \Psi_i^{1-\sigma}
\chi_{\mathcal{B}_{i}},\Psi_j^{\frac{3p-1}{2}}\right\rangle_{L^2},\quad p\geq\frac{7}{3},\\
&\sum_{i=1}^{\nu}\beta_*^2Q\left\langle \Psi_i^{\frac{3p-5}{2}}
\chi_{\mathcal{B}_{i}},\Psi_j^{\frac{3p-1}{2}}\right\rangle_{L^2},\quad 1<p<\frac{7}{3}
\endaligned\right.\notag\\
&=&o(Q).
\end{eqnarray*}
By \eqref{eq0026}, Lemma~\ref{lem0005} and $(iii)$ of $(1)$ of Proposition~\ref{prop0001},
\begin{eqnarray*}
\left|\left\langle \left(\Psi_j^*\right)^{p-2}\mathcal{V}_j\partial_t\Psi_j,\gamma_{\mathcal{N},led,rem,j}\right\rangle_{L^2}\right|
\lesssim\sum_{i=1;i\not=j}^{\nu}\beta_*^3\left\langle \Psi_i^{1-\sigma}
\chi_{\mathcal{B}_{i}},\Psi_j^{\frac{3p-1}{2}}\right\rangle_{L^2}=o(Q).
\end{eqnarray*}
Thus, summarizing the above estimates, we have
\begin{eqnarray*}
\left\langle \mathcal{N}_{j}, \partial_t\Psi_j\right\rangle_{L^2}=o(Q)+\mathcal{O}\left(\|\rho_{**}^{\perp}\|^{2+\ve}\right).
\end{eqnarray*}

{\bf Step.~5}\quad The estimate of $\left\langle \mathcal{N}-\mathcal{N}_{j}, \partial_t\Psi_j\right\rangle_{L^2}$.

Since $\left|\partial_t\Psi\right|\lesssim\Psi$ by \eqref{eq0026}, we can use similar estimates of \eqref{eqnew0020} to obtain
\begin{eqnarray*}
\left|\left\langle \mathcal{N}-\mathcal{N}_{j}, \partial_t\Psi_j\right\rangle_{L^2}-\left\langle\overline{\mathcal{N}}_{rem}+\mathcal{N}_{0}, \partial_t\Psi_j\right\rangle_{L^2}\right|
=o(Q)+\mathcal{O}\left(\|\rho_{**}^{\perp}\|^{2+\ve}\right).
\end{eqnarray*}

{\bf Step.~5.1}\quad The estimate of $\left\langle\mathcal{N}_{0}, \partial_t\Psi_j\right\rangle_{L^2}$.

{\bf Step.~5.1.1}\quad The estimate of $\left\langle\mathcal{N}_{0}-\overline{\mathcal{N}}_{0,1}, \partial_t\Psi_j\right\rangle_{L^2}$, where
$\overline{\mathcal{N}}_{0,1}=A_p\mathcal{U}^{p-2}\rho_0^2+B_p\mathcal{U}^{p-3}\left(3\mathcal{V}\rho_0^2+\rho_0^3\right)$.

By \eqref{eqn0191}, $(1)$ of Proposition~\ref{prop0001} and Lemma~\ref{lemn0001},
\begin{eqnarray*}
\left\langle\mathcal{N}_{0}-\overline{\mathcal{N}}_{0,1}, \partial_t\Psi_j\right\rangle_{L^2}
\lesssim\left\langle\mathcal{U}^{p-2-\sigma}\Psi_j\left(\left|\gamma_{ex}\right|+\left|\gamma_{\mathcal{N},led}+\gamma_{*}\right|\right), \left|\rho_{**}^{\perp}\right|\right\rangle_{L^2}.
\end{eqnarray*}
By Lemma~\ref{lem0005} and $(i)$ of $(1)$ of Proposition~\ref{prop0001},
\begin{eqnarray*}
\left\langle\mathcal{U}^{p-2-\sigma}\Psi_j\left|\gamma_{ex}\right|, \left|\rho_{**}^{\perp}\right|\right\rangle_{L^2}
&\lesssim&\left\{\aligned
&\sum_{i=1}^{\nu}Q\|\Psi_i^{p-1-2\sigma}\Psi_j\|_{L^2(\mathcal{B}_i)}\|\rho_{**}^{\perp}\|,\quad p\geq3,\\
&\sum_{i=1}^{\nu}Q\|\Psi_i^{2(p-2)-\sigma}\Psi_j\|_{L^2(\mathcal{B}_i)}\|\rho_{**}^{\perp}\|,\quad 1<p<3
\endaligned
\right.\\
&=&o(Q)+\mathcal{O}\left(\|\rho_{**}^{\perp}\|^{2+\ve}\right).
\end{eqnarray*}
By Lemma~\ref{lem0005}, $(ii)$ of $(1)$ of Proposition~\ref{prop0001} and Proposition~\ref{propn0001},
\begin{eqnarray*}
\left\langle\mathcal{U}^{p-2}\Psi_j\left|\gamma_{*}\right|, \left|\rho_{**}^{\perp}\right|\right\rangle_{L^2}
&\lesssim&\left(Q+\beta_*^2+\|\rho_{**}^{\perp}\|^{1+\ve}+\|f\|_{H^{-1}}\right)\|\rho_{**}^{\perp}\|\\
&=&o(Q+\beta_*^6)+\mathcal{O}\left(\|\rho_{**}^{\perp}\|^{1+\ve}+\|f\|_{H^{-1}}\right).
\end{eqnarray*}
By Lemma~\ref{lem0005} and $(iii)$ of $(1)$ of Proposition~\ref{prop0001},
\begin{eqnarray*}
\left\langle\mathcal{U}^{p-2}\Psi_j\left|\gamma_{\mathcal{N},led}\right|, \left|\rho_{**}^{\perp}\right|\right\rangle_{L^2}&\lesssim&\left\{\aligned
&\beta_*^2\|\rho_{**}^{\perp}\|,\quad p\geq3,\\
&\left(\beta_*^2+\beta_*^3Q\right)\|\rho_{**}^{\perp}\|,\quad 1<p<3
\endaligned
\right.\\
&=&o(\beta_*^6)+\mathcal{O}\left(\|\rho_{**}^{\perp}\|^{1+\ve}\right).
\end{eqnarray*}
Summarizing the above estimates, we have
\begin{eqnarray*}
\left\langle\mathcal{N}_{0}-\overline{\mathcal{N}}_{0,1}, \partial_t\Psi_j\right\rangle_{L^2} =o(Q+\beta_*^6)+\mathcal{O}\left(\|\rho_{**}^{\perp}\|^{1+\ve}+\|f\|_{H^{-1}}\right).
\end{eqnarray*}

{\bf Step.~5.1.2}\quad The estimate of $\left\langle\mathcal{U}^{p-2}\rho_0^2, \partial_t\Psi_j\right\rangle_{L^2}$.

By $(1)$ of Proposition~\ref{prop0001},
\begin{eqnarray*}
\left\langle\mathcal{U}^{p-2}\rho_0^2, \partial_t\Psi_j\right\rangle_{L^2}
&=&\left\langle\mathcal{U}^{p-2}\left(\gamma_{*}+\gamma_{\mathcal{N},led}\right)^2, \partial_t\Psi_j\right\rangle_{L^2}+\left\langle\mathcal{U}^{p-2}\gamma_{ex}^2, \partial_t\Psi_j\right\rangle_{L^2}\\
&&+2\left\langle\mathcal{U}^{p-2}\left(\gamma_{*}+\gamma_{\mathcal{N},led}\right)\gamma_{ex}, \partial_t\Psi_j\right\rangle_{L^2}.
\end{eqnarray*}
By \eqref{eq0026} and similar estimates of \eqref{eqnew9998},
\begin{eqnarray*}
\left|\left\langle\mathcal{U}^{p-2}\gamma_{ex}^2, \partial_t\Psi_j\right\rangle_{L^2}\right|=o(Q).
\end{eqnarray*}
By \eqref{eq0026} and $(i)$ and $(ii)$ of $(1)$ of Proposition~\ref{prop0001},
\begin{eqnarray*}
\left|\left\langle\mathcal{U}^{p-2}\gamma_{*}\gamma_{ex}, \partial_t\Psi_j\right\rangle_{L^2}\right|
&\lesssim&\left\{\aligned
&\sum_{i=1}^{\nu}o\left(Q\left\|\Psi_i^{p-2\sigma}\Psi_j\right\|_{L^1(\mathcal{B}_i)}\right),\quad p\geq3,\\
&\sum_{i=1}^{\nu}o\left(Q\left\|\Psi_i^{2p-3-\sigma}\Psi_j\right\|_{L^1(\mathcal{B}_i)}\right),\quad 1<p<3
\endaligned\right.\\
&=&o(Q).
\end{eqnarray*}
By \eqref{eq0026} and $(i)$ and $(iii)$ of $(1)$ of Proposition~\ref{prop0001},
\begin{eqnarray*}
\left\langle\mathcal{U}^{p-2}\gamma_{\mathcal{N},led}\gamma_{ex}, \partial_t\Psi_j\right\rangle_{L^2}
=o(Q).
\end{eqnarray*}
By the oddness of $\partial_t\Psi$ in $\bbr$ and $(ii)$ and $(iii)$ of $(1)$ of Proposition~\ref{prop0001},
\begin{eqnarray*}
\left\langle\mathcal{U}^{p-2}\left(\gamma_{*}+\gamma_{\mathcal{N},led}\right)^2, \partial_t\Psi_j\right\rangle_{L^2}&=&\left\langle\left(\mathcal{U}^{p-2}-\left(\Psi_j^*\right)^{p-2}\right)\partial_t\Psi_j, \mathcal{W}_{sym,j}^2\right\rangle_{L^2}\\
&&+\left\langle\mathcal{U}^{p-2 }\partial_t\Psi_j\mathcal{W}_{*,j},2\mathcal{W}_{sym,j}+\mathcal{W}_{*,j}\right\rangle_{L^2},
\end{eqnarray*}
where $\mathcal{W}_{sym,j}=\gamma_{1,j}+\rho_{**,1,j}^{\perp}-\alpha_{j,1}^{**}\Psi_{j}+\gamma_{\mathcal{N},led,j}$
and $\mathcal{W}_{*,j}=\gamma_{*}+\gamma_{\mathcal{N},led}-\mathcal{W}_{sym,j}$.
Similar to \eqref{eqn3147}, by \eqref{eq0026},
\begin{eqnarray*}
\left|\left(\mathcal{U}^{p-2}-\left(\Psi_j^*\right)^{p-2}\right)\partial_t\Psi_j\right|\lesssim\left(\sum_{i=1}^{\nu}\Psi_i^{p-2}\mathcal{U}_i\chi_{\mathcal{B}_i}\right)+\mathcal{U}^{p-1}\chi_{\mathcal{C}\backslash\cup_{j=1}^{\nu}\mathcal{B}_j},
\end{eqnarray*}
thus, by Lemmas~\ref{lem0005}, \ref{lem0007}, \ref{lem0010}, \ref{lem0011} and Proposition~\ref{propn0001},
\begin{eqnarray*}
\left\langle\left(\mathcal{U}^{p-2}-\left(\Psi_j^*\right)^{p-2}\right)\partial_t\Psi_j, \mathcal{W}_{sym,j}^2\right\rangle_{L^2}=o\left(\sum_{j=1}^{\nu}\left\langle\Psi_j^{p-2\sigma}\chi_{\mathcal{B}_j}, \mathcal{U}_j\right\rangle_{L^2}\right)=o(Q).
\end{eqnarray*}
Since by \eqref{eq0026}, Lemma~\ref{lem0007}, $(1)$ of Proposition~\ref{prop0001} and Proposition~\ref{propn0001},
\begin{eqnarray*}
&&\left|\partial_t\Psi_j\mathcal{W}_{*,j}\left(2\mathcal{W}_{sym,j}+\mathcal{W}_{*,j}\right)\right|\\
&\lesssim&\left(\beta_*^2+Q+\|\rho_{**}^{\perp}\|^{1+\ve}+\|f\|_{H^{-1}}\right)^2Q^{1-\sigma}\left(\sum_{j=1}^{\nu}\Psi_j^{1-\sigma}\chi_{\mathcal{B}_j}+\mathcal{U}^{1-\sigma}\chi_{\mathcal{C}\backslash\cup_{j=1}^{\nu}\mathcal{B}_j}\right)
\end{eqnarray*}
in $\mathcal{C}$, we have
\begin{eqnarray*}
\left\langle\mathcal{U}^{p-2 }\partial_t\Psi_j\mathcal{W}_{*,j}, 2\mathcal{W}_{sym,j}+\mathcal{W}_{*,j}\right\rangle_{L^2}
=o(Q+\beta_*^6)+\mathcal{O}\left(\|\rho_{**}^{\perp}\|^{1+\ve}+\|f\|_{H^{-1}}\right).
\end{eqnarray*}
Summarizing the above estimates, we have
\begin{eqnarray*}
\left\langle\mathcal{U}^{p-2}\rho_0^2, \partial_t\Psi_j\right\rangle_{L^2}= o(Q+\beta_*^6)+\mathcal{O}\left(\|\rho_{**}^{\perp}\|^{1+\ve}+\|f\|_{H^{-1}}\right).
\end{eqnarray*}

{\bf Step.~5.1.3}\quad The estimate of $\left\langle\mathcal{U}^{p-3}\mathcal{V}\rho_0^2, \partial_t\Psi_j\right\rangle_{L^2}$.

Clearly, we have
\begin{eqnarray*}
\left\langle\mathcal{U}^{p-3}\mathcal{V}\rho_0^2, \partial_t\Psi_j\right\rangle_{L^2}=\left\langle\mathcal{U}^{p-3}\left(\sum_{i=1;i\not=j}^{\nu}\mathcal{V}_i\right)\rho_0^2, \partial_t\Psi_j\right\rangle_{L^2}+\left\langle\mathcal{U}^{p-3}\mathcal{V}_j\rho_0^2, \partial_t\Psi_j\right\rangle_{L^2}.
\end{eqnarray*}
By \eqref{eqn0191} and applying the same symmetry as in the estimate of $\left\langle\mathcal{U}^{p-2}\rho_0^2, \partial_t\Psi_j\right\rangle_{L^2}$, we have
\begin{eqnarray*}
\left\langle\mathcal{U}^{p-3}\mathcal{V}_j\rho_0^2, \partial_t\Psi_j\right\rangle_{L^2}=o(Q+\beta_*^6)+\mathcal{O}\left(\|\rho_{**}^{\perp}\|^{1+\ve}+\|f\|_{H^{-1}}\right).
\end{eqnarray*}
By \eqref{eq0026} and $(1)$ of Proposition~\ref{prop0001},
\begin{eqnarray*}
&&\left|\left\langle\mathcal{U}^{p-3}\left(\sum_{i=1;i\not=j}^{\nu}\mathcal{V}_i\right)\rho_0^2, \partial_t\Psi_j\right\rangle_{L^2}\right|\\
&\lesssim&\left\langle\mathcal{U}^{p-3}\left|\sum_{i=1;i\not=j}^{\nu}\mathcal{V}_i\right|\gamma_{ex}^2, \Psi_j\right\rangle_{L^2}+\left\langle\mathcal{U}^{p-3}\left|\sum_{i=1;i\not=j}^{\nu}\mathcal{V}_i\right|\left(\gamma_{*}+\gamma_{\mathcal{N},led}\right)^2, \Psi_j\right\rangle_{L^2}.
\end{eqnarray*}
By Lemma~\ref{lem0005} and $(i)$ of $(1)$ of Proposition~\ref{prop0001},
\begin{eqnarray*}
\left\langle\mathcal{U}^{p-3}\left|\sum_{i=1;i\not=j}^{\nu}\mathcal{V}_i\right|\gamma_{ex}^2, \Psi_j\right\rangle_{L^2}
&\lesssim&\left\{\aligned
&\sum_{i=1}^{\nu}\beta_*Q^2\left\langle\Psi_{i}^{\frac{3p-1-4\sigma}{2}}\chi_{\mathcal{B}_i}, \mathcal{U}_i\right\rangle_{L^2},\quad p\geq3,\\
&\sum_{i=1}^{\nu}\beta_*Q^2\left\langle\Psi_{i}^{\frac{7p-13}{2}}\chi_{\mathcal{B}_i}, \mathcal{U}_i\right\rangle_{L^2},\quad 1<p<3
\endaligned\right.\\
&=&o(Q).
\end{eqnarray*}
By Lemma~\ref{lem0005}, $(ii)$ of $(1)$ of Proposition~\ref{prop0001} and Proposition~\ref{propn0001},
\begin{eqnarray*}
&&\left\langle\mathcal{U}^{p-3}\left|\sum_{i=1;i\not=j}^{\nu}\mathcal{V}_i\right|\gamma_{*}^2, \Psi_j\right\rangle_{L^2}\\
&\lesssim&\beta_*\left(Q+\beta_*^2+\|\rho_{**}^{\perp}\|^{1+\ve}+\|f\|_{H^{-1}}\right)^2\sum_{i=1}^{\nu}\left\langle\Psi_{i}^{\frac{3p-1-4\sigma}{2}}\chi_{\mathcal{B}_i}, \mathcal{U}_i\right\rangle_{L^2}\\
&=&o(Q).
\end{eqnarray*}
By Lemma~\ref{lem0005} and $(iii)$ of $(1)$ of Proposition~\ref{prop0001},
\begin{eqnarray*}
\left\langle\mathcal{U}^{p-3}\left|\sum_{i=1;i\not=j}^{\nu}\mathcal{V}_i\right|\gamma_{\mathcal{N},led}^2, \Psi_j\right\rangle_{L^2}=o(Q).
\end{eqnarray*}
Summarizing the above estimates, we have
\begin{eqnarray*}
\left\langle\mathcal{U}^{p-3}\mathcal{V}\rho_0^2, \partial_t\Psi_j\right\rangle_{L^2}=o(Q+\beta_*^6)+\mathcal{O}\left(\|\rho_{**}^{\perp}\|^{1+\ve}+\|f\|_{H^{-1}}\right).
\end{eqnarray*}

{\bf Step.~5.1.4}\quad The estimate of $\left\langle\mathcal{U}^{p-3}\rho_0^3, \partial_t\Psi_j\right\rangle_{L^2}$.

By $(1)$ of Proposition~\ref{prop0001} and Proposition~\ref{propn0001},
\begin{eqnarray*}
\left\langle\mathcal{U}^{p-3}\rho_0^3, \partial_t\Psi_j\right\rangle_{L^2}&=&\mathcal{O}\left(Q^{(2p-1)\wedge 3}\left|\log Q\right|+\left(\beta_*^2+\|\rho_{**}^{\perp}\|^{1+\ve}+\|f\|_{H^{-1}}\right)^3\right)\\
&=&o(Q+\beta_*^6)+\mathcal{O}\left(\|\rho_{**}^{\perp}\|^{1+\ve}+\|f\|_{H^{-1}}\right).
\end{eqnarray*}

Summarizing the estimates from Step.~5.1.1 to Step.~5.1.4, we have
\begin{eqnarray*}
\left\langle\mathcal{N}_{0}, \partial_t\Psi_j\right\rangle_{L^2}=o(Q+\beta_*^6)+\mathcal{O}\left(\|\rho_{**}^{\perp}\|^{1+\ve}+\|f\|_{H^{-1}}\right).
\end{eqnarray*}

{\bf Step.~5.2}\quad The estimate of $\left\langle\overline{\mathcal{N}}_{rem}, \partial_t\Psi_j\right\rangle_{L^2}$.

By Lemma~\ref{lemn0002} and Proposition~\ref{propn0001},
\begin{eqnarray*}
\left\langle\overline{\mathcal{N}}_{rem}, \partial_t\Psi_j\right\rangle_{L^2}&=&C_p\left\langle\mathcal{U}^{p-4}\left(\overline{\mathcal{V}}^4+4\overline{\mathcal{V}}^3\rho_{**}^{\perp}\right), \partial_t\Psi_j\right\rangle_{L^2}\\
&&+D_p\left\langle\mathcal{U}^{p-5}\left(\overline{\mathcal{V}}^5+5\overline{\mathcal{V}}^4\rho_{**}^{\perp}\right), \partial_t\Psi_j\right\rangle_{L^2}\\
&&+o(Q)+\mathcal{O}\left(\beta_*^6+\|\rho_{**}^{\perp}\|^{1+\ve}+\|f\|_{H^{-1}}\right).
\end{eqnarray*}

{\bf Step.~5.2.1}\quad The estimates of $\left\langle \mathcal{U}^{p-4}\overline{\mathcal{V}}^3\rho_{**}^{\perp}, \partial_t\Psi_j\right\rangle_{L^2}$ and $\left\langle \mathcal{U}^{p-5}\overline{\mathcal{V}}^4\rho_{**}^{\perp}, \partial_t\Psi_j\right\rangle_{L^2}$.

Recall that $\overline{\mathcal{V}}=\mathcal{V}+\rho_{0}$ with $\mathcal{V}$ given by \eqref{eqn0040} and $\rho_0$ given in Proposition~\ref{prop0001}.  By \eqref{eq0026}, $(1)$ of Proposition~\ref{prop0001} and Proposition~\ref{propn0001},
\begin{eqnarray*}
\left|\left\langle \mathcal{U}^{p-4}\overline{\mathcal{V}}^3\rho_{**}^{\perp}, \partial_t\Psi_j\right\rangle_{L^2}\right|&\lesssim&\left(Q^{\frac{4p-3}{2}\wedge 3}\left|\log Q\right|^{\frac12}+\left(\beta_*+\|\rho_{**}^{\perp}\|^{1+\ve}+\|f\|_{H^{-1}}\right)^3\right)\|\rho_{**}^{\perp}\|\\
&=&o(Q+\beta_*^6)+\mathcal{O}\left(\|\rho_{**}^{\perp}\|^{1+\ve}+\|f\|_{H^{-1}}\right).
\end{eqnarray*}
and
\begin{eqnarray*}
\left|\left\langle \mathcal{U}^{p-5}\overline{\mathcal{V}}^4\rho_{**}^{\perp}, \partial_t\Psi_j\right\rangle_{L^2}\right|&\lesssim&\left(Q^{\frac{5p-4}{2}\wedge 3}\left|\log Q\right|^{\frac12}+\left(\beta_*+\|\rho_{**}^{\perp}\|^{1+\ve}+\|f\|_{H^{-1}}\right)^4\right)\|\rho_{**}^{\perp}\|\\
&=&o(Q+\beta_*^6)+\mathcal{O}\left(\|\rho_{**}^{\perp}\|^{1+\ve}+\|f\|_{H^{-1}}\right).
\end{eqnarray*}

{\bf Step.~5.2.2}\quad The estimates of $\left\langle \mathcal{U}^{p-4}\overline{\mathcal{V}}^4, \partial_t\Psi_j\right\rangle_{L^2}$ and $\left\langle\mathcal{U}^{p-5}\overline{\mathcal{V}}^5, \partial_t\Psi_j\right\rangle_{L^2}$.

By the oddness of $\partial_t\Psi$ in $\bbr$, Lemmas~\ref{lem0007}, \ref{lem0010} and \ref{lem0011},
\begin{eqnarray*}
&&\left\langle \mathcal{U}^{p-4}\overline{\mathcal{V}}^4, \partial_t\Psi_j\right\rangle_{L^2}\\
&=&\left\langle \left(\mathcal{U}^{p-4}-\left(\Psi_j^*\right)^{p-4}\right)\partial_t\Psi_j, \overline{\mathcal{V}}_{sym,j}^4\right\rangle_{L^2}\\
&&+\left\langle \mathcal{U}^{p-4}\partial_t\Psi_j\overline{\mathcal{V}}_{*,j}, 4\overline{\mathcal{V}}_{sym,j}^3+6\overline{\mathcal{V}}_{sym,j}^2\overline{\mathcal{V}}_{*,j}+4\overline{\mathcal{V}}_{sym,j}\overline{\mathcal{V}}_{*,j}^2+\overline{\mathcal{V}}_{*,j}^3\right\rangle_{L^2},
\end{eqnarray*}
where $\overline{\mathcal{V}}_{sym,j}=\mathcal{V}_j+\gamma_{1,j}+\rho_{**,1,j}^{\perp}-\alpha_{j,1}^{**}\Psi_j+\gamma_{\mathcal{N},led,j}$
and $\overline{\mathcal{V}}_{*,j}=\overline{\mathcal{V}}-\overline{\mathcal{V}}_{sym,j}$.
Similar to \eqref{eqn3147}, by \eqref{eq0026}, we have
\begin{eqnarray*}
\left|\left(\mathcal{U}^{p-4}-\left(\Psi_j^*\right)^{p-4}\right)\partial_t\Psi_j\right|\lesssim\left(\sum_{i=1}^{\nu}\Psi_i^{p-4}\mathcal{U}_i\chi_{\mathcal{B}_i}\right)+\mathcal{U}^{p-1}\chi_{\mathcal{C}\backslash\cup_{j=1}^{\nu}\mathcal{B}_j},
\end{eqnarray*}
thus, by \eqref{eqn0040}, Lemmas~\ref{lem0005}, \ref{lem0007}, \ref{lem0010}, \ref{lem0011} and Proposition~\ref{propn0001},
\begin{eqnarray*}
\left\langle\left(\mathcal{U}^{p-4}-\left(\Psi_j^*\right)^{p-4}\right)\partial_t\Psi_j, \overline{\mathcal{V}}_{sym,j}^4\right\rangle_{L^2}
=o\left(\left\langle\Psi_j^{p-4\sigma}\chi_{\mathcal{B}_j}, \mathcal{U}_j\right\rangle_{L^2}\right)=o(Q).
\end{eqnarray*}
Since by \eqref{eq0026}, \eqref{eqn0040}, Lemma~\ref{lem0007}, $(iii)$ of $(1)$ of Proposition~\ref{prop0001} and Proposition~\ref{propn0001},
\begin{eqnarray*}
&&\left|\partial_t\Psi_j\overline{\mathcal{V}}_{*,j}\left(4\overline{\mathcal{V}}_{sym,j}^3+6\overline{\mathcal{V}}_{sym,j}^2\overline{\mathcal{V}}_{*,j}+4\overline{\mathcal{V}}_{sym,j}\overline{\mathcal{V}}_{*,j}^2+\overline{\mathcal{V}}_{*,j}^3\right)\right|\\
&\lesssim&Q^{1-\sigma}\left(Q+\beta_*+\|\rho_{**}^{\perp}\|^{1+\ve}+\|f\|_{H^{-1}}\right)^3\left(\sum_{j=1}^{\nu}\Psi_j^{3-3\sigma}\chi_{\mathcal{B}_j}+\mathcal{U}^{3-3\sigma}\chi_{\mathcal{C}\backslash\cup_{j=1}^{\nu}\mathcal{B}_j}\right)
\end{eqnarray*}
in $\mathcal{C}$, we have
\begin{eqnarray*}
&&\left\langle \mathcal{U}^{p-4}\partial_t\Psi_j\overline{\mathcal{V}}_{*,j}, 4\overline{\mathcal{V}}_{sym,j}^3+6\overline{\mathcal{V}}_{sym,j}^2\overline{\mathcal{V}}_{*,j}+4\overline{\mathcal{V}}_{sym,j}\overline{\mathcal{V}}_{*,j}^2+\overline{\mathcal{V}}_{*,j}^3\right\rangle_{L^2}\\
&\lesssim&Q^{1-\sigma}\left(Q+\beta_*+\|\rho_{**}^{\perp}\|^{1+\ve}+\|f\|_{H^{-1}}\right)^3\\
&=&o(Q+\beta_*^6)+\mathcal{O}\left(\|\rho_{**}^{\perp}\|^{1+\ve}+\|f\|_{H^{-1}}\right).
\end{eqnarray*}
Summarizing the above estimates, we have
\begin{eqnarray*}
\left\langle \mathcal{U}^{p-4}\overline{\mathcal{V}}^4, \partial_t\Psi_j\right\rangle_{L^2}=o(Q)+\mathcal{O}\left(\beta_*^6+\|\rho_{**}^{\perp}\|^{1+\ve}+\|f\|_{H^{-1}}\right).
\end{eqnarray*}
By \eqref{eqn0191} and applying the same symmetry as in the estimate of $\left\langle\mathcal{U}^{p-4}\overline{\mathcal{V}}^4, \partial_t\Psi_j\right\rangle_{L^2}$, we also have
\begin{eqnarray*}
\left\langle\mathcal{U}^{p-5}\overline{\mathcal{V}}^5, \partial_t\Psi_j\right\rangle_{L^2}=o(Q+\beta_*^6)+\mathcal{O}\left(\|\rho_{**}^{\perp}\|^{1+\ve}+\|f\|_{H^{-1}}\right).
\end{eqnarray*}

Summarizing the estimates from Step.~5.2.1 to Step.~5.2.2, we have
\begin{eqnarray*}
\left\langle \overline{\mathcal{N}}_{rem}, \partial_t\Psi_j\right\rangle_{L^2}=o(Q)+\mathcal{O}\left(\beta_*^6+\|\rho_{**}^{\perp}\|^{1+\ve}+\|f\|_{H^{-1}}\right).
\end{eqnarray*}

The conclusion follows from the estimates from Step.~1 to Step.~5.
\end{proof}

\section{Estimate of $\rho_{**}^{\perp}$}
By the orthogonal conditions of $\rho_{**}^{\perp}$, given by \eqref{eqn2004} and multiplying \eqref{eqn1114} which is given in $(2)$ of Propostion~\ref{prop0001} with $\rho_{**}^{\perp}$ on both sides and integrating by parts, we have
\begin{eqnarray}\label{eqn0060}
\|\rho_{**}^{\perp}\|^2\lesssim\|f\|_{H^{-1}}\|\rho_{**}^{\perp}\|+\left|\left\langle \mathcal{R}_{new,0}, \rho_{**}^{\perp}\right\rangle_{L^2}\right|,
\end{eqnarray}
where $\mathcal{R}_{new,0}$ is given by \eqref{eqn0061} which is given in $(2)$ of Propostion~\ref{prop0001}.  Moreover, we remark that by Lemmas~\ref{lem0003} and \ref{lemn0001}, we have
\begin{eqnarray}\label{eqnew0093}
\mathcal{N}_{rem}=\mathcal{N}_{0}+\overline{\mathcal{N}}_{rem},
\end{eqnarray}
where $\mathcal{N}_{rem}$ is the remaining term in $\mathcal{R}_{new,0}$.  We emphasize once more that we need to eliminate the lower order terms (compared to the $\beta_*^4$ terms) in the data $\mathcal{R}_{new,0}$ which is given in $(2)$ of Proposition~\ref{prop0001} to get the desired stability inequality.  Thus, we
need further decompose $\sum_{j=1}^{\nu}\gamma_{1,j}$ which is given in Lemma~\ref{lem0007}.
\begin{lemma}\label{lemn0003}
Let $d\geq2$, $a<0$ and $b=b_{FS}(a)$.  Then we have the following decomposition
\begin{eqnarray*}
\sum_{j=1}^{\nu}\gamma_{1,j}=\gamma_{1,*}+\sum_{l=1}^{\nu}\alpha_l^{***}\Psi_l,
\end{eqnarray*}
where $\{\alpha_l^{***}\}$ is chosen such that $\left\langle\gamma_{1,*}, \Psi_l\right\rangle=0$ for all $1\leq l\leq\nu$.  Moreover, we have the following estimates
\begin{eqnarray*}
\|\gamma_{1,*}\|_{L^\infty}\lesssim \sum_{l=1}^{\nu}Q^{1-\sigma}\left|\left(\alpha_l^*\right)^p-\alpha_l^*\right|\quad\text{and}\quad\sum_{j=1}^{\nu}\left|\alpha_{j}^{***}\right|\lesssim \sum_{l=1}^{\nu}\left|\left(\alpha_l^*\right)^p-\alpha_l^*\right|.
\end{eqnarray*}
\end{lemma}
\begin{proof}
By the orthogonal conditions of $\gamma_{1,*}$ and multiplying \eqref{eqn0012} with $\Psi_j$ on both sides and integrating by parts, we have
\begin{eqnarray}
\|\Psi\|^2\alpha_j^{***}+\sum_{l=1;l\not=j}^{\nu}\left\langle\Psi_j, \Psi_l\right\rangle\alpha_l^{***}&=&\left\langle\mathcal{R}_{1,j}, \Psi_j\right\rangle_{L^2}+\sum_{l=1;l\not=j}^{\nu}\left\langle\mathcal{R}_{1,l}, \Psi_j\right\rangle_{L^2}\notag\\
&&+\sum_{l=1}^{\nu}p\left\langle\mathcal{U}^{p-1}\gamma_{1,l}, \Psi_j\right\rangle_{L^2}\notag\\
&=&\left\langle\mathcal{R}_{1,j}, \Psi_j\right\rangle_{L^2}+p\sum_{l=1;l\not=j}\alpha_l^{***}\left\langle\mathcal{U}^{p-1}, \Psi_j\Psi_l\right\rangle_{L^2}\notag\\
&&+p\alpha_j^{***}\left\langle\mathcal{U}^{p-1}, \Psi_j^2\right\rangle_{L^2}+\sum_{l=1;l\not=j}^{\nu}\left\langle\mathcal{R}_{1,l}, \Psi_j\right\rangle_{L^2}\notag\\
&&+p\left\langle\left(\mathcal{U}^{p-1}-\left(\Psi_j^*\right)^{p-1}\right)\Psi_j, \gamma_{1,*}\right\rangle_{L^2}\label{eqnewnew0020}
\end{eqnarray}
for all $1\leq j\leq \nu$ and $\gamma_{1,*}$ satisfies the following equation:
\begin{eqnarray}\label{eqn9012}
\left\{\aligned&\mathcal{L}(\gamma_{1,*})=\mathcal{R}_{1,*}-\sum_{i=1}^{\nu}\Psi_{i}^{p-1}\left(c_{1,j,i}\partial_t\Psi_{i}+\sum_{l=1}^{d}\varsigma_{1,j,i,l}w_{i,l}\right),\quad \text{in }\mathcal{C},\\
&\langle \partial_t\Psi_{j}, \gamma_{1,*}\rangle=\langle w_{j,l}, \gamma_{1,*}\rangle=0\quad\text{for all }1\leq j\leq\nu\text{ and }1\leq l\leq d,\endaligned\right.
\end{eqnarray}
where by \eqref{eqn0020},
\begin{eqnarray*}
\mathcal{R}_{1,*}&=&\sum_{l=1}^{\nu}\left(\mathcal{R}_{1,l}-\alpha_l^{***}\left(\Psi_l^p-p\mathcal{U}^{p-1}\Psi_l\right)\right)\notag\\
&=&\sum_{l=1}^{\nu}\left(\left(\alpha_l^*\right)^p-\alpha_l^*-\alpha_l^{***}\left(1-p\left(\alpha_l^*\right)^{p-1}\right)\right)\Psi_l^p\notag\\
&&+\sum_{l=1}^{\nu}p\alpha_l^{***}\left(\mathcal{U}^{p-1}-\left(\Psi_l^*\right)^{p-1}\right)\Psi_l.
\end{eqnarray*}
By Lemma~\ref{lem0005}, \eqref{eqn0020} and \eqref{eqnewnew0020}, we have
\begin{eqnarray*}
\left(1-p\left(\alpha_j^*\right)^{p-1}\right)\alpha_{j}^{***}=\left(\left(\alpha_j^*\right)^p-\alpha_j^*\right)+\sum_{l=1}^{\nu}\mathcal{O}(Q)\left(\left(\alpha_l^*\right)^p-\alpha_l^*\right)+\mathcal{O}\left(Q^{1-\sigma}\right)\|\gamma_{1,*}\|
\end{eqnarray*}
for all $1\leq j\leq\nu$.  Thus, by Lemma~\ref{lem0007}, \eqref{eqn9012}, the orthogonal conditions of $\gamma_{1,j}$ given in \eqref{eqn0012} and the elliptic estimates, we have the desired estimates of $\|\gamma_{1,*}\|_{L^\infty}$ and $\sum_{j=1}^{\nu}\left|\alpha_{j}^{***}\right|$.
\end{proof}

\vskip0.12in

In what follows, we shall estimate $\|\rho_{**}^{\perp}\|$ by \eqref{eqn0060}.
\begin{proposition}\label{propn0003}
Let $d\geq2$, $a<0$ and $b=b_{FS}(a)$.  Then we have
\begin{eqnarray*}
\|\rho_{**}^{\perp}\|\lesssim\beta_*^4+\|f\|_{H^{-1}}.
\end{eqnarray*}
\end{proposition}
\begin{proof}
By \eqref{eqn0061} and \eqref{eqnew0093},
\begin{eqnarray*}
\left\langle \mathcal{R}_{new,0}, \rho_{**}^{\perp}\right\rangle_{L^2}&=&\sum_{i=1}^{\nu}\left\langle (c_{1,ex,i}+c_{1,j,i}+c_{3,led,i}-c_{new,*,i})\Psi_{i}^{p-1}\partial_t\Psi_{i}, \rho_{**}^{\perp}\right\rangle_{L^2}\notag\\
&&+\sum_{i=1}^{\nu}\sum_{l=1}^{d}\left\langle (\varsigma_{2,ex,i,l}+\varsigma_{3,led,i,l}-\varsigma_{new,*,i,l})\Psi_{i}^{p-1}w_{i,l}, \rho_{**}^{\perp}\right\rangle_{L^2}\notag\\
&&+\sum_{j=1}^{\nu}2A_p\left\langle\left(\Psi_j^*\right)^{p-2}\mathcal{V}_j(\rho_*-\gamma_{1,ex}-\gamma_{\mathcal{N},led,j})\chi_{\mathcal{B}_j}, \rho_{**}^{\perp}\right\rangle_{L^2}\notag\\
&&+\sum_{j=1}^{\nu}3B_p\left\langle\left(\Psi_j^*\right)^{p-3}\mathcal{V}_j^2\rho_*\chi_{\mathcal{B}_j}, \rho_{**}^{\perp}\right\rangle_{L^2}+\sum_{j=1}^{\nu}\left\langle\mathcal{R}_{2,j}, \rho_{**}^{\perp}\right\rangle_{L^2}\notag\\
&&+\sum_{j=1}^{\nu}\left\langle\mathcal{O}\left(\beta_*\mathcal{U}_j\left(\beta_*\Psi_{j}^{2p-3}\left(\Psi_j+\rho_*\right)+|\rho_*-\gamma_{1,ex}|\Psi_j^{\frac{3p-5}{2}}\right)\right)\chi_{\mathcal{B}_j}, \rho_{**}^{\perp}\right\rangle_{L^2}\notag\\
&&+\sum_{j=1}^{\nu}\alpha_{j}^{**}\left\langle\left(\mathcal{U}^{p-1}-\Psi_j^{p-1}\right)\Psi_{j}, \rho_{**}^{\perp}\right\rangle_{L^2}+\left\langle\overline{\mathcal{N}}_{rem}, \rho_{**}^{\perp}\right\rangle_{L^2}\notag\\
&&+\left\langle\mathcal{N}_{0}, \rho_{**}^{\perp}\right\rangle_{L^2}+\left\langle\mathcal{O}\left(|\rho_*-\gamma_{1,ex}|\mathcal{U}^{\frac{3p-3}{2}}\right)\chi_{\mathcal{C}\backslash\cup_{j=1}^{\nu}\mathcal{B}_j}, \rho_{**}^{\perp}\right\rangle_{L^2}.
\end{eqnarray*}
By the orthogonal conditions of $\rho_{**}^{\perp}$ given by \eqref{eqn2004} once more,
\begin{eqnarray*}
\left\langle\Psi_{i}^{p-1}(c_{1,ex,i}+c_{1,j,i}+c_{3,led,i}-c_{new,*,i})\partial_t\Psi_{i},\rho_{**}^{\perp}\right\rangle_{L^2}=0
\end{eqnarray*}
and
\begin{eqnarray*}
\left\langle\sum_{l=1}^{d}(\varsigma_{2,ex,j,l}+\varsigma_{3,led,j,l}-\varsigma_{new,*,i,l})w_{i,l},\rho_{**}^{\perp}\right\rangle_{L^2}=0
\end{eqnarray*}
for all $1\leq i,j\leq \nu$, and further by \eqref{eqn0021}, we also have
\begin{eqnarray*}
\left\langle\sum_{i=1}^{\nu}\mathcal{R}_{2,i},\rho_{**}^{\perp}\right\rangle_{L^2}=0.
\end{eqnarray*}

{\bf Step.~1}\quad The estimate of $\left\langle \overline{\mathcal{N}}_{rem}, \rho_{**}^{\perp}\right\rangle_{L^2}$.

By Lemma~\ref{lemn0001} and Propositions~\ref{propn0001} and \ref{propn0002},
\begin{eqnarray*}
\left|\left\langle \overline{\mathcal{N}}_{rem}, \rho_{**}^{\perp}\right\rangle_{L^2}\right|\lesssim\beta_*^{4}\|\rho_{**}^{\perp}\|+\|\rho_{**}^{\perp}\|^{2+\ve}+\|\rho_{**}^{\perp}\|\|f\|_{H^{-1}}.
\end{eqnarray*}

{\bf Step.~2}\quad The estimate of $\left\langle \mathcal{N}_{0}, \rho_{**}^{\perp}\right\rangle_{L^2}$.

By Lemma~\ref{lemn0001}, $(1)$ of Proposition~\ref{prop0001} and Propositions~\ref{propn0001} and \ref{propn0002},
\begin{eqnarray*}
\left|\left\langle \mathcal{N}_{0}, \rho_{**}^{\perp}\right\rangle_{L^2}\right|&\lesssim&\left\|\mathcal{U}^{p-2}\rho_0^2\right\|_{L^2}\|\rho_{**}^{\perp}\|+o(\|\rho_{**}^{\perp}\|^2)\notag\\
&\lesssim&\left(\left\|\mathcal{U}^{p-2}\gamma_{ex}^2\right\|_{L^2}+\|f\|_{H^{-1}}+\beta_*^4\right)\|\rho_{**}^{\perp}\|+o(\|\rho_{**}^{\perp}\|^2).
\end{eqnarray*}
By Lemma~\ref{lem0005}, $(i)$ of $(1)$ of Proposition~\ref{prop0001} and Proposition~\ref{propn0002},
\begin{eqnarray*}
\left\|\mathcal{U}^{p-2}\gamma_{1,ex}^2\right\|_{L^2}^2&\lesssim& \left\{\aligned
&Q^4\sum_{j=1}^{\nu}\left\|\Psi_j\right\|_{L^{2p-4\sigma}(\mathcal{B}_j)}^{2p-4\sigma},\quad p\geq3,\\
&Q^4\sum_{j=1}^{\nu}\left\|\Psi_j\right\|_{L^{6(p-2)}(\mathcal{B}_j)}^{6(p-2)},\quad 1<p<3,
\endaligned\right.\\
&=&o(\beta_*^6+\|\rho_{**}^{\perp}\|^{1+\ve}+\|f\|_{H^{-1}}).
\end{eqnarray*}
Summarizing the above estimates, we have
\begin{eqnarray*}
\left\langle \mathcal{N}_{0}, \rho_{**}^{\perp}\right\rangle_{L^2}
=\mathcal{O}\left(\left(\|f\|_{H^{-1}}+\beta_*^4\right)\|\rho_{**}^{\perp}\|\right)+o(\|\rho_{**}^{\perp}\|^2).
\end{eqnarray*}

{\bf Step.~3}\quad The estimate of $\left\langle\sum_{j=1}^{\nu}\alpha_j^{**}\left(\mathcal{U}^{p-1}-\Psi_j^{p-1}\right)\Psi_{j},\rho_{**}^{\perp}\right\rangle_{L^2}$.

By the orthogonal conditions of $\rho_{**}^{\perp}$ given by \eqref{eqn2004}, we have
\begin{eqnarray*}
\left\langle\sum_{j=1}^{\nu}\alpha_j^{**}\left(\mathcal{U}^{p-1}-\Psi_j^{p-1}\right)\Psi_{j},\rho_{**}^{\perp}\right\rangle_{L^2}=\left\langle\sum_{j=1}^{\nu}\alpha_j^{**}\left(\mathcal{U}^{p-1}-\left(\Psi_{j}^*\right)^{p-1}\right)\Psi_{j},\rho_{**}^{\perp}\right\rangle_{L^2}.
\end{eqnarray*}
Similar to \eqref{eqn3147}, we have
\begin{eqnarray*}
\left|\sum_{j=1}^{\nu}\alpha_j^{**}\left(\mathcal{U}^{p-1}-\Psi_j^{p-1}\right)\Psi_{j}\right|\lesssim\sum_{j=1}^{\nu}\left|\alpha_j^{**}\right|\Psi_{j}^{p-1}\mathcal{U}_j\chi_{\mathcal{B}_j}+\left(\sum_{j=1}^{\nu}\left|\alpha_j^{**}\right|\right)\mathcal{U}^{p}\chi_{\mathcal{C}\backslash\cup_{j=1}^{\nu}\mathcal{B}_j}.
\end{eqnarray*}
Thus, by Lemma~\ref{lem0005} $(2)$ of Proposition~\ref{prop0001} and Propositions~\ref{propn0001} and \ref{propn0002},
\begin{eqnarray*}
\left\langle\sum_{j=1}^{\nu}\alpha_j^{**}\left(\mathcal{U}^{p-1}-\Psi_j^{p-1}\right)\Psi_{j},\rho_{**}^{\perp}\right\rangle_{L^2}&=&\mathcal{O}\left(\sum_{j=1}^{\nu}\left|\alpha_j^{**}\right|Q^{\frac {1}{2}+\sigma}\|\rho_{**}^{\perp}\|\right)\notag\\
&=&\mathcal{O}\left(\left(\beta_*^4+\|f\|_{H^{-1}}\right)\|\rho_{**}^{\perp}\|\right)+o(\|\rho_{**}^{\perp}\|^2).
\end{eqnarray*}

{\bf Step.~4}\quad The estimate of $\sum_{j=1}^{\nu}\left\langle\beta_*^2\Psi_{j}^{2(p-1)}\mathcal{U}_j\chi_{\mathcal{B}_j}, \rho_{**}^{\perp}\right\rangle_{L^2}$.

By Lemma~\ref{lem0005},
\begin{eqnarray*}
\sum_{j=1}^{\nu}\left\langle\beta_*^2\Psi_{j}^{2(p-1)}\mathcal{U}_j\chi_{\mathcal{B}_j}, \rho_{**}^{\perp}\right\rangle_{L^2}&=&\mathcal{O}\left(\beta_*^2Q^{\frac {1}{2}+\sigma}\|\rho_{**}^{\perp}\|\right)\\
&=&\mathcal{O}\left(\left(\beta_*^4+\|f\|_{H^{-1}}\right)\|\rho_{**}^{\perp}\|\right)+o(\|\rho_{**}^{\perp}\|^2).
\end{eqnarray*}

{\bf Step.~5}\quad The estimates of
\begin{eqnarray*}
\sum_{j=1}^{\nu}\left\langle\beta_*^2\mathcal{U}_j\left(\Psi_{j}^{2p-3}\left(\Psi_j+\gamma_{1,ex}\right)\right)\chi_{\mathcal{B}_j}, \rho_{**}^{\perp}\right\rangle_{L^2}
\end{eqnarray*}
and $\left\langle\beta_*(\rho_*-\gamma_{1,ex})\mathcal{U}_*, \rho_{**}^{\perp}\right\rangle_{L^2}$, where $\mathcal{U}_*=\sum_{j=1}^{\nu}\Psi_{j}^{\frac{3p-5}{2}}\mathcal{U}_j\chi_{\mathcal{B}_j}+\mathcal{U}^{\frac{3(p-1)}{2}}\chi_{\mathcal{C}\backslash\left(\cup_{j=1}^{\nu}\mathcal{B}_j\right)}$.

By Lemma~\ref{lem0005} and Proposition~\ref{propn0002}
\begin{eqnarray*}
\left\langle\beta_*^2\mathcal{U}_j\Psi_{j}^{2p-2}\chi_{\mathcal{B}_j}, \rho_{**}^{\perp}\right\rangle_{L^2}=o\left(\beta_*^4\|\rho_{**}^{\perp}\|\right)+\mathcal{O}\left(\|f\|_{H^{-1}}\|\rho_{**}^{\perp}\|+\|\rho_{**}^{\perp}\|^{2+\ve}\right).
\end{eqnarray*}
By Lemma~\ref{lem0005}, $(i)$ of $(1)$ of Proposition~\ref{prop0001} and Proposition~\ref{propn0002},
\begin{eqnarray*}
\left\langle\beta_*^2\mathcal{U}_j\Psi_{j}^{2p-3}\gamma_{1,ex}\chi_{\mathcal{B}_j}, \rho_{**}^{\perp}\right\rangle_{L^2}&=&\left\{\aligned
&\left\langle\beta_*^2\mathcal{U}_j\Psi_{j}^{2p-2-\sigma}\left(Q_j\chi_{\mathcal{B}_{j,+}}+Q_{j-1}\chi_{\mathcal{B}_{j,-}}\right), \rho_{**}^{\perp}\right\rangle_{L^2},\quad p\geq3,\\
&\left\langle\beta_*^2\mathcal{U}_j\Psi_{j}^{3p-5}\left(Q_j\chi_{\mathcal{B}_{j,+}}+Q_{j-1}\chi_{\mathcal{B}_{j,-}}\right), \rho_{**}^{\perp}\right\rangle_{L^2},\quad1<p<3
\endaligned\right.\\
&=&o\left(\beta_*^4\|\rho_{**}^{\perp}\|\right)+\mathcal{O}\left(\|f\|_{H^{-1}}\|\rho_{**}^{\perp}\|+\|\rho_{**}^{\perp}\|^{2+\ve}\right).
\end{eqnarray*}
By $(1)$ of Proposition~\ref{prop0001},
\begin{eqnarray*}
\left\langle\beta_*(\rho_*-\gamma_{1,ex})\mathcal{U}_*, \rho_{**}^{\perp}\right\rangle_{L^2}=\beta_*\left\langle\gamma_{2,ex}+\gamma_{*}+\gamma_{\mathcal{N},led}, \mathcal{U}_*\rho_{**}^{\perp}\right\rangle_{L^2}+\beta_*\left\langle \mathcal{U}_*, \left(\rho_{**}^{\perp}\right)^2\right\rangle_{L^2}.
\end{eqnarray*}
Since $\|\mathcal{U}_*\|_{L^{\infty}}=o(1)$, we have $\beta_*\left\langle \mathcal{U}_*, \left|\rho_{**}^{\perp}\right|^2\right\rangle_{L^2}=o\left(\|\rho_{**}^{\perp}\|^2\right)$.
By Lemma~\ref{lem0005}, $(i)$ of $(1)$ of Proposition~\ref{prop0001} and Proposition~\ref{propn0002}, we have
\begin{eqnarray*}
\beta_*\left|\left\langle\gamma_{2,ex}, \mathcal{U}_*\rho_{**}^{\perp}\right\rangle_{L^2}\right|\
&\lesssim&
\left\{\aligned
&\beta_*^2Q\|\rho_{**}^{\perp}\|\left(\sum_{i=1}^{\nu}\left\langle \Psi_i^{3p-3-2\sigma},\mathcal{U}_i^2\right\rangle_{L^2(\mathcal{B}_{i})}\right)^\frac12,\quad p\geq\frac{7}{3},\\
&\beta_*^2Q\|\rho_{**}^{\perp}\|\left(\sum_{i=1}^{\nu}\left\langle \Psi_i^{6p-10},\mathcal{U}_i^2\right\rangle_{L^2(\mathcal{B}_{i})}\right)^\frac12,\quad 1<p<\frac{7}{3}
\endaligned
\right.\\
&=&o\left(\beta_*^4\|\rho_{**}^{\perp}\|\right)+\mathcal{O}\left(\|f\|_{H^{-1}}\|\rho_{**}^{\perp}\|+\|\rho_{**}^{\perp}\|^{2+\ve}\right).
\end{eqnarray*}
By Lemma~\ref{lem0005}, $(ii)$ and $(iii)$ of $(1)$ of Proposition~\ref{prop0001} and Propositions~\ref{propn0001} and \ref{propn0002}, we have
\begin{eqnarray*}
&&\beta_*\left\langle\gamma_{*}+\gamma_{\mathcal{N},led}, \mathcal{U}_*\rho_{**}^{\perp}\right\rangle_{L^2}\\
&\lesssim&\beta_*^3\|\rho_{**}^{\perp}\|\sum_{j=1}^\nu\left(\left\|\Psi_j^{\frac{3(p-1)-2\sigma}{2}}\mathcal{U}_j\right\|_{L^2(\mathcal{B}_j)}\right)+\|\rho_{**}^{\perp}\|^{2+\ve}+\|f\|_{H^{-1}}\|\rho_{**}^{\perp}\|\\
&&+\beta_*^4Q^{\frac12}\|\rho_{**}^{\perp}\|\\
&=&o\left(\beta_*^4\|\rho_{**}^{\perp}\|\right)+\mathcal{O}\left(\|f\|_{H^{-1}}\|\rho_{**}^{\perp}\|+\|\rho_{**}^{\perp}\|^{2+\ve}\right).
\end{eqnarray*}
Summarizing the above estimates, we have
\begin{eqnarray*}
\left\langle\beta_*\rho_*\mathcal{U}_*, \rho_{**}^{\perp}\right\rangle_{L^2}= o\left(\beta_*^4\|\rho_{**}^{\perp}\|\right)+\mathcal{O}\left(\|f\|_{H^{-1}}\|\rho_{**}^{\perp}\|+\|\rho_{**}^{\perp}\|^{2+\ve}\right).
\end{eqnarray*}

{\bf Step.~6}\quad The estimates of
\begin{eqnarray*}
\sum_{j=1}^{\nu}\left\langle\left(\Psi_j^*\right)^{p-2}\mathcal{V}_j(\rho_*-\gamma_{1,ex}-\gamma_{\mathcal{N},led,j})\chi_{\mathcal{B}_j},\rho_{**}^{\perp}\right\rangle_{L^2}
\end{eqnarray*}
and
\begin{eqnarray*}
\sum_{j=1}^{\nu}\left\langle\left(\Psi_j^*\right)^{p-3}\mathcal{V}_j^2\rho_*\chi_{\mathcal{B}_j},\rho_{**}^{\perp}\right\rangle_{L^2}.
\end{eqnarray*}

By $(1)$ of Proposition~\ref{prop0001},
\begin{eqnarray*}
&&\sum_{j=1}^{\nu}\left\langle\left(2A_p\left(\Psi_j^*\right)^{p-2}\mathcal{V}_j(\rho_*-\gamma_{1,ex}-\gamma_{\mathcal{N},led,j})+3B_p\left(\Psi_j^*\right)^{p-3}\mathcal{V}_j^2\rho_*\right)\chi_{\mathcal{B}_j},\rho_{**}^{\perp}\right\rangle_{L^2}\\
&=&\left\langle\mathcal{V}_*\rho_0,\rho_{**}^{\perp}\right\rangle_{L^2}-\sum_{j=1}^{\nu}\left\langle 2A_p\left(\Psi_j^*\right)^{p-2}\mathcal{V}_j(\gamma_{1,ex}+\gamma_{\mathcal{N},led,j})\chi_{\mathcal{B}_j},\rho_{**}^{\perp}\right\rangle_{L^2}+o\left(\|\rho_{**}^{\perp}\|^2\right).
\end{eqnarray*}
where
\begin{eqnarray*}
\mathcal{V}_*=\sum_{j=1}^{\nu}\left(2A_p\left(\Psi_j^*\right)^{p-2}\mathcal{V}_j+3B_p\left(\Psi_j^*\right)^{p-3}\mathcal{V}_j^2\right)\chi_{\mathcal{B}_j}.
\end{eqnarray*}
By Lemma~\ref{lem0005}, $(i)$ of $(1)$ of Proposition~\ref{prop0001} and Propositions~\ref{propn0001} and \ref{propn0002}, we have
\begin{eqnarray*}
\left|\left\langle \gamma_{2,ex}, \mathcal{V}_*\rho_{**}^{\perp}\right\rangle_{L^2}\right|&\lesssim&\left\{\aligned
&\beta_*^2\|\rho_{**}^{\perp}\|\left(\sum_{j=1}^{\nu}Q_j\left\|\Psi_j^{\frac{3p-1-2\sigma}{2}}\right\|_{L^2(\mathcal{B}_{j})}\right),\quad p\geq\frac{7}{3},\\
&\beta_*^2\|\rho_{**}^{\perp}\|\left(\sum_{j=1}^{\nu-1}Q_j\left\|\Psi_j^{3p-4}\right\|_{L^2(\mathcal{B}_{j})}\right),\quad 1<p<\frac{7}{3}
\endaligned
\right.\\
&=&o\left(\beta_*^4\|\rho_{**}^{\perp}\|\right)+\mathcal{O}\left(\|f\|_{H^{-1}}\|\rho_{**}^{\perp}\|+\|\rho_{**}^{\perp}\|^{2+\ve}\right).
\end{eqnarray*}
By the orthogonal conditions of $\rho_{**}^{\perp}$ given in $(2)$ of Proposition~\ref{prop0001}, Lemmas~\ref{lem0007} and \ref{lem0012}, $(iii)$ of $(1)$ of Proposition~\ref{prop0001} and Propositions~\ref{propn0001} and \ref{propn0002} and Lemma~\ref{lemn0003} that
\begin{eqnarray*}
&&\sum_{j=1}^{\nu}\left\langle\left(\Psi_j^*\right)^{p-2}\mathcal{V}_j(\sum_{l=1}^{\nu}\gamma_{1,l}+\gamma_{\mathcal{N},led}-\gamma_{\mathcal{N},led,j})\chi_{\mathcal{B}_j},\rho_{**}^{\perp}\right\rangle_{L^2}\\
&&+\sum_{j=1}^{\nu}\left\langle\left(\Psi_j^*\right)^{p-3}\mathcal{V}_j^2\left(\sum_{l=1}^{\nu}\gamma_{1,l}+\gamma_{\mathcal{N},led}\right)\chi_{\mathcal{B}_j},\rho_{**}^{\perp}\right\rangle_{L^2}\\
&\lesssim&\sum_{j=1}^{\nu}\left|\left\langle\gamma_{1,*}+\gamma_{\mathcal{N},led,rem,j,*}+\rho_{**,2}^{\perp}+\sum_{l=1;l\not=j}^\nu\left(\alpha_l^{***}-\alpha_{l,1}^{**}\right)\Psi_l, \Psi_j^{p-2}\mathcal{V}_j\rho_{**}^{\perp}\right\rangle_{L^2(\mathcal{B}_j)}\right|\\
&&+\left|\left\langle\left(\alpha_j^{***}-\alpha_{j,1}^{**}\right)\Psi_j, \Psi_j^{p-2}\mathcal{V}_j\rho_{**}^{\perp}\right\rangle_{L^2(\mathcal{C}\backslash\mathcal{B}_j)}\right|+\beta_*^2\|\rho_{**}^{\perp}\|\left(\beta_*^{2}+\|\rho_{**}^{\perp}\|^{1+\ve}+\|f\|_{H^{-1}}\right)\\
&\lesssim&\beta_*^4\|\rho_{**}^{\perp}\|+\|\rho_{**}^{\perp}\|\|f\|_{H^{-1}}+\|\rho_{**}^{\perp}\|^{2+\ve}.
\end{eqnarray*}
Summarizing the above estimates, we have
\begin{eqnarray*}
&&\left|\sum_{j=1}^{\nu}\left\langle\left(2A_p\left(\Psi_j^*\right)^{p-2}\mathcal{V}_j(\rho_*-\gamma_{\mathcal{N},led,j})+3B_p\left(\Psi_j^*\right)^{p-3}\mathcal{V}_j^2\rho_*\right)\chi_{\mathcal{B}_j},\rho_{**}^{\perp}\right\rangle_{L^2}\right|\notag\\
&\lesssim&  \beta_*^4\|\rho_{**}^{\perp}\|+\|f\|_{H^{-1}}\|\rho_{**}^{\perp}\|+\|\rho_{**}^{\perp}\|^{2+\ve}.
\end{eqnarray*}

The conclusion follows from the estimates from Step.~1 to Step.~6.
\end{proof}

\section{Estimate of $\beta_*$ and proof of $(a)$ of Theorem~\ref{thmn0001}}
By multiplying \eqref{eq0014} with $\mathcal{V}_j$ on both sides and integrating by parts, the orthogonal conditions of $\rho_*$ and the oddness of $\{\mathcal{V}_i\}$ on $\mathbb{S}^{d-1}$, we have
\begin{eqnarray}\label{eqn0023}
-\left\langle f, \mathcal{V}_j\right\rangle_{H^1}&=&\left\langle \mathcal{R}_{2,j}, \mathcal{V}_j\right\rangle_{L^2}+\left\langle \mathcal{N}, \mathcal{V}_j\right\rangle_{L^2}+\left\langle \mathcal{L}_{j,ex}(\rho_*), \mathcal{V}_j\right\rangle_{L^2}\notag\\
&&+\sum_{i=1;i\not=j}^{\nu}\left\langle \mathcal{R}_{2,i}, \mathcal{V}_j\right\rangle_{L^2}+\left\langle \mathcal{R}_{2,ex}, \mathcal{V}_j\right\rangle_{L^2}
\end{eqnarray}
for all $j=1,2,\cdots,\nu$.
\begin{proposition}\label{propn0004}
Let $d\geq2$, $a<0$ and $b=b_{FS}(a)$.  Then we have
\begin{eqnarray*}
\sum_{j=1}^{\nu}\left(p\left((\alpha_j^*)^{p-1}-1\right)\left\|\Psi_{j}^{p-1}\mathcal{V}_j^2\right\|_{L^1}+\left\langle \mathcal{N}_{j}, \mathcal{V}_j\right\rangle_{L^2}+\left\langle f, \mathcal{V}_j\right\rangle_{H^1}\right)=o(\beta_*^4)+\mathcal{O}\left(\beta_*\|f\|_{H^{-1}}\right),
\end{eqnarray*}
where $\mathcal{N}_{j}$ is given by \eqref{eqnew9999}.
\end{proposition}
\begin{proof}
By the oddness of $\{\mathcal{V}_i\}$ on $\mathbb{S}^{d-1}$, we have
\begin{eqnarray}\label{eqn0027}
\left\langle \mathcal{N}_{j}, \mathcal{V}_j\right\rangle_{L^2}&=&B_p\left\langle\left(\Psi_j^*\right)^{p-3}, \mathcal{V}_j^4\right\rangle_{L^2}+3B_p\left\langle\left(\Psi_j^*\right)^{p-3}\mathcal{V}_j^3,\rho_*\right\rangle_{L^2}\notag\\
&&+2A_p\left\langle\left(\Psi_j^*\right)^{p-2}\mathcal{V}_j^2, \rho_*\right\rangle_{L^2}
\end{eqnarray}
for all $1\leq j\leq\nu$, where $\mathcal{N}_{j}$ is given by \eqref{eqnew9999}.  By \eqref{eqn0023} and \eqref{eqn0027}, we have
\begin{eqnarray}\label{eqn0029}
-\left\langle f, \mathcal{V}_j\right\rangle_{H^1}&=&\left\langle \mathcal{R}_{2,j}, \mathcal{V}_j\right\rangle_{L^2}+\left\langle \mathcal{N}_{j}, \mathcal{V}_j\right\rangle_{L^2}+\left\langle \mathcal{L}_{j,ex}(\rho_*), \mathcal{V}_j\right\rangle_{L^2}+\left\langle \mathcal{R}_{2,ex}, \mathcal{V}_j\right\rangle_{L^2}\notag\\
&&+\left\langle \mathcal{N}-\mathcal{N}_{j}, \mathcal{V}_j\right\rangle_{L^2}+\sum_{i=1;i\not=j}^{\nu}\left\langle \mathcal{R}_{2,i}, \mathcal{V}_j\right\rangle_{L^2}.
\end{eqnarray}
As in the proof of Proposition~\ref{prop0001}, the rest of the proof is to estimate every terms in \eqref{eqn0029}.

{\bf Step.~1}\quad The estimate of $\sum_{j=1}^{\nu}\left\langle \mathcal{R}_{2,j}, \mathcal{V}_j\right\rangle_{L^2}$.

By \eqref{eqn0021},
\begin{eqnarray*}
\sum_{j=1}^{\nu}\left\langle \mathcal{R}_{2,j}, \mathcal{V}_j\right\rangle_{L^2}=\sum_{j=1}^{\nu}p\left((\alpha_j^*)^{p-1}-1\right)\left\|\Psi_{j}^{p-1}\mathcal{V}_j^2\right\|_{L^1}.
\end{eqnarray*}

{\bf Step.~2}\quad The estimate of $\sum_{i=1;i\not=j}^{\nu}\left\langle \mathcal{R}_{2,i}, \mathcal{V}_j\right\rangle_{L^2}$.

By \eqref{eqn0021}, Lemma~\ref{lem0005} and Propositions~\ref{propn0001} and \ref{propn0002},
\begin{eqnarray*}
\left|\sum_{i=1;i\not=j}^{\nu}\left\langle \mathcal{R}_{2,i}, \mathcal{V}_j\right\rangle_{L^2}\right|&\lesssim&\sum_{i=1;i\not=j}^{\nu}\beta_*^2\left|(\alpha_i^*)^{p-1}-1\right|\left\langle\Psi_{i}^{\frac{3p-1}{2}}, \Psi_{j}^{\frac{p+1}{2}}\right\rangle_{L^2}\notag\\
&=&o(\beta_*^4+\|\rho_{**}^{\perp}\|^2)+\mathcal{O}\left(\beta_*\|f\|_{H^{-1}}\right).
\end{eqnarray*}

{\bf Step.~3}\quad The estimate of $\left\langle \mathcal{R}_{2,ex}, \mathcal{V}_j\right\rangle_{L^2}$.

By \eqref{eqn3147}, Lemma~\ref{lem0005} and Propositions~\ref{propn0001} and \ref{propn0002},
\begin{eqnarray*}
\left|\left\langle \mathcal{R}_{2,ex}, \mathcal{V}_j\right\rangle_{L^2}\right|\lesssim\sum_{i=1}^{\nu}\beta_*^2\left\langle\Psi_{i}^{2p-1}, \mathcal{U}_i\right\rangle_{L^2(\mathcal{B}_{j})}=o(\beta_*^4+\|\rho_{**}^{\perp}\|^2)+\mathcal{O}\left(\beta_*\|f\|_{H^{-1}}\right).
\end{eqnarray*}

{\bf Step.~4}\quad The estimate of $\left\langle \mathcal{L}_{j,ex}(\rho_*), \mathcal{V}_j\right\rangle_{L^2}$.

By \eqref{eqn0044} and $(i)$ and $(ii)$ of $(1)$ of Proposition~\ref{prop0001},
\begin{eqnarray*}
\left\langle \mathcal{L}_{j,ex}(\rho_*), \mathcal{V}_j\right\rangle_{L^2}&=&p\left\langle\left(\mathcal{U}^{p-1}-\left(\Psi_j^*\right)^{p-1}\right)\rho_*, \mathcal{V}_j\right\rangle_{L^2}\\
&=&\mathcal{O}\left(\left\langle\beta_*\mathcal{U}_{**}, \gamma_{2,ex}+\gamma_{\mathcal{N},led}+\rho_{**}^{\perp}\right\rangle_{L^2}\right),
\end{eqnarray*}
where $\mathcal{U}_{**}=\sum_{j=1}^{\nu}\Psi_j^{\frac{3(p-1)}{2}}\mathcal{U}_j\chi_{\mathcal{B}_j}+\mathcal{U}^{\frac{3p-1}{2}}\chi_{\mathcal{C}\backslash\cup_{i=1}^{\nu}\mathcal{B}_i}$.
By Lemmas~\ref{lem0005}, \ref{lem0008} and Proposition~\ref{propn0002},
\begin{eqnarray*}
\beta_*\left\langle \mathcal{U}_{**}, \gamma_{2,ex}\right\rangle_{L^2}
&\lesssim&\left\{\aligned
&\sum_{j=1}^{\nu}\beta_*^2Q\left\langle \Psi_j^{\frac{3p-1-2\sigma}{2}},\mathcal{U}_j\right\rangle_{L^2(\mathcal{B}_j)},\quad p\geq\frac{7}{3},\\
&\sum_{j=1}^{\nu}\beta_*^2Q\left\langle \Psi_j^{3p-4},\mathcal{U}_j\right\rangle_{L^2(\mathcal{B}_j)},\quad 1<p<\frac{7}{3}
\endaligned\right.\\
&=&o(\beta_*^4+\|\rho_{**}^{\perp}\|^2)+\mathcal{O}\left(\beta_*\|f\|_{H^{-1}}\right).
\end{eqnarray*}
By Lemma~\ref{lem0005}, $(iii)$ of $(1)$ of Proposition~\ref{prop0001} and Proposition~\ref{propn0002},
\begin{eqnarray*}
\beta_*\left\langle \mathcal{U}_{**}, \gamma_{\mathcal{N},led}\right\rangle_{L^2}
&\lesssim&\beta_*^3Q+\left\{\aligned
&\sum_{j=1}^{\nu}\beta_*^4Q\left\langle\Psi_j^{\frac{3p-1-2\sigma}{2}}\chi_{\mathcal{B}_j},\mathcal{U}_j\right\rangle_{L^2},\quad p\geq3,\\
&\sum_{j=1}^{\nu}\beta_*^4Q\left\langle\Psi_j^{3p-4}\chi_{\mathcal{B}_j},\mathcal{U}_j\right\rangle_{L^2},\quad 1<p<3
\endaligned\right.\\
&=&o(\beta_*^4+\|\rho_{**}^{\perp}\|^2)+\mathcal{O}\left(\beta_*\|f\|_{H^{-1}}\right).
\end{eqnarray*}
By Lemma~\ref{lem0005} and Proposition~\ref{propn0002},
\begin{eqnarray*}
\beta_*\left\langle \mathcal{U}_{**}, \left|\rho_{**}^{\perp}\right|\right\rangle_{L^2}
\lesssim\sum_{i=1}^{\nu}\beta_*\left\|\Psi_i^{\frac{3(p-1)}{2}}\mathcal{U}_i\right\|_{L^2(\mathcal{B}_j)}\|\rho_{**}^{\perp}\|=o(\beta_*^4+\|\rho_{**}^{\perp}\|^2)+\mathcal{O}\left(\beta_*\|f\|_{H^{-1}}\right).
\end{eqnarray*}
Summarizing the above estimates, we have
\begin{eqnarray*}
\left\langle \mathcal{L}_{j,ex}(\rho_*), \mathcal{V}_j\right\rangle_{L^2}=o(\beta_*^4)+\mathcal{O}\left(\beta_*\|f\|_{H^{-1}}\right).
\end{eqnarray*}

{\bf Step.~5}\quad The estimate of $\left\langle \mathcal{N}-\mathcal{N}_{j}, \mathcal{V}_j\right\rangle_{L^2}$.

Similar to \eqref{eqnew0020}, by the oddness of $w_{j,l}$ on $\mathbb{S}^{d-1}$,
\begin{eqnarray*}
\left\langle \mathcal{N}-\mathcal{N}_{j}, \mathcal{V}_j\right\rangle_{L^2}&=&
\sum_{i=1;i\not=j}^{\nu}A_p\left\langle\left(\Psi_i^*\right)^{p-2}\left(\mathcal{V}_i^2+2\mathcal{V}_i\rho_*\right)\chi_{\mathcal{B}_i}, \mathcal{V}_j\right\rangle_{L^2}\notag\\
&&+\sum_{i=1;i\not=j}^{\nu}B_p\left\langle\left(\Psi_i^*\right)^{p-3}\left(\mathcal{V}_i^3+\mathcal{V}_i^2\rho_*\right)\chi_{\mathcal{B}_i}, \mathcal{V}_j\right\rangle_{L^2}\notag\\
&&+2A_p\left\langle\left(\Psi_j^*\right)^{p-2}\mathcal{V}_j\rho_*\chi_{\mathcal{C}\backslash\mathcal{B}_j}, \mathcal{V}_j\right\rangle_{L^2}\notag\\
&&+B_p\left\langle\left(\Psi_j^*\right)^{p-3}\left(\mathcal{V}_j^3+3\mathcal{V}_j^2\rho_*\right)\chi_{\mathcal{C}\backslash\mathcal{B}_j}, \mathcal{V}_j\right\rangle_{L^2}\notag\\
&&+\sum_{i=1}^{\nu}\left\langle\mathcal{O}\left(\beta_*|\rho_*|\Psi_{i}^{\frac{3p-5}{2}}\mathcal{U}_i+\beta_*^2\Psi_{i}^{2p-2}\mathcal{U}_i\right)\chi_{\mathcal{B}_i}, \mathcal{V}_j\right\rangle_{L^2}\notag\\
&&+\left\langle\mathcal{O}\left(\mathcal{U}^{p-2}\mathcal{V}^{2}+\beta_*|\rho_*|\mathcal{U}^{\frac{3(p-1)}{2}}\right)\chi_{\mathcal{C}\backslash\left(\cup_{j=1}^{\nu}\mathcal{B}_j\right)}, \mathcal{V}_j\right\rangle_{L^2}\notag\\
&&+\left\langle\overline{\mathcal{N}}_{rem}, \mathcal{V}_j\right\rangle_{L^2}+\left\langle\mathcal{N}_{0}, \mathcal{V}_j\right\rangle_{L^2}.
\end{eqnarray*}

{\bf Step.~5.1}\quad The estimate of $\sum_{i=1;i\not=j}^{\nu}\left\langle\left(\Psi_i^*\right)^{p-2}\mathcal{V}_i^2\chi_{\mathcal{B}_i}, \mathcal{V}_j\right\rangle_{L^2}$.

By Lemma~\ref{lem0005} and Proposition~\ref{propn0002},
\begin{eqnarray*}
\left|\sum_{i=1;i\not=j}^{\nu}\left\langle\left(\Psi_i^*\right)^{p-2}\mathcal{V}_i^2\chi_{\mathcal{B}_i}, \mathcal{V}_j\right\rangle_{L^2}\right|\lesssim\beta_*^3\left\langle\Psi_i^{2p-1},\Psi_j^{\frac{p+1}{2}} \right\rangle_{L^2}
=o(\beta_*^4+\|\rho_{**}^{\perp}\|^2)+\mathcal{O}\left(\beta_*\|f\|_{H^{-1}}\right).
\end{eqnarray*}

{\bf Step.~5.2}\quad The estimate of $\sum_{i=1;i\not=j}^{\nu}\left\langle\left(\Psi_i^*\right)^{p-3}\mathcal{V}_i^3\chi_{\mathcal{B}_i}, \mathcal{V}_j\right\rangle_{L^2}$.

By \eqref{eqn0191}, we also have
\begin{eqnarray*}
\left|\sum_{i=1;i\not=j}^{\nu}\left\langle\left(\Psi_i^*\right)^{p-3}\mathcal{V}_i^3\chi_{\mathcal{B}_i}, \mathcal{V}_j\right\rangle_{L^2}\right|&\lesssim&\left|\sum_{i=1;i\not=j}^{\nu}\left\langle\left(\Psi_i^*\right)^{p-2}\mathcal{V}_i^2\chi_{\mathcal{B}_i}, \mathcal{V}_j\right\rangle_{L^2}\right|\\
&=&o(\beta_*^4+\|\rho_{**}^{\perp}\|^2)+\mathcal{O}\left(\beta_*\|f\|_{H^{-1}}\right).
\end{eqnarray*}

{\bf Step.~5.3}\quad The estimate of $\sum_{i=1;i\not=j}^{\nu}\left\langle\left(\Psi_i^*\right)^{p-2}\mathcal{V}_i\rho_*\chi_{\mathcal{B}_i}, \mathcal{V}_j\right\rangle_{L^2}$.

By $(1)$ of Proposition~\ref{prop0001},
\begin{eqnarray*}
\left\langle\left(\Psi_i^*\right)^{p-2}\mathcal{V}_i\rho_*\chi_{\mathcal{B}_i}, \mathcal{V}_j\right\rangle_{L^2}=\left\langle\left(\Psi_i^*\right)^{p-2}\mathcal{V}_i\mathcal{V}_j\chi_{\mathcal{B}_i}, \rho_{**}^{\perp}+\gamma_{ex}+\gamma_{*}+\gamma_{\mathcal{N},led}\right\rangle_{L^2}.
\end{eqnarray*}
By Lemma~\ref{lem0005}, $(i)$ of $(1)$ of Proposition~\ref{prop0001} and Proposition~\ref{propn0002},
\begin{eqnarray*}
\left|\left\langle \left(\Psi_i^*\right)^{p-2}\mathcal{V}_i\mathcal{V}_j\chi_{\mathcal{B}_i}, \gamma_{ex}\right\rangle_{L^2}\right|
&\lesssim&\left\{\aligned
&\beta_*^2Q\left\langle \Psi_i^{\frac{3p-1-2\sigma}{2}},\Psi_j^{\frac{p+1}{2}}\right\rangle_{L^2(\mathcal{B}_i)},\quad p\geq3,\\
&\beta_*^2Q\left\langle \Psi_i^{\frac{5p-7}{2}},\Psi_j^{\frac{p+1}{2}}\right\rangle_{L^2(\mathcal{B}_i)},\quad 1<p<3
\endaligned\right.\\
&=&o(\beta_*^4+\|\rho_{**}^{\perp}\|^2)+\mathcal{O}\left(\beta_*\|f\|_{H^{-1}}\right).
\end{eqnarray*}
By Lemma~\ref{lem0005}, $(ii)$ and $(iii)$ of $(1)$ of Proposition~\ref{prop0001} and Propositions~\ref{propn0001} and \ref{propn0002},
\begin{eqnarray*}
\left|\left\langle \left(\Psi_i^*\right)^{p-2}\mathcal{V}_i\mathcal{V}_j\chi_{\mathcal{B}_i}, \gamma_{*}+\gamma_{\mathcal{N},led}\right\rangle_{L^2}\right|
&\lesssim&\beta_*^4\left\langle \Psi_i^{\frac{3p-1-2\sigma}{2}},\Psi_j^{\frac{p+1}{2}}\right\rangle_{L^2(\mathcal{B}_i)}\\
&&+o(\beta_*^4+\|\rho_{**}^{\perp}\|^2)+\beta_*\|f\|_{H^{-1}}\\
&=&o(\beta_*^4)+\mathcal{O}\left(\beta_*\|f\|_{H^{-1}}\right).
\end{eqnarray*}
By Lemma~\ref{lem0005} and Proposition~\ref{propn0002},
\begin{eqnarray*}
\left\langle \left(\Psi_i^*\right)^{p-2}\mathcal{V}_i\mathcal{V}_j\chi_{\mathcal{B}_i}, \rho_{**}^{\perp}\right\rangle_{L^2}
&\lesssim&\beta_*^2\|\rho_{**}^{\perp}\|\left(\left\langle \Psi_i^{3p-3},\Psi_j^{p+1}\right\rangle_{L^2(\mathcal{B}_i)}\right)^{\frac12}\\
&=&o(\beta_*^4+\|\rho_{**}^{\perp}\|^2)+\mathcal{O}\left(\beta_*\|f\|_{H^{-1}}\right).
\end{eqnarray*}
Summarizing the above estimates, we have
\begin{eqnarray*}
\sum_{i=1;i\not=j}^{\nu}\left\langle\left(\Psi_i^*\right)^{p-2}\mathcal{V}_i\rho_*\chi_{\mathcal{B}_i}, \mathcal{V}_j\right\rangle_{L^2}=o(\beta_*^4+\|\rho_{**}^{\perp}\|^2)+\mathcal{O}\left(\beta_*\|f\|_{H^{-1}}\right).
\end{eqnarray*}

{\bf Step.~5.4}\quad The estimate of $\sum_{i=1;i\not=j}^{\nu}\left\langle\left(\Psi_i^*\right)^{p-3}\mathcal{V}_i^2\rho_*\chi_{\mathcal{B}_i}, \mathcal{V}_j\right\rangle_{L^2}$.

By \eqref{eqn0191}, we have
\begin{eqnarray*}
\left|\sum_{i=1;i\not=j}^{\nu}\left\langle\left(\Psi_i^*\right)^{p-3}\mathcal{V}_i^2\rho_*\chi_{\mathcal{B}_i}, \mathcal{V}_j\right\rangle_{L^2}\right|&\lesssim&\left|\sum_{i=1;i\not=j}^{\nu}\left\langle\left(\Psi_i^*\right)^{p-2}\mathcal{V}_i\rho_*\chi_{\mathcal{B}_i}, \mathcal{V}_j\right\rangle_{L^2}\right|\\
&=&o(\beta_*^4+\|\rho_{**}^{\perp}\|^2)+\mathcal{O}\left(\beta_*\|f\|_{H^{-1}}\right).
\end{eqnarray*}

{\bf Step.~5.5}\quad The estimate of $\left\langle\left(\Psi_j^*\right)^{p-2}\mathcal{V}_j\rho_*\chi_{\mathcal{C}\backslash\mathcal{B}_j}, \mathcal{V}_j\right\rangle_{L^2}$.

By $(1)$ of Proposition~\ref{prop0001},
\begin{eqnarray*}
\left\langle\left(\Psi_j^*\right)^{p-2}\mathcal{V}_j\rho_*\chi_{\mathcal{C}\backslash\mathcal{B}_j}, \mathcal{V}_j\right\rangle_{L^2}=
\left\langle\left(\Psi_j^*\right)^{p-2}\mathcal{V}_j^2\chi_{\mathcal{C}\backslash\mathcal{B}_j}, \rho_{**}^{\perp}+\gamma_{1,ex}+\gamma_{*}+\gamma_{\mathcal{N},led}\right\rangle_{L^2}.
\end{eqnarray*}
By Lemmas~\ref{lem0005}, \ref{lem0006} and Proposition~\ref{propn0002},
\begin{eqnarray*}
\left|\left\langle\left(\Psi_j^*\right)^{p-2}\mathcal{V}_j^2\chi_{\mathcal{C}\backslash\mathcal{B}_j}, \gamma_{1,ex}\right\rangle_{L^2}\right|
&\lesssim&\left\{\aligned
&\sum_{i=1;i\not=j}^{\nu}\beta_*^2Q\left\langle \Psi_i^{1-\sigma},\Psi_j^{2p-1}\right\rangle_{L^2(\mathcal{B}_i)},\quad p\geq3,\\
&\sum_{i=1;i\not=j}^{\nu}\beta_*^2Q\left\langle \Psi_i^{p-2},\Psi_j^{2p-1}\right\rangle_{L^2(\mathcal{B}_i)},\quad 1<p<3
\endaligned\right.\\
&=&o(\beta_*^4+\|\rho_{**}^{\perp}\|^2)+\mathcal{O}\left(\beta_*\|f\|_{H^{-1}}\right).
\end{eqnarray*}
By Lemma~\ref{lem0005}, $(ii)$ and $(iii)$ of $(1)$ of Proposition~\ref{prop0001} and Propositions~\ref{propn0001} and \ref{propn0002},
\begin{eqnarray*}
\left\langle \left(\Psi_j^*\right)^{p-2}\mathcal{V}_j^2\chi_{\mathcal{C}\backslash\mathcal{B}_j}, \gamma_{*}+\gamma_{\mathcal{N},led}\right\rangle_{L^2}
&\lesssim&\sum_{i=1;i\not=j}^{\nu}\beta_*^4\left\langle \Psi_i^{1-\sigma},\Psi_j^{2p-1}\right\rangle_{L^2(\mathcal{B}_i)}\\
&&+o(\beta_*^4+\|\rho_{**}^{\perp}\|^2)+\beta_*\|f\|_{H^{-1}}\\
&=&o(\beta_*^4+\|\rho_{**}^{\perp}\|^2)+\mathcal{O}\left(\beta_*\|f\|_{H^{-1}}\right).
\end{eqnarray*}
By Lemma~\ref{lem0005} and Proposition~\ref{propn0002},
\begin{eqnarray*}
\left|\left\langle\left(\Psi_j^*\right)^{p-2}\mathcal{V}_j^2\chi_{\mathcal{C}\backslash\mathcal{B}_j}, \rho_{**}^{\perp}\right\rangle_{L^2}\right|&\lesssim&\beta_*^2\|\rho_{**}^{\perp}\|\left(\int_{\mathcal{C}\backslash\mathcal{B}_j}\Psi_j^{4p-2}d\mu\right)^{\frac12}\\
&=&o(\beta_*^4+\|\rho_{**}^{\perp}\|^2)+\mathcal{O}\left(\beta_*\|f\|_{H^{-1}}\right).
\end{eqnarray*}
Summarizing the above estimates, we have
\begin{eqnarray*}
\left\langle\left(\Psi_j^*\right)^{p-2}\mathcal{V}_j\rho_*\chi_{\mathcal{C}\backslash\mathcal{B}_j}, \mathcal{V}_j\right\rangle_{L^2}=o(\beta_*^4)+\mathcal{O}\left(\beta_*\|f\|_{H^{-1}}\right).
\end{eqnarray*}

{\bf Step.~5.6}\quad The estimate of $\left\langle\left(\Psi_j^*\right)^{p-3}\left(\mathcal{V}_j^3+3\mathcal{V}_j^2\rho_*\right)\chi_{\mathcal{C}\backslash\mathcal{B}_j}, \mathcal{V}_j\right\rangle_{L^2}$.

By Lemma~\ref{lem0005} and $(1)$ of Proposition~\ref{prop0001},
\begin{eqnarray*}
\left\langle\left(\Psi_j^*\right)^{p-3}\left(\mathcal{V}_j^3+3\mathcal{V}_j^2\rho_*\right)\chi_{\mathcal{C}\backslash\mathcal{B}_j}, \mathcal{V}_j\right\rangle_{L^2}&=&3\left\langle\left(\Psi_j^*\right)^{p-3}\mathcal{V}_j^3\chi_{\mathcal{C}\backslash\mathcal{B}_j}, \gamma_{2,ex}+\gamma_{\mathcal{N},led}\right\rangle_{L^2}\\
&&3\left\langle\left(\Psi_j^*\right)^{p-3}\mathcal{V}_j^3\chi_{\mathcal{C}\backslash\mathcal{B}_j}, \rho_{**}^{\perp}\right\rangle_{L^2}+o(\beta_*^4).
\end{eqnarray*}
By Lemmas~\ref{lem0005} and \ref{lem0008}, $\left|\left\langle\left(\Psi_j^*\right)^{p-3}\mathcal{V}_j^3\chi_{\mathcal{C}\backslash\mathcal{B}_j}, \gamma_{2,ex}\right\rangle_{L^2}\right|=o(\beta_*^4)$.
By Lemma~\ref{lem0005} and $(iii)$ of $(1)$ of Proposition~\ref{prop0001},
\begin{eqnarray*}
\left|\left\langle \left(\Psi_j^*\right)^{p-3}\mathcal{V}_j^3\chi_{\mathcal{C}\backslash\mathcal{B}_j}, \gamma_{\mathcal{N},led}\right\rangle_{L^2}\right|
=o(\beta_*^4).
\end{eqnarray*}
By Lemma~\ref{lem0005} and Proposition~\ref{propn0002},
\begin{eqnarray*}
\left|\left\langle\left(\Psi_j^*\right)^{p-3}\mathcal{V}_j^3\chi_{\mathcal{C}\backslash\mathcal{B}_j}, \rho_{**}^{\perp}\right\rangle_{L^2}\right|&\lesssim&\beta_*^3\|\rho_{**}^{\perp}\|\left(\int_{\mathcal{C}\backslash\mathcal{B}_j}\Psi_j^{5p-3}d\mu\right)^{\frac12}\\
&=&o(\beta_*^4+\|\rho_{**}^{\perp}\|^2)+\mathcal{O}\left(\beta_*\|f\|_{H^{-1}}\right).
\end{eqnarray*}
Summarizing the above estimates, we have
\begin{eqnarray*}
\left\langle\left(\Psi_j^*\right)^{p-3}\left(\mathcal{V}_j^3+3\mathcal{V}_j^2\rho_*\right)\chi_{\mathcal{C}\backslash\mathcal{B}_j}, \mathcal{V}_j\right\rangle_{L^2}=o(\beta_*^4+\|\rho_{**}^{\perp}\|^2)+\mathcal{O}\left(\beta_*\|f\|_{H^{-1}}\right).
\end{eqnarray*}

{\bf Step.~5.7}\quad The estimate of $\sum_{i=1}^{\nu}\left\langle\beta_*\rho_*\Psi_{i}^{\frac{3p-5}{2}}\mathcal{U}_i\chi_{\mathcal{B}_i}, \mathcal{V}_j\right\rangle_{L^2}$.

By $(1)$ of Proposition~\ref{prop0001},
\begin{eqnarray*}
\left|\left\langle\beta_*\rho_*\Psi_{i}^{\frac{3p-5}{2}}\mathcal{U}_i\chi_{\mathcal{B}_i}, \mathcal{V}_j\right\rangle_{L^2}\right|&\lesssim&\beta_*^2\left\langle\Psi_i^{2p-2}\mathcal{U}_i\chi_{\mathcal{B}_{i}}, \left|\rho_{**}^{\perp}+\gamma_{ex}+\gamma_{*}+\gamma_{\mathcal{N},led}\right|\right\rangle_{L^2}.
\end{eqnarray*}
By Lemma~\ref{lem0005}, $(i)$ of $(1)$ of Proposition~\ref{prop0001}, and Proposition~\ref{propn0002},
\begin{eqnarray*}
\left|\beta_*^2\left\langle\Psi_i^{2p-2}\mathcal{U}_i\chi_{\mathcal{B}_{i}}, \gamma_{ex}\right\rangle_{L^2}\right|&\lesssim&\left\{\aligned
&\beta_*^2Q\left\langle \Psi_i^{2p-1-\sigma},\mathcal{U}_i\right\rangle_{L^2(\mathcal{B}_{i,+})},\quad p\geq3,\\
&\beta_*^2Q\left\langle \Psi_i^{3p-4},\mathcal{U}_i\right\rangle_{L^2(\mathcal{B}_{i,+})},\quad 1<p<3
\endaligned\right.\\
&=&o(\beta_*^4+\|\rho_{**}^{\perp}\|^2)+\mathcal{O}\left(\beta_*\|f\|_{H^{-1}}\right).
\end{eqnarray*}
By Lemma~\ref{lem0005}, $(ii)$ and $(iii)$ of $(1)$ of Proposition~\ref{prop0001} and Propositions~\ref{propn0001} and \ref{propn0002},
\begin{eqnarray*}
\left|\beta_*^2\left\langle\Psi_i^{2p-2}\mathcal{U}_i\chi_{\mathcal{B}_{i}}, \gamma_{*}+\gamma_{\mathcal{N},led}\right\rangle_{L^2}\right|
&\lesssim&\beta_*^4Q+o(\beta_*^4+\|\rho_{**}^{\perp}\|^2)+\beta_*\|f\|_{H^{-1}}\\
&=&o(\beta_*^4+\|\rho_{**}^{\perp}\|^2)+\mathcal{O}\left(\beta_*\|f\|_{H^{-1}}\right).
\end{eqnarray*}
By Lemma~\ref{lem0005} and Proposition~\ref{propn0002},
\begin{eqnarray*}
\left|\beta_*^2\left\langle\Psi_i^{2p-2}\mathcal{U}_i\chi_{\mathcal{B}_{i}}, \rho_{**}^{\perp}\right\rangle_{L^2}\right|
&\lesssim&\beta_*^2\|\rho_{**}^{\perp}\|\left(\left\langle \Psi_i^{4p-4},\mathcal{U}_i^2\right\rangle_{L^2(\mathcal{B}_{i})}\right)^{\frac12}\\
&=&o(\beta_*^4+\|\rho_{**}^{\perp}\|^2)+\mathcal{O}\left(\beta_*\|f\|_{H^{-1}}\right).
\end{eqnarray*}
Summarizing the above estimates, we have
\begin{eqnarray*}
\sum_{i=1}^{\nu}\left\langle\beta_*\rho_*\Psi_{i}^{\frac{3p-5}{2}}\mathcal{U}_i\chi_{\mathcal{B}_i}, \mathcal{V}_j\right\rangle_{L^2}=o(\beta_*^4+\|\rho_{**}^{\perp}\|^2)+\mathcal{O}\left(\beta_*\|f\|_{H^{-1}}\right).
\end{eqnarray*}

{\bf Step.~5.8}\quad The estimates of
\begin{eqnarray*}
\sum_{i=1}^{\nu}\left\langle\beta_*^2\Psi_{i}^{2p-2}\mathcal{U}_i\chi_{\mathcal{B}_i}, \mathcal{V}_j\right\rangle_{L^2}\quad\text{and}\quad\left\langle\mathcal{U}^{p-2}\mathcal{V}^2\chi_{\mathcal{C}\backslash\cup_{i=1}^{\nu}\mathcal{B}_i}, \mathcal{V}_j\right\rangle_{L^2}.
\end{eqnarray*}

By Lemma~\ref{lem0005} and Proposition~\ref{propn0002},
\begin{eqnarray*}
\left|\sum_{i=1}^{\nu}\left\langle\beta_*^2\Psi_{i}^{2p-2}\mathcal{U}_i\chi_{\mathcal{B}_i}, \mathcal{V}_j\right\rangle_{L^2}\right|&\lesssim&\beta_*^3\sum_{i=1}^{\nu}\left\langle\Psi_i^{\frac{5p-3}{2}},\mathcal{U}_i \right\rangle_{L^2(\mathcal{B}_{i})}\\
&=&o(\beta_*^4+\|\rho_{**}^{\perp}\|^2)+\mathcal{O}\left(\beta_*\|f\|_{H^{-1}}\right)
\end{eqnarray*}
and
\begin{eqnarray*}
\left|\left\langle\mathcal{U}^{p-2}\mathcal{V}^2\chi_{\mathcal{C}\backslash\cup_{i=1}^{\nu}\mathcal{B}_i}, \mathcal{V}_j\right\rangle_{L^2}\right|&\lesssim&\beta_*^3\int_{\mathcal{C}\backslash\cup_{i=1}^{\nu}\mathcal{B}_i}\mathcal{U}^{\frac{5p-1}{2}}d\mu\\
&=&o(\beta_*^4+\|\rho_{**}^{\perp}\|^2)+\mathcal{O}\left(\beta_*\|f\|_{H^{-1}}\right).
\end{eqnarray*}

{\bf Step.~5.9}\quad The estimate of $\left\langle\beta_*\rho_{*}\mathcal{U}^{\frac{3(p-1)}{2}}\chi_{\mathcal{C}\backslash\cup_{i=1}^{\nu}\mathcal{B}_i}, \mathcal{V}_j\right\rangle_{L^2}$.

By $(1)$ of Proposition~\ref{prop0001},
\begin{eqnarray*}
\left|\left\langle\beta_*\rho_{*}\mathcal{U}^{\frac{3(p-1)}{2}}\chi_{\mathcal{C}\backslash\cup_{i=1}^{\nu}\mathcal{B}_i}, \mathcal{V}_j\right\rangle_{L^2}\right|&\lesssim&\beta_*^2\left\langle\mathcal{U}^{2p-1}\chi_{\mathcal{C}\backslash\cup_{i=1}^{\nu}\mathcal{B}_i}, \gamma_{ex}+\gamma_{*}+\gamma_{\mathcal{N},led}\right\rangle_{L^2}\\
&&+\beta_*^2\left\langle\mathcal{U}^{2p-1}\chi_{\mathcal{C}\backslash\cup_{i=1}^{\nu}\mathcal{B}_i}, \rho_{**}^{\perp}\right\rangle_{L^2}.
\end{eqnarray*}
Lemma~\ref{lem0005}, $(1)$ of Proposition~\ref{prop0001} and Propositions~\ref{propn0001} and \ref{propn0002},
\begin{eqnarray*}
\left|\left\langle\beta_*\rho_{*}\mathcal{U}^{\frac{3(p-1)}{2}}\chi_{\mathcal{C}\backslash\cup_{i=1}^{\nu}\mathcal{B}_i}, \mathcal{V}_j\right\rangle_{L^2}\right|&\lesssim&\beta_*^4\left\|\mathcal{U}^{2p-\sigma}\right\|_{L^1(\mathcal{C}\backslash\cup_{i=1}^{\nu}\mathcal{B}_i)}+\beta_*\|f\|_{H^{-1}}\notag\\
&&+\beta_*^2\|\rho_{**}^{\perp}\|\left\|\mathcal{U}^{2p-1}\right\|_{L^1(\mathcal{C}\backslash\cup_{i=1}^{\nu}\mathcal{B}_i)}+o(\|\rho_{**}^{\perp}\|^2)\notag\\
&=&o(\beta_*^4+\|\rho_{**}^{\perp}\|^2)+\mathcal{O}\left(\beta_*\|f\|_{H^{-1}}\right).
\end{eqnarray*}

{\bf Step.~5.10}\quad The estimate of $\left\langle\mathcal{N}_{0}, \mathcal{V}_j\right\rangle_{L^2}$.

By \eqref{eqn0191}, \eqref{eqn0190} and Lemma~\ref{lemn0001},
\begin{eqnarray*}
\left|\left\langle\mathcal{N}_{0}, \mathcal{V}_j\right\rangle_{L^2}\right|\lesssim\beta_*\left\langle\mathcal{U}^{p-2}\Psi_j^{\frac{p+1}{2}}, \gamma_{ex}^2+\left|\gamma_{\mathcal{N},led}+\gamma_{*}\right|^2\right\rangle_{L^2}+\beta_*\|\rho_{**}^{\perp}\|^2.
\end{eqnarray*}
By Lemma~\ref{lem0005}, $(i)$ of $(1)$ of Proposition~\ref{prop0001}, and Propositions~\ref{propn0002} and \ref{propn0003},
\begin{eqnarray*}
\beta_*\left\langle\mathcal{U}^{p-2}\Psi_j^{\frac{p+1}{2}}, \gamma_{ex}^2\right\rangle_{L^2}
&\lesssim&\left\{\aligned
&\beta_*Q^2\int_{\mathcal{B}_j}\Psi_j^{\frac{3p+1-4\sigma}{2}}d\mu,\quad p\geq3,\\
&\beta_*Q^2\int_{\mathcal{B}_j}\Psi_j^{\frac{7p-11}{2}}d\mu,\quad 1<p<3
\endaligned\right.\\
&=&o(\beta_*^4+\|\rho_{**}^{\perp}\|^2)+\mathcal{O}\left(\beta_*\|f\|_{H^{-1}}\right).
\end{eqnarray*}
By Lemma~\ref{lem0005}, $(ii)$ and $(iii)$ of $(1)$ of Proposition~\ref{prop0001} and Propositions~\ref{propn0001} and \ref{propn0002},
\begin{eqnarray*}
\beta_*\left\langle\mathcal{U}^2\Psi_j^{\frac{p+1}{2}}, \left|\gamma_{\mathcal{N},led}+\gamma_{*}\right|^2\right\rangle_{L^2}
&\lesssim&\beta_*\left(\beta_*^2+\|f\|_{H^{-1}}\right)^{2}+o(\beta_*^4+\|\rho_{**}^{\perp}\|^2)+\mathcal{O}\left(\beta_*\|f\|_{H^{-1}}\right)\\
&=&o(\beta_*^4+\|\rho_{**}^{\perp}\|^2)+\mathcal{O}\left(\beta_*\|f\|_{H^{-1}}\right).
\end{eqnarray*}
Summarizing the above estimates, we have
\begin{eqnarray*}
\left\langle\mathcal{N}_{0}, \mathcal{V}_j\right\rangle_{L^2}=o(\beta_*^4+\|\rho_{**}^{\perp}\|^2)+\mathcal{O}\left(\beta_*\|f\|_{H^{-1}}\right).
\end{eqnarray*}

{\bf Step.~5.11}\quad The estimate of $\left\langle\overline{\mathcal{N}}_{rem}, \mathcal{V}_j\right\rangle_{L^2}$.

By Lemma~\ref{lem0001} and Propositions~\ref{propn0001}, \ref{propn0002} and \ref{propn0003},
\begin{eqnarray*}
\left|\left\langle\overline{\mathcal{N}}_{rem}, \mathcal{V}_j\right\rangle_{L^2}\right|\lesssim\beta_*\left(\beta_*^4+\|\rho_{**}^{\perp}\|^{1+\ve}+\|f\|_{H^{-1}}\right)=o(\beta_*^4+\|\rho_{**}^{\perp}\|^2)+\mathcal{O}\left(\beta_*\|f\|_{H^{-1}}\right).
\end{eqnarray*}

By summarizing the estimates from Step.~5.1 to Step.~5.11, we have
\begin{eqnarray*}
\left\langle \mathcal{N}-\mathcal{N}_{j}, \mathcal{V}_j\right\rangle_{L^2}
=o(\beta_*^4+\|\rho_{**}^{\perp}\|^2)+\mathcal{O}\left(\beta_*\|f\|_{H^{-1}}\right).
\end{eqnarray*}

The conclusion follows from the estimates from Step.~1 to Step.~5 and Proposition~\ref{propn0003}.
\end{proof}

\vskip0.12in

With Proposition~\ref{propn0004} in hands, we can finally estimate $\beta_*$.
\begin{proposition}\label{propn0005}
Let $d\geq2$, $a<0$ and $b=b_{FS}(a)$.  Then we have $\beta_*\lesssim\|f\|_{H^{-1}}^{\frac{1}{3}}$.
\end{proposition}
\begin{proof}
By Lemmas~\ref{lem0005}, \ref{lem0011} and \ref{lem0012}, $(1)$ of Proposition~\ref{prop0001} and Propositions~\ref{propn0001}, \ref{propn0002} and \ref{propn0003},
\begin{eqnarray}\label{eqnewnew4443}
\left\langle\left(\Psi_j^*\right)^{p-2}\mathcal{V}_j^2, \rho_*\right\rangle_{L^2}&=&\left\langle\left(\Psi_j^*\right)^{p-2}\mathcal{V}_j^2, \gamma_{ex}+\gamma_*+\gamma_{\mathcal{N},led}+\rho_{**}^{\perp}\right\rangle_{L^2}\notag\\
&=&\left\langle\left(\Psi_j^*\right)^{p-2}\mathcal{V}_j^2, \sum_{i=1}^{\nu}(\gamma_{1,i}+\alpha_{i,0}^{**}\Psi_i)+\gamma_{\mathcal{N},led}\right\rangle_{L^2}\notag\\
&&+o(\beta_*^4+\beta_*\|f\|_{H^{-1}}).
\end{eqnarray}
We write $\sum_{i=1}^{\nu}\gamma_{1,i}+\gamma_{\mathcal{N},led}=\sum_{j=1}^{\nu}\overline{\alpha}_{j,*}\Psi_j+\overline{\gamma}_{**}^{\perp}$
such that $\left\langle\Psi_j, \overline{\gamma}_{**}^{\perp}\right\rangle$ for all $1\leq j\leq \nu$.  Then by \eqref{eqn0052} and $(i)$ of $(1)$ of Proposition~\ref{prop0001},
\begin{eqnarray}\label{eqnewnew6678}
\left\langle\sum_{i=1}^{\nu}(\overline{\alpha}_{i,*}+\alpha_{i,0}^{**}\Psi_i)+\gamma_{\mathcal{N},led},\Psi_j\right\rangle&=&\left\langle\sum_{i=1}^{\nu}(\gamma_{1,i}+\alpha_{i,0}^{**}\Psi_i)+\gamma_{\mathcal{N},led},\Psi_j\right\rangle\notag\\
&=&-\left\langle\gamma_{1,ex},\Psi_j\right\rangle\notag\\
&=&\mathcal{O}(Q).
\end{eqnarray}
It follows from \eqref{eqnewnew4443} and Propositions~\ref{propn0002} and \ref{propn0003} that
\begin{eqnarray}\label{eqnewnew4444}
\left\langle\left(\Psi_j^*\right)^{p-2}\mathcal{V}_j^2, \rho_*\right\rangle_{L^2}
&=&\left\langle\left(\Psi_j^*\right)^{p-2}\mathcal{V}_j^2, \overline{\gamma}_{**}^{\perp}\right\rangle_{L^2}+o(\beta_*^4+\beta_*\|f\|_{H^{-1}}).
\end{eqnarray}
By \eqref{eqn0012}, \eqref{eqn0015} and Lemma~\ref{lemn0003}, we know that $\overline{\gamma}_{**}^{\perp}$ satisfies
\begin{eqnarray}\label{eqn5512}
\left\{\aligned&\mathcal{L}(\overline{\gamma}_{**}^{\perp})=\overline{\mathcal{R}}_{1,**}-\sum_{i=1}^{\nu}\Psi_{i}^{p-1}\left(c_{1,j,i}\partial_t\Psi_{i}+\sum_{l=1}^{d}\varsigma_{1,j,i,l}w_{i,l}\right),\quad \text{in }\mathcal{C},\\
&\langle\Psi_{j}, \overline{\gamma}_{**}^{\perp}\rangle=\langle \partial_t\Psi_{j}, \overline{\gamma}_{**}^{\perp}\rangle=\langle w_{j,l}, \overline{\gamma}_{**}^{\perp}\rangle=0\quad\text{for all }1\leq j\leq\nu\text{ and all }1\leq l\leq d,\endaligned\right.
\end{eqnarray}
where
\begin{eqnarray*}
\overline{\mathcal{R}}_{1,**}&=&\mathcal{N}_{led}+\sum_{l=1}^{\nu}\left(\left(\alpha_l^*\right)^p-\alpha_l^*-\overline{\alpha}_{l,*}\left(1-p\left(\alpha_l^*\right)^{p-1}\right)\right)\Psi_l^p\notag\\
&&+\sum_{l=1}^{\nu}p\overline{\alpha}_{l,*}\left(\mathcal{U}^{p-1}-\left(\Psi_l^*\right)^{p-1}\right)\Psi_l
\end{eqnarray*}
with $\mathcal{N}_{led}$ given by \eqref{eqn3045}.  Since by \eqref{eqn0068}, \eqref{eqnewnew6678} and Propositions~\ref{propn0001}, \ref{propn0002} and \ref{propn0003}, we have $\sum_{i=1}^{\nu}\left|\overline{\alpha}_{i,*}\right|=\mathcal{O}\left(\beta_*^2+\|f\|_{H^{-1}}\right)$.  Thus, by the orthogonal conditions of $\overline{\gamma}_{**}^{\perp}$, multiplying \eqref{eqn5512} with $\overline{\gamma}_{**}^{\perp}$ on both sides and integrating by parts, Lemma~\ref{lem0005} and Propositions~\ref{propn0002} and \ref{propn0003}, we have
\begin{eqnarray}
\left\langle\mathcal{L}(\overline{\gamma}_{**}^{\perp})-\mathcal{N}_{led}, \overline{\gamma}_{**}^{\perp}\right\rangle_{L^2}=o\left(\beta_*^4+\beta_*\|f\|_{H^{-1}}\right)+\mathcal{O}\left(\|f\|_{H^{-1}}^2\right).
\label{eqnewnew1220}
\end{eqnarray}
Since $Q\to0$ and $\beta_*\to0$ as $\|f\|_{H^{-1}}\to0$, by Lemma~\ref{lem0005}, Propositions~\ref{propn0001}, \ref{propn0002} and \ref{propn0003}, it is easy to see that
\begin{eqnarray*}
\left\langle \mathcal{L}(\overline{\gamma}_{**}^{\perp})-\mathcal{N}_{led}, \overline{\gamma}_{**}^{\perp}\right\rangle_{L^2}&=&\sum_{j=1}^{\nu}\left\langle \mathcal{L}_j(\overline{\gamma}_{**}^{\perp})-\frac{p(p-1)}{2}\left(\Psi_j^*\right)^{p-2}\mathcal{V}_j^2, \overline{\gamma}_{**}^{\perp}\right\rangle_{L^2}\\
&&+o(\beta_*^4)+\mathcal{O}(\beta_*\|f\|_{H^{-1}})
\end{eqnarray*}
which, together with \eqref{eqnewnew4444}, \eqref{eqnewnew1220} and Propositions~\ref{propn0001}, \ref{propn0002}, \ref{propn0003} and \ref{propn0004}, implies that
\begin{eqnarray}\label{eqn0030}
&&\sum_{j=1}^{\nu}\left(\left\langle f, \mathcal{V}_j\right\rangle_{H^1}-\frac{\left\langle\Psi_j^{p-1},\mathcal{V}_j^2\right\rangle_{L^2}}{\alpha_j^*\|\Psi\|_{L^{p+1}}^{p+1}}\left\langle f, \Psi_j\right\rangle_{H^1}\right)\notag\\
&=&\sum_{j=1}^{\nu}2\left\langle \mathcal{L}_j(\overline{\gamma}_{**}^{\perp})-p(p-1)\left(\Psi_j^*\right)^{p-2}\mathcal{V}_j^2, \overline{\gamma}_{**}^{\perp}\right\rangle_{L^2}\notag\\
&&-\sum_{j=1}^{\nu}\left(\frac{p^2(p-1)(\alpha_j^*)^{p-3}\left(\left\langle\Psi_j^{p-1},\mathcal{V}_j^2\right\rangle_{L^2}\right)^2}{\|\Psi\|_{L^{p+1}}^{p+1}}+\frac{p(p-1)(p-2)}{6}\left\langle\left(\Psi_j^*\right)^{p-3}, \mathcal{V}_j^4\right\rangle_{L^2}\right)\notag\\
&&+o(\beta_*^4)+\mathcal{O}\left(\beta_*\|f\|_{H^{-1}}+\|f\|_{H^{-1}}^2\right).
\end{eqnarray}
The conclusion then follows from applying the estimates in \cite[Section~4.3]{FP2024} and the orthogonal conditions of $\overline{\gamma}_{**}^{\perp}$ given in \eqref{eqn5512} into \eqref{eqn0030}.
\end{proof}

\vskip0.12in

We are now ready to given the proof of $(a)$ of Theorem~\ref{thmn0001}.

\noindent\textbf{Proof of $(a)$ of Theorem~\ref{thmn0001}:} The conclusions for $\nu\geq2$ follows immediately from Lemma~\ref{lem0002} and Propositions~\ref{propn0001}, \ref{propn0002}, \ref{propn0003}, \ref{propn0004} and \ref{propn0005}.  For $\nu=1$, there is no interaction between bubbles, that is, we have $Q=0$.  Thus, the conclusion for $\nu=1$ follows from Lemma~\ref{lem0002} and Propositions~\ref{propn0001}, \ref{propn0003}, \ref{propn0004} and \ref{propn0005}
\hfill$\Box$

\section{Optimal example and proof of $(b)$ of Theorem~\ref{thmn0001}}
Let $R>0$ be a sufficiently large parameter and $\beta>0$ is a sufficiently small parameter.  We shall use the function, given by
\begin{eqnarray*}
v=\Psi+\Psi_{R}+\beta(w_d+w_{R,d}):=\Gamma_R+\beta\Phi_R
\end{eqnarray*}
to construct a optimal example of the stability stated in Theorem~\ref{thmn0001}, where
$\Psi_R=\Psi(t-R)$ and, as above, $w_{d}=\Psi^{\frac{p+1}{2}}\theta_{d}$ and $w_{R,d}=w_d(t-R)$.  It is easy to see that
\begin{eqnarray*}
\frac32\left(S_{FS}^{-1}\right)^{\frac{p+1}{p-1}}<\|v\|_{H^{1}}^2<\frac52\left(S_{FS}^{-1}\right)^{\frac{p+1}{p-1}}.
\end{eqnarray*}
Moreover, since $\Psi(t)$ is the unique positive solution of \eqref{eq0006} for $d\geq2$, $a<0$ and $b=b_{FS}(a)$, by Lemmas~\ref{lem0001} and \ref{lem0004},
\begin{eqnarray}\label{eqqnew0001}
\Xi&:=&-\Delta_{\theta}v-\partial_t^2v+\Lambda_{FS}v-v^{p}\notag\\
&=&\Psi^p+\Psi_R^p+p\beta\left(\Psi^{p-1}w_d+\Psi_R^{p-1}w_{R,d}\right)-\left(\Gamma_R+\beta\Phi_R\right)^p\notag\\
&=&\Psi^p+\Psi_R^p-\Gamma_R^p+p\beta\left(\left(\Psi^{p-1}-\Gamma_R^{p-1}\right)w_d+\left(\Psi_R^{p-1}-\Gamma_R^{p-1}\right)w_{R,d}\right)\notag\\
&&-\left(A_p\beta^2\Psi^{p-2}w_d^2+B_p\beta^3\Psi^{p-3}w_d^3\right)\chi_{\mathcal{B}}-\left(\beta^2\Gamma_R^{p-3}\Phi_R^{2}\left(A_p\Gamma_R+B_p\beta\Phi_R\right)\right)\chi_{\mathcal{C}\backslash\left(\mathcal{B}\cup\mathcal{B}_R\right)}\notag\\
&&-\left(A_p\beta^2\Psi_R^{p-2}w_{R,d}^2+B_p\beta^3\Psi_R^{p-3}w_{R,d}^3\right)\chi_{\mathcal{B}_R}+\Xi_{rem}
\end{eqnarray}
where $A_p=\frac{p(p-1)}{2}$, $B_p=\frac{p(p-1)(p-2)}{6}$,
\begin{eqnarray*}
\mathcal{B}=\left[-\frac{R}{2}, \frac{R}{2}\right]\times\mathbb{S}^{d-1},\quad\mathcal{B}_R=\left[\frac{R}{2}, \frac{3R}{2}\right]\times\mathbb{S}^{d-1}
\end{eqnarray*}
and
\begin{eqnarray*}
\Xi_{rem}=\mathcal{O}\left(\beta^2\left(\Psi^{2(p-1)}\Psi_R\chi_{\mathcal{B}}+\Psi_R^{2(p-1)}\Psi\chi_{\mathcal{B}_R}\right)+\beta^4\Gamma_R^4\right).
\end{eqnarray*}
We denote
\begin{eqnarray}\label{eqqnew0005}
\Xi_1&=&\left(\Gamma_R^p-\Psi^p-\Psi_R^p\right)+p\beta\left(\left(\Gamma_R^{p-1}-\Psi^{p-1}\right)w_d+\left(\Gamma_R^{p-1}-\Psi_R^{p-1}\right)w_{R,d}\right)\notag\\
&:=&\Xi_{1,1}+\beta\Xi_{1,2}
\end{eqnarray}
and
\begin{eqnarray}\label{eqqnew0006}
\Xi_2&=&A_p\beta^2\left(\Psi^{p-2}w_{d}^2\chi_{\mathcal{B}}+\Psi_R^{p-2}w_{R,d}^2\chi_{\mathcal{B}_R}+\Gamma_R^{p-2}\Phi_R^{2}\chi_{\mathcal{C}\backslash\left(\mathcal{B}\cup\mathcal{B}_R\right)}\right)\notag\\
&&+B_p\beta^3\left(\Psi^{p-3}w_d^3\chi_{\mathcal{B}}+\Psi_R^{p-3}w_{R,d}^3\chi_{\mathcal{B}_R}+\Gamma_R^{p-3}\Phi_R^{3}\chi_{\mathcal{C}\backslash\left(\mathcal{B}\cup\mathcal{B}_R\right)}\right)\notag\\
&:=&\beta^2\Xi_{2,1}+\beta^3\Xi_{2,2}.
\end{eqnarray}
Applying Lemmas~\ref{lem0006}, \ref{lem0008} and \ref{lem0009}, we immediately have the following.
\begin{lemma}\label{lemq1001}
Let $d\geq2$, $a<0$ and $b=b_{FS}(a)$.  Then the following equation
\begin{eqnarray}\label{eqqnew1002}
\left\{\aligned&-\Delta_{\theta}\varrho_{i,j}-\partial_t^2\varrho_{i,j}+\Lambda_{FS}\varrho_{i,j}-p\Gamma_R^{p-1}\varrho_{i,j}=\Xi_{i,j}+\vartheta_{i,j},\quad \text{in }\mathcal{C},\\
&\langle \partial_t\Psi, \varrho_{i,j}\rangle=\langle \partial_t\Psi_R, \varrho_{i,j}\rangle=\langle w_{l}, \varrho_{i,j}\rangle=\langle w_{R,l}, \varrho_{i,j}\rangle=0\text{ for all }1\leq l\leq d,\endaligned\right.
\end{eqnarray}
is uniquely solvable, where $\Xi_{i,j}$ is given by \eqref{eqqnew0005} and \eqref{eqqnew0006}, and
\begin{eqnarray}\label{eqqnew0003}
\vartheta_{i,j}=\Psi^{p-1}\left(c_{i,j}\partial_t\Psi+\sum_{l=1}^{d}\varsigma_{i,j,l}w_{l}\right)+\Psi_R^{p-1}\left(c_{R,i,j}\partial_t\Psi_R+\sum_{l=1}^{d}\varsigma_{R,i,j,l}w_{R,l}\right)
\end{eqnarray}
with $c_{i,j}, c_{R,i,j}$ and $\{\varsigma_{i,j,l}\}$ and $\{\varsigma_{R,i,j,l}\}$ being chosen such that the right hand side of the equation~\eqref{eqqnew1002} is orthogonal to $\Psi^{p-1}\partial_t\Psi$, $\Psi_R^{p-1}\partial_t\Psi_R$, $\left\{\Psi^{p-1}w_{l}\right\}$ and $\left\{\Psi_R^{p-1}w_{R,l}\right\}$ in $L^2(\mathcal{C})$.  Moreover, $\varrho_{1,1}$ is even on $\mathbb{S}^{d-1}$ and $\varrho_{1,2}$ is odd on $\mathbb{S}^{d-1}$ with
\begin{eqnarray*}
1\gtrsim\left\{\aligned
&\|\varrho_{1,1}\|_{\sharp},\quad p\geq3,\\
&\|\varrho_{1,1}\|_{\natural,1,*},\quad 1<p<3,
\endaligned\right.
\quad\text{and}\quad
1\gtrsim\left\{\aligned
&\|\varrho_{1,2}\|_{\sharp},\quad p\geq\frac{7}{3},\\
&\|\varrho_{1,2}\|_{\natural,2,*},\quad 1<p<\frac{7}{3},
\endaligned\right.
\end{eqnarray*}
while, $\varrho_{2,1}$ is even on $\mathbb{S}^{d-1}$ and $\varrho_{2,2}$ is odd on $\mathbb{S}^{d-1}$ with
\begin{eqnarray*}
1\gtrsim\sup_{(t,\theta)\in\mathcal{C}}\frac{|\varrho_{2,1}|+|\varrho_{2,2}|}{\Psi^{1-\sigma}+\Psi_R^{1-\sigma}}
\end{eqnarray*}
and $\sigma>0$ is chosen to satisfy $\beta^2\leq\frac{1}{8}\Gamma_R^\sigma$ in $\overline{\mathcal{B}}_{*}=[-R,2R]\times\mathbb{S}^{d-1}$.
\end{lemma}

\vskip0.12in

Let $\varrho_*=\varrho_{1,1}+\beta\varrho_{1,2}+\beta^2\varrho_{2,1}+\beta^3\varrho_{2,2}$.
Then we have the following.
\begin{proposition}\label{propq1001}
Let $d\geq2$, $a<0$ and $b=b_{FS}(a)$.  Then
\begin{eqnarray}\label{eqqnew0020}
\|\varrho_*\|\sim\beta^2+\left\{\aligned
&Q_R,\quad p>2,\\
&Q_R\left|\log Q_R\right|,\quad p=2,\\
&Q_R^{\frac{p}{2}},\quad 1<p<2,
\endaligned
\right.
\end{eqnarray}
where $Q_R=e^{-\sqrt{\Lambda_{FS}}R}$.
\end{proposition}
\begin{proof}
By using the test functions
\begin{eqnarray*}
\widetilde{\varrho}_R(t)=\left\{\aligned&1,\quad \frac{R}{2}-3\leq t\leq \frac{R}{2}-2,\\
&0,\quad t\leq\frac{R}{2}-4 \text{ or }t\geq \frac{R}{2}-1.\endaligned\right.
\end{eqnarray*}
for $1<p<2$,
\begin{eqnarray*}
\widetilde{\varrho}_R(t)=\left\{\aligned&1,\quad \frac{R}{4}\leq t\leq \frac{R}{2}-2,\\
&0,\quad t\leq \frac{R}{4}-1 \text{ or }t\geq \frac{R}{2}-1.\endaligned\right.
\end{eqnarray*}
for $p=2$ and
\begin{eqnarray*}
\widehat{\varrho}_R(t)=\left\{\aligned&1,\quad T_*\leq t\leq T_*+1,\\
&0,\quad t\leq T_*-1 \text{ or }t\geq T_*+2,\endaligned\right.
\end{eqnarray*}
with $T_*>0$ sufficiently large for $p>2$ to \eqref{eqqnew1002}, as that in the proof of \cite[Proposition~6.2]{WW2022}, we can show that
\begin{eqnarray*}
\|\varrho_{1,1}\|\gtrsim\left\{\aligned
&Q_R,\quad p>2,\\
&Q_R\left|\log Q_R\right|,\quad p=2,\\
&Q_R^{\frac{p}{2}},\quad 1<p<2,
\endaligned
\right.
\end{eqnarray*}
which, together with \eqref{eqqnew0005}, Lemma~\ref{lemq1001} and multiplying \eqref{eqqnew1002} of $\varrho_{1,1}$ with $\varrho_{1,1}$ on both sides and integrating by parts,  implies that
\begin{eqnarray*}
\|\varrho_{1,1}\|\sim\left\{\aligned
&Q_R,\quad p>2,\\
&Q_R\left|\log Q_R\right|,\quad p=2,\\
&Q_R^{\frac{p}{2}},\quad 1<p<2.
\endaligned
\right.
\end{eqnarray*}
Similar to \eqref{eqqnew0009}, by \eqref{eqqnew0005}, Lemma~\ref{lemq1001} and multiplying \eqref{eqqnew1002} of $\varrho_{1,2}$ with $\varrho_{1,2}$ on both sides and integrating by parts, we also have
\begin{eqnarray*}
\|\varrho_{1,2}\|\lesssim\left\{\aligned
&Q_R,\quad p>2,\\
&Q_R\left|\log Q_R\right|,\quad p=2,\\
&Q_R^{\frac{p}{2}},\quad 1<p<2.
\endaligned
\right.
\end{eqnarray*}
By \eqref{eqqnew0006}, Lemma~\ref{lemq1001} and multiplying \eqref{eqqnew1002} of $\varrho_{2,2}$ with $\varrho_{2,2}$ on both sides and integrating by parts, it is also easy to see that $\|\varrho_{2,2}\|\lesssim1$.  It remains to estimate $\|\varrho_{2,1}\|$.  By \eqref{eqqnew0006}, Lemma~\ref{lemq1001} and multiplying \eqref{eqqnew1002} of $\varrho_{2,1}$ with $\varrho_{2,1}$ on both sides and integrating by parts, it is also easy to see that $\|\varrho_{2,1}\|\lesssim1$.  For the lower bound of $\|\varrho_{2,1}\|$, we recall that the spherical harmonics on $\mathbb{S}^{d-1}$, denoted by $\{\mathcal{Y}_{j,l}\}$ with $j=0,1,2,\cdots$ and $l=1,2,\cdots,l_j$ for some $l_j\in\bbn$, form an orthogonal basic of $L^2(\mathbb{S}^{d-1})$ with span$_{1\leq l\leq l_j}\{\mathcal{Y}_{j,l}\}$ forming the eigenspace of the $j$th eigenvalue of $-\Delta_\theta$ on $L^2\left(\mathbb{S}^{d-1}\right)$, where $\Delta_{\theta}$ is the Laplace-Beltrami operator on $\mathbb{S}^{d-1}$.  Moreover, it is well known that the eigenvalues of $-\Delta_\theta$ on $L^2\left(\mathbb{S}^{d-1}\right)$ are given by $j(j+d-2)$.
The first eigenvalue $0$ is simple with eigenfunctions $\mathcal{Y}_{0,1}=1$, the eigenfunctions of the second eigenvalue $d-1$ are precisely $\mathcal{Y}_{1,l}=\theta_l$ for $1\leq l\leq d$.  It is also well known that $\mathcal{Y}_{2,d}=\theta_d^2-\frac{1}{d}$ is a spherical harmonic on $\mathbb{S}^{d-1}$ with degree $2$ (cf. \cite[(2.6)]{S2007} or \cite[(4.9)]{FP2024}).  Now, by \eqref{eqqnew0006}, Lemma~\ref{lemq1001} and multiplying \eqref{eqqnew1002} of $\varrho_{2,1}$ with $\mathcal{Y}_{2,d}$ on both sides and integrating by parts,we have
\begin{eqnarray*}
\|\varrho_{2,1}\|\gtrsim\left\langle \varrho_{2,1}, \mathcal{Y}_{2,d}\right\rangle-p\left\langle \Gamma_{R}\varrho_{2,1}, \mathcal{Y}_{2,d}\right\rangle_{L^2}=\left\langle \Xi_{2,1}, \mathcal{Y}_{2,d}\right\rangle_{L^2}\gtrsim\left\|\Psi^{p-2}\mathcal{Y}_{2,d}^2\right\|_{L^1}.
\end{eqnarray*}
Thus, by $\varrho_*=\varrho_{1,1}+\beta\varrho_{1,2}+\beta^2\varrho_{2,1}+\beta^3\varrho_{2,2}$, we have the desired estimate of $\|\varrho\|$ given by \eqref{eqqnew0020}.
\end{proof}

\vskip0.12in

We define
\begin{eqnarray}\label{eqqnew0020}
f_*:=-\Delta_{\theta}(\varrho_*+v)-\partial_t^2(\varrho_*+v)+\Lambda_{FS}(\varrho_*+v)-(v+\varrho_*)^p.
\end{eqnarray}
Then by \eqref{eqqnew0001} and Lemma~\ref{lemq1001},
\begin{eqnarray}\label{eqqnew1230}
f_*&=&\left(-\Delta_{\theta}\varrho_*-\partial_t^2\varrho_*+\Lambda_{FS}\varrho_*-p\Gamma_R^{p-1}\varrho_*\right)+\Psi^p+\Psi_R^p+p\beta\left(\Psi^{p-1}w_d+\Psi_R^{p-1}w_{R,d}\right)\notag\\
&&+p\Gamma_R^{p-1}\varrho_*-(\Gamma_R+\beta\Phi_R+\varrho_*)^p\notag\\
&=&\vartheta_{1,1}+\beta\vartheta_{1,2}+\beta^2\vartheta_{2,1}+\beta^3\vartheta_{2,2}+\Xi_{1,1}+\beta\Xi_{1,2}+\beta^2\Xi_{2,1}+\beta^3\Xi_{2,2}\notag\\
&&+\Psi^p+\Psi_R^p-\Gamma_R^p+p\beta\left(\left(\Psi^{p-1}-\Gamma_R\right)w_d+\left(\Psi_R^{p-1}-\Gamma_R\right)w_{R,d}\right)\notag\\
&&+p\Gamma_R^{p-1}(\varrho_*+\beta\Phi_R)+\Gamma_R^p-(\Gamma_R+\beta\Phi_R+\varrho_*)^p\notag\\
&=&\vartheta_{1,1}+\beta\vartheta_{1,2}+\beta^2\vartheta_{2,1}+\beta^3\vartheta_{2,2}+\beta^2\Xi_{2,1}+\beta^3\Xi_{2,2}-\mathcal{N}_{\varrho_*},
\end{eqnarray}
where
$\vartheta_{i,j}$ is given by \eqref{eqqnew0003} and by Lemmas~\ref{lem0003}, \ref{lem0004} and \ref{lemn0001},
\begin{eqnarray*}
\mathcal{N}_{\varrho_*}
&=&A_p\Gamma_P^{p-2}\left((\beta\Phi_R)^2+2\beta\Phi_R\varrho_*+\varrho_*^2\right)\\
&&+B_p\Gamma_P^{p-3}\left((\beta\Phi_R)^3+3(\beta\Phi_R)^2\varrho_*+3(\beta\Phi_R)\varrho_*^2+\varrho_*^3\right)\\
&&+\mathcal{O}\left(\Gamma_R^{p-4\sigma}\left(\beta+Q_R^{\frac{2\wedge p}{2}}\left|\log Q_R\right|\right)^4+\left(\beta^2\Gamma_R^{1-\sigma}\right)^{1+\ve}\chi_{\mathcal{C}\backslash\overline{\mathcal{B}}_{*}}\right)\\
&=&\beta^2\Xi_{2,1}+\beta^3\Xi_{2,2}+\mathcal{N}_{\varrho_*,rem},
\end{eqnarray*}
where
\begin{eqnarray*}
\mathcal{N}_{\varrho_*,rem}&=&2A_p\beta\left(\Psi^{p-2}w_{d}\chi_{\mathcal{B}}+\Psi_R^{p-2}w_{R,d}\chi_{\mathcal{B}_R}+\Gamma_R^{p-2}\Phi_R\chi_{\mathcal{C}\backslash\left(\mathcal{B}\cup\mathcal{B}_R\right)}\right)\varrho_*\notag\\
&&+3B_p\beta^2\left(\Psi^{p-3}w_d^2\chi_{\mathcal{B}}+\Psi_R^{p-3}w_{R,d}^2\chi_{\mathcal{B}_R}\right)\varrho_*\notag\\
&&+2\beta A_p\left((\Gamma_R^{p-2}\Phi_R-\Psi^{p-2}w_{d})\chi_{\mathcal{B}}+(\Gamma_R^{p-2}\Phi_R-\Psi_{R}^{p-2}w_{R,d}\chi_{\mathcal{B}_R})\right)\varrho_*\\
&&+\mathcal{O}\left(\beta^2\left(\left(\Psi^{2p-2}\Psi_R+\Psi^{2p-3}\Psi_R\varrho_*\right)\chi_{\mathcal{B}}+\left(\Psi_R^{2p-2}\Psi+\Psi_R^{2p-3}\Psi\varrho_*\right)\chi_{\mathcal{B}_R}\right)\right)\\
&&+\mathcal{O}\left(\Gamma_R^{p-4\sigma}\left(\beta+Q_R^{\frac{2\wedge p}{2}}\left|\log Q_R\right|\right)^4+\left(\beta^2\Gamma_R^{1-\sigma}\right)^{1+\ve}\chi_{\mathcal{C}\backslash\overline{\mathcal{B}}_{*}}\right).
\end{eqnarray*}
By Lemmas~\ref{lem0013} and \ref{lemq1001}, we immediately have the following.
\begin{lemma}\label{lemq1002}
Let $d\geq2$, $a<0$ and $b=b_{FS}(a)$.  Then the following equation,
\begin{eqnarray}\label{eqqnew1102}
\left\{\aligned&-\Delta_{\theta}\varrho_{1,1,*}-\partial_t^2\varrho_{1,1,*}+\Lambda_{FS}\varrho_{1,1,*}-p\Gamma_R^{p-1}\varrho_{1,1,*}=\Xi_{1,1,*}+\vartheta_{1,1,*},\quad \text{in }\mathcal{C},\\
&\langle \partial_t\Psi, \varrho_{1,1,*}\rangle=\langle \partial_t\Psi_R, \varrho_{1,1,*}\rangle=\langle w_{l}, \varrho_{1,1,*}\rangle=\langle w_{R,l}, \varrho_{1,1,*}\rangle=0\text{ for all }1\leq l\leq d,\endaligned\right.
\end{eqnarray}
is uniquely solvable,  where $\Xi_{1,1,*}=2A_p\Gamma_R^{p-2}\Phi_R\varrho_{1,1}$ and
and
\begin{eqnarray*}
\vartheta_{1,1,*}=\Psi^{p-1}\left(c_{1,1,*}\partial_t\Psi+\sum_{l=1}^{d}\varsigma_{1,1,*,l}w_{l}\right)+\Psi_R^{p-1}\left(c_{R,1,1,*}\partial_t\Psi_R+\sum_{l=1}^{d}\varsigma_{R,1,1,*,l}w_{R,l}\right)
\end{eqnarray*}
with $c_{1,1,*}, c_{R,1,1,*}$ and $\{\varsigma_{1,1,*,l}\}$ and $\{\varsigma_{R,1,1,*,l}\}$ being chosen such that the right hand side of the equation~\eqref{eqqnew1102} is orthogonal to $\Psi^{p-1}\partial_t\Psi$, $\Psi_R^{p-1}\partial_t\Psi_R$, $\left\{\Psi^{p-1}w_{l}\right\}$ and $\left\{\Psi_R^{p-1}w_{R,l}\right\}$ in $L^2(\mathcal{C})$.  Moreover, $\varrho_{1,1,*}$ is odd on $\mathbb{S}^{d-1}$ with
\begin{eqnarray*}
1\gtrsim\left\{\aligned
&\|\varrho_{1,1,*}\|_{\sharp},\quad p\geq\frac{7}{3},\\
&\|\varrho_{1,1,*}\|_{\natural,2,*},\quad 1<p<\frac{7}{3}.
\endaligned\right.
\end{eqnarray*}
\end{lemma}

\vskip0.12in

Let $\varrho=\varrho_{*}+\varrho_{1,1,*}$ and define
\begin{eqnarray}\label{eqqnew1020}
f:=-\Delta_{\theta}(\varrho+v)-\partial_t^2(\varrho+v)+\Lambda_{FS}(\varrho+v)-(v+\varrho)^p.
\end{eqnarray}
Then by \eqref{eqqnew1230} and Lemma~\ref{lemq1002},
\begin{eqnarray*}
f=\vartheta_{1,1}+\beta(\vartheta_{1,2}+\vartheta_{1,1,*})+\beta^2\vartheta_{2,1}+\beta^3\vartheta_{2,2}-\mathcal{N}_{\varrho_*,rem,1},
\end{eqnarray*}
where
\begin{eqnarray*}
\mathcal{N}_{\varrho_*,rem,1}&=&3B_p\beta^2\left(\Psi^{p-3}w_d^2\chi_{\mathcal{B}}+\Psi_R^{p-3}w_{R,d}^2\chi_{\mathcal{B}_R}\right)\varrho_*\notag\\
&&+\mathcal{O}\left(\beta^2\left(\left(\Psi^{2p-2}\Psi_R+\Psi^{2p-3}\Psi_R\varrho_*\right)\chi_{\mathcal{B}}+\left(\Psi_R^{2p-2}\Psi+\Psi_R^{2p-3}\Psi\varrho_*\right)\chi_{\mathcal{B}_R}\right)\right)\\
&&+\mathcal{O}\left(\Gamma_R^{p-4\sigma}\left(\beta+Q_R^{\frac{2\wedge p}{2}}\left|\log Q_R\right|\right)^4+\left(\beta^2\Gamma_R^{1-\sigma}\right)^{1+\ve}\chi_{\mathcal{C}\backslash\overline{\mathcal{B}}_{*}}\right).
\end{eqnarray*}
\begin{proposition}\label{propq0001}
Let $d\geq2$, $a<0$ and $b=b_{FS}(a)$.  Then
\begin{eqnarray*}
\|f\|_{H^{-1}}\sim\beta^3+Q_R,
\end{eqnarray*}
where $Q_R=e^{-\sqrt{\Lambda_{FS}}R}$.
\end{proposition}
\begin{proof}
For the sake of simplicity, we redenote $\vartheta_{1,2}+\vartheta_{1,1,*}$ by $\vartheta_{1,2}$, $c_{1,2}+c_{1,1,*}$ by $c_{1,2}$, $c_{R,1,2}+c_{R,1,1,*}$ by $c_{R,1,2}$, $\varsigma_{1,2,l}+\varsigma_{1,1,*,l}$ by $\varsigma_{1,2,l}$ and $\varsigma_{R,1,2,l}+\varsigma_{R,1,1,*,l}$ by $\varsigma_{R,1,2,l}$.
As in the proof of Lemma~\ref{lem0012}, by Lemma~\ref{lemq1001}, the orthogonality of $\Psi^{p-1}\partial_t\Psi$ and $\left\{\Psi^{p-1}w_{l}\right\}$ in $L^2(\mathcal{C})$ and the oddness of $w_d$ on $\mathbb{S}^{d-1}$,
\begin{eqnarray*}
-\left\langle \Psi^{p-1}\partial_t\Psi, \Xi_{i,j}\right\rangle_{L^2}=\left\|\Psi^{p-1}\partial_t\Psi\right\|_{L^2}^2c_{i,j}+\left\langle \Psi^{p-1}\partial_t\Psi, \Psi_R^{p-1}\partial_t\Psi_R\right\rangle_{L^2}c_{R,i,j}
\end{eqnarray*}
and
\begin{eqnarray*}
-\left\langle \Psi_R^{p-1}\partial_t\Psi_R, \Xi_{i,j}\right\rangle_{L^2}=\left\|\Psi^{p-1}\partial_t\Psi\right\|_{L^2}^2c_{R,i,j}+\left\langle \Psi^{p-1}\partial_t\Psi, \Psi_R^{p-1}\partial_t\Psi_R\right\rangle_{L^2}c_{i,j}
\end{eqnarray*}
while for all $1\leq l\leq d$,
\begin{eqnarray*}
-\left\langle \Psi^{p-1}w_{j}, \Xi_{i,j}\right\rangle_{L^2}=\sum_{l=1}^{d}\left(\left\langle \Psi^{p-1}w_j, \Psi^{p-1}w_{l}\right\rangle_{L^2}\varsigma_{i,j,l}+\left\langle \Psi^{p-1}w_j, \Psi_R^{p-1}w_{R,i,j,l}\right\rangle_{L^2}\varsigma_{R,i,j,l}\right)
\end{eqnarray*}
and
\begin{eqnarray*}
-\left\langle \Psi_R^{p-1}w_{R,j}, \Xi_{i,j}\right\rangle_{L^2}=\sum_{l=1}^{d}\left(\left\langle \Psi_R^{p-1}w_{R,j}, \Psi^{p-1}w_{l}\right\rangle_{L^2}\varsigma_{i,j,l}+\left\langle \Psi_R^{p-1}w_{R,j}, \Psi_R^{p-1}w_{R,i,j,l}\right\rangle_{L^2}\varsigma_{R,i,j,l}\right).
\end{eqnarray*}
It follows from Lemma~\ref{lem0005} that
\begin{eqnarray*}
\left\{\aligned
&c_{i,j}=-B_*\left\langle \Psi^{p-1}\partial_t\Psi, \Xi_{i,j}\right\rangle_{L^2}+\mathcal{O}\left(Q_R^p\left|\log Q_R\right|\left\langle \Psi_R^{p-1}\partial_t\Psi_R, \Xi_{i,j}\right\rangle_{L^2}\right),\\
&c_{R,i,j}=-B_*\left\langle \Psi_R^{p-1}\partial_t\Psi_R, \Xi_{i,j}\right\rangle_{L^2}+\mathcal{O}\left(Q_R^p\left|\log Q_R\right|\left\langle \Psi^{p-1}\partial_t\Psi, \Xi_{i,j}\right\rangle_{L^2}\right)
\endaligned\right.
\end{eqnarray*}
and
\begin{eqnarray*}
\left\{\aligned
&\varsigma_{i,j,l}=-B_{**}\left\langle \Psi^{p-1}w_{j}, \Xi_{i,j}\right\rangle_{L^2}+\mathcal{O}\left(Q_R^{\frac{3p-1}{2}}\left|\log Q_R\right|\left\langle \Psi_R^{p-1}w_{R,j}, \Xi_{i,j}\right\rangle_{L^2}\right),\\
&\varsigma_{R,i,j,l}=-B_{**}\left\langle \Psi_R^{p-1}w_{R,j}, \Xi_{i,j}\right\rangle_{L^2}+\mathcal{O}\left(Q_R^{\frac{3p-1}{2}}\left|\log Q_R\right|\left\langle \Psi^{p-1}w_{j}, \Xi_{i,j}\right\rangle_{L^2}\right)
\endaligned\right.
\end{eqnarray*}
for all $1\leq l\leq d$, where $B_*=\left\|\Psi^{p-1}\partial_t\Psi\right\|_{L^2}^2$ and $B_{**}=\left\|\Psi^{p-1}w_d\right\|_{L^2}^2$.  Thus, by \eqref{eq0026} and Lemma~\ref{lem0005}, the oddness of $\partial_t\Psi$ in $\bbr$ and the oddness of $w_d$ on $\mathbb{S}^{d-1}$,
\begin{eqnarray}\label{eqqnew0023}
c_{1,1}\sim c_{R,1,1}\sim Q_R\quad\text{and}\quad \sum_{l=1}^{d}\left(\left|\varsigma_{1,1,l}\right|+\left|\varsigma_{R,1,1,l}\right|\right)=0.
\end{eqnarray}
Similarly, we also have
\begin{eqnarray}\label{eqqnew0024}
|c_{1,2}|+|c_{R,1,2}|=0\quad\text{and}\quad\sum_{l=1}^{d}\left(\left|\varsigma_{1,2,l}\right|+\left|\varsigma_{R,1,2,l}\right|\right)\lesssim Q_R.
\end{eqnarray}
Again, by \eqref{eq0026} and Lemma~\ref{lem0005}, the oddness of $\partial_t\Psi$ in $\bbr$ and the oddness of $w_d$ on $\mathbb{S}^{d-1}$, we have
\begin{eqnarray}\label{eqqnew0025}
c_{2,1}\sim c_{R,2,1}\sim Q_R^p,\quad \sum_{l=1}^{d}\left(\left|\varsigma_{2,1,l}\right|+\left|\varsigma_{R,2,1,l}\right|\right)=0
\end{eqnarray}
and
\begin{eqnarray}\label{eqqnew0026}
c_{2,2}=c_{R,2,2}=0,\quad \sum_{l=1}^{d}\left(\left|\varsigma_{2,2,l}\right|+\left|\varsigma_{R,2,2,l}\right|\right)\sim1
\end{eqnarray}
It follows that $\left\|\vartheta_{1,1}+\beta\vartheta_{1,2}+\beta^2\vartheta_{2,1}+\beta^3\vartheta_{2,2}\right\|_{L^2}^2\sim\beta^3+Q_R$.  On the other hand, by Lemmas~\ref{lem0005} and \ref{lemq1001},
\begin{eqnarray*}
\left\|\mathcal{N}_{\varrho_*,rem,1}\right\|_{L^2}=o(\beta^4+Q_R).
\end{eqnarray*}
Thus, we must have $\|f\|_{H^{-1}}\sim\beta^3+Q_R$.
\end{proof}

\vskip0.12in

We decompose $\varrho=\widetilde{\alpha}\Psi+\widetilde{\alpha}_R\Psi_R+\widetilde{\varrho}^{\perp}$ where $\widetilde{\alpha}$ and $\widetilde{\alpha}_R$ are chosen such that $\left\langle \Psi, \widetilde{\varrho}^{\perp}\right\rangle=0$ and $\left\langle \Psi_R, \widetilde{\varrho}^{\perp}\right\rangle=0$.  It follows from Lemma~\ref{lemq1001} that
\begin{eqnarray*}
\widetilde{\alpha}\|\Psi\|^2+\mathcal{O}\left(Q_R\right)\widetilde{\alpha}_R=\left\langle \Xi_{1,1}+\beta\Xi_{1,2}+\beta^2\Xi_{2,1}+\beta^3\Xi_{2,2}, \Psi\right\rangle_{L^2}
\end{eqnarray*}
and
\begin{eqnarray*}
\widetilde{\alpha}_R\|\Psi\|^2+\mathcal{O}\left(Q_R\right)\widetilde{\alpha}=\left\langle \Xi_{1,1}+\beta\Xi_{1,2}+\beta^2\Xi_{2,1}+\beta^3\Xi_{2,2}, \Psi_R\right\rangle_{L^2}.
\end{eqnarray*}
By \eqref{eqqnew0005} and \eqref{eqqnew0006},
\begin{eqnarray}\label{eqqnew0021}
\widetilde{\alpha}=(1+o(1))\widetilde{\alpha}=\left(\frac{p\int_{\mathcal{C}}\Psi^{p-1}d\mu}{\|\Psi\|^2}+o(1)\right)Q_R+\left(\frac{A_p\int_{\mathcal{C}}\Psi^{2p}d\mu}{\|\Psi\|^2}+o(1)\right)\beta^2.
\end{eqnarray}
\begin{proposition}\label{propq0003}
Let $d\geq2$, $a<0$ and $b=b_{FS}(a)$.  If $Q_R\lesssim \beta^3$ then $\left\|\beta\Phi_R+\widetilde{\varrho}^{\perp}\right\|\sim\left\|f\right\|_{H^{-1}}^{\frac13}$ as $\beta\to0$.
\end{proposition}
\begin{proof}
If $Q_R\lesssim \beta^3$ then by Propositions~\ref{propq1001} and \ref{propq0001} and \eqref{eqqnew0021}, we have $\left\|\beta\Phi_R+\widetilde{\varrho}^{\perp}\right\|\sim\beta\sim\left\|f\right\|_{H^{-1}}^{\frac13}$.
\end{proof}

\vskip0.12in

For the sake of simplicity, we denote $\widetilde{v}=v+\varrho$.  We shall decompose $\widetilde{v}$ as in lemma~\ref{lem0002} by considering the following variational problem:
\begin{eqnarray}\label{eqqnew0011}
\inf_{\overrightarrow{\alpha}_{\nu}\in\left(\bbr_+\right)^{2}, \overrightarrow{s}_{2}\in\bbr^\nu}\left\|\widetilde{v}-\sum_{j=1}^{2}\alpha_j\Psi_{s_j}\right\|^2.
\end{eqnarray}
Clearly, as \eqref{eqn0001}, the variational problem~\eqref{eqqnew0011} has minimizers, say $(\widetilde{\alpha}^*_{1}, \widetilde{\alpha}^*_{2}, \widetilde{s}^*_{1}, \widetilde{s}^*_{2})$, satisfying
\begin{eqnarray}\label{eqqnew1005}
\sum_{j=1}^2\left|\widetilde{\alpha}^*_{j}-1\right|\to0\quad\text{and}\quad\left|\widetilde{s}^*_{1}-\widetilde{s}^*_{2}\right|\to+\infty
\end{eqnarray}
as $R\to+\infty$ and $\beta\to0$.
\begin{proposition}\label{propq0002}
Let $d\geq2$, $a<0$ and $b=b_{FS}(a)$.  Then for $R>0$ sufficiently large and $\beta>0$ sufficiently small, the variational problem~\eqref{eqqnew0011} has a unique minimizer, say $(\widetilde{\alpha}^*_{1}, \widetilde{\alpha}^*_{2}, \widetilde{s}^*_{1}, \widetilde{s}^*_{2})$, satisfying
\begin{eqnarray*}
\widetilde{s}^*_{1}=\mathcal{O}\left(\left(\beta^2+Q_R\right)^2\right),\quad\widetilde{s}^*_{2}=R+\mathcal{O}\left(\left(\beta^2+Q_R\right)^2\right)
\end{eqnarray*}
and
\begin{eqnarray*}
\widetilde{\alpha}^*_{1}-1=\frac{\left\langle\rho, \Psi\right\rangle}{\|\Psi\|^2}+\mathcal{O}\left(\left(\beta^2+Q_R\right)^2\right),\quad \widetilde{\alpha}^*_{2}-1=\frac{\left\langle\rho, \Psi_R\right\rangle}{\|\Psi\|^2}+\mathcal{O}\left(\left(\beta^2+Q_R\right)^2\right).
\end{eqnarray*}
\end{proposition}
\begin{proof}
Since $(\widetilde{\alpha}^*_{1}, \widetilde{\alpha}^*_{2}, \widetilde{s}^*_{1}, \widetilde{s}^*_{2})$ is a minimizer of the variational problem~\eqref{eqqnew0011} and $\Psi$ and $\partial_t\Psi$ are solutions of \eqref{eq0006} and \eqref{eq0016}, respectively, we have
\begin{eqnarray}\label{eqqnew0017}
0=\left\langle \widetilde{v}-\sum_{i=1}^{2}\widetilde{\alpha}^*_{i}\Psi_{\widetilde{s}^*_{i}}, \Psi_{\widetilde{s}^*_{j}}\right\rangle=\left\langle \widetilde{v}-\sum_{i=1}^{2}\widetilde{\alpha}^*_{i}\Psi_{\widetilde{s}^*_{i}}, \Psi_{\widetilde{s}^*_{j}}^{p}\right\rangle_{L^2}
\end{eqnarray}
and
\begin{eqnarray}\label{eqqnew0016}
0=\left\langle \widetilde{v}-\sum_{i=1}^{2}\widetilde{\alpha}^*_{i}\Psi_{\widetilde{s}^*_{i}}, \partial_t\Psi_{\widetilde{s}^*_{j}}\right\rangle=\left\langle \widetilde{v}-\sum_{i=1}^{2}\widetilde{\alpha}^*_{i}\Psi_{\widetilde{s}^*_{i}}, p\Psi_{\widetilde{s}^*_{j}}^{p-1}\partial_t\Psi_{\widetilde{s}^*_{j}}\right\rangle_{L^2}
\end{eqnarray}
for all $j=1,2$.  By the oddness of $w_d$ on $\mathbb{S}^{d-1}$, the oddness of $\partial_t\Psi$ in $\bbr$, \eqref{eqqnew1005} and \eqref{eqqnew0016}, we have $\sum_{i=1}^{2}|\widetilde{\alpha}^*_{i}|\lesssim1$ and
\begin{eqnarray*}
\left\langle \Gamma_R+\varrho, p\Psi_{\widetilde{s}^*_{j}}^{p-1}\partial_t\Psi_{\widetilde{s}^*_{j}}\right\rangle_{L^2}=\mathcal{O}\left(\left\langle \Psi_{\widetilde{s}^*_{i}}, \Psi_{\widetilde{s}^*_{j}}^{p-1}\partial_t\Psi_{\widetilde{s}^*_{j}}\right\rangle_{L^2}\right)=\mathcal{O}\left(Q_R^{\frac{\left|\widetilde{s}^*_{1}-\widetilde{s}^*_{2}\right|}{R}}\right).
\end{eqnarray*}
Recall that $\varrho=\varrho_{1,1}+\beta\varrho_{1,2}+\beta^2\varrho_{2,1}+\beta^3\varrho_{2,2}$.  Thus, by Lemma~\ref{lemq1001} and the fact that $\partial_t\Psi$ solves \eqref{eq0016}, we have
\begin{eqnarray*}
\left\langle \varrho, p\Psi_{\widetilde{s}^*_j}^{p-1}\partial_t\Psi_{\widetilde{s}^*_j}\right\rangle_{L^2}&=&\left\langle \Xi_{1,1}+\vartheta_{1,1}+\beta(\Xi_{1,2}+\vartheta_{1,2})+\beta^2(\Xi_{2,1}+\vartheta_{2,1}), \partial_t\Psi_{\widetilde{s}^*_j}\right\rangle_{L^2}\\
&&+\left\langle \beta^3(\Xi_{2,2}+\vartheta_{2,2}), \partial_t\Psi_{\widetilde{s}^*_j}\right\rangle_{L^2}+\left\langle p\Gamma_R^{p-1}\varrho, \partial_t\Psi_{\widetilde{s}^*_j}\right\rangle_{L^2}.
\end{eqnarray*}
It follows from \eqref{eq0026}, \eqref{eqqnew0023}, \eqref{eqqnew0024}, \eqref{eqqnew0025}, \eqref{eqqnew0026}, the oddness of $w_d$ on $\mathbb{S}^{d-1}$, the oddness of $\partial_t\Psi$ in $\bbr$ and Lemmas~\ref{lem0005} and \ref{lemq1001} that
\begin{eqnarray}\label{eqqnew0018}
\left|\left\langle\varrho, \Psi_{\widetilde{s}^*_j}^{p-1}\partial_t\Psi_{s_j}\right\rangle_{L^2}\right|&\lesssim&\left\langle \left|\Gamma_R^{p-1}\varrho\right|, \Psi_{\widetilde{s}^*_j}\right\rangle_{L^2}+\left\langle \left|\Xi_{1,1}+\beta^2\Psi_R^{p-2}w_{R,d}^2\chi_{\mathcal{B}_R}\right|, \Psi_{\widetilde{s}^*_j}\right\rangle_{L^2}\notag\\
&&+(c_{1,1}+c_{R,1,1})+\beta^2(c_{2,1}+c_{R,2,1})\notag\\
&\lesssim&\beta^2+Q_R.
\end{eqnarray}
On the other hand, for every $s_j\leq\frac{R}{2}$, by Lemma~\ref{lem0005},
\begin{eqnarray}\label{eqqnew0019}
\left\langle\Gamma_R, \Psi_{s_j}^{p-1}\partial_t\Psi_{s_j}\right\rangle_{L^2}=\left\langle\Psi, \Psi_{s_j}^{p-1}\partial_t\Psi_{s_j}\right\rangle_{L^2}+\mathcal{O}\left(Q_R^{\frac{R-s_j}{R}}\right).
\end{eqnarray}
Note that by the evenness of $\Psi$ and the oddness of $\partial_t\Psi$ in $\bbr$, $\left\langle\Psi, \Psi_{s_j}^{p-1}\partial_t\Psi_{s_j}\right\rangle_{L^2}=0$ has a uniquely nondegenerate solution $s_j=0$ on $(-\infty, \frac{R}{2}]$.  Thus, by \eqref{eqqnew0018}, \eqref{eqqnew0019}, the symmetry of $\Gamma_R$ about $s=\frac{R}{2}$, for $R>0$ sufficiently large, the solution of \eqref{eqqnew0017} and \eqref{eqqnew0016} must satify
\begin{eqnarray}\label{eqqnew0031}
\widetilde{s}^*_{1}=\mathcal{O}\left(\beta^2+Q_R\right)\quad\text{and}\quad \widetilde{s}^*_{2}=R+\mathcal{O}\left(\beta^2+Q_R\right).
\end{eqnarray}
which, together with \eqref{eqqnew0017}, implies that
\begin{eqnarray*}
\left\langle \Gamma_R+\varrho, \Psi_{\widetilde{s}^*_{j}}^{p}\right\rangle_{L^2}=\widetilde{\alpha}^*_{j}\|\Psi\|^2+\mathcal{O}\left(\left\langle \Psi_{\widetilde{s}^*_{i}}, \Psi_{\widetilde{s}^*_{j}}^{p}\right\rangle_{L^2}\right)=\widetilde{\alpha}^*_{j}\|\Psi\|^2+\mathcal{O}\left(Q_R\right)
\end{eqnarray*}
for all $j=1,2$.  Similar to \eqref{eqqnew0018} and \eqref{eqqnew0019}, we have
\begin{eqnarray*}
\left\langle \Gamma_R+\varrho, \Psi_{\widetilde{s}^*_{j}}^{p}\right\rangle_{L^2}=\|\Psi\|^2+\mathcal{O}\left(\beta^2+Q_R\right).
\end{eqnarray*}
Thus, we also have
\begin{eqnarray}\label{eqqnew0030}
\widetilde{\alpha}^*_{j}=1+\mathcal{O}\left(\beta^2+Q_R\right).
\end{eqnarray}
Now, by \eqref{eqqnew0016} once more, the oddness of $\partial_t^3\Psi$, the Taylor expansion and the orthogonal conditions of $\rho$ given in \eqref{eqqnew1002}, we have
\begin{eqnarray*}
0&=&\left\langle \widetilde{v}-\sum_{i=1}^{2}\widetilde{\alpha}^*_{i}\Psi_{\widetilde{s}^*_{i}}, \partial_t\Psi_{\widetilde{s}^*_{j}}\right\rangle\notag\\
&=&\left\langle \Psi^p, \partial_t\Psi_{\widetilde{s}^*_{1}}\right\rangle_{L^2}+\left\langle \Psi_R-\widetilde{\alpha}^*_{2}\Psi_{\widetilde{s}^*_{2}}, p\Psi_{\widetilde{s}^*_{1}}^{p-1}\partial_t\Psi_{\widetilde{s}^*_{1}}\right\rangle_{L^2}+\left\langle \varrho, p\Psi_{\widetilde{s}^*_{1}}^{p-1}\partial_t\Psi_{\widetilde{s}^*_{1}}\right\rangle_{L^2}\notag\\
&=&-\left\langle \Psi^p, \partial_t^2\Psi\right\rangle_{L^2}\widetilde{s}^*_{1}+\mathcal{O}\left(\left(\widetilde{s}^*_{1}\right)^3\right)+(1-\widetilde{\alpha}^*_{2})\left\langle \Psi_{\widetilde{s}^*_{2}}, p\Psi_{\widetilde{s}^*_{1}}^{p-1}\partial_t\Psi_{\widetilde{s}^*_{1}}\right\rangle_{L^2}\notag\\
&&+\left\langle \partial_t\Psi_R, p\Psi_{\widetilde{s}^*_{1}}^{p-1}\partial_t\Psi_{\widetilde{s}^*_{1}}\right\rangle_{L^2}(\widetilde{s}^*_{2}-R)+\mathcal{O}\left(\left(\widetilde{s}^*_{2}-R\right)^2+\|\varrho\|_{L^\infty}\widetilde{s}^*_{1}\right),
\end{eqnarray*}
which, together with \eqref{eqqnew0031} and \eqref{eqqnew0030}, implies that $\widetilde{s}^*_{1}=\mathcal{O}\left(\left(\beta^2+Q_R\right)^2\right)$.  Similarly, we also have $\widetilde{s}^*_{2}=R+\mathcal{O}\left(\left(\beta^2+Q_R\right)^2\right)$.  Again, by \eqref{eqqnew0017}, \eqref{eqqnew0030} and the Taylor expansion,
\begin{eqnarray*}
\left\langle \Gamma_R+\varrho, \Psi_{\widetilde{s}^*_{1}}^{p}\right\rangle_{L^2}&=&\widetilde{\alpha}^*_{1}\|\Psi\|^2+\left\langle \Psi_{\widetilde{s}^*_{2}}, \Psi_{\widetilde{s}^*_{1}}^{p}\right\rangle_{L^2}+\mathcal{O}\left(\left(\beta^2+Q_R\right)^2\right)\\
&=&\widetilde{\alpha}^*_{1}\|\Psi\|^2+\left\langle \Psi_{R}, \Psi_{\widetilde{s}^*_{1}}^{p}\right\rangle_{L^2}+\mathcal{O}\left(\left(\beta^2+Q_R\right)^2\right),
\end{eqnarray*}
which, together with
\begin{eqnarray*}
\left\langle \Gamma_R+\varrho, \Psi_{\widetilde{s}^*_{1}}^{p}\right\rangle_{L^2}=\|\Psi\|^2+\left\langle \Psi_{R}, \Psi_{\widetilde{s}^*_{1}}^{p}\right\rangle_{L^2}+\left\langle\rho, \Psi_{\widetilde{s}^*_{1}}^{p}\right\rangle_{L^2},
\end{eqnarray*}
implies that $\widetilde{\alpha}^*_{1}-1=\frac{\left\langle\varrho, \Psi\right\rangle}{\|\Psi\|^2}+\mathcal{O}\left(\left(\beta^2+Q_R\right)^2\right)$.
Similarly, we also have $\widetilde{\alpha}^*_{2}-1=\frac{\left\langle\varrho, \Psi_R\right\rangle}{\|\Psi\|^2}+\mathcal{O}\left(\left(\beta^2+Q_R\right)^2\right)$.
\end{proof}

\vskip0.12in

Let $\widetilde{v}_\pm=\max\{\pm\widetilde{v}, 0\}$.  Then $\widetilde{v}=\widetilde{v}_+-\widetilde{v}_-$ and by \eqref{eqqnew0020},
\begin{eqnarray}\label{eqqnew1020}
-\Delta_{\theta}\widetilde{v}_+-\partial_t^2\widetilde{v}_++\Lambda_{FS}\widetilde{v}_+-\widetilde{v}_+^p=f+\mathcal{G}(\widetilde{v}_-):=f_{\widetilde{v}_+},
\end{eqnarray}
where $\mathcal{G}(\widetilde{v}_-)=-\Delta_{\theta}\widetilde{v}_--\partial_t^2\widetilde{v}_-+\Lambda_{FS}\widetilde{v}_--\widetilde{v}_-^p$.

\noindent\textbf{Proof of $(b)$ of Theorem~\ref{thmn0001}:}
Recall that we have the decomposition
\begin{eqnarray}\label{eqqnew0028}
\widetilde{v}=v+\widetilde{\alpha}\Psi+\widetilde{\alpha}_R\Psi_R+\widetilde{\varrho}^{\perp}
\end{eqnarray}
in $H^1(\mathcal{C})$,
where by the orthogonal conditions of $\widetilde{\varrho}^{\perp}$ and \eqref{eqqnew0021},
\begin{eqnarray*}
\left\langle\varrho, \Psi\right\rangle=\widetilde{\alpha}\|\Psi\|^2+\mathcal{O}\left(\left(\beta^2+Q_R\right)^2\right)\quad\text{and}\quad\left\langle\varrho, \Psi_R\right\rangle=\widetilde{\alpha}_R\|\Psi\|^2+\mathcal{O}\left(\left(\beta^2+Q_R\right)^2\right).
\end{eqnarray*}
It follows from Proposition~\ref{propq0002} that
\begin{eqnarray*}
\widetilde{\alpha}^*_{1}=1+\widetilde{\alpha}+\mathcal{O}\left(\left(\beta^2+Q_R\right)^2\right)\quad\text{and}\quad\widetilde{\alpha}^*_{2}=1+\widetilde{\alpha}_R+\mathcal{O}\left(\left(\beta^2+Q_R\right)^2\right),
\end{eqnarray*}
which, together with Proposition~\ref{propq0002} once more and the Taylor expansion, implies that
\begin{eqnarray}\label{eqqnew0029}
\widetilde{v}&=&\sum_{j=1}^2\widetilde{\alpha}^*_{j}\Psi_{\widetilde{s}^*_{j}}+\widetilde{\varrho}\notag\\
&=&\Gamma_R+\widetilde{\alpha}\Psi+\widetilde{\alpha}_R\Psi_R+\widetilde{\varrho}+\mathcal{O}\left(\left(\beta^2+Q_R\right)^2\right)
\end{eqnarray}
in $H^1(\mathcal{C})$.  By \eqref{eqqnew0028} and \eqref{eqqnew0029}, we have
\begin{eqnarray*}
\widetilde{\varrho}=\beta\Phi_R+\widetilde{\varrho}^{\perp}+\mathcal{O}\left(\left(\beta^2+Q_R\right)^2\right).
\end{eqnarray*}
Thus, by \eqref{eqqnew0011}, Propositions~\ref{propq0003} and \ref{propq0002}, we have
\begin{eqnarray}\label{eqqnew0032}
\inf_{\overrightarrow{\alpha}_{\nu}\in\left(\bbr_+\right)^{2}, \overrightarrow{s}_{2}\in\bbr^\nu}\left\|\widetilde{v}-\sum_{j=1}^{2}\alpha_j\Psi_{s_j}\right\|\sim\|f\|_{H^{-1}}^{\frac13}.
\end{eqnarray}
By Lemma~\ref{lemq1001}, we know that $|\beta\Phi_R+\varrho_{1,1}+\beta\varrho_{1,2}|\lesssim\Gamma_R$ in $\mathcal{C}$ for sufficiently small $\beta$ and sufficiently large $R$.  Thus, $0\leq\widetilde{v}_-\leq\beta^2\varrho_{2,1}+\beta^3\varrho_{2,2}|$ in $\mathcal{C}$.  It follows from \eqref{eqqnew0020}, Lemma~\ref{lemq1001} and Proposition~\ref{propq0001} that
\begin{eqnarray*}
\|\widetilde{v}_-\|^2\lesssim\left\langle f, \widetilde{v}_-\right\rangle_{L^2}=\mathcal{O}\left(\left(\beta^2+Q_R\right)^2\right),
\end{eqnarray*}
which, together with \eqref{eqqnew1020} and \eqref{eqqnew0032}, implies that $\widetilde{v}_+$ is the desired functions.
\hfill$\Box$

\begin{remark}\label{rmkn0001}
The optimal example of Theorem~\ref{thmn0001} in this section, given by $\widetilde{v}=v+\varrho$, precisely describes the relation between $\|f\|_{H^{-1}}$ and
\begin{eqnarray*}
\inf_{\overrightarrow{\alpha}_{\nu}\in\left(\bbr_+\right)^{2}, \overrightarrow{s}_{2}\in\bbr^\nu}\left\|\widetilde{v}-\sum_{j=1}^{2}\alpha_j\Psi_{s_j}\right\|.
\end{eqnarray*}
Indeed, we have $\|f\|_{H^{-1}}\sim\beta^3+Q_R$ and
\begin{eqnarray*}
\inf_{\overrightarrow{\alpha}_{\nu}\in\left(\bbr_+\right)^{2}, \overrightarrow{s}_{2}\in\bbr^\nu}\left\|\widetilde{v}-\sum_{j=1}^{2}\alpha_j\Psi_{s_j}\right\|\sim \beta+\left\{\aligned
&Q_R,\quad p>2,\\
&Q_R\left|\log Q_R\right|,\quad p=2,\\
&Q_R^{\frac{p}{2}},\quad 1<p<2.
\endaligned
\right.
\end{eqnarray*}
If the interaction of two bubbles is much smaller than their projections on nontrivial kernel, that is, $\beta^3\gtrsim Q_R$, then we have
\begin{eqnarray*}
\inf_{\overrightarrow{\alpha}_{\nu}\in\left(\bbr_+\right)^{2}, \overrightarrow{s}_{2}\in\bbr^\nu}\left\|\widetilde{v}-\sum_{j=1}^{2}\alpha_j\Psi_{s_j}\right\|\sim\|f\|_{H^{-1}}^{\frac13}.
\end{eqnarray*}
If the interaction of two bubbles is much large than their projections on nontrivial kernel, that is,
\begin{eqnarray*}
\beta\lesssim\left\{\aligned
&Q_R,\quad p>2,\\
&Q_R\left|\log Q_R\right|,\quad p=2,\\
&Q_R^{\frac{p}{2}},\quad 1<p<2,
\endaligned
\right.
\end{eqnarray*}
then we have
\begin{eqnarray*}
\inf_{\overrightarrow{\alpha}_{\nu}\in\left(\bbr_+\right)^{2}, \overrightarrow{s}_{2}\in\bbr^\nu}\left\|\widetilde{v}-\sum_{j=1}^{2}\alpha_j\Psi_{s_j}\right\|\sim\left\{\aligned
&\|f\|_{H^{-1}},\quad p>2,\\
&\|f\|_{H^{-1}}\left|\log \|f\|_{H^{-1}}\right|,\quad p=2,\\
&\|f\|_{H^{-1}}^{\frac{p}{2}},\quad 1<p<2.
\endaligned
\right.
\end{eqnarray*}
If the interaction of two bubbles is somehow comparable with their projections on nontrivial kernel, that is $\beta^3\lesssim Q_R$ and
\begin{eqnarray*}
\beta\gtrsim\left\{\aligned
&Q_R,\quad p>2,\\
&Q_R\left|\log Q_R\right|,\quad p=2,\\
&Q_R^{\frac{p}{2}},\quad 1<p<2,
\endaligned
\right.
\end{eqnarray*}
then we have
\begin{eqnarray*}
\inf_{\overrightarrow{\alpha}_{\nu}\in\left(\bbr_+\right)^{2}, \overrightarrow{s}_{2}\in\bbr^\nu}\left\|\widetilde{v}-\sum_{j=1}^{2}\alpha_j\Psi_{s_j}\right\|\sim&\|f\|_{H^{-1}}^{t}
\end{eqnarray*}
with
\begin{eqnarray*}
\left\{\aligned
&\frac13\geq t\leq1,\quad p>2,\\
&\frac13\leq t\leq 1+o(1),\quad p=2,\\
&\frac13\leq t\leq\frac{p}{2},\quad 1<p<2.
\endaligned
\right.
\end{eqnarray*}
\end{remark}


\begin{thebibliography}{100}
\bibitem{ADM2022}
W. Ao, A. DelaTorre, M. del Mar Gonz\'alez, Symmetry and symmetry breaking
for the fractional Caffarelli-Kohn-Nirenberg inequality, {\it J. Funct. Anal.,}
{\bf282} (2022), 109438.

\bibitem{A2023}
S. Aryan, Stability of Hardy Littlewood Sobolev inequality under bubbling, {\it Calc. Var.,}  {\bf62} (2023), Article No. 23.

\bibitem{A1976}
T. Aubin, Probl\`emes isop\'erim\'etriques de Sobolev.  {\it J. Differential Geometry,} {\bf11} (1976), 573--598.

\bibitem{BWW2003}
T. Bartsch, T. Weth, M. Willem, A Sobolev inequality with remainder term and critical equations on domains with topology for the polyharmonic operator, {\it Calc. Var.,} {\bf18} (2003), 253--268.

\bibitem{BE1991}
G. Bianchi, H. Egnell, A note on the Sobolev inequality.  {\it J. Funct. Anal.,} {\bf100} (1991), 18--24.

\bibitem{BGKM2022}
M. Bhakta, D. Ganguly, D. Karmakar, S. Mazumdar, Sharp quantitative stability of Struwe’s decomposition of the Poincar\'e-Sobolev inequalities on the hyperbolic space: Part I, preprint, arXiv:2211.14618.

\bibitem{BGRS2014}
S. Bobkov, N. Gozlan, C. Roberto, P. Samson, Bounds on the deficit in the logarithmic Sobolev inequality, {\it J. Funct. Anal.,} {\bf267} (2014), 4110--4138.

\bibitem{BL1985}
H. Brezis, E. Lieb, Sobolev inequalities with remainder terms. {\it J. Funct. Anal.,} {\bf62} (1985), 73--86.

\bibitem{C2017}
E. Carlen, Duality and stability for functional inequalities, {\it Ann. Fac. Sci. Toulouse Math.,} {\bf26} (2017), 319--350.

\bibitem{CF2013}
E. Carlen, A. Figalli, Stability for a GNS inequality and the log-HLS inequality, with application
to the critical mass Keller-Segel equation, {\it Duke Math. J.,} {\bf162} (2013), 579--625.

\bibitem{CFL2014}
E. Carlen, R. Frank, E. Lieb, Stability estimates for the lowest eigenvalue of a Schr\"odinger operator, {\it Geom. Funct. Anal.,} {\bf24} (2014), 63--84.

\bibitem{CKN1984}
L. Caffarelli, R. Kohn, L. Nirenberg, First order interpolation inequalities with weights.  {\it Compos. Math.,} {\bf53} (1984), 259--275.

\bibitem{CW2001}
F. Catrina, Z.-Q. Wang, On the Caffarelli-Kohn-Nirenberg inequalities: sharp constants, existence (and nonexistence), and symmetry of extremal functions. {\it Comm. Pure Appl. Math.,} {\bf54} (2001), 229--258.

\bibitem{CFLL2024}
C. Cazacu, J. Flynn, N. Lam, G. Lu, Caffarelli-Kohn-Nirenberg identities, inequalities and their stabilities, {\it J. Math. Pures Appl.,} {\bf182} (2024), 253--284.

\bibitem{CK2024}
H. Chen, S. Kim, Sharp quantitative stability of the Yamabe problem, preprint, 2024, arXiv:2404.13961.

\bibitem{CFW2013}
S. Chen, R. L. Frank, T. Weth, Remainder terms in the fractional Sobolev inequality, {\it Indiana Univ. Math. J.,} {\bf62} (2013), 1381--1397.

\bibitem{CLT2023}
L. Chen, G. Lu, H. Tang, Sharp stability of log-Sobolev and Moser-Onofri inequalities on the sphere, {\it J. Funct. Anal.,} {\bf285} (2023), 110022.

\bibitem{CC1993}
K. Chou, W. Chu, On the best constant for a weighted Sobolev-Hardy inequality.  {\it J. London Math. Soc.,} {\bf48} (1993), 137--151.

\bibitem{C2006}
A. Cianchi, A quantitative Sobolev inequality in BV, {\it J. Funct. Anal.,} {\bf237} (2006), 466--481.

\bibitem{CFMP2009}
A. Cianchi, N. Fusco, F. Maggi, A. Pratelli, The sharp Sobolev inequality in quantitative form, {\it J. Eur. Math. Soc.,} {\bf11} (2009), 1105--1139.

\bibitem{CFM2017}
G. Ciraolo, A. Figalli, F. Maggi, A quantitative analysis of metrics on $\bbr^N$ with
almost constant positive scalar curvature, with applications to fast diffusion flows.  {\it Int. Math. Res. Not.,} {\bf2018} (2017), 6780--6797.

\bibitem{DSW2024}
B. Deng, L. Sun, J. Wei, Optimal quantitative estimates of Struew's decomposition.  {\it Duke Math. J.} to appear, arXiv:2103.15360.

\bibitem{DSW2023}
B. Deng, L. Sun, J. Wei, Non-degeneracy and quantitative stability of half-harmonic maps from $\bbr$ to $\mathbb{S}$, {\it Adv. Math.,} {\bf420} (2023), 108979.

\bibitem{DGK2024}
N. De Nitti, F. Glaudo, and T. K\"onig. Non-degeneracy, stability and symmetry for the fractional Caffarelli-Kohn-Nirenberg inequality, preprint, 2024. arXiv: 2403.02303.

\bibitem{NK2023}
N. De Nitti and T. K\"onig. Stability with explicit constants of the critical points of the fractional Sobolev inequality and applications to fast diffusion. {\it J. Funct. Anal.,} {\bf285} (2023), 110093.

\bibitem{D2011}
J. Dolbeault, Sobolev and Hardy-Littlewood-Sobolev inequalities: duality and fast diffusion, {\it Math. Res. Lett.,} {\bf18} (2011), 1037--1050.

\bibitem{DE2022}
J. Dolbeault, M. J. Esteban, Hardy-Littlewood-Sobolev and related inequalities: stability, EMS Press, Berlin, 2022, 247–268.

\bibitem{DEFFL2022}
J. Dolbeault, M. J. Esteban, A. Figalli, R. L. Frank, M. Loss, Stability for the Sobolev inequality with explicit constants. Preprint, arXiv:2209.08651 [Math. AP].

\bibitem{DELT2009}
J. Dolbeault, M. J. Esteban, M. Loss, G. Tarantello, On the symmetry of extremals for the Caffarelli-Kohn-Nirenberg inequalities.  {\it Adv. Nonlinear Stud.,} {\bf9} (2009), 713--726.

\bibitem{DEL2012}
J. Dolbeault, M. J. Esteban, M. Loss, Symmetry of extremals of functional inequalities via spectral estimates for linear operators.  {\it J. Math. Phys.,} {\bf53} (2012), article 095204, 18 pp.

\bibitem{DEL2016}
J. Dolbeault, M. J. Esteban, M. Loss, Rigidity versus symmetry breaking via nonlinear flows on cylinders and Euclidean spaces, {\it Invent. math.,} {\bf206} (2016), 397--440.

\bibitem{DET2008}
J. Dolbeault, M. J. Esteban, G. Tarantello, The role of Onofri type inequalities in the symmetry properties of extremals for Caffarelli-Kohn-Nirenberg inequalities, in two space dimensions, {\it Ann. Sc. Norm. Super. Pisa Cl. Sci.,} {\bf5} (2008), 313-341.

\bibitem{DJ2014}
J. Dolbeault, G. Jankowiak, Sobolev and Hardy-Littlewood-Sobolev inequalities, {\it J. Differential Equations,} {\bf257} (2014), 1689--1720.

\bibitem{DT2016}
J. Dolbeault, G. Toscani, Stability results for logarithmic Sobolev and Gagliardo-Nirenberg inequalities, {\it Int. Math. Res. Not.,}  {\bf2017} (2016), 473--498.

\bibitem{EPT1989}
H. Egnell, F. Pacella, M. Tricarico, Some remarks on Sobolev inequalities. {\it Nonlinear Anal.,} {\bf13} (1989), 671--681.

\bibitem{E2015}
R. Eldan, A two-sided estimate for the Gaussian noise stability deficit, {\it Invent. Math.,} {\bf201} (2015), 561--624.

\bibitem{ENS2022}
M. Engelstein, R. Neumayer, L. Spolaor, Quantitative stability for minimizing Yamabe metrics, {\it Trans. Amer. Math. Soc. Ser. B,} {\bf9} (2022), 395--414.

\bibitem{FI2016}
M. Fathi, E. Indrei, M. Ledoux, Quantitative logarithmic Sobolev inequalities and stability estimates, {\it Discrete Contin. Dyn. Syst.,} {\bf36} (2016), 6835--6853.

\bibitem{F2013}
A. Figalli, Stability in geometric and functional inequalities, European Congress of Mathematics, 585--599, Eur. Math. Soc., Zurich, 2013.

\bibitem{FG2021}
A. Figalli, F. Glaudo, On the Sharp Stability of Critical Points of the Sobolev Inequality.  {\it Arch. Rational Mech. Anal.,} {\bf237} (2020), 201--258.

\bibitem{FI2013}
A. Figalli, E. Indrei, A sharp stability result for the relative isoperimetric inequality
inside convex cones, {\it J. Geom. Anal.,} {\bf23} (2013), 938--969.

\bibitem{FJ2015}
A. Figalli, D. Jerison, Quantitative stability for sumsets in $\bbr^N$, {\it J. Eur. Math. Soc.,}  {\bf17} (2015), 1079--1106.

\bibitem{FJ2017}
A. Figalli, D. Jerison, Quantitative stability for the Brunn-Minkowski inequality, {\it Adv.
Math.,} {\bf314} (2017), 1--47.

\bibitem{FMP2010}
A. Figalli, F. Maggi, A. Pratelli, A mass transportation approach to quantitative
isoperimetric inequalities, {\it Invent. Math.,} {\bf182} (2010), 167--211.

\bibitem{FMP2009}
A. Figalli, F. Maggi, A. Pratelli, A refined Brunn-Minkowski inequality for convex sets, {\it Ann. Inst. H. Poincar\'e Anal. Non Lin\'eaire,} {\bf26} (2009), 2511--2519.

\bibitem{FMP2013}
A. Figalli, F. Maggi, A. Pratelli, Sharp stability theorems for the anisotropic Sobolev and log-Sobolev inequalities on functions of bounded variation, {\it Adv. Math.,} {\bf242} (2013), 80--101.

\bibitem{FN2019}
A. Figalli, R. Neumayer, Gradient stability for the Sobolev inequality: the case $p\geq2$, {\it J. Eur. Math. Soc.,} {\bf21} (2019), 319--354.

\bibitem{FZ2020}
A. Figalli, Y. R.-Y. Zhang, Sharp gradient stability for the Sobolev inequality, {\it Duke Math. J.,} to appear, arXiv:2003.04037.

\bibitem{FS2003}
V. Felli, M. Schneider, Perturbation results of critical elliptic equations of Caffarelli-Kohn-Nirenberg type. {\it J. Differential Equations,} {\bf191} (2003), 121--142.

\bibitem{FIPR2017}
F. Feo, E. Indrei, M. R. Posteraro, C. Roberto, Some Remarks on the Stability of the Log-Sobolev Inequality for the Gaussian Measure, {\it Potential Anal.,} {\bf47} (2017), 37--52.

\bibitem{F2023}
R. L. Frank, The sharp Sobolev inequality and its stability: an introduction, preprint, arXiv2304.03115 [math.AP].

\bibitem{F2022}
R. L. Frank, Degenerate stability of some Sobolev inequalities, {\it Ann. Inst. H. Poincar\'e Anal. Non Lin\'eaire,} {\bf39} (2022), 1459--1484.

\bibitem{FP2024}
R. L. Frank, J. W. Peteranderl, Degenerate stability of the Caffarelli-Kohn-Nirenberg inequality along the Felli-Schneider curve,  {\it Calc. Var.,} {\bf63} (2024), Paper No. 44.

\bibitem{F1989}
B. Fuglede, Stability in the isoperimetric problem for convex or nearly spherical domains in $\bbr^N$, {\it Trans. Amer. Math. Soc.,} {\bf314} (1989) 619--638.

\bibitem{F2015}
N. Fusco, The quantitative isoperimetric inequality and related topics, {\it Bull. Math. Sci.,} {\bf5} (2015), 517--607.

\bibitem{FMP2007}
N. Fusco, F. Maggi, A. Pratelli, The sharp quantitative Sobolev inequality for functions of bounded variation, {\it J. Funct. Anal.,} {\bf244} (2007), 315--341.

\bibitem{IM2014}
E. Indrei, D. Marcon, A quantitative log-Sobolev inequality for a two parameter
family of functions, {\it Int. Math. Res. Not.,}{\bf20} (2014), 5563--5580.

\bibitem{IK2021}
E. Indrei, D. Kim, Deficit estimates for the logarithmic Sobolev inequality, {\it Differ. Integral Equ.,} {\bf34}, (2021), 437--466.

\bibitem{J1992}
R. Hall, A quantitative isoperimetric inequality in n-dimensional space, {\it J. Reine Angew. Math.,} {\bf428} (1992), 161--176.

\bibitem{K2022}
T. K\"onig, On the sharp constant in the Bianchi-Egnell stability inequality, {\it Bull. Lond. Math. Soc.,} {\bf55} (2023), 2070--2075.

\bibitem{K2023}
T. K\"onig, Stability for the Sobolev inequality: Existence of a minimizer,  {\it J. Eur. Math. Soc.}, to appear, arXiv2211.14185.

\bibitem{K2023-1}
T. K\"onig, An exceptional property of the one-dimensional Bianchi-Egnell inequality, {\it Calc. Var.,} {\bf63} (2024), Paper No. 123.

\bibitem{L1983}
E. Lieb, Sharp constants in the Hardy-Littlewood-Sobolev and related inequalities. {\it Ann. of Math. (2),} {\bf118} (1983), 349--374.

\bibitem{LWW2024}
T. Li, J. Wei, Y. Wu, Infinitely many nonradial positive solutions for multi-species nonlinear Schrodinger systems in $\mathbb{R}^{N}$, {\it J. Differential Equations,} {\bf381} (2024), 340--396.

\bibitem{LW2004}
C.-S. Lin, Z.-Q. Wang, Symmetry of extremal functions for the Caffarrelli-Kohn-Nirenberg inequalities. {\it Proc. Amer. Math. Soc.,} {\bf132} (2004), 1685--1691.

\bibitem{L2005}
A. Loiudice, Improved Sobolev inequalities on the Heisenberg group, {\it Nonlinear Anal.,} {\bf62} (2005), 953--962.

\bibitem{LW2000}
G. Lu, J. Wei, On a Sobolev inequality with remainder term, {\it Proc. Amer. Math. Soc.,} {\bf128} (2000), 75--84.

\bibitem{M2008}
F. Maggi, Some methods for studying stability in isoperimetric type problems, {\it Bull. Amer. Math. Soc.,} {\bf45} (2008), 367--408.

\bibitem{N2020}
R. Neumayer, A note on strong-form stability for the Sobolev inequality, {\it Calc. Var.,} {\bf59} (2020), Paper No. 25.

\bibitem{N2019}
V. Nguyen, The sharp Gagliardo-Nirenberg-Sobolev inequality in quantitative form, {\it J. Funct. Anal.,} {\bf277} (2019), 2179--2208.

\bibitem{NV2024}
F. Nobili, I. Y. Violo, Stability of Sobolev inequalities on Riemannian manifolds with Ricci curvature lower bounds, {\it Adv. Math.,} {\bf440} (2024), 109521.

\bibitem{R2014}
B. Ruffini, Stability theorems for Gagliardo-Nirenberg-Sobolev inequalities: a reduction principle to the radial case, {\it Rev. Mat. Complut.,} {\bf27} (2014), 509--539.

\bibitem{S2007}
M. Schneider, A priori estimates for the scalar curvature equation on $\mathbb{S}^3$,
{\it Calc. Var.,} {\bf29} (2007), 521--560.

\bibitem{S2016}
F. Seuffert, An extension of the Bianchi-Egnell stability estimate to Bakry, Gentil, and Ledoux's generalization of the Sobolev inequality to continuous dimensions, {\it J. Funct. Anal.,} {\bf273} (2017), 3094--3149.

\bibitem{S1984}
M. Struwe, A global compactness result for elliptic boundary value problems involving limiting nonlinearities.  {\it Math. Z.,} {\bf187} (1984), 511--517.

\bibitem{T1976}
G. Talenti, Best constant in Sobolev inequality.  {\it Ann. Mat. Pura Appl. (4),} {\bf110} (1976), 353--372.

\bibitem{WW2003}
Z.-Q. Wang, M. Willem, Caffarelli-Kohn-Nirenberg inequalities with remainder terms,
{\it J. Funct. Anal.,} {\bf203} (2003), 550--568.

\bibitem{WW2022}
J. Wei, Y. Wu, On the stability of the Caffarelli-Kohn-Nirenberg inequality, {\it Math. Ann.,} {\bf384} (2022), 1509--1546.

\bibitem{WW2024}
J. Wei, Y. Wu, Stability of the Caffarelli-Kohn-Nirenberg inequality: the existence of minimizers, preprint, 2023, arXiv:2308.04667.

\bibitem{W78}
F. B. Weissler, Logarithmic Sobolev inequalities for the heat-diffusion semigroup, {\it Trans. Amer. Math. Soc.,} {\bf237} (1978), 255--269.

\bibitem{ZZ2024}
Y. Zhou, W. Zou, Degenerate stability of critical points of the Caffarelli-Kohn-Nirenberg inequality along the Felli-Schneider curve, preprint, 2024, arXiv:2407.10849.
\end{thebibliography}
\end{document}